\documentclass[a4paper,10pt,reqno]{amsart}
\usepackage[latin1]{inputenc}
\usepackage{dsfont}
\usepackage{amsxtra}
\usepackage{amsmath}
\usepackage{amscd}
\usepackage{amssymb}
\usepackage{amsfonts}
\usepackage[all]{xy}
\usepackage{mathrsfs}
\usepackage{amsthm} 
\usepackage{color}

\theoremstyle{plain}
\newtheorem{thm}{Theorem}[section]
\newtheorem{lemma}[thm]{Lemma}
\newtheorem{claim}[thm]{Claim}
\newtheorem{prop}[thm]{Proposition}

\newtheorem{cor}[thm]{Corollary}

\newtheorem{df-prop}[thm]{Definition-Proposition}

\theoremstyle{definition}
\newtheorem{df}[thm]{Definition}

\theoremstyle{remark}
\newtheorem{rk}[thm]{Remark}
\newtheorem{ex}[thm]{Example}

\numberwithin{equation}{section}

\def\coker{\operatorname{coker}\nolimits}

\def\minj{\operatorname{\!-inj}\nolimits}
\def\mmod{\operatorname{\!-mod}\nolimits}
\def\mproj{\operatorname{\!-proj}\nolimits}
\usepackage{mathabx}

\def\bbC{\mathbb{C}}

\def\bbL{\mathbb{L}}

\def\bbN{\mathbb{N}}

\def\bbP{\mathbb{P}}
\def\bbQ{\mathbb{Q}}
\def\bbR{\mathbb{R}}

\def\bbZ{\mathbb{Z}}

\def\scrA{\mathscr{A}}

\def\scrC{\mathscr{C}}
\def\scrD{\mathscr{D}}
\def\scrE{\mathscr{E}}

\def\scrG{\mathscr{G}}

\def\scrI{\mathscr{I}}

\def\scrL{\mathscr{L}}

\def\scrO{\mathscr{O}}
\def\scrP{\mathscr{P}}
\def\scrQ{\mathscr{Q}}
\def\scrR{\mathscr{R}}
\def\scrS{\mathscr{S}}

\def\scrV{\mathscr{V}}

\def\scrX{\mathscr{X}}

\def\fraks{\mathfrak{s}}
\def\frakl{\mathfrak{l}}

\def\frakL{\mathfrak{L}}
\def\frakM{\mathfrak{M}}

\def\frakP{\mathfrak{P}}

\def\frakS{\mathfrak{S}}

\def\calA{\mathcal{A}}

\def\calC{\mathcal{C}}

\def\calE{\mathcal{E}}

\def\calG{\mathcal{G}}
\def\calH{\mathcal{H}}
\def\calI{\mathcal{I}}

\def\calO{\mathcal{O}}

\def\calQ{\mathcal{Q}}
\def\calR{\mathcal{R}}
\def\calS{\mathcal{S}}

\def\frakL{\mathfrak{L}}

\def\frakb{\mathfrak{b}}

\def\frakg{\mathfrak{g}}
\def\frakh{\mathfrak{h}}
\def\frakl{\mathfrak{l}}
\def\frakm{\mathfrak{m}}
\def\frakn{\mathfrak{n}}
\def\frakp{\mathfrak{p}}

\def\fraku{\mathfrak{u}}

\def\frakt{\mathfrak{t}}

\def\fraksl{\mathfrak{sl}}
\def\frakgl{\mathfrak{gl}}

\def\bfone{\mathbf{1}}
\def\bfa{\mathbf{a}}
\def\bfb{\mathbf{b}}
\def\bfc{\mathbf{c}}
\def\bfg{\mathbf{g}}

\def\bfi{\mathbf{i}}
\def\bfj{\mathbf{j}}

\def\bfm{\mathbf{m}}

\def\bfp{\mathbf{p}}

\def\bft{\mathbf{t}}

\def\bfw{\mathbf{w}}

\def\bfL{\mathbf{L}}
\def\bfR{\mathbf{R}}
\def\bfS{\mathbf{S}}

\def\bfH{\mathbf{H}}

\def\bfM{\mathbf{M}}

\def\bfO{\mathbf{O}}

\def\bft{\mathbf{t}}
\def\bfA{\mathbf{A}}

\def\bfF{\mathbf{F}}
\def\bfP{\mathbf{P}}
\def\bfT{\mathbf{T}}
\def\bfV{\mathbf{V}}

\def\lam{\lambda}
\def\Lam{\Lambda}
\def\al{\alpha}
\def\geqs{\geqslant}
\def\leqs{\leqslant}
\def\det{{{\rm{det}}}}

\def\llangle{{\langle\!\langle}}
\def\rrangle{{\rangle\!\rangle}}

\def\MMod{{{\mbox{-}}\!\operatorname{Mod}}}
\def\ie{{\em i.e.}}

\def\iso{{\buildrel\sim\over\to}}
\def\cas{{\operatorname{cas}\nolimits}}
\def\diag{{\operatorname{diag}\nolimits}}
\def\afcas{{\mathbf{cas}}}

\def\Bbbk{{\operatorname{k}\nolimits}}
\def\BGG{\operatorname{\leqslant_{\tiny \rm{b}}}}
\def\lub{{{}^\beta}}

\def\Frac{\operatorname{Frac}\nolimits}

\def\tilt{{\operatorname{tilt}\nolimits}}

\def\Hom{\operatorname{Hom}\nolimits}

\def\End{\operatorname{End}\nolimits}
\def\Mat{\operatorname{Mat}\nolimits}

\def\min{{\operatorname{min}\nolimits}}

\def\wt{\operatorname{wt}\nolimits}

\def\ct{{\operatorname{cont}\nolimits}}

\def\cone{{\operatorname{cone}\nolimits}}

\def\op{{{\operatorname{{op}}\nolimits}}}

\def\Ext{\operatorname{Ext}\nolimits}
\def\Spec{\operatorname{Spec}\nolimits}
\def\Res{\operatorname{Res}\nolimits}
\def\Ind{\operatorname{Ind}\nolimits}

\def\Indc{\scrI\!{nd}}

\def\Coker{\operatorname{Coker}\nolimits}
\def\Ker{\operatorname{Ker}\nolimits}

\def\proj{{\operatorname{proj}\nolimits}}
\def\inj{{\operatorname{inj}\nolimits}}

\def\IM{\operatorname{IM}\nolimits}

\def\KZ{\operatorname{KZ}\nolimits}

\def\Im{\operatorname{Im}\nolimits}

\def\Id{\operatorname{Id}\nolimits}

\def\id{\operatorname{id}\nolimits}

\def\Irr{\operatorname{Irr}\nolimits}

\def\res{\operatorname{res}\nolimits}

\def\pro{{\operatorname{{{\lim\limits_{\leftarrow}}}}}}
\def\ind{{\operatorname{{{\lim\limits_{\rightarrow}}}}}}
\def\t{{\tau}}

\def\bfwt{{\bfw\bft}}

\def\OInd{{}^\calO\!\mathrm{\Ind}}
\def\ORes{{}^\calO\!\mathrm{\Res}}
\def\HInd{{}^\bfH\!\mathrm{\Ind}}
\def\HRes{{}^\bfH\!\mathrm{\Res}}
\def\Otimes{{\textstyle{\bigotimes}}}

\author{R. Rouquier, P. Shan, M. Varagnolo, E. Vasserot}
\email{rouquier@math.ucla.edu}
\email{peng.shan@unicaen.fr, pengshan@mit.edu}
\email{michela.varagnolo@u-cergy.fr}
\email{vasserot@math.jussieu.fr}

\title
[]
{categorifications
and cyclotomic rational double affine Hecke algebras}

\begin{document}
\begin{abstract}
Varagnolo and Vasserot conjectured an equivalence between the category $\calO$
for CRDAHA's and a subcategory of
an affine parabolic category O of type $A$. We prove this conjecture. 
As applications, we  prove a
conjecture of Rouquier on the dimension of simple modules of CRDAHA's 
and a conjecture of Chuang-Miyachi on the Koszul duality for the category $\calO$ of CRDAHA's.
\end{abstract}
\thanks{This research was partially supported by the ANR grant number ANR-10-BLAN-0110. The first author was partially supported by the NSF grant DMS-1161999.}

\maketitle

\setcounter{tocdepth}{3}

\tableofcontents

\section{Introduction}

Rational Double affine Hecke algebras (RDAHA for short) have been introduced by 
Etingof and Ginzburg in 2002.
They are associative algebras associated with a complex reflection group $W$ and  a parameter $c$.
Their representation theory is
similar to the representation theory of semi-simple Lie algebras. In particular, they admit a category 
$\calO$ which is 
analogous to the BGG category O. This category is highest weight with the standard modules labeled by irreducible representations of $W$. 
Representations in $\calO$ are infinite dimensional in general, but they admit a character.
An important question is to determine the characters of simple modules.

One of the most important family of RDAHA's is the cyclotomic one (CRDAHA for short),
where $W=G(\ell,1,n)$ is the wreath product
of $S_n$ and $\bbZ/\ell\bbZ$. One reason is that the representation theory of CRDAHA's
is closely related to the representation theory of Ariki-Koike algebras, and that the latter are important
in group theory. Another reason is that the category $\calO$ 
of CRDAHA's is closely related to the representation  
theory of affine Kac-Moody algebras, see e.g. \cite{E}, \cite{SV}, \cite{VV}.
A third reason, is that this category has a very rich structure called a categorical action of an affine Kac-Moody 
algebra. This action on $\calO$ was constructed previously in \cite{S}.
Such structures have been introduced recently in representation theory and have already had remarkable 
applications, see e.g., \cite{CR}, \cite{KhL}, \cite{R2}.

The structure of $\calO$ depends heavily on the parameter $c$. 
For generic  values of $c$ the category is semi-simple. The most non semi-simple case (which is also the most 
complicated one) occurs when $c$ takes a particular form of rational numbers, see \eqref{eq:paarameters}. 
For these parameters Rouquier made a conjecture to determine the characters of simple modules in 
$\calO$ \cite{R1}.
Roughly speaking, this conjecture says that the Jordan-H\"older multiplicities of the standard 
modules in $\calO$ are given by some parabolic Kazhdan-Lusztig polynomials.
This conjecture was known to be true in the particular case $\ell=1$ \cite{R1}.
Motivated by this conjecture, Varagnolo-Vasserot introduced in \cite{VV} a new category $\bfA$ which is a 
subcategory of an affine parabolic category $\bfO$ at a negative level and should be viewed as an affine and 
higher level analogue of the category of polynomial
representations of $GL_N$. 
They conjectured that there should be an equivalence of highest weight categories between 
$\calO$ and $\bfA.$

In this paper we prove Varagnolo-Vasserot's conjecture (Theorem \ref{thm:main3}). A first consequence is a proof of 
Rouquier's conjecture (Theorem \ref{thm:rouquier}).
A second remarkable application is a proof that the category $\calO$ is Koszul (Theorem \ref{thm:chuang-miyachi}), 
yielding a proof of a conjecture of Chuang-Miyachi \cite{CM}, because
the affine parabolic category $\bfO$ is Koszul by  \cite{SVV}.

Our proof is based on an extension of Rouquier's theory of highest weight covers developed in \cite{R1}. 
Basically, it says that two highest weight covers of the same algebra are equivalent as highest weight 
categories if they satisfy a so called \emph{$1$-faithful} condition and if the highest weight orders on both covers are compatible. 
Here, we'll prove that
such an equivalence holds in a situation where the covers are only $0$-faithful (see Proposition \ref{prop:key}).

The category $\calO$ is a highest weight cover over 
the module category $\calH$ of the Ariki-Koike algebra via the KZ functor introduced in \cite{GGOR}.
It is a $0$-faithful cover and if the parameters of the  RDAHA satisfy some technical 
condition, then it is even $1$-faithful.
A similar functor $\Phi:\bfA\to\calH$ was introduced in \cite{VV} using the Kazhdan-Lusztig fusion product on 
the affine category $\bfO$ at a negative level.
A previous work of Dunkl and Griffeth \cite{DG} allows to show without much difficulty that there is a highest 
weight order on $\calO$ which refines the linkage order on $\bfA$. A difficult part of the proof consists of 
showing that the functor $\Phi$ is indeed a cover, meaning that it is an exact quotient functor, and that it has the same faithfulness properties as the KZ functor. Once this is done, the equivalence between $\calO$ and $\bfA$ follows directly from the unicity of $1$-faithful covers if the technical condition on parameters mentioned above is satisfied.
To prove the equivalence without this condition, we 
need to replace KZ and $\Phi$ by some other covers, see the end of the introduction for more details on this.

A key ingredient in our proof is a deformation argument. More precisely, the highest weight
categories $\bfA$, $\bfO$ admit deformed versions over a regular local ring $R$ of dimension 2. 
Some technical results prove that the 
Kazhdan-Lusztig tensor product can also be deformed properly, which allows us to define the deformed version of $\Phi$. Next, a theorem of 
Fiebig asserts that the structure of the category $\bfO$ of a Kac-Moody algebra only depends on the associated Coxeter system \cite{F2}. 
In particular, the localization of $\bfA$
at a height one prime ideal $\frakp\subset R$ can be described in simpler terms. Two cases appear, either $\frakp$ is subgeneric or generic. In the first case, 
considered in \S \ref{sec:red1}, the category $\bfA$ reduces to an analog subcategory $A$ inside the parabolic category $\scrO$ of $\mathfrak{gl}_N$ 
with 2 blocs. The latter is closely 
related to the higher level Schur-Weyl duality studied by Brundan and Kleshchev in \cite{BK1}. In the second case,
considered in \S \ref{sec:red2}, the category $\bfA$ reduces to the corresponding category for $\ell=1$, which is precisely the Kazhdan-Lusztig category 
associated with affine Lie algebras at negative levels.
Finally, we show that to prove the desired properties of the functor $\Phi$ it is enough to check them for the localization of $\Phi$ at each 
height one prime ideal $\frakp$ and this permits to conclude.

Now, let us say a few words concerning the organization of this paper.

Section 2 contains some basic facts on highest weight categories and some developments on the theory of highest weight covers in \cite{R1}.

Section 3  is a reminder on Hecke algebras, q-Schur algebras and categorifications. 

Section 4 contains basic facts on the parabolic category $\scrO$ of $\frakg\frakl_N$ and the subcategory $A\subset\scrO$ introduced in \cite{BK1}. 
The results in \cite{BK1} are not enough for us since we need to consider
a deformed category $A$ with \emph{integral} deformation parameters. The new material is gathered in \S \ref{sec:2blocs}.

In Section 5 we consider the affine parabolic category $\bfO$ (at a negative level).
The monoidal structure on $\bfO$ is defined later in \S \ref{sec:KL}.
Using this monoidal structure we construct a categorical action on $\bfO$
in \S \ref{sec:categorification}. Then, we define the subcategory $\bfA\subset\bfO$
in \S \ref{sec:BKKB}.
The rest of the section is devoted to the deformation argument and the proof that $\bfA$ is a highest weight cover 
of the module category of a cyclotomic Hecke algebra satisfying some faithfulness conditions.

In Section 6 we first give a reminder on the category $\calO$ of
CRDAHA's, following \cite{GGOR}, \cite{R1}.
Then, we prove our main theorems in \S \ref{sec:comparison1}, \ref{sec:comparison2}
using the results from \S \ref{sec:refined}.
This yields a proof of Varagnolo-Vasserot's conjecture \cite{VV}.
For the clarity of the exposition we separate the cases of rational and irrational levels,
although both proofs are very similar.

In Section 7 we give some applications of our main theorem, including proofs for Rouquier's conjecture 
and Chuang-Miyachi's conjecture.

Section 8 is a reminder on the Kazhdan-Lusztig tensor product on the affine category 
$\bfO$ at a negative level.
We generalize their construction in order to get a monoidal structure on arbitrary parabolic categories,
deformed over an analytic 2-dimensional regular local ring.
Several technical results concerning the Kazhdan-Lusztig tensor product are postponed to the appendix.

To finish, let us explain the relation of this work with other recent works.

The case of irrational level (proved in Theorem \ref{thm:comparisonI})
was conjectured in \cite[rem.~8.10$(b)$]{VV}, as a degenerate analogue of the main conjecture
\cite[conj.~8.8]{VV}. There, it was mentioned that it should follow from \cite[thm.~C]{BK1}.
In the dominant case, this has been proved recently \cite[thm.~6.9.1]{GL}.

While we were writing this paper I. Losev made public several papers with some overlaps with ours.
In \cite{L1}, \cite{L2} he developed a general formalism of categorical actions on highest weight categories.
Then, he used this formalism in \cite{L3} to  prove that the category $\bfA$ is equipped 
with a categorical action, induced by the categorical action on $\bfO$ introduced in \cite{VV}
(using the Kazhdan-Lusztig fusion product).
The categorical action on $\bfA$ gives an independent proof of Theorem \ref{thm:isom} $(a)$, $(b)$.
Finally, he proposed a combinatorial approach to prove that $\bfA$ is a 1-faithful highest weight cover of the cyclotomic Hecke algebra
under some technical condition on the parameters of the CRDAHA. 

A first version of our paper was announced in July 2012 and has been presented at several occasions since then.
There, we proved this 1-faithfulness for $\bfA$ (and the Varagnolo-Vasserot's conjecture) under the same condition on the parameters by a deformation argument similar, but weaker, to the one used in the present paper. 

The proof which we give in this article avoids this technical condition on the parameters.
It uses an idea introduced later, in \cite{L3}. 
There, I. Losev replaces the highest weight cover $\bfA$ of  the cyclotomic Hecke algebra $\bfH$ by
a highest weight cover, by $\bfA$, of a bigger algebra than $\bfH$, which has better properties.

After this paper was written, B. Webster sent us a copy of a preliminary
version of his recent preprint \cite{Wb} proposing another proof of Rouquier's
conjecture which does not use the affine parabolic category $\bfO$.

Note that our construction does not use any categorical action on $\bfA$.
It only uses representation theoretic arguments.
However, since Theorem \ref{thm:main3} yields an equivalence between $\bfA$ and $\calO$, 
we can recover a categorical action on $\bfA$ from our theorem and
the main result of \cite{S}. This is explained in \S \ref{sec:7.4}.

\vspace{3mm}

\section{Highest weight categories}
\label{diag1}
In the paper the symbol $R$ will always denote a noetherian commutative domain (with 1).
We denote by $K$ its fraction field. When $R$ is a local ring, we denote
by $\Bbbk$ its residue field and  by $\frakm$ its maximal ideal.

\subsection{Rings and modules}
For any $R$-module $M,$ let $M^*=\Hom_R(M,R)$ denote the dual module.
An \emph{$S$-point} of $R$ is a morphism
$\chi:R\to S$ of commutative rings with 1.
If $\chi$ is a morphism of local rings, we say that it is a \emph{local $S$-point}.
We write $SM=M(\chi)=M\otimes_RS$. 
If $\phi$ is a $R$-module homomorphism, we abbreviate also
$S\phi=\phi\otimes_RS$.

Let $\frakP$, $\frakM$ be the spectrum and the maximal spectrum of $R$.
Let $\frakP_1\subset\frakP$ be the subset of height 1 prime ideals.
For each $\frakp\in\frakP,$ let $R_\frakp$ denote the localization of $R$ at $\frakp$.
The maximal ideal of $R_\frakp$ is $\frakm_\frakp=R_\frakp\,\frakp$ and
its residue field is $\Bbbk_\frakp=\Frac(R/\frakp)$.

A \emph{closed $\Bbbk$-point} of $R$ is a quotient 
$R\to R/\frakm=\Bbbk$ where $\frakm\in\frakM$.
To unburden the notation we may write $\Bbbk\in\frakM$.

A \emph{finite projective} $R$-algebra is an $R$-algebra which is finitely generated and projective 
as an $R$-module.

We will mainly be interested in the case where $R$ is a local ring.
In this case, any projective module is free by Kaplansky's theorem.
Therefore, we'll use indifferently the words free or projective.

\vspace{3mm}

\subsection{Categories}\label{ss:notation}
Given $A$ a ring, we denote by $A^\op$ the opposite ring in which the order of
multiplication is reversed. Given $\scrC$ is a category, let $\scrC^\op$ be the opposite
category.

An {\it $R$-category} $\scrC$ is an additive
category enriched over the tensor category of $R$-modules.
All the functors $F$ on
$\scrC$ are assumed to be $R$-linear.
We denote the identity element in the endomorphism ring $\End(F)$ again by $F$.
We denote the identity functor on $\scrC$ by $1_\scrC$.
We say that $\scrC$ is {\it Hom-finite} if the Hom spaces are finitely
generated over $R$. 
If the category $\scrC$ is abelian or exact, let $K_0(\scrC)$ be the Grothendieck group and write
$[\scrC]=K_0(\scrC)\otimes_\bbZ\bbC$.
If $\scrC$ is additive, it is an exact category with split exact sequences and $[\scrC]$
is the complexified split Grothendieck group.
Let $[M]$ denote the class of an object $M$ of $\scrC.$

\smallskip

Assume now that $\scrC$ is abelian and has enough projectives.
We say that an object $M\in\scrC$
 is projective over $R$ if $\Hom_\scrC(P,M)$ is a projective
$R$-module for all projective objects $P$ of $\scrC$.
The full subcategory $\scrC\cap R\mproj$ of objects of $\scrC$ projective over $R$ is
an exact subcategory and the canonical functor
$D^b(\scrC\cap R\mproj)\to D^b(\scrC)$ is fully faithful.
An object $X\in\scrC$ 
which is projective over $R$ is \emph{relatively $R$-injective} if  $\Ext^1_\scrC(Y,X)=0$
for all objects $Y$ of $\scrC$ that are projective over $R$.

If $\scrC$ is the category $A\mmod$ of finitely generated (left) modules over a finite
projective $R$-algebra
$A$, then an object $X\in\scrC$ is projective over $R$ if and only if it is
projective as an $R$-module. It is
relatively $R$-injective if in addition
the dual $X^*=\Hom_R(X,R)$ is a projective right $A$-module.
If there is no risk of confusion we will say injective instead of relatively
$R$-injective. We put $\scrC^*=A^\op\mmod$. The functor
$\Hom_R(\bullet,R):\scrC^\op\to\scrC^*$ restricts to an equivalence of exact
categories $\scrC^\op\cap R\mproj\iso\scrC^*\cap R\mproj$.

\smallskip
We denote by
$\Irr(\scrC)$
the sets of isomorphism classes of simple objects of $\scrC$.
Let $\scrC^\proj,\scrC^\inj\subset\scrC$ be the full subcategories of projective
and of relatively $R$-injective objects.
If $\scrC=A\mmod$, we abbreviate $\Irr(A)=\Irr(\scrC)$,
$A\mproj=\scrC^\proj$ and $A\minj=\scrC^\inj$.

Given an $S$-point $R\to S$ and $\scrC=A\mmod$, we can form the $S$-category
$S\scrC=SA\mmod$. Given another $R$-category $\scrC'$ as above and an exact ($R$-linear)
functor
$F:\scrC\to\scrC'$, then $F$ is represented by a projective object $P\in\scrC$. 
We set $SF=\Hom_{S\scrC}(SP,\bullet):S\scrC\to S\scrC'$.

Let $\scrA$ be a \emph{Serre} subcategory of $\scrC$.
The canonical embedding functor $h:\scrA\to\scrC$
has a left adjoint $h^\ast$ which takes an object $M$ in $\scrC$ to its maximal quotient in $\scrC$ which 
belongs to $\scrA$. It admits also a
right adjoint $h^!$ which takes an object $M$ in $\scrC$ to its maximal subobject in $\scrC$ which belongs to 
$\scrA$. The functor
$h^\ast$ is right exact, while $h^!$ is left exact. The functor $h$ is fully faithful. 
Hence the adjunction morphisms
$h^\ast h\to 1_\scrA$ and $1_\scrA\to h^!h$
are isomorphisms.
By definition, the adjunction morphisms
$1_\scrC\to h h^\ast$ and $hh^!\to 1_\scrC$ are respectively an epimorphism and a monomorphism. 

Here, and in the rest of the paper, we use the
following notation : a composition of functors $E$ and $F$ is written as $EF$
while a composition of morphisms of
functors $\psi$ and $\phi$ is written as $\psi\circ\phi$.

\vspace{3mm}

\subsection{Highest weight categories over local rings}
Let $R$ be  a commutative local ring.
We recall and complete some basic facts about highest
weight categories over $R$ (cf \cite[\S 4.1]{R1} and \cite{CPS},
 \cite[\S 2]{DuSc}).

Let $\scrC$ be an abelian $R$-category which is equivalent to the
category $A\mmod$ of finitely generated modules over a finite projective $R$-algebra $A$.

The category $\scrC$ is a \textit{highest weight $R$-category} if it is equipped
with a poset of isomorphism classes of objects
$(\Delta(\scrC),\leqslant)$ called the \emph{standard objects}
satisfying the following conditions:
\begin{itemize}
\item the objects of $\Delta(\scrC)$ are projective over $R$
\item given $M\in\scrC$ such that $\Hom_{\scrC}(D,M)=0$ for all $D\in\Delta(\scrC)$,
we have $M=0$
\item given $D\in\Delta(\scrC)$, there is $P\in\scrC^\proj$ and a surjection 
$f:P\twoheadrightarrow D$ such that $\ker f$ has a (finite) filtration whose successive
quotients are objects $D'\in\Delta$ with $D'>D$.
\item given $D\in\Delta$, we have $\End_{\scrC}(D)=R$
\item given $D_1,D_2\in\Delta$ with $\Hom_{\scrC}(D_1,D_2)\not=0$, we have $D_1\le D_2$.
\end{itemize}

The partial order $\leqslant$ is called the \emph{highest weight order} of $\scrC$.
We write $\Delta(\scrC)=\{\Delta(\lambda)\}_{\lambda\in\Lambda}$, for
$\Lambda$ an indexing poset. Note that if $\le'$ is an order coarser than $\le$ (ie,
$\lambda\le\mu$ implies $\lambda\le'\mu$), then $\scrC$ is also a highest weight category
relative to the order $\le'$.

An {\it equivalence of highest weight categories}
$\scrC'\xrightarrow{\sim}\scrC$ is an equivalence which induces a bijection
$\Delta(\scrC')\xrightarrow{\sim}\Delta(\scrC)$.
A \emph{highest weight subcategory} is a full Serre subcategory
$\scrC'\subset\scrC$ that is a highest weight category with poset
$\Delta(\scrC')$ an ideal of $\Delta(\scrC)$ (ie, if $D'\in\Delta(\scrC')$,
$D\in\Delta(\scrC)$ and $D'<D$, then $D'\in\Delta(\scrC')$).

Highest weight categories come with associated projective, injective, tilting and
costandard objects, as described in the next proposition.

\begin{prop}
Let $\scrC$ be a highest weight $R$-category. Given $\lambda\in\Lambda$, there
are indecomposable objects 
 $P(\lambda)\in\scrC^\proj$, $I(\lambda)\in\scrC^\inj$,
 $T(\lambda)\in\scrC$ and $\nabla(\lambda)\in\scrC$
(the projective, injective,
tilting and costandard objects associated with $\lambda$),
unique up to isomorphism such that
\begin{itemize}
\item[{\small($\nabla$)}] $\Hom_{\scrC}(\Delta(\mu),\nabla(\lambda))\simeq \delta_{\lambda\mu}R$
and $\Ext^1_{\scrC}(\Delta(\mu),\nabla(\lambda))=0$ for all $\mu\in\Lambda$,
\item[{\small($P$)}] there is a surjection $f:P(\lambda)\twoheadrightarrow\Delta(\lambda)$ 
such that $\ker f$ has a filtration whose successive
quotients are $\Delta(\mu)$'s with $\mu>\lambda$,
\item[{\small($I$)}] there is an injection $f:\nabla(\lambda)\hookrightarrow I(\lambda)$
such that $\mathrm{coker} f$ has a filtration whose successive
quotients are $\nabla(\mu)$'s with $\mu>\lambda$,
\item[{\small($T$)}] there is an injection $f:\Delta(\lambda)\hookrightarrow T(\lambda)$
and a surjection $g:T(\lambda)\twoheadrightarrow \nabla(\lambda)$ such that
$\mathrm{coker} f$ (resp. $\ker g$) has a 
filtration whose successive
quotients are $\Delta(\mu)$'s (resp. $\nabla(\mu)$'s) with $\mu<\lambda$.
\end{itemize}

We have the following properties of those objects.

\medskip
\noindent
$\bullet\ $ $\nabla(\lambda)$, $\Delta(\lambda)$, $P(\lambda)$, $I(\lambda)$
and $T(\lambda)$ are projective over $R$.

\smallskip
\noindent
$\bullet\ $
Given a commutative local $R$-algebra $S$, then $S\scrC$ is a highest weight 
$S$-category on the poset $\Lambda$ with standard objects
$S\Delta(\lambda)$
and costandard objects $S\nabla(\lambda)$. If $R\to S$ is a local $S$-point,
then the projective, injective and tilting objects associated
with $\lambda$ are $SP(\lambda)$, $SI(\lambda)$ and $ST(\lambda)$.

\smallskip
\noindent
$\bullet\ $
$\scrC^*$ is a highest weight $R$-category on the poset
$\Lambda$ with standard objects $\Delta^*(\lambda)=\nabla(\lambda)^*$ and
with $P^*(\lambda)=I(\lambda)^*$, $I^*(\lambda)=P(\lambda)^*$,
$\nabla^*(\lambda)=\Delta(\lambda)^*$ and $T^*(\lambda)=T(\lambda)^*$.
\end{prop}

\begin{proof}
Note that the statements of the proposition are classical when $R$ is a field.

The existence of the objects $\nabla(\lambda)$ giving $\scrC^\op$ the
structure of a highest weight category and satisfying the $\Hom$ and
$\Ext$ conditions is given by \cite[Proposition 4.19]{R1}.
The unicity follows from Lemma \ref{lem:B7} below. The description of the projective,
tilting and injective objects of $\scrC^*$ is clear.

It is shown in \cite[Proposition 4.14]{R1} that $S\scrC$ is a highest
weight category with $\Delta(S\scrC)=S\Delta(\scrC)$. We denote by
$P_S(\lambda)$, $I_S(\lambda)$, etc. the projective, injective, etc. of
$S\scrC$ associated with $\lambda$.

\smallskip
The existence of $P(\lambda)$ is granted in the definition of highest
weight categories.
We show by descending induction on $\lambda$ that 
$\Bbbk P(\lambda)\simeq P_\Bbbk(\lambda)$. This is clear if $\lambda$ is maximal,
for then $P(\lambda)=\Delta(\lambda)$. We have $\Bbbk P(\lambda)=P_\Bbbk(\lambda)
\oplus Q$, where $Q$ is a direct sum of $P_\Bbbk(\mu)$'s with
$\mu>\lambda$. By induction, $P_\Bbbk(\mu)=\Bbbk P(\mu)$, hence $Q$ lifts to
$\tilde{Q}\in\scrC^\proj$, and there are maps $f:\tilde{Q}\to P(\lambda)$
and $g:P(\lambda)\to \tilde{Q}$ such that $\Bbbk(gf)=\id_Q$. Since $R$ is local
and $\tilde{Q}$ is a finitely generated projective $R$-module, we deduce that
$gf$ is an automorphism of $\tilde{Q}$, hence $\tilde{Q}$ is a direct summand
of $P(\lambda)$, so $\tilde{Q}=0$ and $\Bbbk P(\lambda)=P_\Bbbk(\lambda)$.
The unicity of $P(\lambda)$ is then clear, since given $M,N\in\scrC^\proj$,
we have $\Bbbk\Hom_{\scrC}(M,N)\iso \Hom_{\Bbbk\scrC}(\Bbbk M,\Bbbk N)$.

Given $R\to S$ a local point, the residue field $\Bbbk'$ of $S$ is a field
extension
of $\Bbbk$. Since $\Bbbk A$ is a split $\Bbbk$-algebra, it follows that given
$P$ a projective indecomposable $\Bbbk A$-module, then $\Bbbk'P$ is a projective
indecomposable $\Bbbk'A$-module. We deduce that $P_{\Bbbk'}(\lambda)\simeq \Bbbk'\otimes_\Bbbk
\Bbbk P(\lambda)$, hence $P_S(\lambda)\simeq SP(\lambda)$.

The statements about $I(\lambda)$ follow from those about $P(\lambda)$
by duality.

The statements about $T(\lambda)$ are proven in the same way as those
for $P(\lambda)$, using Proposition \ref{prop:introhw}(b) below.
\end{proof}

Note that
$(\scrC, \Delta(\scrC))$ is a highest weight $R$-category if and only if
$(\Bbbk\scrC, \Bbbk\Delta(\scrC))$ is a highest weight $\Bbbk$-category 
and the objects of $\Delta(\scrC)$ are projective over $R$,
see \cite[thm.~4.15]{R1}.
Note also that $\Delta(\lambda)$ has a unique simple quotient
$L(\lambda)$, and $\Irr(\scrC)=\{L(\lambda)\}_{\lambda\in\Lambda}$.


\medskip
Let $\scrC^\Delta$ and $\scrC^\nabla$ be the full 
subcategories of $\scrC$ consisting of the 
{\it $\Delta$-filtered} and {\it $\nabla$-filtered} objects, i.e., 
objects having a finite
filtration whose successive  quotients are standard, costandard respectively.
These are exact subcategories of $\scrC$.
Note that every object of $\scrC^\Delta$ has a finite projective resolution, where
the kernels of the differentials are in $\scrC^\Delta$. As a consequence, the
canonical functor $D^b(\scrC^\Delta)\to D^b(\scrC)$ is fully faithful. Similarly,
the canonical functor $D^b(\scrC^\nabla)\to D^b(\scrC)$ is fully faithful, as
every object of $\scrC^\nabla$ has a finite relatively $R$-injective resolution.

\begin{lemma}\label{deltaequivalence}
Let $\scrC$, $\scrC'$ be highest weight $R$-categories. An exact functor
$\Phi:\scrC\to\scrC'$ which restricts to an equivalence $\Phi:\scrC^{\Delta}\iso{\scrC'}^\Delta$ 
is an equivalence of highest weight categories $\scrC\iso\scrC'$.
\end{lemma}

\begin{proof}
Since $\Phi$ identifies the projective objects in $\scrC$ and $\scrC'$, it induces an equivalence of their bounded homotopy categories, hence
an equivalence $D^b(\scrC)\to D^b(\scrC')$. Since $\Phi$ is exact, we are done.
\end{proof}

Let $\scrC^\tilt=\scrC^\Delta\cap\scrC^\nabla$ be the full 
subcategory of $\scrC$ consisting of the \emph{tilting objects}, i.e., 
the objects which are both $\Delta$-filtered and $\nabla$-filtered.

Let $T=\bigoplus_{\lambda\in\Lambda}T(\lambda)$.
The \emph{Ringel dual} of $\scrC$ is the category
$\scrC^\diamond=\End_\scrC(T)^\op\mmod$. It is a highest weight
category on the poset $\Lambda^\op$.
The functor $\Hom(T,\bullet):\scrC\to\scrC^\diamond$ 
restricts to an equivalence
$\scrR:\scrC^{\nabla}\iso(\scrC^\diamond)^\Delta$,
called the \emph{Ringel equivalence}. We have
$\scrR(\nabla(\lambda))=\Delta^{\!\diamond}(\lambda)$, 
$\scrR(T(\lambda))\simeq P^\diamond(\lambda)$ and
$\scrR(I(\lambda))\simeq T^\diamond(\lambda)$ for $\lambda\in\Lambda$,
see \cite[Proposition 4.26]{R1}.
The highest weight category $\scrC$ is determined, up to equivalence, by $\scrC^\diamond$ and
we put $(\scrC^\diamond)^\blackdiamond=\scrC$.
There is an equivalence of highest weight categories $\scrC\iso
\scrC^{\diamond\diamond}$ such that
the composition 
$$\scrC^\proj\iso (\scrC^{\diamond\diamond})^\proj\xrightarrow[\sim]{\scrR^{-1}}
(\scrC^\diamond)^\tilt\xrightarrow[\sim]{\scrR^{-1}}\scrC^\inj$$
is isomorphic to the Nakayama duality $\Hom_A(\bullet,A)^*$. This provides also
an equivalence of highest weight categories $\scrC^\blackdiamond\iso
\scrC^\diamond$.

Now, for $M\in\scrC$ we set
\begin{equation}
\label{eq:lcdrcd}
\gathered
\mathrm{lcd}_\scrC(M)=
\min\{i\,;\,\exists\mu\in\Lambda, \Ext^i(M,T(\mu))\not=0\},\\
\mathrm{rcd}_\scrC(M)=
\min\{i\,;\,\exists\mu\in\Lambda, \Ext^i(T(\mu),M)\not=0\}.
\endgathered
\end{equation}

\begin{lemma}
Assume $R$ is a field.
Let $\lambda\in\Lambda$. Then
\begin{align*}
\min\{i\,;\, \exists\mu\in\Lambda, \Ext^i(L(\lambda),T(\mu))\not=0\}&=
\min\{i\,;\,\exists\mu\in\Lambda, \Ext^i(L(\lambda),\Delta(\mu))\not=0\}\\
&=\min\{i\,;\,\exists M\in\scrC^\Delta, \Ext^i(L(\lambda),M)\not=0\}
\end{align*}
\end{lemma}

\begin{proof}
Let $c_1$, $c_2$ and $c_3$ be the quantities defined by the terms
involving respectively $T(\mu)$'s, $\Delta(\mu)$'s and $M\in\scrC^\Delta$ in
the first two equalities.
It is clear that $c_1\ge c_2=c_3$.

Take $\mu$ minimal such that $\Ext^{c_2}(L(\lambda),\Delta(\mu))\not=0$.
There is an exact sequence $0\to \Delta(\mu)\to T(\mu)\to
M\to 0$ where $M$ has a filtration with subquotients $\Delta(\nu)$'s where
$\nu<\mu$. We deduce that $\Ext^{c_2}(L(\lambda),T(\mu))\not=0$, hence
$c_1\le c_2$.
\end{proof}

\medskip
Let us recall a few facts on base change for highest weight categories.

\vspace{2mm}

\begin{prop} \label{prop:introhw} 
Let $\scrC$ be a highest weight $R$-category, and let $R\to S$ be a local $S$-point.
For any $M,N\in\scrC$ the following holds :

(a) if $S$ is $R$-flat then $S\Ext^d_\scrC(M,N)=\Ext^d_{S\scrC}(SM,SN)$ for all $d\in\bbN,$

(b) if either $M\in\scrC^\proj$ or ($M\in\scrC^\Delta$ and $N\in\scrC^\nabla$), then we have
$S\Hom_\scrC(M,N)=\Hom_{S\scrC}(SM,SN)$,

(c) if $M$ is $R$-projective then $M\in\scrC^\proj$ (resp.~ $M\in\scrC^\tilt,$ $\scrC^\Delta$, $\scrC^\inj$)
if and only if $\Bbbk M\in\Bbbk\scrC^\proj$ (resp.~ $\Bbbk M\in\Bbbk\scrC^\tilt,$ $\Bbbk\scrC^\Delta$,
$\Bbbk\scrC^\inj$),

(d) if either ($M\in\scrC^\proj$ and $N$ is $R$-projective) or ($M\in\scrC^\Delta$ and $N\in\scrC^\nabla$) 
then $\Hom_\scrC(M,N)$ is $R$-projective.
\end{prop}

\vspace{.5mm}

\begin{proof}
Part $(a)$ is [Bourbaki, \emph{Alg\`ebre}, \text{ch.}~X, \S 6, prop.~7.c].

The statements in $(b)$, $(d)$ are clear if $M$ is a free $A$-module, and are preserved under taking direct 
summands, so they hold for $M\in\scrC^\proj$.

Let $M\in\scrC^\Delta$ and $N\in\scrC^\nabla$. We have
$\Ext^1_\scrC(M,N)=\Ext^1_{S\scrC}(SM,SN)=0$. As a consequence, if $M$ is an extension
of $M_1,M_2\in\scrC^\Delta$ and the statements $(b)$, $(d)$ hold for $M_i,N$, then they hold for
$M,N$. We proceed now by descending induction on $\lambda$ to prove that the statement for 
$M=\Delta(\lambda)$. There is an exact sequence $0\to M'\to P(\lambda)\to \Delta(\lambda)\to 0$,
where $M'$ is an extension of $\Delta(\lambda')$'s with $\lambda'>\lambda$. 
The statements $(b)$, $(d)$ hold for $P(\lambda)$ and, by induction, for $M'$. Hence, they hold for $M$.

Part $(c)$ is \cite[prop.~4.30]{R1}.
\end{proof}

\begin{prop}
The indecomposable projective (resp. relatively $R$-injective, tilting)
objects of $\scrC$ are the $P(\lambda)$ (resp. $I(\lambda)$, $T(\lambda)$),
for $\lambda\in\Lambda$.
\end{prop}

\begin{proof}
The statements are classical for $\Bbbk\scrC$, and Proposition
\ref{prop:introhw}(b),(c) reduce to that case.
\end{proof}

\vspace{2mm}

Let us quote the following easy result for a later use.

\vspace{2mm}

\begin{prop}\label{claim:eqHWC} 
(a) Let $\scrC_1$, $\scrC_2$ be highest weight $\Bbbk$-categories.
An equivalence of abelian $\Bbbk$-categories
$F:\scrC_1\to\scrC_2$ which induces a morphism of posets
$\Irr(\scrC_1)\to\Irr(\scrC_2)$ is an equivalence of highest weight categories.

(b) Let $\scrC_1$, $\scrC_2$ be highest weight $R$-categories. An equivalence
of abelian $R$-categories $F:\scrC_1\to\scrC_2$ which induces an equivalence of highest weight
$\Bbbk$-categories $\Bbbk F:\Bbbk\scrC_1\to\Bbbk\scrC_2$ is an equivalence of highest weight  $R$-categories.
\end{prop}

\vspace{.5mm}

\begin{proof}
For part $(a)$ we need to prove that $F$ maps $\Delta(\scrC_1)$ to $\Delta(\scrC_2)$. 
An equivalence of abelian categories $F$ takes indecomposable projective
objects to indecomposable projective objects. So it preserves the standard
modules,
as $\Delta(\lambda)$ is the largest quotient of $P(\lambda)$ all of whose
composition factors are $L(\mu)$'s with $\mu<\lambda$. Part (b) follows
from Proposition \ref{prop:introhw}(c).
\end{proof}

Next, we state some basic facts on $\nabla$ and $\Delta$-filtered modules. The
situation over a base ring that is not a field is slightly more complicated.

\vspace{2mm}

\begin{lemma}\label{lem:B7} Let $\scrC$ be a highest weight
category over $R$ and let $M\in\scrC$. The following conditions are equivalent:

\begin{itemize}
\item[(a)] $\Ext^1_\scrC(\Delta(\lambda),M)=0$ for all $\lambda\in \Lambda$
\item[(b)] there is a filtration
$0=M_0\subset M_1\subset\cdots\subset M_r=M$ and there are elements $\lambda_i\in\Lambda$
such that $M_i/M_{i-1}\simeq\nabla(\lambda_i)\otimes_R\Hom_\scrC(\Delta(\lambda_i),M)$
with $\lambda_i\neq\lambda_j$ for $i\neq j$ and $\lambda_i<\lambda_j$ implies $i<j$
\item[(c)] there is a filtration
$0=M_0\subset M_1\subset\cdots\subset M_r=M$, there are elements $\lambda_i\in\Lambda$
and there are $R$-modules $U_i$ such that
$M_i/M_{i-1}\simeq\nabla(\lambda_i)\otimes_R U_i$.
\end{itemize}
If the conditions above hold and $M$ is projective over $R$, then
$M\in\scrC^\nabla$.
\end{lemma}

\vspace{.5mm}

\begin{proof} 
Assume $(b)$. Let $\lambda,\mu\in\Lambda$ and $U\in R\mmod$. We have
$\Ext^{>0}_\scrC(\Delta(\lambda),\nabla(\mu))=0$ and 
$\Hom_\scrC(\Delta(\lambda),\nabla(\mu))\in R\mproj$. We deduce that
\begin{multline*}
\Ext^{>0}_\scrC(\Delta(\lambda),\nabla(\mu)\otimes_R U)=
H^{>0}(R\Hom_\scrC(\Delta(\lambda),\nabla(\mu)\otimes_R^L U))\simeq \\
\simeq H^{>0}(R\Hom_\scrC(\Delta(\lambda),\nabla(\mu))\otimes_R^L U)=
\Ext^{>0}_\scrC(\Delta(\lambda),\nabla(\mu))\otimes_R U=0.
\end{multline*}
This shows $(b)\Rightarrow(a)$.

\smallskip
Now, assume $(a)$.
Let $\lambda\in\Lambda$ be minimal such that $\Hom_\scrC(\nabla(\lambda),M)\neq 0.$
Fix an element $\mu\leqslant\lambda$ (no assumption on $\mu$ if
$\Hom_\scrC(\nabla(\lambda),M)=0$ for all $\lambda\in\Lambda$). There is an exact sequence
$0\to\Delta(\mu)\to T(\mu)\to M'\to 0$, where $M'$ is filtered by $\Delta(\nu)$'s with $\nu<\mu$.
So, we have $\Ext^1_\scrC(M',M)=0$.
Hence the canonical map $\Hom_\scrC(T(\mu),M)\to \Hom_\scrC(\Delta(\mu),M)$ is surjective.
There is an exact sequence
$0\to M''\to T(\mu)\to\nabla(\mu)\to 0$, where $M''$ is filtered by $\nabla(\nu)$'s with $\nu<\mu.$
Since $\Hom_\scrC(M'',M)=0$, the canonical map $\Hom_\scrC(\nabla(\mu),M)\to \Hom_\scrC(T(\mu),M)$
is an isomorphism. Consequently, the composition $\Delta(\mu)\to T(\mu)\to \nabla(\mu)$ induces a surjective
map $\Hom_\scrC(\nabla(\mu),M)\to \Hom_\scrC(\Delta(\mu),M)$.

If $\mu\neq\lambda$, we have $\Hom_\scrC(\nabla(\mu),M)=0$, hence $\Hom_\scrC(\Delta(\mu),M)=0$.
This shows that the canonical map $\Hom_\scrC(T(\lambda),M)\to\Hom_\scrC(\Delta(\lambda),M)$ is an 
isomorphism. Hence, we have canonical isomorphisms
\begin{equation*}
\Hom_\scrC(\nabla(\lambda),M)\iso\Hom_\scrC(T(\lambda),M)\iso\Hom_\scrC(\Delta(\lambda),M).
\end{equation*}

Now, set $U=\Hom_\scrC(\Delta(\lambda),M)$. We have
$$\aligned
\Hom_\scrC(\nabla(\lambda)\otimes_RU,M)
&\simeq\Hom_R(U,\Hom_\scrC(\nabla(\lambda),M))\\
&\simeq\Hom_R(U,\Hom_\scrC(\Delta(\lambda),M))\\
&\simeq\Hom_\scrC(\Delta(\lambda)\otimes_RU,M).
\endaligned
$$
So, the canonical map $\Delta(\lambda)\otimes_RU\to M$ factors through a map
$f:\nabla(\lambda)\otimes_RU\to M$.

If $\mu\neq\lambda$, we have $\Hom_\scrC(\Delta(\mu),\nabla(\lambda))=0$.
Further, we have an isomorphism
$$\Hom_\scrC(\Delta(\lambda),f):\Hom_\scrC(\Delta(\lambda),\nabla(\lambda))\otimes_R U\iso
\Hom_\scrC(\Delta(\lambda),M).$$
Consequently, the map $f=\Hom_\scrC(A,f)$ is injective.
Hence, since $$\Ext^2_\scrC(\Delta(\mu),\nabla(\lambda)\otimes_RU)=0$$ for all $\mu$, 
the long exact sequence gives a surjective map
$$ \Ext^1_\scrC(\Delta(\mu),M)\to \Ext^1_\scrC(\Delta(\mu),\Coker(f)).$$
The left hand side is 0 by assumption, we deduce that
$\Ext^1_\scrC(\Delta(\mu),\Coker(f))=0.$
We have
$$\{\mu\in\Lambda\,;\,\Hom_\scrC(\Delta(\mu),\Coker(f))\neq 0\}\subset
\{\mu\in\Lambda\,;\,\Hom_\scrC(\Delta(\mu),M)\neq 0\}\setminus\{\lambda\}.$$
Therefore, by induction on the set $\{\mu\in\Lambda\,;\,\Hom_\scrC(\Delta(\mu),M)\neq 0\},$ we get that 
$\Coker(f)$ has a filtration as required. Since we have an exact sequence
$$0\to\nabla(\lambda)\otimes_RU\to M\to\Coker(f)\to 0,$$
we deduce that $M$ has also a filtration as required.

\smallskip
Assume now $M$ is projective over $R$ and consider a filtration
as in (b). We show that $\Hom_\scrC(\Delta(\lambda),M)$ is projective
over $R$ for all $\lambda$ by induction on $r$.
There is an exact sequence
$0\to L\to P(\lambda_r)\to \Delta(\lambda_r)\to 0$ where $L$ is filtered
by $\Delta(\mu)$'s with $\mu>\lambda_r$, so we have
$\Hom(\Delta(\lambda_r),M)\simeq\Hom(P(\lambda_r),M)$. We deduce that
$\Hom(\Delta(\lambda_r),M)$ is projective over $R$.  By induction,
given $i\le r-1$, then
$\Hom(\Delta(\lambda_i),M_{r-1})\simeq
\Hom(\Delta(\lambda_i),M)$ is projective over $R$ and the result
follows.
\end{proof}

\vspace{3mm}

\subsection{Highest weight covers}
\subsubsection{Definition and characterizations}

Let $\scrC$ be a highest weight $R$-category and let $B$ be a finite projective $R$-algebra.
Consider a quotient functor $F:\scrC\to B\mmod$
in the general sense of \cite[sec.~III.I]{Gab}, \ie, there is
$P\in\scrC^\proj$ and there are isomorphisms
$B\iso\End_\scrC(P)^\op$ and $F\iso\Hom_\scrC(P,\bullet)$.
We denote by $G$ a right adjoint of $F$ and by $\eta:1\to GF$ the unit.

\smallskip
We say that $F$ is
\begin{itemize}
\item a \emph{highest weight cover} if it is fully faithful on $\scrC^\proj$
\item \emph{d-faithful} for some $d\in\bbZ_{\geqslant -1}$ if
$\Ext^i_\scrC(M,N)=0$ for all
$M\in\scrC$ with $F(M)=0$, $N\in\scrC^\Delta$ and $i\leqslant d+1$.
\end{itemize}

As Lemma \ref{lem:B2} below shows, if $F$ is $d$-faithful for some $d\ge 0$, then
it is a highest weight cover.

\smallskip
We denote by $(B\mmod)^{F\Delta}$ the full exact subcategory of $B\mmod$
of objects with a filtration whose successive quotients are in
$F(\Delta)$. Let $F^\Delta:\scrC^\Delta\to (B\mmod)^{F\Delta}$ be the restriction of $F$.

\smallskip
We provide some characterizations of $d$-faithfulness.

\begin{lemma}\label{lem:B2}
Let $F$ be a quotient functor.
Let $d\in\bbZ_{\ge 0}$ and let $\scrE=\scrC^\Delta$, $\scrE=\Delta(\scrC)$ or
$\scrE=\scrC^\tilt$.
The following conditions are equivalent

\begin{itemize}
\item[(i)] $F$ is $d$-faithful
\item[(ii)] given
$M\in\scrC$ with $F(M)=0$ and $N\in\scrE$, we have
$\Ext^{\le d+1}_\scrC(M,N)=0$
\item[(iii)] given $N\in\scrE$, we have 
$H^{\leqslant d}(\cone(N\xrightarrow{\eta} RGF(N)))=0$.
\item[(iv)] given $M\in\scrC$, $N\in\scrE$ and $i\leqslant d$, then $F$ induces
an isomorphism \\ $\Ext^i_\scrC(M,N)\iso\Ext^i_B(FM,FN)$,
\item[(v)] given $M\in\scrC^\proj$, $N\in\scrE$ and $i\leqslant d$, then $F$ induces
an isomorphism $\Ext^i_\scrC(M,N)\iso\Ext^i_B(FM,FN)$.
\end{itemize}
If $R$ is a field, these conditions are equivalent to
\begin{itemize}
\item[(vi)] given $\lambda\in\Lambda$ with $FL(\lambda)=0$, then
$\mathrm{lcd}_\scrC(L(\lambda))>d+1$.
\end{itemize}
\end{lemma}

\vspace{.5mm}

\begin{proof}
Note that (ii) in the case $\scrE=\scrC^\Delta$ is the statement (i).
It is clear that (ii) for $\scrE=\Delta(\scrC)$ is equivalent to (ii) for
$\scrE=\scrC^\Delta$, and these imply (ii) for $\scrE=\scrC^\tilt$.
Assume (ii) holds in the case $\scrE=\scrC^\tilt$. 
Let $M\in\scrC$ with $F(M)=0$.
We prove by induction on $\lambda$ that 
$\Ext^{\le d+1}_\scrC(M,\Delta(\lambda))=0$.

There is an exact sequence $0\to\Delta(\lambda)\to T(\lambda)\to L\to 0$, where $L$ has
a filtration by $\Delta(\mu)$'s with $\mu<\lambda$. We have
$\Ext^{\le d+1}_\scrC(M,T(\lambda))=0$ and, by induction, we have
$\Ext^{\le d+1}_\scrC(M,L)=0$. We deduce that 
$\Ext^{\le d+1}_\scrC(M,\Delta(\lambda))=0$. So, (ii) holds for $\scrE=\scrC^\Delta$.

\smallskip
Let $X=\cone(N\xrightarrow{\eta} RGF(N))$. We have $F(H^i(X))=0$ for all $i$.
Given $Y\in D^b(\scrC)$ such that $F(Y)=0$, we have
$$\Hom_{D^b(\scrC)}(Y,RGF(N))\simeq\Hom_{D^b(B)}(F(Y),F(N))=0,$$ hence
$\Hom_{D^b(\scrC)}(Y,X[i])\simeq\Hom_{D^b(\scrC)}(Y,N[i+1])$ for all $i$.

\smallskip
Assume (ii).
As usual, let $\tau_{\leqslant m}$ denote the \emph{canonical truncation}
which takes a complex $C=(C^n,d^n)$ to the subcomplex
$$\tau_{\leqslant m}(C)=\{\dots\to C^{m-1}\to\Ker(d^m)\to 0\to\dots\}.$$
Taking $Y=\tau_{\le d}(X)$ above, we obtain
$\Hom_{D^b(\scrC)}(\tau_{\le d}(X),X)=0$, hence $\tau_{\le d}(X)=0$. So, (iii) holds.

\smallskip
Note that the canonical map
$\Ext^i_\scrC(M,N)\to\Ext^i_B(FM,FN)$ is an isomorphism if and only if
the canonical map
$\Ext^i_\scrC(M,N)\to\Hom_{D^b(\scrC)}(M,RGFN[i])$ is an isomorphism.

\smallskip
Assume (iii). Applying $\Hom(M,-)$ to the distinguished triangle 
$N\to RGF(N)\to X\rightsquigarrow$, we deduce that (iv) holds.

\smallskip
It is clear that (iv)$\Rightarrow$(v). Assume (v).
It follows from Lemma \ref{le:Extplus1} that the canonical map
$\Ext^{d+1}_\scrC(M,N)\to \Ext^{d+1}_B(F(M),F(N))$ is injective for all $M\in\scrC$, and
(ii) follows.

\medskip
Assume finally $R$ is a field. The assertion (ii), in the case $\scrE=\scrC^\tilt$,
follows from the case $M$ simple: that is assertion (vi).
\end{proof}

\begin{rk}\label{rem:2.9}
We leave it to the reader to check that the first three equivalences in Lemma
\ref{lem:B2} hold when $d=-1$.
\end{rk}

\begin{lemma}
\label{le:Extplus1}
Let $F$ be an exact functor, let $d\ge -1$ and let $N\in\scrC$.
Assume $F$ induces
\begin{itemize}
\item an isomorphism $\Ext^i_\scrC(P,N)\iso \Ext^i_B(F(P),F(N))$ for $P\in\scrC^\proj$ and
$i\le d$
\item an injection $\Ext^{d+1}_\scrC(P,N)\hookrightarrow\Ext^{d+1}_B(F(P),F(N))$ 
for $P\in\scrC^\proj$.
\end{itemize}
Then, $F$ induces
\begin{itemize}
\item an isomorphism $\Ext^i_\scrC(M,N)\iso \Ext^i_B(F(M),F(N))$ for $M\in\scrC$ and
$i\le d$
\item an injection $\Ext^{d+1}_\scrC(M,N)\hookrightarrow \Ext^{d+1}_B(F(M),F(N))$ for
$M\in\scrC$.
\end{itemize}
\end{lemma}

\begin{proof}
We prove by induction on $i$ the first statement of the lemma.
Consider an exact sequence $0\to M'\to P\to M\to 0$ with $P\in\scrC^\proj$. We have
a commutative diagram with exact horizontal sequences
$$
{\tiny
\xymatrix{
\Ext^i_\scrC(P,N)\ar[r]\ar[d]_\sim & \Ext^i_\scrC(M',N)\ar[r]\ar[d] &
\Ext^{i+1}_\scrC(M,N)\ar[r]\ar[d] & \Ext^{i+1}_\scrC(P,N)\ar[d]\ar[r]&
\Ext^{i+1}_\scrC(M',N)\ar[d] \\
\Ext^i_B(FP,FN)\ar[r] & \Ext^i_B(FM',FN)\ar[r] &
\Ext^{i+1}_B(FM,FN)\ar[r] & \Ext^{i+1}_B(FP,FN)\ar[r]&\Ext^{i+1}_B(FM',FN)
}}$$
where the fourth vertical map is an isomorphism for $i+1\le d$ and is injective
for $i=d$. By induction, the second vertical map is an isomorphism, hence
the third vertical map is injective. So, we have shown that the canonical map
$\Ext^{i+1}_\scrC(L,N)\to \Ext^{i+1}_B(F(L),F(N))$ is injective for all $L\in\scrC$,
in particular for $L=M'$. 
If $i+1\le d$, we deduce that the third vertical map is an isomorphism.
\end{proof}

Let us summarize some of the results above.

\begin{cor}
\label{co:char01faith}
Let $F:\scrC\to B\mmod$ be a quotient functor.
\begin{itemize}
\item $F$ is $(-1)$-faithful if and only if $F^\Delta$ is faithful
\item $F$ is a highest weight cover if and only if $\eta(M):M\to GF(M)$ is an
isomorphism for all $M\in\scrC^\proj$
\item $F$ is $0$-faithful if and only if $F^\Delta$ is fully faithful
if and only if $\eta(M):M\to GF(M)$ is an
isomorphism for all $M\in\scrC^\Delta$
\item $F$ is $1$-faithful if and only if $F^\Delta$ is an equivalence.
\end{itemize}
\end{cor}

\smallskip
The next two lemmas relate highest weight covers of $\scrC$,
$\scrC^*$ and $\scrC^\diamond$.

\begin{lemma}
\label{le:hwdual}
Consider a highest weight cover $F=\Hom_\scrC(P,\bullet):\scrC\to B\mmod$.
Then $F^*=\Hom_{\scrC^*}(\Hom_A(P,A),\bullet):\scrC^*\to B^\op\mmod$ is a highest
weight cover.

 Let $d\ge 0$. Then,
$F$ is $d$-faithful if and only if $F^*$ induces isomorphisms
$\Ext^i_{\scrC^*}(M,N)\iso \Ext^i_{B^\op}(F^*M,F^*N)$ for all $M,N\in (\scrC^*)^\nabla$ and
$i\le d$.
\end{lemma}

\begin{proof}
There is a commutative diagram
$$\xymatrix{
(\scrC^\Delta)^\op \ar[rr]_\sim^{\Hom_R(\bullet,R)} \ar[d]_F && (\scrC^*)^\nabla
\ar[d]^{F^*} \\
(B\mmod)^\op\cap R^\proj \ar[rr]_{\Hom_R(\bullet,R)}^\sim && B^\op\mmod\cap R^\proj
}$$
since
\begin{eqnarray*}
\Hom_{A^\op}(\Hom_A(P,A),\Hom_R(\bullet,R))&\simeq&
\Hom_R(\Hom_A(P,A)\otimes_A\bullet,R)\\&\simeq&
\Hom_R(\Hom_A(P,\bullet),R).
\end{eqnarray*}

The lemma follows, since (higher) extensions can be computed in the exact subcategories 
appearing in the diagram.
\end{proof}

The next lemma is clear.

\begin{lemma}
\label{le:hwcoverRingel}
Let $T\in\scrC^\nabla$ and consider a finite projective $R$-algebra $B$ with a morphism
of algebras $\phi:B\to\End_\scrC(T)^\op$. Let $F=\Hom_\scrC(T,\bullet)$, $P=\scrR(T)$ and
$F^\diamond=\Hom_{\scrC^\diamond}(P,\bullet):\scrC^\diamond\to B\mmod$.

The functor $F^\diamond$ is a highest weight weight cover if and only if
$T$ is tilting, $F$ is fully faithful on $\scrC^\tilt$ and $\phi$ is an isomorphism.

The functor $F^\diamond$ is $d$-faithful if and only if
$T$ is tilting, $\phi$ is an isomorphism and
$F$ induces isomorphisms
$\Ext^i_{\scrC}(M,N)\iso \Ext^i_{B}(FM,FN)$ for all $M,N\in \scrC^\nabla$ and
$i\le d$.
\end{lemma}

\vspace{2mm}

Recall that an $R$-algebra $B$ is \emph{Frobenius} if and only if
$B^*=\Hom_R(B,R)$ is isomorphic to $B$ as a $B$-module.

\begin{lemma}
\label{le:tiltfromFrob}
Let $F=\Hom_\scrC(P,\bullet):
\scrC\to B\mmod$ be a $0$-faithful functor. If $B$ is a Frobenius
algebra, then $P$ is tilting.
\end{lemma}

\begin{proof}
Let $\lambda\in\Lambda$.
By Lemma \ref{le:Extplus1}, we have an injection
$$\Ext^1_\scrC(\Delta(\lambda),P)\hookrightarrow
\Ext^1_B(F\Delta(\lambda),F(P)).$$ Since $F(P)=B$ is relatively $R$-injective
and $F\Delta(\lambda)$ is projective over $R$, we deduce that
$\Ext^1_B(F\Delta(\lambda),F(P))=0$, hence
$\Ext^1_\scrC(\Delta(\lambda),P)=0$. It follows from Lemma \ref{lem:B7}
that $P$ is tilting.
\end{proof}

\begin{lemma}
\label{le:projfromFrob}
Let $\scrC$ be a highest weight category, $T\in\scrC^\tilt$ and
$B=\End_{\scrC}(T)^\op$. Assume the restriction of
$\Hom_\scrC(T,\bullet)$ to $\scrC^\nabla$ is fully faithful and $B$ is a Frobenius algebra.
Then $T$ is projective.
\end{lemma}

\begin{proof}
This follows from Lemma \ref{le:tiltfromFrob} applied to $\scrC^\diamond$, cf
Lemma \ref{le:hwcoverRingel}.
\end{proof}

\subsubsection{Base change}
Let $S$ be a local commutative flat $R$-algebra. If $F$ is $d$-faithful, then
$SF$ is $d$-faithful, and the converse holds if $S$ is faithfully flat over $R$
(for example, if it is a local $S$-point).

\begin{lemma}
\label{prop:RR}
Let $F$ be a quotient functor.

If $KF$ is $(-1)$-faithful, then $F$ is $(-1)$-faithful.

Assume $R$ is a regular local ring.
If $R_\frakp F$ is $0$-faithful (resp. is a highest weight cover)
for all $\frakp\in \frakP_1$, then $F$ is $0$-faithful (resp.
is a highest weight cover).
\end{lemma}

\begin{proof}
The first statement is obvious, since objects of $\scrC^\Delta$ are projective
over $R$.

Assume now $F$ is $(-1)$-faithful.
Let $M\in\scrC^\Delta$. Consider the exact sequence
$$0\to M\xrightarrow{\eta(M)} GFM\to \coker\eta(M)\to 0.$$
Assume $R_\frakp \coker\eta(M)=0$ for all $\frakp\in \frakP_1$. 
Then, the support of $\coker\eta(M)$ has codimension $\ge 2$, hence
$\Ext^1_R(\coker\eta(M),M)=0$, since $M$ is projective over $R$. It follows
that $\coker\eta(M)$ is a direct summand of the torsion-free module
$GF(M)$, hence $\coker\eta(M)=0$.
The lemma follows.
\end{proof}

\vspace{2mm}
The corollary below is immediate.

\begin{cor}\label{cor:2.16}
Let $\scrC$ be a highest weight category, $T\in\scrC^\tilt$ and
$B=\End_{\scrC}(T)^\op$. 
Let $F=\Hom_\scrC(T,\bullet)$.
Assume $R$ is a regular local ring.
Then, the restriction of
$F$ to $\scrC^\nabla$ is fully faithful 
if the restriction of
$R_\frakp F$ to $R_\frakp \scrC^\nabla$ is fully faithful for all $\frakp\in \frakP_1$.
\end{cor}

\begin{proof}
Let $P=\scrR(T)$ and $F^\diamond=\Hom_{\scrC^\diamond}(P,\bullet):\scrC^\diamond\to B\mmod$.
The restriction of
$F$ to $\scrC^\nabla$ is fully faithful 
if and only if  the restriction of
$F^\diamond$ to $(\scrC^\diamond)^\nabla$ is fully faithful.
Now, $F^\diamond$ is a quotient functor because $T$ is tilting.
Thus, by Lemma \ref{prop:RR}, if $R_\frakp F^\diamond$ is $0$-faithful 
for all $\frakp\in \frakP_1$, then $F^\diamond$ is $0$-faithful.
Finally, by Lemma \ref{le:hwcoverRingel},
$R_\frakp F^\diamond$ is $0$-faithful if the restriction of
$R_\frakp F$ to $R_\frakp \scrC^\nabla$ is fully faithful.
\end{proof}

The following key result generalizes \cite[prop.~4.42]{R1}.

\begin{prop}
\label{prop:0to1}
Assume $R$ is regular.
If $\Bbbk F$ is $d$-faithful, then
$F$ is $d$-faithful. If in addition
$KF$ is $(d+1)$-faithful, then $F$ is $(d+1)$-faithful.
\end{prop}

\begin{proof}
Assume $\Bbbk F$ is $d$-faithful.
Let $M\in\scrC$ with $F(M)=0$ and let $N\in\scrC^\Delta$. We have
$R\Hom_{\Bbbk\scrC}(kM,kN)\simeq \Bbbk\otimes_R^\bbL R\Hom_{\scrC}(M,N)$.
Let $C$ be a bounded complex of finitely generated
projective $R$-modules quasi-isomorphic to 
$R\Hom_{\scrC}(M,N)$ and with $C^{<r}=0$. We assume $r$ is maximal
with this property. Then, $\Ext^r_{\Bbbk\scrC}(\Bbbk M,\Bbbk N)\simeq
H^r(\Bbbk C)\not=0$, hence $r>d+1$, so $\Ext^{\le d+1}_{\scrC}(M,N)=0$. It follows that
$F$ is $d$-faithful.

Assume now $KF$ is $(d+1)$-faithful. Then $H^{d+2}(C)$ is a torsion
$R$-module. If it is non-zero, then $C^{d+1}\not=0$ a contradiction. So,
$H^{d+2}(C)=0$ and $F$ is $(d+1)$-faithful.
\end{proof}

\subsubsection{Uniqueness results}
\label{sec:unique}
We assume in this section that $R$ is normal.

Let $B'$ be an $R$-algebra, finitely generated and projective over $R$, and
such that $KB'$ is split semi-simple.

Fix a poset structure on $\Irr(KB')$.
 Given $E\in\Irr(KB')$, let $(KB')_{\le E}$ (resp. $(KB')_{<E}$)
be the sum of the simple $KB'$-submodules of $KB'$
isomorphic to some $F\le E$ (resp. $F<E$).

 We say that a family
$\{S(E)\}_{E\in\Irr(KB')}$ of $B'$-modules, finitely generated and projective
over $R$, are {\it Specht modules} for $B'$ if
$$(B'\cap (KB')_{\le E})/(B'\cap (KB')_{<E})\simeq S(E)^{\dim_KE}
\text{ for }E\in\Irr(K'B).$$

Note that $KS(E)\simeq E$ and
$\End_{B'}(S(E))=R$. So, if
$\{S'(E)\}_{E\in\Irr(KB')}$ is another family of Specht modules, then
$S'(E)\simeq S(E)$ for all $E$: the Specht modules are unique, up to
isomorphism (if they exist).

The same construction with the  opposite order on $\Irr(KB')$ leads to
the {\em dual Specht modules} $S'(E)\in B\mmod$ with $KS'(E)\simeq E$.

Assume that  the $K$-algebra $KB$ is semi-simple 
and that $F$ is a highest weight cover.
Then the $K$-category $K\scrC$ is split semi-simple and we have 
an equivalence $KF:K\scrC\iso KB\mmod$.
So, the functor $KF$ induces a bijection $\Irr(K\scrC)\iso\Irr(KB)$ and we put
$S(\lambda)_K=KF(\Delta(\lambda))\in\Irr(KB)$. The
highest weight order on $\Irr(K\scrC)$ yields a partial order on $\Irr(KB).$

We will say that $F$ is a \emph{highest weight cover of
$B$ for the order on $\Irr(KB)$ coming from the one on $\Irr(K\scrC)$}.

The next lemma follows from \cite[Lemma 4.48]{R1}.

\begin{lemma}
\label{le:Korder}
Let $F$ be a highest weight cover and assume $KB$ is semi-simple. Then
$B$ has Specht modules $S(\lambda)=F(\Delta(\lambda))$ and dual
Specht modules $S'(\lambda)=F(\nabla(\lambda))$.
\end{lemma}

\begin{prop}\label{prop:key}
Let $F:\scrC\to B\mmod$ and $F':\scrC'\to B\mmod$ be highest
weight covers. Assume $R$ is regular, $B$ is Frobenius, and $KB$ is semi-simple.

Assume that
\begin{itemize}
\item the order on $\Irr(KB)$ induced by $(\scrC,F)$ refines, or is refined by, the
order induced by $(\scrC',F')$
\item $F$ is fully faithful on $\scrC^\Delta$ and on $\scrC^\nabla$
\item $F'$ is fully faithful on $\scrC^{\prime\Delta}$ and on
$\scrC^{\prime\nabla}$
\item
\begin{itemize}
\item[(a)] $F(P(\lambda))\in F'(\scrC^{\prime\proj})$ for all
$\lambda\in\Lambda$ such that
$\mathrm{lcd}_{k\scrC}(L(\lambda))\le 1$ and
$F(I(\lambda))\in F'(\scrC^{\prime\inj})$ for all
$\lambda\in\Lambda$ such that
$\mathrm{rcd}_{k\scrC}(L(\lambda))\le 1$.

or

\item[(b)] $F(T(\lambda))\in F'(\scrC^{\prime\tilt})$ for all
$\lambda\in\Lambda$ such that
$\mathrm{lcd}_{k\scrC^\diamond}(L^\diamond(\lambda))\le 1$
or $\mathrm{rcd}_{k\scrC^\diamond}(L^\diamond(\lambda))\le 1$
\end{itemize}
\end{itemize}
Then,
there is an equivalence of highest weight categories
$\Phi:\scrC\iso\scrC'$ such that $F'\Phi\simeq F$.
\end{prop}

\begin{proof}
Lemma \ref{le:Korder} shows there is a bijection $p:\Lambda\iso\Lambda'$
such that $F(\Delta(\lambda))\simeq F'(\Delta'(p(\lambda)))$. 
If the order on $\Lambda'$ is finer than the one induced by $\Lambda$, then
we replace it by that coarser order. The category $\scrC'$ is still highest weight.
In the other case, we change the order on $\Lambda$. Then $p$ becomes an isomorphism
of posets and we
identify $\Lambda$ and $\Lambda'$ through $p$.

Let $\scrO=\scrC^\diamond$ and
$\scrO'=\scrC^{\prime\diamond}$. Lemma \ref{le:tiltfromFrob} shows that $P$ is
tilting. So, $\scrR(P)$ is tilting and projective and, 
identifying $\scrC^\blackdiamond$ with $\scrC^\diamond$,
we have $\scrR^{-1}(P)\simeq \scrR(P)$.
Since $F$ is fully faithful on $\scrC^\nabla$, it follows from Lemma
\ref{le:hwcoverRingel} that
$F^\diamond=\Hom_{\scrC^\diamond}(\scrR(P),-)$ is $0$-faithful. 
Similarly,  we deduce that
$F^\diamond$ is fully faithful on
$(\scrC^\diamond)^\nabla$, since $F$ is $0$-faithful. 
We prove in the same way that $F^{\prime\diamond}=
\Hom_{\scrC^{\prime\diamond}}(\scrR(P'),\bullet)$
is fully faithful on $(\scrC^{\prime\diamond})^\Delta$ and on
$(\scrC^{\prime\diamond})^\nabla$.

We have $F(P(\lambda))\in F'(\scrC^{\prime\proj})$ if and only if
$F^\blackdiamond(T^\blackdiamond(\lambda))\in F^{\prime\blackdiamond}
((\scrC^{\prime\blackdiamond})^\tilt)$.
Similarly, we have $F(I(\lambda))\in F'(\scrC^{\prime\inj})$ if and only if
$F^\diamond(T^\diamond(\lambda))\in F^{\prime\diamond}((\scrC^{\prime\diamond})^\tilt)$.

Since $\scrC^\diamond\simeq\scrC^\blackdiamond$ as highest weight categories, we deduce that the case (a) of the proposition for $(\scrC,\scrC',F,F')$ is equivalent to
the case (b) of the proposition for $(\scrC^\diamond,\scrC^{\prime\diamond},
F^\diamond,F^{\prime\diamond})$. We assume from now on that we are in case (a).

Let $\tilde{P}=P\oplus
\bigoplus_{\mathrm{lcd}_{k\scrC}(L(\lambda))\le 1}P(\lambda)$,
let $\tilde{B}=\End_\scrC(\tilde{P})^\op$ and let
$\tilde{F}=\Hom_\scrC(\tilde{P},\bullet):\scrC\to \tilde{B}\mmod$. 
This is a $1$-faithful cover by Lemma \ref{lem:B2} and Proposition
\ref{prop:0to1}. So the functor $\tilde{F}$ restricts to
an equivalence $\tilde{F}^\Delta:\scrC^\Delta\iso
(\tilde{B}\mmod)^{\tilde{F}\Delta}$, with inverse 
$\Hom_{\tilde{B}}(\tilde{F}(A),\bullet)$.

Consider $\tilde{P}'\in \scrC^{\prime\proj}$ such that
$F'(\tilde{P}')\simeq F(\tilde{P})$.
Fixing such an isomorphism, we obtain
an isomorphism $\tilde{B}\iso \End_{\scrC'}(\tilde{P}')^\op$.
Note that $P'$ is a direct summand of
$\tilde{P}'$, since $F'(P')\simeq B\simeq F(P)$.
Let
$\tilde{F}'=\Hom_{\scrC'}(\tilde{P}',\bullet):\scrC'\to \tilde{B}\mmod$, a highest weight
cover. Lemma \ref{le:Korder} shows that $\tilde{F}'(\Delta'(\lambda))\simeq
\tilde{F}(\Delta(\lambda))$ for all $\lambda\in\Lambda$.

Let $i$ be the
idempotent of $\tilde{B}$ such that $\tilde{P}i=P$. 
The right action of $B$ on $P$ provides
an isomorphism $B\iso i\tilde{B}i$. This equips $\tilde{B}i$ with a structure
of $(\tilde{B},B)$-bimodule. 
Let $\hat{F}=\Hom_{\tilde{B}}(\tilde{B}i,\bullet):\tilde{B}\mmod\to B\mmod$.

We have an isomorphism
$\hat{F}\circ \tilde{F}\iso
\Hom_\scrC(\tilde{P}\otimes_{\tilde{B}}\tilde{B}i,\bullet)$, hence
$\hat{F}\circ \tilde{F}\iso F$. Similarly, we have
an isomorphism $\hat{F}\circ \tilde{F}'\iso F'$.
Consider the exact functor
$$\Phi=\Hom_{\scrC'}(\tilde{P}'\otimes_{\tilde{B}}\tilde{F}(A),\bullet)
\simeq \Hom_{\tilde{B}}(\tilde{F}(A),\bullet)\circ \tilde{F}^{\prime\Delta}:
(\scrC')^{\Delta}\to \scrC^\Delta.$$
We have an isomorphism $\tilde{F}^\Delta\circ\Phi\iso \tilde{F}^{\prime\Delta}$
and there is a commutative diagram
$$\xymatrix{
\scrC^{\prime\Delta}\ar[rrrr]^{\Phi} \ar[drr]^{\tilde{F}^{\prime\Delta}}
\ar@{_{(}->}[ddrr]_{F^{\prime\Delta}}
&& && \scrC^\Delta \ar[dll]_{\tilde{F}^\Delta}^\sim
\ar@{^{(}->}[ddll]^{F^\Delta} \\
&& (\tilde{B}\mmod)^{\tilde{F}\Delta} \ar[d]_-{\hat{F}^\Delta} \\
&& (B\mmod)^{F\Delta}
}$$

\smallskip

Since $F^\Delta$ is fully faithful and $\tilde{F}^\Delta$ is an equivalence,
we deduce that $\hat{F}^\Delta$ is fully faithful. Since $F^{\prime\Delta}$
is fully faithful, we deduce that $\tilde{F}^{\prime\Delta}$ is fully faithful.
It follows that $\Phi$ is fully faithful. Note that $\Phi(\Delta'(\lambda))\simeq\Delta(\lambda)$ for all $\lambda\in\Lambda$. Since $\tilde{F}(P)=\tilde{B}i\simeq\tilde{F}'(P')$, we have $\Phi(P')\simeq P$. 


Define
$$\tilde{\Psi}=\Hom_\scrC(\Phi(A'),\bullet):\scrC\to\scrC'$$
Since $\Phi(A')\in\scrC^{\Delta}$, it follows that $\tilde{\Psi}$ is exact on
$\scrC^\nabla$.
We have 
$$\tilde{\Psi}(P)\simeq\Hom_\scrC(\Phi(A'),\Phi(P'))\simeq\Hom_{\scrC'}(A',P')
\simeq P'.$$
 Let us fix an isomorphism $\tilde{\Psi}(P)\iso P'$.
Let $I\subset\Lambda$ be an ideal. Define $(KP)_I$ as the
sum of the simple submodules of $KP$ isomorphic to
$K\nabla(\mu)$ for some $\mu\in I$. Let $P_I= P\cap (KP)_I$.
Given $\lambda\in\Lambda$, we have $P_{\le\lambda}/P_{<\lambda}\simeq
\nabla(\lambda)^n$ for some $n>0$, since $P$ is tilting (Lemma \ref{le:tiltfromFrob})
and $KP$ is a progenerator
of $K\scrC$. We have $K\tilde{\Psi}((KP)_I)=(KP')_I$ for all ideals $I\subset
\Lambda$, hence
$\tilde{\Psi}(\nabla(\lambda))\simeq \nabla'(\lambda)$ for all $\lambda\in
\Lambda$.
We deduce that $\tilde{\Psi}$ restricts to 
an exact functor $\Psi:\scrC^\nabla\to\scrC^{\prime\nabla}$. We have
\begin{align*}
\Phi(A')\otimes_{\scrC'}P'&\simeq \Hom_{\scrC'}(\tilde{P}'\otimes_{\tilde{B}}
\tilde{F}(A),P')\simeq \Hom_{\tilde{B}}(\tilde{F}(A),\tilde{F}'(P')) \\
&\simeq 
\Hom_{\tilde{B}}(\tilde{F}(A),\tilde{F}(P))\simeq \Hom_{\scrC}(A,P)\simeq P,
\end{align*}
hence
$$F'\circ\tilde{\Psi}=\Hom_{\scrC'}(P',\Hom_{\scrC}(\Phi(A'),\bullet))\simeq
\Hom_{\scrC}(\Phi(A')\otimes_{\scrC'}P',\bullet)\simeq \Hom_{\scrC}(P,\bullet)=F.$$
Since $F^\nabla$ and $F^{\prime\nabla}$ are fully faithful, we deduce that
$\Psi$ is fully faithful.

We now apply what we have proven to $\scrC^*$ and $\scrC^{\prime*}$ (cf Lemma
\ref{le:hwdual}).
We obtain a full faithful exact functor $\Psi_*:\scrC^{*\nabla}
\to \scrC^{'*\nabla}$, hence a fully faithful exact functor
$\Upsilon=\Psi(\bullet^*)^*:\scrC^{\Delta}\to\scrC^{\prime\Delta}$
such that $\Upsilon(\Delta(\lambda))\simeq\Delta(\lambda')$ for all $\lambda\in\Lambda$.
The composition $\Phi\Upsilon$ is a fully faithful exact endofunctor of $\scrC^\Delta$
and $F\Phi\Upsilon\simeq F$. It follows that $\Phi\Upsilon$ fixes isomorphism classes
of objects, hence it is an equivalence. Similarly, $\Upsilon\Phi$ is an equivalence,
hence $\Phi$ is an equivalence $(\scrC')^{\Delta}\to \scrC^\Delta$. The proposition follows from Lemma \ref{deltaequivalence}.
\end{proof}

\subsubsection{Covers of truncated polynomial rings in one variable}

Let $I$ be a non-empty finite poset and
$\{q_i\}_{i\in I}$ a family of elements of $R$. We denote by
$\bar{q}_i$ the image of $q_i$ in $\Bbbk$.
We assume that given $i,j\in I$, then $\bar{q}_i=\bar{q}_j$ if and only
$i\leqslant j$ or $j\leqslant i$.

Let $B=R[T]/\bigl(\prod_{i\in I}(T-q_i)\bigr)$. This is a free $R$-algebra,
with basis $(1,T,\ldots,T^{d-1})$.
Given $j\in I$, let 
$$S_j=R[T]/(T-q_j) \text{ and }
Y_j=R[T]/\bigl(\prod_{i\in I, i\geqslant j}(T-q_i)\bigr).$$
We put
$Y=\bigoplus_{j\in I}Y_j$,
$A=\End_B(Y)^\op$, $G=\Hom_B(Y,\bullet):B\mmod\to A\mmod$, $P=G(B)$ and
$F=\Hom_A(P,\bullet):A\mmod\to B\mmod$. Let $\Delta(j)$ be the quotient
of $G(Y_j)$ by the subspace of maps $Y\to Y_j$ that factor through
$Y_{j'}$ for some $j'>j$.

\begin{prop}
\label{pr:coverrank1}
$\scrC=A\mmod$ is a highest weight $R$-category on the poset $I$ with standard
objects the $\Delta(j)$'s. The functor $F$ is a $(-1)$-faithful
highest cover of $B$ and we have
$$F(\Delta(j))\simeq S_j,\ F(P(j))\simeq Y_j,\ P(j)=G(Y_j).$$
If $q_i{\not=}q_j$ for $i\not=j$, then $F$ is a $0$-faithful cover of $B$.

Assume $\scrC'$ is a highest weight $R$-category with poset $I$ and
$F':\scrC'\to B\mmod$ is a highest weight cover.
If $R$ is a field or $KF'(\Delta(j))\simeq KS_j$ for all $j$, then
there is an equivalence of highest weight categories $\Phi:\scrC\iso
\scrC'$ such that $F'\Phi\simeq F$.
\end{prop}

\begin{proof}
Let $\bar{I}$ be the quotient of $I$ by the relation $i\sim j$ if
$\bar{q}_i=\bar{q}_j$. We have a block decomposition
$B\simeq \bigoplus_{J\in\bar{I}} R[T]/\bigl(\prod_{i\in J}(T-q_i)\bigr)$, and
if the proposition holds for the individual blocks, then it holds for $B$.
As a consequence, it is enough to prove the proposition when $\bar{q}_i=
\bar{q}_j$ for all $i,j\in I$. Choosing $i\in I$ and replacing
$T$ by $T-q_i$, we can assume further that $\bar{q}_i=0$ for all $i\in I$.
Since the poset structure on $I$ is now a total order, we can
assume $I=\{0,\ldots,d-1\}$ with the usual order,
for some $d\ge 1$.

\smallskip
Assume first $R$ is a field. Then $B=R[T]/T^d$. Note that
$Y_j=R[T]/T^{d-j}$ and that
$\{Y_j\}_{j\in I}$ is a complete set of representatives of
isomorphism classes of indecomposable $B$-modules.
Denote by $e_j$ the idempotent of $A$ corresponding to the projection onto
$Y_j$. Then, the projective indecomposable $A$-modules are the
$P(j)=Ae_j$, $j\in I$. Note that $\End(P(d-1))=R$. Let $L=Ae_{d-1}A$.
We have $L^2=L$, $L\simeq P(d-1)^d$ as left $A$-modules and
$A/L\simeq\End_{R[T]/(T^{d-1})}\bigl(\bigoplus_{0\le i\le d-2}R[T]/
(T^{d-i-1})\bigr)^\op$. It follows that
$A\mmod$ is a highest weight category on the
poset $I$, with $\Delta(j)=Ae_j/Ae_{j+1}Ae_j$
[CPS, ``Finite-dimensional algebras and highest weight categories'', 
Lemma 3.4]. Let us state some properties of $\scrC$, that can be easily checked.
The module $\Delta(j)$ is uniserial, with composition series
$L(j),L(j-1),\ldots, L(0)$, starting from the head. We have
$[P(j):\Delta(i)]=1$ if $i\ge j$, and
$[P(j):\Delta(i)]=0$ otherwise. The module $P=P(0)$ is
projective and injective, while $P(d-1)=\Delta(d-1)$.
Note that $F$ is exact and its restriction to $A\mproj$ is fully faithful.
Since every $\Delta(j)$ embeds in $P$, it follows that $F$ is 
$(-1)$-faithful. Note that $F(\Delta(j))\simeq R$.

Consider now $\scrC'$ and $F'$ as in the proposition. Since
$\scrC'$ has $d$ non-isomorphic projective indecomposable modules, it follows
that $\{F'(P'(j))\}_{j\in I}=\{Y_j\}_{j\in I}$. As a consequence, there
is a permutation $\sigma$ of $I$ and
an equivalence 
$\Phi:\scrC\mproj\iso\scrC'\mproj$ such that $\Phi(P(\sigma(j)))\simeq P'(j)$
and $F'\Phi\simeq F$. Such an equivalence extends to an equivalence
$\Phi:\scrC\iso\scrC'$, and
$F'\Phi\simeq F$. So, $\scrC$ is a highest weight category with
the order given by $i\leqslant' j$ if $\sigma(i)\leqslant\sigma(j)$. Note that
$\End(P(j))=R$ if and only if $j=d-1$. It follows that $d-1$ must be
maximal for the order $\leqslant'$, and considering the quotient algebra
$A/L$ as above, one sees by induction that $\leqslant'=\leqslant$, \ie,
$\sigma=1$, hence $\Phi$ is an equivalence of highest weight categories.
This shows the proposition when $R$ is a field.

\smallskip
Assume now $R$ is a general local ring. The $R$-modules $\Delta(j)$ are
free and $\Bbbk A\simeq\End_{\Bbbk B}(\Bbbk Y)$. We deduce that 
$\scrC$ is a highest weight category and $F$
is a $(-1)$-faithful highest weight cover. If $KB$ is semi-simple,
it follows from Proposition 2.18 that $F$ is $0$-faithful (the regularity
of $R$ is not necessary here).

We consider finally $\scrC'$ and $F'$ as in the proposition. Since the
canonical map $\Bbbk\Hom_B(Y_i,Y_j)\to\Hom_{\Bbbk B}(\Bbbk Y_i,\Bbbk Y_j)$ is an isomorphism
for all $i,j$, we deduce that $\Bbbk F'$ is a highest weight cover, hence
equivalent to $\Bbbk F$. As a consequence, $F'$ is $0$-faithful and
$\Bbbk F'(P'(j))\simeq \Bbbk Y_j$ for all $j$. We deduce that
$[P'(j):\Delta(i)]=\delta_{i\ge j}$, and
it follows that $[KF'(P'(j))]=[KS_j]+\cdots+[KS_{d-1}]$ in 
$K_0(KB\mmod)$.
There is a surjective morphism of $B$-modules $B\to \Bbbk F'(P'(j))$. It
lifts to a surjective morphism of $B$-modules $B\to F'(P'(j))$. Since
$F'(P'(j))$ is free over $R$, there is a subset $J$ of $I$ of
cardinality $j$ with $F'(P'(j))\simeq B/\bigl(\prod_{i\in J}(T-q_i)\bigr)$.
It follows that $[KF'(P'(j))]=\sum_{i{\not\in} J}[KS'_i]$, hence
$F'(P'(j))\simeq Y_j$, as $\{q_i\}_{i\in J}=\{q_i\}_{i\ge j}$.
The proposition follows.
\end{proof}

\vspace{2mm}

\subsection{Complement on symmetric algebras}

Let $R$ be a commutative noetherian ring.
Let $B$ be an $R$-algebra. We say that $B$ is \emph{symmetric} if it is
a finitely generated projective $R$-module and
$B$ is isomorphic to $B^*$ as a $(B,B)$-bimodule.

\hspace{2mm}

\begin{prop} \label{prop:inj}
Let $B$ be a symmetric $R$-algebra.
Assume $R$ is a domain with field of fractions $K$ and
$KB$ is a split semi-simple algebra. Let $\psi$ be an $R$-algebra endomorphism of $B$.

If $K\psi$ is an automorphism of $KB$ that induces the identity map
on $K_0(KB)$, then $\psi$ is an automorphism.
\end{prop}

\vspace{.5mm}

\begin{proof}
Let $t\in\Hom_R(B,R)$ be a symmetrizing form for $B$, the image of $1$ through an
isomorphism of $(B,B)$-bimodules $B\iso B^*$. Note that $t([B,B])=0$.

Since $KB$ is split semi-simple, the character map  is an isomorphism
$K\otimes_\bbZ K_0(KB)\to \Hom_K(KB/[KB,KB],K)$. We deduce that
$\psi$ induces the identity on $KB/[KB,KB]$, hence $t\circ\psi=t$.

Consider a maximal ideal $\frakm$ of $R$, and let $\Bbbk=R/\frakm$. The $\Bbbk$-algebra
$\Bbbk B$ is symmetric, with symmetrizing form $\Bbbk t$ and $(\Bbbk t)\circ
(\Bbbk \psi)=\Bbbk t$. It follows that $\Bbbk t(\ker(\Bbbk\psi))=0$, hence
$\ker(\Bbbk\psi)=0$, since the kernel of a symmetrizing form contains no nonzero
ideal. We deduce that $\Bbbk\psi$ is an isomorphism.

We have shown that $(R/\frakm)\psi$ is onto for every maximal ideal $\frakm$ of $R$.
It follows that $\psi$ is onto, hence it is an isomorphism, since $B$ is a finitely
generated projective $R$-module.
\end{proof}

\vspace{3mm}

\section{Hecke algebras, q-Schur algebras and categorifications}
\label{sec:catego}
Let  $R$ be a $\bbC[q,q^{-1}]$-algebra. 
Let $q_R$ be the image of $q$ in $R$.  
If no confusion is possible, we may abbreviate $q=q_R$.

\subsection{Quivers}\label{sec:quivers}

Assume that $q_R\neq 1$. For any subset $\scrI\subset R^\times$ 
we associate a quiver $\scrI(q)$ with set of vertices $\scrI$ and with an arrow $i\to i\,q_R$
whenever $i, i\,q_R\in \scrI$. 
We may abbreviate $\scrI=\scrI(q)$ when there is no risk of confusion.
Note that we do not assume $\scrI(q)$ to be connected or $\scrI$ to be finite.
We will assume that $(q^{\bbZ}\scrI(q))/q^{\bbZ}$ is finite.

Let $Q_{R,1},\ldots,Q_{R,\ell}\in\scrI$ such that
$\scrI=\bigcup_{p=1}^\ell \scrI_p$,
where $\scrI_p=\scrI\cap q^{\bbZ}_R\,Q_{R,p}$.
We write $i\equiv j$ if $i\in q^{\bbZ}j$.
Each equivalence class has a representative (possibly more than one) in the set
$\{Q_{R,1},Q_{R,2},\dots,Q_{R,\ell}\}$.

If $\scrI(q)$ is stable under multiplication by $q_R^\bbZ$, and
$q_R$ is not a root of $1$, then each $\scrI_p$ is isomorphic to the quiver $A_\infty$.
If $\scrI(q)$ is stable under multiplication by $q_R^\bbZ$, and
$q_R$ is a primitive $e$-th of $1$, then each
$\scrI_p$ is isomorphic to the quiver $A^{(1)}_{e-1}$.

For any subset $I\subset R$ we consider also the quiver $I_1$
with the set of vertices $I$ and with an arrow $i\to i+1$
whenever $i, i+1\in I$.  We may abbreviate $I=I_1$.

\vspace{3mm}

\subsection{Kac-Moody algebras associated with a quiver}\label{ss:quiver}
Let $(a_{ij})$ be the generalized Cartan
matrix associated with the quiver $\scrI$ and
let $\mathfrak{sl}_{\!\scrI}$ be the (derived) Kac-Moody algebra over $\bbC$
associated with $(a_{ij})$.
The Lie algebra $\mathfrak{sl}_{\!\scrI}$ is generated by $E_i$, $F_i$ with $i\in \scrI,$
subject to the usual relations.
Fix a subset $\Omega\subset[1,\ell]$ such that
$\scrI$ is the disjoint union 
$\scrI=\bigsqcup_{p\in \Omega}\scrI_p$.
We have a Lie algebra decomposition $\fraksl_{\scrI}=\bigoplus_{p\in \Omega}\fraksl_{\scrI_p}$.

For each $i\in\scrI$, let $\al_i,$ $\check\al_i$ be the simple root and coroot
corresponding to $E_i$ and let $\Lam_i$ be the $i$-th fundamental weight. Set
$Q=\bigoplus_{i\in \scrI}\bbZ\al_i$ and
$Q^+=\bigoplus_{i\in \scrI}\bbN\al_i$. 
Set $P=\bigoplus_{i\in \scrI}\bbZ \Lam_i$
and
$P^+=\bigoplus_{i\in \scrI}\bbN\Lam_i$.

Let $X$ be the free abelian group with basis $\{\varepsilon_i\,;\,i\in\scrI\}$.
The assignment $\alpha_i\mapsto \varepsilon_i-\varepsilon_{iq}$
yields additive maps $Q,Q^+\to X$. If $\scrI$ is bounded below
then we may identify $\Lambda_i$ with the (finite) sum $\sum_{d\in\bbN}\varepsilon_{iq^{-d}}$.
So, we may consider $P,P^+$ as subsets of $X$.

We will write $P=P_{\!\scrI}$, $Q=Q_{\!\scrI}$, $Q^+=Q_{\!\scrI}^+$ and $X=X_{\!\scrI}$ if necessary.
For $\al\in Q^+$ of height $d$ we write
$\scrI^\al=\{\bfi=(i_1,\ldots,i_d)\in \scrI^d\,;\,\al_{i_1}+\cdots+\al_{i_d}=\al\}.$
The set $\scrI^\al$ is an orbit for the action of the symmetric 
group $\frakS_d$ on $\scrI^d$ by permutation. Each $\frakS_d$-orbit in
$\scrI^d$ is of this form.

For any subset $I\subset R$ we consider also the quiver $I_1$
which yields in the same way as above  a Cartan datum and a Lie algebra $\mathfrak{sl}_I$.

\vspace{3mm}

\subsection{Partitions}\label{sec:combinatorics}
Set
$\bbZ^\ell(n)=\{(\nu_1,\dots,\nu_\ell)\in\bbZ^\ell\,;\,
\nu_1+\dots+\nu_\ell=n\}$,
$\scrC_n^\ell=\{\nu\in\bbZ^\ell(n);\;\nu_p\geqslant 0,\ \forall p\},$
and $\scrC_{n,+}^\ell=\{\nu\in\bbZ^\ell(n);\;\nu_p> 0,\ \forall p\}.$
An element of $\scrC^\ell_n$ is a \emph{composition} 
of $n$ into $\ell$ parts.
We will say that the composition $\nu$ is \emph{dominant} if it satisfies the inequalities
$\nu_1\geqslant \nu_2\geqslant\cdots\geqslant\nu_\ell,$ and that it is \emph{anti-dominant} if we have
$\nu_1\leqslant \nu_2\leqslant\cdots\leqslant\nu_\ell.$

Let $\scrP_n$ be the set of \emph{partitions} of $n$, i.e.,
the set of non-increasing sequences of positive integers
with sum $n$. 
For $\lambda\in\scrP_n$, let $|\lam|=n$ be the weight of $\lambda$, 
let $l(\lam)$ be the number of parts in $\lam$ and
let ${}^t\lambda$ be the transposed partition.
We associate to $\lam$ the \emph{Young diagram} $Y(\lambda)$
with $\lam_i$ boxes in the $i$-th row.
Let $\scrP^\ell_n$ be the set of \emph{$\ell$-partitions} of $n$, i.e., the
set of $\ell$-tuples of partitions 
$\lam=(\lam^1,\ldots,\lam^\ell)$ with
$\sum_p|\lam^p|=n.$ 
Let $\scrP=\bigsqcup_n\scrP_n$ and
$\scrP^\ell=\bigsqcup_n\scrP^\ell_n$.
For each $\nu\in\scrC_n^\ell$ and $d\in[1,n]$ we set 
$\scrP^\nu=\{\lam\in\scrP^\ell\,;\,l(\lam^p)\leqs\nu_p\}$ with
$\scrP^\nu_d=\scrP^\nu\cap\scrP^\ell_d.$

Let $A\in Y(\lambda)$ be the box which lies in 
the $i$-th row and $j$-th column of the diagram of
$\lam^p$. Consider the element
$p(A)=p$ in $[1,\ell]$. Given  $Q_{R,1},Q_{R,2},\dots,Q_{R,\ell}\in\scrI$, we set
$q\text{-}\!\res^Q(A)=q^{j-i}_RQ_{R,p}.$ 
For $\lam$, $\mu\in\scrP^\ell$ 
we write $q\text{-}\!\res^Q(\mu-\lam)=a$ if $\mu$ is
obtained by adding a box of residue $a$ to the Young diagram
associated with $\lam$.

We may write $q\text{-}\!\res^s(A)=q\text{-}\!\res^Q(A)$ and $\ct^s(A)=s_p+j-i$,
where $s_p$ is a formal symbol such that $q^{s_p}_R=Q_{R,p}$.
We call  $q\text{-}\!\res^s(A)$ the \emph{shifted residue} of $A$ and
$\ct^s(A)$ its \emph{shifted content}.
We may also abbreviate $Q_p=Q_{R,p}$.

Let $\Gamma$ be the group of $\ell$-th roots of 1 in $\bbC^\times$.
Let $\frakS_d$ be the symmetric group on $d$ letters and $\Gamma_d$ be the semi-direct product
$\frakS_d\ltimes\Gamma^d$, where $\Gamma^d$ is the Cartesian product of $d$ copies of $\Gamma$.
The group $\Gamma_d$ is a complex reflection group.
The set
$\Irr(\bbC \Gamma_d)$ is identified with $\scrP^\ell_d$ in such a way that $\lambda$ is associated with
the module $\scrX(\lambda)_\bbC$ induced from the 
$\Gamma_{|\lambda^1|}\times\ldots\times\Gamma_{|\lambda^\ell|}$-module
$\phi_{\lambda^1}\chi^\ell\otimes
\phi_{\lambda^2}\chi\otimes\cdots\otimes\phi_{\lambda^\ell}\chi^{\ell-1}.$
Here $\phi_{\lambda^p}$ is the irreducible $\bbC\frakS_{|\lambda^p|}$-module associated with the partition 
$\lambda^p$ and $\chi^p$ is the one dimensional $\Gamma^{|\lambda^p|}$-module given by the $p$-th 
power of the determinant. 

Note that this labeling agrees with \cite[sec.~6]{R1}, \cite[sec.~1.5]{VV} 
but it differs from that of \cite[sec.~2.3.4]{GL}.

\vspace{3mm}

\subsection{Hecke algebras}
\label{sec:3.4}
\subsubsection{Cyclotomic Hecke algebras}
Write $\bfH_{R,0}=R$. For $d\geqs 1,$ the \emph{affine Hecke algebra}
$\bfH_{R,d}$ is the $R$-algebra generated by
$T_1,\ldots,T_{d-1},$ $X^{\pm 1}_1,\ldots,X^{\pm 1}_d$
subject to the relations
\begin{equation*}
\begin{array}{l}
(T_i+1)(T_i-q_R)=0,\\
T_iT_{i+1}T_i=T_{i+1}T_iT_{i+1},\quad T_iT_j=T_jT_i \ \textrm{if}\ |i-j|>1,\\
X_iX_j=X_jX_i,\quad X_iX_i^{-1}=X_i^{-1}X_i=1,\\
T_iX_{i}T_i=q_RX_{i+1},\quad X_iT_j=T_jX_i\ \textrm{if}\ i-j\neq 0,1.
\end{array}
\end{equation*}

The \emph{cyclotomic Hecke algebra} is the quotient
$\bfH^Q_{R,d}$ of $\bfH_{R,d}$ by the two-sided ideal generated by
$\prod_{p=1}^\ell(X_1-Q_{R,p})$.


If $\ell=1$, then the $R$-algebra $\bfH^Q_{R,d}$ is generated by $T_i$ with $i\in[1,d)$. It does not depend on the choice of the unit $Q_{R,1}$.
In this case we write $\bfH^+_{R,d}=\bfH^Q_{R,d}$.

Given $s=(s_1,\dots,s_\ell)$ as above, we write $\bfH^s_{R,d}=\bfH^Q_{R,d}.$
For any $d<d'$, the $R$-algebra embedding
$\bfH_{R,d}\to \bfH_{R,d'}$ given by
$T_i\mapsto T_i,$ $X_j\mapsto X_j$ for $i\in[1,d),$ $j\in[1,d],$
induces an embedding
$\bfH^s_{R,d}\to \bfH^s_{R,d'}$.
The $R$-algebra $\bfH^s_{R,d'}$ is free as a left and as a right $\bfH^s_{R,d}$-module. 
This yields a pair of exact adjoint functors
$(\Ind^{d'}_d,\Res^{d'}_d)$ between
$\bfH^s_{R,d'}\mmod$ and $\bfH^s_{R,d}\mmod.$
For $d\leqslant d'$ there is also an algebra embedding 
$\bfH^+_{R,d}\to \bfH^s_{R,d'}$ given by $T_i\mapsto T_i$ for $i\in[1,d).$
It yields a pair of exact adjoint functors
$(\Ind^{d',s}_{d,+},\Res^{d',s}_{d,+})$ between
$\bfH^+_{R,d}\mmod$ and $\bfH^s_{R,d'}\mmod$.

Now, assume that $R=K$ is a field.
Any finite dimensional $\bfH^s_{K,d}$-module $M$
can be decomposed into (generalized) weight spaces
$M=\bigoplus_{\bfi\in \scrI^d} M_\bfi,$ with
$M_\bfi=\{v\in M\,;\,(X_r-i_r)^nv=0,\
r\in[1,d],\, n\gg 0\}.$
See \cite[sec.~4.1]{BK3} and the references there for details.
Decomposing the regular module, we get a system
of orthogonal idempotents $\{1_\bfi\,;\,\bfi\in K^d\}$
in $\bfH^s_{K,d}$ such that $1_\bfi M=M_\bfi$
for each finite dimensional module
$M$ of $\bfH^s_{K,d}$. 

Given $\al\in Q^+$ of height $d$, we set
$1_\al=\sum_{\bfi\in K^\al} 1_\bfi.$
The nonzero $1_\al$'s are the primitive central idempotents in $\bfH^s_{K,d}$, i.e.,
the algebra $\bfH^s_{K,\al}=1_\al \bfH^s_{K,d}$
is either zero or a single block of $\bfH^s_{K,d}$ \cite{B3,LM}. 

\vspace{3mm}

\subsubsection{Degenerate cyclotomic Hecke algebras}
In the same way we can consider the
\emph{degenerate Hecke algebra} $H_{R,d}$ and the
\emph{degenerate cyclotomic Hecke algebra}
$H^s_{R,d}$ introduced in \cite{BK1}.  We assume here $s\in R^\ell$.
The algebra $H_{R,d}$ is generated by elements 
$t_1,\dots,t_{d-1},x_1,\dots,x_d$ subject to the relations 
\begin{equation*}
\begin{array}{l}
t_i^2=1,\\
t_it_{i+1}t_i=t_{i+1}t_it_{i+1},\quad t_it_j=t_jt_i \ \textrm{if}\ |i-j|>1,\\
x_ix_j=x_jx_i,\\
t_ix_{i+1}=x_{i}t_i+1,\quad x_it_j=t_jx_i\ \textrm{if}\ i-j\neq 0,1.
\end{array}
\end{equation*}
The degenerate cyclotomic Hecke algebra $H^s_{R,d}$ is the quotient of $H_{R,d}$ by the two-sided ideal generated by the 
element $\prod_{p=1}^\ell(x_1-s_{R,p})$.

The representation theory of $H^s_{R,d}$ is very similar to that of
$\bfH^s_{R,d}$. For instance, if $R=K$ is a field then the primitive central idempotents in 
$H^s_{K,d}$
are again labeled by the elements $\al\in Q^+$ of height $d$, which permits us to define
$H^s_{K,\alpha}=1_\alpha H^s_{K,d}$ as above.
For any subset $I\subset K$ we set
$H^s_{I}=\bigoplus_{\al\in Q_{\!I}^+} H^s_{K,\al},$
$H^s_{I,d}=H^s_{I}\cap H^s_{K,d}.$
See e.g., \cite[sec.~3]{BK3} for more details. 

\vspace{3mm}

\subsubsection{Representations}
We will use the following properties of $H^s_{R,d}$ and $\bfH^s_{R,d}$ :

\begin{itemize}

\item the $R$-algebras $\bfH_{R,d}^s$ and 
$H^s_{R,d}$ are both symmetric  by \cite{MM}, \cite[app.~A]{BK1},

\item
the  $K$-algebra
$\bfH^s_{K,d}$ is split semi-simple if and only if 
\begin{equation}\label{(A)}
\prod_{i=1}^d(1+q_K+\cdots+q_K^{i-1})\,\prod_{u<v}\prod_{-d<r<d}(q_K^r\,Q_{K,u}-Q_{K,v})\neq 0.
\end{equation}
\end{itemize}

Now, set $\zeta=\exp(2\sqrt{-1}\pi/\ell)$.
If $q_K=1$ and $Q_{K,p}=\zeta^{p-1},$ 
then $\bfH^s_{K,d}$ is the algebra
$K\Gamma_d$ of the group $\Gamma_d$.
Therefore, if $\bfH^s_{K,d}$ is semi-simple, then the set
$\Irr(\bfH^s_{K,d})$ is canonically identified with $\Irr(K\Gamma_d)$ by Tits' deformation 
Theorem. For each $\lambda\in\scrP^\ell_d,$ one can define a \emph{Specht module}
$S(\lambda)^{s,q}_R$ of $\bfH^s_{R,d}$ as in \S \ref{sec:unique},
using the dominance
order $\unlhd$ on $\scrP^\ell_d$, cf \S \ref{sec:CycloSchur} below. It is free
over $R$, and specializes to
$\scrX(\lambda)_\bbC$ as
$q_R\mapsto 1$ and $Q_{R,p}\mapsto\zeta^{p-1}$.
The Specht modules $S(\lambda)^{s}_R$ of $H^s_{R,d}$ with $\lambda\in\scrP^\ell_d$ are defined similarly.

Now, assume that $R$ is an analytic deformation ring in the sense of \S
\ref{sec:analytic} below. Set $\scrI=\bigcup_{p=1}^\ell q^{s_p+\bbZ}_R$ and $I=\bigcup_{p=1}^\ell(s_p+\bbZ)$.
The multiplication by $q_R$ and the shift by 1 equips the sets $\scrI$, $I$ with structures of quivers $\scrI(q)$, $I_1$ as explained in \S \ref{sec:quivers}.

\vspace{2mm}

\begin{prop}\label{prop:isomHecke}
Assume that $R$ is a local ring.

(a) 
The blocks $\bfH^s_{R,\alpha}$  of 
$\bfH^s_{R,d}$ (resp.~ the blocks $H^s_{R,\alpha}$ of $H^s_{R,d}$) are labeled
by the elements $\alpha\in Q^+_\scrI$ (resp.~$\alpha\in Q^+_I$)
of height $d$. We have
$\Bbbk\bfH^s_{R,\alpha}=\bfH^s_{\Bbbk,\alpha}$ and
$\Bbbk H^s_{R,\alpha}=H^s_{\Bbbk,\alpha}$ for each $\alpha$.

(b) Assume that the map
$\exp(-2\pi\sqrt{-1}\,\bullet/\kappa)$ yields an isomorphism of quivers
$\beta:I_1\to\scrI(q)$.
Given an element $\alpha\in Q_I^+$, let $\alpha$ denote also its image in 
$Q^+_\scrI$.
Then, we have an $R$-algebra isomorphism
$\alpha_R:\bfH^s_{R,\alpha}\iso\, H^s_{R,\alpha}$ such that $\alpha_R(S(\lambda)_R^s)\simeq S(\lambda)_R^{s,q}$ for each $\lambda$.
\end{prop}

\vspace{.5mm}

\begin{proof}
Part $(a)$ is obvious, because the primitive central idempotents of
$\bfH^s_{\Bbbk,d}$, $H^s_{\Bbbk,d}$ lift to
$\bfH^s_{R,d}$, $H^s_{R,d}$ since $R$ is henselian.

More precisely, given $\alpha$ in $Q^+_{\scrI_\Bbbk}$ or in $Q^+_{I_\Bbbk}$,
to lift the idempotent $1_\alpha$ in 
$\bfH^s_{\Bbbk,d}$, $H^s_{\Bbbk,d}$ into an
idempotent  in 
$\bfH^s_{R,d}$, $H^s_{R,d}$, we first consider
the idempotent  in 
$\bfH^s_{K,d}$, $H^s_{K,d}$ given by the sum of all  $1_\bfi$'s, with 
$\bfi$ in $\scrI^d=\scrI^d_R$ or  in $I^d=I^d_R,$ such that the residue class of $\bfi$ in 
$\Bbbk^d$ is a summand of $\alpha$. Note that, although there may be an infinite number of such tuples $\bfi$,
this sum contains only a finite number of non zero terms. 
A standard computation in linear algebra implies that it belongs indeed to $\bfH^s_{R,d}$, $H^s_{R,d}$,
yielding an idempotent which specializes to $1_\alpha$.

Now, we concentrate on part $(b)$.
Note that \cite[sec.~3.5, 4.5]{BK3}, \cite[\S 3.2.6]{R2} yield
a $K$-algebra isomorphism 
$\alpha_K:\bfH^s_{K,\alpha}\iso H^s_{K,\alpha}$. 
We will prove that the isomorphism $\alpha_K$ in \cite{R2} (which differs from the one in \cite{BK3}) restricts to an isomorphism
$\alpha_R:\bfH^s_{R,\alpha}\iso H^s_{R,\alpha}$.

We have the following formulae
$$\begin{gathered}
\alpha_K^{-1}(1_\bfi)=1_\bfj\ \text{where}\ \bfj=\beta(\bfi),\\
\alpha_K^{-1}(x_r1_\bfi)=(j_r^{-1}X_r-1+i_r)1_\bfj,\\
\alpha_K^{-1}((t_r+1)1_\bfi)=(T_r+1){X_r-X_{r+1}-j_r\over X_r-qX_{r+1}}1_\bfj \ \text{if}\ i_r=i_{r+1},\\
\alpha_K^{-1}((t_r+1)1_\bfi)=(T_r+1){X_r-X_{r+1}\over X_r-qX_{r+1}+j_r}1_\bfj \ \text{if}\ i_r=i_{r+1}+1,\\
\alpha_K^{-1}((t_r+1)1_\bfi)=(T_r+1){\alpha_K^{-1}(x_r)-\alpha_K^{-1}(x_{r+1})-1\over X_r-qX_{r+1}}
{X_r-X_{r+1}\over \alpha_K^{-1}(x_r)-\alpha_K^{-1}(x_{r+1})}
1_\bfj \ \text{else}.\\
\end{gathered}$$

Let $\bfP\subset\bfH^s_{R,d}$ and $P\subset H^s_{R,d}$ be the $R$-subalgebras generated by the $X_r$'s and the $x_r$'s respectively.
We may assume that $R$ is in general position.
Then, the $K$-algebras $\bfH^s_{K,d}$, $H^s_{K,d}$ are semi-simple, and the same is true for $K\bfP$ and $KP$.
Therefore, we have $x_r1_\bfi=i_r1_\bfi$ and $X_r1_\bfj=j_r1_\bfj=\beta(i_r)1_\bfj=\exp(-2\pi\sqrt{-1}\,\alpha_K^{-1}(x_r))1_\bfj.$
We deduce that $\alpha_K^{-1}(P)=\bfP$. 

Now, we have
$$\gathered
{\alpha_K^{-1}(x_r)-\alpha_K^{-1}(x_{r+1})-1\over X_r-qX_{r+1}}=q^{-1}X_{r+1}^{-1}
{\alpha_K^{-1}(x_r)-\alpha_K^{-1}(x_{r+1})-1\over
\exp(-2\pi\sqrt{-1}\,(\alpha_K^{-1}(x_r)-\alpha_K^{-1}(x_{r+1})-1)/\kappa)-1},
\\
{X_r-X_{r+1}\over \alpha_K^{-1}(x_r)-\alpha_K^{-1}(x_{r+1})}=
X_{r+1}{\exp(-2\pi\sqrt{-1}\,(\alpha_K^{-1}(x_r)-\alpha_K^{-1}(x_{r+1}))/\kappa)-1\over \alpha_K^{-1}(x_r)-\alpha_K^{-1}(x_{r+1})}.
\endgathered$$
Therefore, both expressions are units in  $\bfP$.
Hence $\alpha_K$ restricts to an isomorphism
$\alpha_R:\bfH^s_{R,\alpha}\iso H^s_{R,\alpha}$.


The isomorphism $\alpha_R(S(\lambda)_R^s)\simeq S(\lambda)_R^{s,q}$ follows
from the unicity of Specht modules.
\end{proof}

\vspace{3mm}

\subsection{Cyclotomic $q$-Schur algebras}
\label{sec:CycloSchur}
For each $\lambda\in\scrP^\ell_d,$ we consider the elements
$w_\lambda=\sum_{w\in\frakS_\lambda}T_w$
and
$x_\lambda=\prod_{p=1}^\ell\prod_{i=1}^{a_p}(X_{i}-Q_{R,p})$
where
$a_p=|\lambda^1|+\cdots+|\lambda^{p-1}|$ and $\frakS_\lambda$ is the parabolic subgroup
of $\frakS_d$ associated with $\lambda$.
The $R$-algebra $\bfS_{R,d}^s=\End_{\bfH^s_{R,d}}
\big(\bigoplus_\lambda w_\lambda x_\lambda\bfH^s_{R,d}\big)$
is called the \emph{cyclotomic $q$-Schur} algebra \cite{DJM}.

The category $\bfS_{R,d}^s\mmod$ is a highest weight category whose standard objects are the 
\emph{Weyl modules} $W(\lambda)^{s,q}_R$ labeled by multipartitions $\lambda\in\scrP^\ell_d$.
The highest weight order is given by the \emph{dominance order} $\unlhd$ on $\scrP^\ell_d$.
The algebra $\bfS^s_{R,d}$ is Ringel self-dual, 
see \cite[prop.~4.3, cor.~7.3]{M3}.

There is a double centralizer property for $\bfS_{R,d}^s$
and $\bfH_{R,d}^s$ which produces a highest weight cover
$\Xi^s_{R,d}:\bfS_{R,d}^s\mmod\to \bfH_{R,d}^s\mmod$,
called the \emph{cyclotomic q-Schur functor} \cite[sec.~5]{M}, \cite{R1}.
The Specht module $S(\lambda)_R^{s,q}$ is the image of $W(\lambda)_R^{s,q}$ under this functor.
If $R=K$ is a field, then the  $K$-algebra $\bfS_{K,d}^s$
is semi-simple if and only if
condition \eqref{(A)} holds.

Using $H^s_{R,d}$ instead of $\bfH^s_{R,d},$ we define 
the \emph{degenerate cyclotomic $q$-Schur algebra} $S^s_{R, d}$
and the cyclotomic $q$-Schur functor
$\Xi^s_{R,d}:S_{R,d}^s\mmod\to H_{R,d}^s\mmod$
in a similar way.
See \cite{AMR}, \cite{BK1} for details. All the results on $\bfS^s_{R,d}$ recalled
above have direct analogues for $S^s_{R, d}$, see e.g., \cite[sec.~6.6]{GL}.
In particular, the Specht module $S(\lambda)_R^{s}$ is the image of the Weyl module $W(\lambda)_R^{s}$ 
by the $q$-Schur functor.

\vspace{3mm}

\subsection{Categorical actions on abelian categories}
\label{sec:3.6}

Let $\scrC$ be an abelian $R$-category.

\begin{df} 
A \emph{pre-categorification} (or \emph{pre-categorical action}) on $\scrC$ is a tuple $(E,F,X,T)$ where
($E$, $F$) is an adjoint pair of exact functors $\scrC\to\scrC$
and $X\in\End(E)$, $T\in\End(E^2)$ are endomorphisms of functors such that
\begin{itemize}
\item for each $d\in\bbN$, there is an $R$-algebra homomorphism 
$\phi_{E^d}:\bfH_{R,d}\to\End(E^d)$ given by
$X_k\mapsto E^{d-k}XE^{k-1},$
$T_l\mapsto E^{d-l-1}TE^{l-1}$ for
$k\in [1,d],$ $l\in[1,d),$

\item the functor $E$ is isomorphic to a right adjoint of $F$.
\end{itemize}
\end{df}

\vspace{.5mm}

\begin{rk}
\label{rem:3.3}
Given a pair of adjoint functors $(E,F)$, the adjunction yields a canonical $R$-algebra isomorphism
$\End(F^d)=\End(E^d)^\op$ for each $d\in\bbN$, see e.g.,  \cite[sec.~4.1.2]{CR}.
Under this isomorphism, the morphisms $X,$ $T$ yield morphisms 
$X\in\End(F)$, $T\in\End(F^2)$ which induces
an $R$-algebra homomorphism $\phi_{F^d}:\bfH_{R,d}\to\End(F^d)^\op.$
\end{rk}

\vspace{2mm}

Now, assume that $R=K$ is a field and that $\scrC$ is Hom-finite.
Let $\scrI=\scrI(q)$.

\vspace{2mm}

\begin{df}[\cite{CR, R2}]\label{df:cat1}
An $\mathfrak{sl}_{\scrI}$\emph{-categorification} (or \emph{categorical action}) on $\scrC$ is the datum of a pre-categorification
$(E,F,X,T)$ and a decomposition $\scrC=\bigoplus_{\lam\in X}\scrC_\lam$.
For $i\in \scrI$ let $F_i,$ $E_i$ be the generalized $i$-eigenspaces
of $X$ acting on $F,$ $E$ respectively. We assume in addition that
\begin{itemize}

\item we have $F=\bigoplus_{i\in \scrI} F_i$ and $E=\bigoplus_{i\in \scrI} E_i$,

\item the action of $E_i$, $F_i$, $i\in \scrI$ on $[\scrC]$ gives an integrable representation of
$\mathfrak{sl}_{\scrI}$,

\item we have $E_i(\scrC_\lam)\subset\scrC_{\lam+\alpha_i}$ and
$F_i(\scrC_\lam)\subset\scrC_{\lam-\alpha_i}$.
\end{itemize}
\end{df}

\vspace{.5mm}

\begin{rk} The constructions above have a degenerate analogue.
Given $I\subset R$ and $\mathfrak{sl}_I$ as above,
the definition of a pre-categorification and of an $\mathfrak{sl}_I$-categorification
is the same, with $\bfH_{R,d}$ replaced by $H_{R,d}$ and
$\mathfrak{sl}_{\scrI}$ by $\mathfrak{sl}_I$.
In particular, for each $d\in\bbN$ there is an $R$-algebra homomorphism 
$\phi_{E^d}:H_{R,d}\to\End(E^d)$ given by
$X_k\mapsto E^{d-k}XE^{k-1},$
$T_l\mapsto E^{d-l-1}TE^{l-1}$.
\end{rk}

\vspace{.5mm}

\begin{ex} \label{ex:3.6}
Let $R=K$ be a field which is an analytic algebra, see \S \ref{sec:analytic}.
Let $s$ be as in \S \ref{sec:combinatorics}, and
$\Lambda=\Lambda^s=\sum_{p=1}^\ell\Lambda_{Q_{p}}.$ 
Let $\bfH^s_{\scrI,d}=\bigoplus_\al \bfH^s_{K,\al}$,
where $\al$ runs over elements of $Q_{\!\scrI}^+$ of height $d$.

The abelian $K$-category
$\scrL(\Lambda)_\scrI=\bigoplus_{d\in\bbN}\bfH^s_{\scrI,d}\mmod$
decomposes as 
$\scrL(\Lambda)_\scrI=\bigoplus_{\al\in Q^+_\scrI}
\scrL(\Lambda)_{\scrI,\Lambda-\al}$
with
$\scrL(\Lambda)_{\scrI,\Lambda-\al}=\bfH^s_{K,\al}\mmod$.

The endofunctors $E=\bigoplus_{d\in\bbN}\Res_d^{d+1}$ and $F=\bigoplus_{d\in\bbN}\Ind_d^{d+1}$ 
of $\scrL(\Lambda)_\scrI$ are exact and biadjoint. The right multiplication on 
$\bfH^s_{\scrI,d+1}$ by $X_{d+1}$ yields an endomorphism 
of the functor $\Ind_d^{d+1}$, denoted again by $X_{d+1}$.
The right multiplication by $T_{d+1}$ yields an endomorphism of $\Ind_d^{d+2}.$
We define $X\in\End(F)$ and $T\in\End(F^2)$ by $X=\bigoplus_{d\in\bbN}X_{d+1}$ and 
$T=\bigoplus_{d\in\bbN}T_{d+1}$.

The tuple $(E, F, X, T)$ and the decomposition above give an 
$\mathfrak{sl}_{\!\scrI}$-categorification of $\bfL(\Lambda)$ (the simple
$\mathfrak{sl}_{\!\scrI}$-module with highest weight $\Lambda$) on $\scrL(\Lambda)_\scrI$, 
called the \emph{minimal} 
$\mathfrak{sl}_{\!\scrI}$-\emph{categorification} of highest weight $\Lambda$.

In the degenerate case,
the induction and restriction functors give an abelian 
$\mathfrak{sl}_I$-categorification of $\bfL(\Lambda)$ on 
$\scrL(\Lambda)_I=\bigoplus_{d\in\bbN} H^s_{I,d}\mmod$, called again the minimal
$\mathfrak{sl}_I$-categorification of highest weight $\Lambda$. 

\end{ex}

\vspace{3mm}

\section{The category $\scrO$}

Fix integers $\ell,N\geqslant 1$ and fix a composition $\nu\in\scrC^\ell_{N,+}$.

\subsection{Deformation rings}\label{sec:deformring}
A \emph{deformation ring} is a regular commutative noetherian 
$\bbC$-algebra $R$ with 1  equipped with a 
$\bbC$-algebra homomorphism $\bbC[\bbC^\times\times\bbC^{\ell}]\to R$. 
Let $\kappa_R,$ $\tau_{R,p}$ be the images in $R$ of 
the standard coordinates $z,z_1,\dots,z_\ell$ on $\bbC^\times$ and $\bbC^{\ell}$. 
Set $\tau_R=(\tau_{R,1},\dots,\tau_{R,\ell})$.
Define $s_{R,1},\ldots,s_{R,\ell}\in R$ by
$s_{R,p}=\nu_p+\t_{R,p}.$
We may abbreviate $s_p=s_{R,p}$, $\kappa=\kappa_R$ and $\tau_p=\tau_{R,p}$.
For any $S$-point $\chi:R\to S$ we write 
$\kappa_S=\chi(\kappa_R)$ and
$\tau_{S,p}=\chi(\tau_{R,p})$.

A \emph{local deformation ring} is a deformation ring $R$ which is a local ring such that 
the residue class $\tau_{\Bbbk,p}$ of $\tau_{R,p}$ is 0 for each $p$. We will denote by $-e$ the residue class 
$\kappa_{\Bbbk}$ of $\kappa_R$. \emph{We will always assume that $e$ is a positive integer}.


\vspace{1mm}

\begin{rk}
Let $R$ be a deformation ring.
Then, for each $\frakp\in\frakP$, the local ring $R_\frakp$ is regarded as a deformation ring
with deformation parameters $\kappa_{R_\frakp}$, $\tau_{R_\frakp}$. 
It may not be a 
local deformation ring, since we may have $\tau_{R,p}\notin\frakp$.
\end{rk}

\vspace{1mm}

We will say that the deformation ring $R$ is in
\emph{general position} if the elements in
$\{\tau_{R,u}-\tau_{R,v}+a\,\kappa_R+b,\,\kappa_R-c\,;\,a,b\in\bbZ,\,c\in\bbQ,\,u\neq v\}$ 
are pairwise coprime.

\vspace{1mm}

\begin{ex}
Given a tuple $\zeta=(\zeta_1,\dots,\zeta_\ell)$ in $\bbC^\ell$,
we have the deformation ring 
$\bbC[\bbC^\times\times\bbC^{\ell}]\to R=\bbC[\tau,\kappa,\kappa^{-1}]$ such that
$z\mapsto\kappa$ and $z_p\mapsto \zeta_p\,\tau$.
It is in general position if $\zeta$ is generic.
\end{ex}

\vspace{3mm}

\subsection{Lie algebras}\label{sec:liealgebras}
Let $R$ be a deformation ring. 

Set $\frakg_R=\frakg\frakl_{R,N}$.
Let $U(\frakg_R)$ be the enveloping algebra (over $R$) of $\frakg_R$.
Let $\frakt_R\subset\frakb_R\subset\frakg_R$
be the diagonal torus and
the Borel Lie subalgebra of upper triangular matrices. Let $\frakp_{R,\nu}\subset\frakg_R$ be the parabolic 
subalgebra spanned
by $\frakb_R$ and the Levi subalgebra
$\frakm_{R,\nu}=\mathfrak{gl}_{R,\nu_1}\oplus\cdots\oplus
\mathfrak{gl}_{R,\nu_\ell}.$

Let $e_{i,j}\in\frakg_R$ be the $(i,j)$-matrix unit, and set $e_i=e_{i,i}$.
Let $(\epsilon_i)$ be the basis of $\frakt^*_R$ dual to $(e_i)$. It identifies
$\frakt^*_R$ with $R^N$.
In a similar way we identify $\frakt_R=R^N$.

Let $\Pi,$ $\Pi^+$ be the sets of roots of $\frakg_R$ and $\frakb_R$.
We say that $\nu$ is \emph{regular} if $\frakm_{R,\nu}=\frakt_R$.
Let $\Pi_\nu$ be the set of roots of $\frakm_{R,\nu}$. Set
$\Pi^+_\nu=\Pi^+\cap\Pi_\nu$.

The dot action of the Weyl group $W$ on $\frakt^*_R$ is given by $w\bullet\lambda=w(\lambda+\rho)-\rho$, 
where $\rho=(0,-1,\dots,1-N)$. 
Two weights are \emph{linked} if they belong to the same orbit of the $\bullet$-action.

Consider the partition $[1,N]= J^\nu_1\sqcup J_2^\nu\sqcup\cdots\sqcup J^\nu_\ell$ given by
$i_p=1+\nu_1+\ldots+\nu_{p-1},$ $j_p=i_{p+1}-1$ and $J^\nu_p=[i_p,j_p]$
For each $k\in J^\nu_p$ we define $p_k=p$.
Set 
$\det_p=\sum_{i\in J^\nu_p}\epsilon_i$ and $\det=\sum_{p=1}^\ell\det_p.$

The weights in the subset $P=\bbZ^N$ of $P_R=R^N$
are called {\it integral weights}. 
Given a subset $S\subset R,$  we write
$S^\nu=\{\lam\in S^N\,;\,\lam_i- \lam_{i+1}\in\bbN,\,\forall i\neq j_1,j_2,\ldots,j_\ell\}.$
We call $P^\nu_R=R^\nu$ the set of the \emph{$\nu$-dominant} weights in $P_R$.

An $\ell$-partition $\lam\in\scrP^\nu$ can be viewed as an element in $\bbN^\nu$ by adding
zeroes to the right of each partition $\lam^p$ such that $l(\lam^p)\leqslant\nu_p$, i.e., we
identify the $\ell$-partition 
$\lambda=(\lambda^1,\lambda^2,\dots,\lambda^\ell)$ 
with the $N$-tuple 
$(\lambda^10^{\nu_1-l(\lambda^1)}\dots\lambda^\ell 0^{\nu_\ell-l(\lambda^\ell)}).$

Similarly, we can view the tuple $\tau_R\in R^\ell$ as a weight in $P_R$ by identifying it with
$\tau_R=\sum_p\tau_{R,p}\,\det_p$. To simplify we may abbreviate $\tau=\tau_R$.

Set $\rho_\nu=(\nu_1,\nu_1-1,\ldots,1,\ldots,\nu_\ell, \nu_\ell-1,\ldots,1)$. So, we have
$\rho_\nu+\t=(s_1,s_1-1,\dots,\t_{R,1}+1,s_2,s_2-1,\dots,\t_{R,2}+1,\ldots,\t_{R,\ell}+1).$
We identify the set of $\ell$-partitions $\scrP^\nu$ with a subset of $P^\nu_R$ via the injective map
\begin{equation}\label{varpi}
 \varpi : \scrP^\nu\to P^\nu+\t,\quad\lambda\mapsto \lambda+\rho_\nu+\t-\rho.
\end{equation}

The Casimir elements are 
$\omega=\sum_{i,j=1}^Ne_{ij}\otimes e_{ji}$ and
$\cas=\sum_{i,j=1}^Ne_{ij}e_{ji}$.
We may write $\omega_N=\omega,$ $\cas_N=\cas$ to avoid confusions.

\vspace{3mm}

\subsection{Definition of the category $\scrO$} \label{sec:catO}
A $\frakt_R$-module $M$ is called a \emph{weight $\frakt_R$-module} if it is a direct sum
of its \emph{weight submodules}
$M_\lambda=\{m\in M\,;\,xm=\lambda(x)m,\,x\in\frakt_R\}$ as $\lambda$ runs over $P_R$.
Let $\scrO^\nu_R$ be the $R$-category of finitely generated $U(\frakg_R)$-modules which are
weight $\frakt_R$-modules and such that the action of $U(\frakp_{R,\nu})$ is locally finite over $R$.

For $\lambda\in P^\nu_R,$ we consider the 
$U(\frakm_{R,\nu})$-module 
$V(\lambda)_{R,\nu}=V(\lambda')_{\bbC,\nu}\otimes R_{\lambda-\lambda'},$
where $\lambda'\in P^\nu$ is such that $\lambda-\lambda'$ is a character
of $\frakm_{R,\nu}$, $R_{\lambda-\lambda'}$ is $R$,
equipped with the representation of $\frakm_{R,\nu}$ given by this character, and
$V(\lambda')_{\bbC,\nu}$ is the finite-dimensional simple $\frakm_\nu$-module with
highest weight $\lambda'$.
We view $V(\lambda)_{R,\nu}$ as a $\frakp_{R,\nu}$-module  and define the parabolic
Verma module
$M(\lambda)_{R,\nu}=U(\frakg_R)\otimes_{U(\frakp_{R,\nu})}V(\lambda)_{R,\nu}.$
If $\nu$ is regular, we abbreviate $M(\lambda)_R=M(\lambda)_{R,\nu}$.

For $\lambda\in P_{K}^\nu$, let
$L(\lambda)_K$ be the unique simple quotient of
$M(\lambda)_{K,\nu}$.

Let $\scrO^\nu_{R,\t}$ 
be the full subcategory of $\scrO^\nu_R$ 
consisting of the modules whose weights belong to $P+\t$.
Note that $M(\lambda)_{R,\nu}\in\scrO^\nu_{R,\tau}$ if and only if $\lambda\in P^\nu+\tau,$
and that $\scrO^\nu_{K,\t}$ is the Serre subcategory of $\scrO^\nu_K$ 
generated by all the simple modules $L(\lambda)_K$ with $\lambda\in P^\nu+\tau$. 
For $\lambda\in\scrP^\nu$ we set
$\Delta(\lambda)_{R,\tau}=M(\varpi(\lambda))_{R,\nu}$.

If $R=\bbC$ or if $\tau=0$ we drop the subscripts
$R$ or $\tau$ from the notation.

\vspace{3mm}

\subsection{Definition of the category $A$}
 \label{sec:ATft}
 Let $R$ be a deformation ring.
Assume that $R$ is either a field or a local ring.

The category $\scrO^{\nu}_{R,\tau}$ is a highest weight $R$-category with 
$\Delta(\scrO^{\nu}_{R,\tau})=\{M(\lambda)_{R,\nu};\,\lambda\in P^\nu+\tau\}.$
If $R$ is a local ring with residue field $\Bbbk$, the specialization at $\Bbbk$ identifies
the poset $\Delta(\scrO^{\nu}_{R,\tau})$ with
$\Delta(\scrO^{\nu}_{\Bbbk,\tau})$. 

The partial order is given by the \emph{BGG-ordering} on $P^\nu_R$, which is
the smallest partial order such that
$\lambda\leqslant\lambda'$ if 
$[M(\lambda')_{\Bbbk,\nu}:L(\lambda)_\Bbbk]\neq 0.$

The \emph{linkage ordering} 
on $P^\nu_R$ is the
transitive and reflexive closure of the relation $\Uparrow$ given by
$\lambda\Uparrow\lambda'$
if and only if there are
$\beta\in\Pi(\lambda')$, $w\in W_\nu$ such that 
$\beta\notin\Pi_\nu$
and
$\lambda=ws_\beta\bullet\lambda'\in\lambda'-\bbN\Pi^+$ 
modulo $\frakm\,\widehat P_R$.

The linkage ordering and the BGG-ordering are known to be equivalent. 
We will use indifferently one or the other in the rest of the text.

\vspace{2mm}

\begin{df-prop} \label{prop:Z+} Assume that $\tau_{\Bbbk,u}-\tau_{\Bbbk,v}\notin\bbN^\times$ for each $u<v$.
There are unique highest weight $R$-subcategories 
$A^\nu_{R,\tau}, A^\nu_{R,\tau}\{d\}$ of $\scrO^{\nu}_{R,\t}$ 
with $\Delta(A^\nu_{R,\tau})=\{\Delta(\lambda)_{R,\tau}\,;\,\lambda\in\scrP^\nu\}$ and
$\Delta(A^\nu_{R,\tau}\{d\})=\{\Delta(\lambda)_{R,\tau}\,;\,\lambda\in\scrP^\nu_d\}.$
\end{df-prop}

\vspace{.5mm}

\begin{proof} It is enough to assume that $R=K$ is a field and to
prove that $\Delta(A^{\nu}_{K,\tau})$ is an ideal of the poset $\Delta(\scrO^{\nu}_{K,\tau}).$
To do so, we must check that if $\lambda\in\scrP^\nu$, $\mu\in P^\nu+\tau$ and 
$\beta\in\Pi\setminus\Pi_\nu,$ $w\in W_\nu$ are such that 
$\mu=ws_\beta\bullet\varpi(\lambda)$ and $\varpi(\lambda)-\mu\in\bbN\Pi^+$, 
then we have $\mu\in\varpi(\scrP^\nu)$.
Write $\beta=\alpha_{k,l}$ with $k<l$ and $k=i_u+x\leqslant j_u$, $l=i_v+y\leqslant j_v$. 
For each $a,b\in K$ we write $a>b$ if and only if $a-b\in\bbN^\times$. Then, we have
$u<v$ and 
\begin{equation}\label{eq333}
\lambda_k+s_{K,u}-x>\lambda_l+s_{K,v}-y,
\end{equation}
where $\lambda$ is viewed as a $N$-tuple
$(\lambda_1,\lambda_2,\dots,\lambda_N)$. We have
$$\gathered
\{(\mu+\rho)_a\,;\,i_u\leqslant a\leqslant j_u\}=\{\lambda_a+s_{K,u}-(a-i_u)
\,;\,i_u\leqslant a\leqslant j_u,\,a\neq k\}\cup\\
\cup\{\lambda_l+s_{K,v}-y\},\\
\{(\mu+\rho)_b\,;\,i_v\leqslant b\leqslant j_v\}=\{\lambda_b+s_{K,v}-(b-i_v)
\,;\,i_v\leqslant b\leqslant j_v,\,b\neq l\}\cup\\
\cup\{\lambda_k+s_{K,u}-x\}.
\endgathered$$
To prove that $\mu\in\varpi(\scrP^\nu)$,
we must check that
$$\min\{(\mu+\rho)_a\,;\,i_u\leqslant a\leqslant j_u\}\geqslant \tau_{K,u}+1,\quad
\min\{(\mu+\rho)_b\,;\,i_v\leqslant b\leqslant j_v\}\geqslant \tau_{K,v}+1.$$
By \eqref{eq333} and the assumption in the lemma, we have $\tau_{K,v}-\tau_{K,u}\in\bbN$.
Hence, the first inequality is true, 
because for any $\lambda\in\scrP^\nu,$ $i_u\leqslant a\leqslant j_u,$ we have
$\lambda_a+s_{K,u}-(a-i_u)\geqslant\tau_{K,u}+1,$ and 
$\lambda_l+s_{K,v}-y\geqslant\tau_{K,v}+1\geqslant\tau_{K,u}+1.$
Now, to prove the second one, observe that by \eqref{eq333} we have
$$\min\{(\mu+\rho)_b\,;\,i_v\leqslant b\leqslant j_v\}\geqslant
\min\{\lambda_b+s_{K,v}-(b-i_v)\,;\,i_v\leqslant b\leqslant j_v\}\geqslant\tau_{K,v}+1.$$
\end{proof}

\vspace{3mm}

\subsection{The categorical action on $\scrO$}
\label{sec:4.5}
Let $V_R$ be the natural representation of $\frakg_R$ on $R^N$. 
Let $V^*_R=\Hom_R(V_R,R)$ be the dual representation. 
We have a pre-categorical action $(e,f,X,T)$ on $\scrO^{\nu}_{R,\t}$
such that $$e(M)=M\otimes_RV_R^*, \quad f(M)=M\otimes_RV_R,$$ $X_M\in\End(f(M))$ is the left multiplication by 
the Casimir element $\omega$, and $T_M\in\End(f^2(M))$ is the left 
multiplication by $1\otimes\omega$,
see e.g., \cite[sec.~3.4]{BK1}.

Now, assume that $R=K$ is a field. Set $I=\{\tau_{K,1},\dots,\tau_{K,\ell}\}+\bbZ$. 

For each $\mu\in P^\nu+\tau,$ we write
$\wt(\mu)=\sum_{k=1}^N\varepsilon_{\langle\mu+\rho,\epsilon_k\rangle}$.
We have $\wt(\mu)\in X_{I}$ if and only if $\langle\mu,\epsilon_k\rangle\in I$ for all $k$.
Note that $\wt(\mu)=\sum_{i\in I}(m_i(\mu)-m_{i+1}(\mu))\,\Lambda_i$, where
$m_i(\mu)=\sharp\{k\in[1,N]\,;\,\langle\mu+\rho\,;\,\epsilon_k\rangle=i\}$.

For each $\lambda\in X_{I},$ 
let $\scrO^\nu_{K,\tau,\lambda}\subset \scrO^\nu_{K,\tau}$ be the Serre subcategory generated by  the 
modules
$L(\mu)_K$ such that $\mu\in P^\nu+\tau$ and
$\wt(\mu)=\lambda$.
The linkage principle yields the decomposition
$\scrO^\nu_{K,\tau}=\bigoplus_{\lambda\in X_I}\scrO^{\nu}_{K,\tau,\lambda}$.
This decomposition yields an $\fraks\frakl_I$-categorical action on
$\scrO_{K,\tau}^{\nu}.$

Let $V_I$ be the natural representation of $\fraksl_I$. It is
a representation with the basis $\{v_i\,;\,i\in I\}$. 
We have the following formulas, see, e.g., \cite[lem.~4.3]{BK1}.

\vspace{2mm}

\begin{prop} \label{prop:actiontf}
For $\lambda, \mu\in P^\nu_K$  we write 
$\lambda \overset{i}\to\mu$ if $\mu+\rho$ is obtained from 
$\lambda+\rho$ by replacing an entry equal to $i$ by $i+1$.

$(a)$ $f_i(M(\lambda)_{K,\nu})$ has a $\Delta$-filtration with sections of the form 
$M(\mu)_{K,\nu}$, one for each $\mu$ such that $\lambda \overset{i}\to\mu$,

$(b)$  $e_i(M(\lambda)_{K,\nu})$ has a $\Delta$-filtration with sections of the form 
$M(\mu)_{K,\nu}$, one for each $\mu$ such that $\mu \overset{i}\to\lambda$,

$(c)$ the elements $[L(\mu)_K]$, $[M(\mu)_{K,\nu}]$ in $[\scrO^\nu_{K,\tau}]$ 
are homogeneous of weight $\wt(\mu)$,

$(d)$ as an $\fraksl_I$-module, we have
$[\scrO^\nu_{K,\tau}]=\bigotimes_{p=1}^\ell\bigwedge^\nu(V_I)$.
\end{prop}

\vspace{3mm}

\subsection{Definition of the functor $\Phi$}\label{sec:catA}
Recall that $R$ is a deformation ring which is either a field or a local ring.

Let $h:A_{R,\tau}^{\nu}\to\scrO^\nu_{R,\tau}$ be the canonical embedding. 
Its left adjoint is $h^*$. Consider the endofunctors
$E,F$ of $A_{R,\t}^{\nu}$ given by $E=h^* e h$ and $F=h^* f h$. 
Since $f$ preserves the subcategory $A_{R,\t}^{\nu},$ we have
$F=f|_{A_{R,\t}^{\nu}}$.
So $F$ is exact and $(E,F)$ is an adjoint pair. 
Further, the endomorphisms $X,T$ of $f,f^2$ yield endomorphisms of
$F,F^2$.

Next,  consider the module
$T_{R,d}=T^\nu_{R,\tau}\{d\}=f^d(\Delta(\emptyset)_{R,\t})$
in $A^{\nu}_{R,\t}\{d\}$.
The algebra homomorphism $\phi_{f^d}$ 
factors through an $R$-algebra homomorphism \cite[lem.~3.4]{BK1}
$$\varphi^s_{R,d}:H^s_{R,d}\to\End_{A^\nu_{R,\tau}}(T_{R,d})^\op
=\End_{\scrO^\nu_{R,\tau}}(T_{R,d})^\op.$$ 
Composing $\Hom_{A^\nu_{R,\tau}}(T_{R,d},\bullet)$ with the pullback by 
$\varphi^s_{R,d}$ we get a functor
$$\Phi_{R,d}^s:A^\nu_{R,\tau}\to H^s_{R,d}\mmod.$$

\vspace{1mm}

\begin{rk}\label{rem:4.5}
To avoid confusions we may write
$A_{R,\tau}^\nu(N)=A^{\nu}_{R,\t}$,
$T_{R,d}(N)=T_{R,d}.$ 
\end{rk}

\vspace{.5mm}

\begin{rk} \label{rk:4.7}
For each $\frakp\in\frakP$,
the pre-categorification $(e,f,X,T)$ on $\scrO^{\nu}_{R,\t}$
yields a pre-categorification on $\scrO^{\nu}_{R_\frakp,\t}$
and $\scrO^{\nu}_{\Bbbk_\frakp,\t}$ by base-change.
It yields also a tuple
$(E,F,X,T)$ on $A^{\nu}_{R_\frakp,\t}$
and $A^{\nu}_{\Bbbk_\frakp,\t}$ as above.
In particular, this yields a module $T_{R_\frakp,d}$  in 
$A^\nu_{R_\frakp,\tau},$ an
$R_\frakp$-algebra homomorphism 
$\varphi^s_{R_\frakp,d}:H^s_{R_\frakp,d}\to\End_{A^\nu_{R_\frakp,\tau}}(T_{R_\frakp,d})^\op,$
and a functor
$\Phi_{R_\frakp,d}^s : A^\nu_{R_\frakp,\tau}\{d\}\to H^s_{R_\frakp,d}\mmod$.
\end{rk}

\vspace{2mm}

Now, assume that $R=K$ is a field and recall the following.

\vspace{2mm}

\begin{prop}[\cite{BK1}]\label{prop:isomBK}
Let  $\tau_{K,u}-\tau_{K,v}\notin\bbZ^\times$  all $u,v$. 

(a) Assume that $\nu_p\geqs d$ for all $p$. Then, the map $\varphi^s_{K,d}$ is a $K$-algebra isomorphism 
$H^s_{K,d}\to\End_{A^\nu_{K,\tau}}(T_{K,d})^\op$.

(b) Assume that $\nu$ is either dominant or anti-dominant. Then, the category  $A^\nu_{K,\tau}$ is a sum of blocks of  $\scrO^\nu_{K,\tau}$, 
the functors $E$, $F$ are biadjoint, the module
$T_{K,d}$ is projective in $A^\nu_{K,\tau}$ and a simple module of $A^\nu_{K,\tau}$ is a submodule
of a parabolic Verma module if and only if it lies in the top of $T_{K,d}$.

(c) Assume that $\tau_{K,u}-\tau_{K,v}\neq 0$ for all $u\neq v$ and that $\nu_p\geqs d$ for all $p$. Then, the category 
$A^\nu_{K,\tau}$ is split semi-simple.
Assume further that $\nu$ is either dominant or anti-dominant. 
Then $\Phi_{K,d}^s$ is an equivalence of 
$K$-categories which maps $\Delta(\lambda)_{K,\tau}$ to  $S(\lambda)^s_K$.
\end{prop}

\vspace{.5mm}

\begin{proof}
For $\nu$ dominant, part $(a)$ is proved in \cite[thm.~5.13, cor.~6.7]{BK1}.
For non-dominant $\nu,$ a proof is given in \cite[lem.~5.5]{B3} using \cite{BK1}. 
It can also be proved using \cite[lem.~5.4]{R2}.

Part $(b)$ is proved in \cite{BK1}.
For instance, the bi-adjointness of $E,F$ is obvious because $A^\nu_{K,\tau}$ is a sum of blocs of  $\scrO^\nu_{K,\tau}$,
and to prove the third claim
one checks first that $T_{K,0}$ is projective and then
that the functor $F$ preserves projective modules.
The last claim of $(b)$ is proved in \cite[thm.~4.8]{BK1}.

The first statement of $(c)$ follows from the linkage principle. 
By \cite[lem.~4.2]{BK1}, the module $T_{K,d}$ is a projective generator in this case. 
Therefore, the functor $\Phi_{K,d}^s$ is an equivalence of $K$-categories. It maps $\Delta(\lambda)_{K,\tau}$ to $S(\lambda)^s_K$ by \cite[thm.~6.12]{BK1}.
\end{proof}

\vspace{2mm}

\begin{rk}
Assume that $\nu_p\geqslant d$ and $\tau_{K,u}-\tau_{K,v}\notin\bbZ^\times$ for each $p,u,v$.
Then, the tuple $(E,F,X,T)$ 
define a pre-categorical action on $A^\nu_{K,\tau}$.
\end{rk}

\vspace{3mm}

\subsection{The category $A$ with two blocks}
\label{sec:2blocs}
If $\tau_{K,u}-\tau_{K,v}\in\bbZ_{<0}$ for some $u<v$, then
the category $A^\nu_{K,\tau}$ is well defined but it may not be a sum of blocks of  $\scrO^\nu_{K,\tau}$. 
In this section we generalize Proposition \ref{prop:isomBK} in order to allow 
\emph{integral} deformation parameters.
To simplify, we'll assume that $\ell=2$.
This is enough for our purpose. 
Similar results can be obtain for arbitrary $\ell$.
Note that, for $\ell=2,$ the composition $\nu$ is always either dominant or anti-dominant.

The aim of this section is to prove the following.

\vspace{2mm}

\begin{prop}\label{prop:4.25}
Assume that $\ell=2$, $\nu_1,\nu_2\geqslant d$ and $\tau_{K,1}-\tau_{K,2}\notin\bbN^\times$.
Put $s=\nu+\tau$.
Then, the following hold

(a) $\varphi^s_{K,d}$ is an isomorphism $H^s_{K,d}\to\End_{A^\nu_{K,\tau}}(T_{K,d})^\op$,

(b)  $T_{K,d}$ is projective in $A^\nu_{K,\tau}$,

(c) a simple module of $A^\nu_{K,\tau}$ is a submodule
of a parabolic Verma module if and only if it lies in the top of $T_{K,d}$.
\end{prop}

\vspace{2mm}

In order to prove this, we first prove the following.

\vspace{2mm}

\begin{prop}\label{prop:shift}
Assume that $\ell=2$, $\nu_1,\nu_2\geqslant d$ and $\tau_{K,1}-\tau_{K,2}\in\bbZ_{<0}$.
Set $\nu'=(\nu'_1,\nu'_2)$ and $\tau'_K=(\tau'_{K,1},\tau'_{K,2})$ with $\nu'=\nu+(0,1)$,
$\tau'_{K}=\tau_{K}-(0,1)$. Put $s=\nu+\tau$ and $s'=\nu'+\tau'$.
Then, we have $s=s'$ and there is an equivalence of highest weight categories
$A^\nu_{K,\tau}\{d\}\simeq A^{\nu'}_{K,\tau'}\{d\}$ which intertwines the morphisms
$\varphi^s_{K,d},\varphi^{s'}_{K,d}$ and the functors $\Phi_{K,d}^s,\Phi_{K,d}^{s'}$.
\end{prop}

\vspace{.5mm}

\begin{proof} The proof is rather long and consists of several steps.

Write $\frakg=\frakgl_{K,N}$, $\frakg'=\frakgl_{K,N+1}$ and
$e_{N+1}=\diag(0,\ldots,0,1)$.
Set also $\frakn=\bigoplus_{i=1}^N K e_{N+1,i}$ and $\fraku=\bigoplus_{i=1}^{N+1}K e_{i,N+1}$.

Fix $\varkappa\in K$.
Let $\frakg\MMod$ be the category of all $\frakg$-modules.
We define the functors
$$\calR:\frakg'\MMod\to \frakg\MMod,\quad M\mapsto \Ker_M(e_{N+1}-\varkappa)$$
$$\calI:\frakg\MMod\to\frakg'\MMod,\quad M\mapsto U(\frakg')\otimes_{U(\frakp)}(M
\otimes_K K_\varkappa)$$
where $\frakp=\frakp_{K,N,1}$ is the standard parabolic of type $(N,1)$ and $K_\varkappa$ is the obvious $\frakgl_{K,1}$-module. 
Let $\frakm=\frakm_{K,N,1}$ be the Levi subalgebra of $\frakp$.

Let $\calC_{\geqslant\varkappa}\subset\frakg'\MMod$ be the full subcategory of modules for
which $e_{N+1}$ is semi-simple with weights in $\varkappa+\bbN$.
The functor $\calR$ restricts to an exact functor $\calC_{\geqslant\varkappa}\to\frakg\MMod$, and
since $U(\frakg')=K[\frakn]\otimes_K U(\frakp)$, 
the functor $\calI$ takes values in $\calC_{\geqslant\varkappa}$.

\vspace{2mm}

\begin{lemma}
The functor $\calI:\frakg\MMod\to\calC_{\geqslant\varkappa}$ is exact, fully faithful, and is left adjoint to
$\calR:\calC_{\geqslant\varkappa}\to\frakg\MMod$.
\end{lemma}

\vspace{.5mm}

\begin{proof}
Let us first prove the adjointness. Given $M\in\frakg\MMod$, $L\in\frakg'\MMod$, we have 
$\Hom_{\frakg'}(\calI(M),L)\simeq
\Hom_{\frakg}(M,\Hom_{\bar\frakn}(K_\varkappa,L)).$
If $L\in \calC_{\geqslant\varkappa}$, then we have 
$\Hom_{\bar\frakn}(K_\varkappa,L)=
\Hom_{K e_{N+1}}(K_\varkappa,L)$.
We deduce that there is an isomorphism
$\Hom_{\frakg'}(\calI(M),L)\simeq
\Hom_{\frakg}(M,\calR(L)).$
So $\calI$ is left adjoint to $\calR$.

Now, let us prove the fully faithfulness of $\calI$. We have
$U(\frakg')\otimes_{U(\fraku)}K=K[\frakn]\otimes_K U(\frakg)$
as $(\frakm,\frakg)$-bimodules. The left $\frakm$-action comes from the adjoint action of
$K e_{N+1}$ on $\frakn$ and the diagonal adjoint action of $\frakg$. The right $\frakg$-action
is the opposite of the adjoint action of $\frakg$ on itself.
We have $\calI(M)\simeq K[\frakn]\otimes_K (M\otimes_K K_\varkappa)$ as an $\frakm$-module.
We deduce that the unit $1\to\calR\calI$ is invertible.
\end{proof}

\vspace{2mm}

\begin{lemma}
\label{le:extclosed}
Let $\calA,\calA'$ be two abelian artinian categories, and 
$\calI:\calA\to\calA'$ a fully faithful functor
with an exact right adjoint $\calR$. Then, the following hold

(a) the full subcategory $\Im(\calI)$ of $\calA'$ is extension closed,

(b) if $\calR$ induces an isomorphism $[\calA]\to[\calA']$ then $\calR,$ $\calI$ are inverse
equivalences of categories.
\end{lemma}

\vspace{.5mm}

\begin{proof}
The functor $\calI$ is a right exact, hence $\calI\calR$ is also right exact.
Given an exact sequence $0\to \calI(M)\to L\to \calI(M')\to 0$ in $\calA'$ with
$M,M'\in\calA$, we obtain a commutative diagram whose rows are exact sequences
$$\xymatrix{
0\ar[r] & \calI(M)\ar[r] & L\ar[r] & \calI(M')\ar[r] & 0 \\
 & \calI\calR\calI(M)\ar[r]\ar[u] & \calI\calR(L)\ar[r]\ar[u] & \calI\calR\calI(M')\ar[r]\ar[u]& 0.
}$$
The vertical maps are given by the counit $\calI\calR\to 1$.
Since $\calI$ is fully faithful, the unit $1\to \calR\calI$ is an isomorphism.
Thus, the left and right vertical maps are invertible.
It follows that the two sequences are actually isomorphic,
hence $\Im(\calI)$ is extension-closed. This proves part $(a)$.

To prove $(b)$, since $1\simeq \calR\calI$, it is enough to check that the counit  is an isomorphism 
$\calI\calR\to 1$.
Since $\calR$ is exact and since $\calR\calI\calR\iso \calR$ by adjunction,
for each $M\in\calA$ the kernel and the cokernel of $\calI\calR(M)\iso M$ are killed by $\calR$. Hence their
classes in the Grothendieck groups are 0. Hence they are both 0.
\end{proof}

\vspace{2mm}

\begin{cor}
The full subcategory $\Im(\calI)$ of $\calC_{\geqslant\varkappa}$
is extension-closed and $\calI,$ $\calR$ induce inverse equivalences  $\frakg\MMod\simeq\Im(\calI)$.
\end{cor}

\vspace{.5mm}

Let $\frakt$, $\frakt'$ be the Cartan subalgebras of $\frakg$, $\frakg'$.
Set $P_K=\frakt^*$, $P'_K=(\frakt')^*$.
We abbreviate $\scrO=\scrO_{K}(N)$ and $\scrO'=\scrO_{K}(N+1)$.
Given $\lambda\in P_K$, let
$M(\lambda)=M(\lambda)_{K}$ be the corresponding Verma module in $\scrO$.
For $\lambda'\in P'_K$, we define
$M(\lambda')\in\scrO'$ similarly.

We have $\calI(M(\lambda))\simeq M(\lambda'),$
where $\lambda'=\lambda+\varkappa\epsilon_{N+1}$.
Thus, we have  $\calR M(\lambda')\simeq \calR\calI(M(\lambda))\simeq M(\lambda).$
We deduce that $\calI,$ $\calR$ are inverse equivalences between the
category of $\Delta$-filtered $\frakg$-modules in $\scrO$ and the category 
of $\frakg'$-modules which are extensions of objects
$M(\lambda')$ with $\lambda'\in P_K+\varkappa\epsilon_{N+1}$.


Now, fix $d,\nu,\nu',\tau_K,\tau'_K$ as in Proposition \ref{prop:shift}.
Put $\varkappa=\tau_{2,K}+N$.
We abbreviate $\scrO^\nu=\scrO^\nu_{K}(N)$ and $\scrO^{\nu'}=\scrO^{\nu'}_{K}(N+1)$.
Write also $A=A_{K,\tau}^\nu(N)$ and $A'=A^{\nu'}_{K,\tau'}(N+1)$.
Let $\varpi$, $\varpi'$ be the maps \eqref{varpi} associated with the parabolic categories
$\scrO^\nu$, $\scrO^{\nu'}$.
For $\lambda\in P^\nu_K$, $\lambda'\in P^{\nu'}_K$ let
$M(\lambda)_\nu$, $M(\lambda')_{\nu'}$ be the parabolic
Verma modules $M(\lambda)_{K,\nu}$, $M(\lambda')_{K,\nu'}$ in $\scrO^\nu$, $\scrO^{\nu'}$.

Consider the sets of weights
$\calE(d)=\{\varpi(\lambda)\,;\,\lambda\in\scrP^2_d\}$ and
$\calE'(d)=\{\varpi'(\lambda)\,;\,\lambda\in\scrP^2_d\}$ in
$P^\nu_K$, $P^{\nu'}_K$ respectively.
Since $\nu_1,\nu_2\geqslant d$, we have
an isomorphism of posets
$\calE(d)\to\calE'(d)$ such that $\lambda\mapsto \lambda'=\lambda+\varkappa\epsilon_{N+1}$.


Let $\calQ:\scrO\to\scrO^\nu$ be the functor sending a module to its largest quotient
in $\scrO^\nu$. This is the left adjoint to the inclusion functor $\scrO^\nu\to\scrO$.
We define $\calQ':\scrO'\to\scrO^{\nu'}$ in the same way.

\vspace{2mm}

\begin{lemma}
The functors $\calQ'\calI$, $\calR$ induce inverse equivalences of highest categories
$A\{d\}\simeq A'\{d\}$. 
\end{lemma}

\vspace{.5mm}

\begin{proof}
Let $\lambda\in P^\nu_K$ and $\lambda'=\lambda+\varkappa\epsilon_{N+1}$.
Assume $\lambda'\in P^{\nu'}_K$. Let $\{\alpha_i\,;\,i\in I_\nu\}$ be the set of simple roots in $\Pi^+_\nu$.
There is an exact sequence
$$\bigoplus_{i\in I_{\nu'}}M(s_i\bullet\lambda')\to M(\lambda')\to
M(\lambda')_{\nu'}\to 0.$$
We have $s_i\bullet\lambda{\not\in} P^\nu_K$ for $i\in I_\nu$, hence
$\calQ M(s_i\bullet\lambda)=0$. 
So, for $i\neq n$ we have
$$\calQ\calR M(s_i\bullet\lambda')\simeq \calQ\calR M(s_i\bullet\lambda+\varkappa\epsilon_{N+1})\simeq
\calQ\calR\calI M(s_i\bullet\lambda)\simeq
\calQ M(s_i\bullet\lambda)=0.$$
On the other
hand, we have $\calR M(s_N\bullet\lambda')=0$ because
$M(s_N\bullet\lambda')\in\calC_{>\varkappa}$. 
Since $\calQ\calR$ is right exact, this yields an isomorphism
$\calQ\calR M(\lambda')_{\nu'}\simeq
\calQ\calR M(\lambda').$
Note that $\calR$ restricts to a functor $\scrO^{\nu'}\cap\calC_{\geqslant\varkappa}\to\scrO^\nu$.
We deduce that
$$\calR M(\lambda')_{\nu'}\simeq \calQ\calR M(\lambda')_{\nu'}\simeq 
\calQ\calR M(\lambda')\simeq \calQ M(\lambda)\simeq
M(\lambda)_\nu.$$
Thus, $\calR$ restricts to an exact functor $A'\{d\}^\Delta\to A\{d\}^\Delta$.
Since $A'\{d\}^\Delta$ contains a progenerator for $A'\{d\}$,
$\calR$ is right exact and $A\{d\}$ is preserved under taking quotients, 
we deduce that $\calR$ restricts to an exact functor $A'\{d\}\to A\{d\}.$ 
For a future use,  note also that $\calR$ yields an isomorphism
$[A'\{d\}]\to[A\{d\}]$.

Let $\calS$ be the endofunctor of $\scrO'$ sending a module to the quotient
by its largest submodule on which $e_{N+1}$ doesn't have the eigenvalue $\varkappa$.
Let us consider the functor $\calS\calI$ on $\scrO$.
It is right exact and takes values in $\calC_{\geqslant\varkappa}$.
For $N\in\scrO$, the module $\calS\calI(N)$ is the quotient of $\calI(N)$
by its largest submodule contained in $\calC_{>\varkappa}$.
Since $\calR$ is exact and vanishes on $\calC_{>\varkappa}$, we deduce that
$1\simeq \calR\calI\simeq \calR\calS\calI$ on $\scrO$.

Next, for $\lambda\in\calE(d)$
the counit $\calI\calR\to 1$ yields a map 
$\calI M(\lambda)_\nu\to M(\lambda')_{\nu'}$
which is obviously surjective. Let $M$ be its kernel.
Applying the exact functor $\calR$ to the exact sequence
$0\to M\to \calI M(\lambda)_\nu\to M(\lambda')_{\nu'}\to 0$ yields the exact sequence
$0\to \calR(M)\to M(\lambda)_\nu\to M(\lambda)_\nu\to 0.$
We deduce that $\calR(M)=0$. 
Since $M\in\calC_{\geqslant\varkappa}$, this implies that $M\in\calC_{>\varkappa}$. 
Thus, applying the right exact functor $\calS$ to the exact sequence above
yields the isomorphism
$\calS\calI M(\lambda)_\nu\simeq \calS M(\lambda')_{\nu'}.$
Now, the constituents of $M(\lambda')_{\nu'}$ have a highest weight of the
form $\mu'$ for some $\mu\in\frakt^*$, because $M(\lambda')_{\nu'}\in A'\{d\}$.
Hence, the only submodule of $M(\lambda')_{\nu'}$ contained in $\calC_{>\varkappa}$ is 0.
So $\calS M(\lambda')_{\nu'}\simeq 
M(\lambda')_{\nu'}$, hence
$\calS\calI M(\lambda)_\nu\simeq M(\lambda')_{\nu'}$.

Now, consider an exact sequence $0\to M_1\to M\to M_2\to 0$ in $A\{d\}^\Delta$.
Since $\calS\calI$ is right exact, we have an exact sequence
$\calS\calI(M_1)\to \calS\calI(M)\to \calS\calI(M_2)\to 0$. 
By induction on the length of a $\Delta$-filtration, we have $\calS\calI(M_1),\calS\calI(M_2)\in A'\{d\}$.
Thus, the image of the map $\calS\calI(M_1)\to \calS\calI(M)$ lies in $A'\{d\}$, hence
$\calS\calI(M)\in A'\{d\}$. We deduce that
$\calS\calI(A\{d\}^\Delta)\subset A'\{d\}$.
Since $A\{d\}^\Delta$ contains a progenerator for $A\{d\}$,
$\calS\calI$ is right exact and $A'\{d\}$ is preserved under taking quotients, 
we deduce that $\calS\calI$ restricts to a functor $A\{d\}\to A'\{d\}$. 

Finally, let us consider the functor $\calQ'\calI$.
Since $\calR$ takes $\scrO^{\nu'}\cap\calC_{\geqslant\varkappa}$ to $\scrO^\nu$, the functor
$\calQ'\calI:\scrO^{\nu}\to\scrO^{\nu'}\cap\calC_{\geqslant\varkappa}$ is left adjoint to $\calR$.
So $\calQ'\calI$ is right exact and we have an exact sequence
$$\bigoplus_{i\in I_\nu}M(s_i\bullet\lambda)\to M(\lambda)\to
M(\lambda)_\nu\to 0.$$
Since $s_i\bullet\lambda'{\not\in} P^{\nu'}_K$ for $i\in I_\nu$, we have
$\calQ'M(s_i\bullet\lambda')=0$, hence
$\calQ'\calI M(s_i\bullet\lambda)\simeq \calQ' M(s_i\bullet\lambda')=0.$
We deduce
that 
$$
\calQ'\calI M(\lambda)_\nu\simeq \calQ'\calI M(\lambda)\simeq \calQ'M(\lambda')\simeq
M(\lambda')_{\nu'}.$$
Therefore, since $\calQ'\calI$ is right exact and $\calQ'\calI M(\lambda)_\nu\simeq M(\lambda')_{\nu'}$, 
the same argument  as for $\calS\calI $, see above,
implies that $\calQ'\calI$ restricts to a functor $A\{d\}\to A'\{d\}$ which is left adjoint to $\calR$.

Next, we compare the functors $\calQ'\calI,$ $\calS\calI$ on $A\{d\}.$
For each $N\in A\{d\}$ we write $\calS\calI(N)=\calI(N)/L$ and $\calQ'\calI(N)=\calI(N)/M$. 
Since $d<\nu'_2=\nu_2+1$ and $\calQ'\calI(N)\in A'\{d\}$, the constituents of
$\calQ'\calI (N)$  are in $\calC_{\geqslant\varkappa}\setminus\calC_{>\varkappa}.$
Hence, the constituents of $I(N)$ which are in $\calC_{>\varkappa}$ are contained in $M$.
Since $L\in\calC_{>\varkappa}$, we deduce that $L\subset M$.
Thus we have an epimorphism $\calS\calI \to \calQ'\calI$ on $A\{d\}$.
Hence, since $\calR$ is exact, the isomorphism $1\to \calR\calS\calI$ and the unit $1\to \calR\calQ'\calI$ 
yield a commutative triangle
$$\xymatrix{1\ar[r]^-\sim\ar[dr]&\calR\calS\calI\ar@{>>}[d]\\&\calR\calQ'\calI,}$$
from which we deduce that the unit is surjective.
Now, by adjunction, composing the unit and counit gives the identity
$\calR\to \calR\calQ'\calI\calR\to \calR.$ Hence the unit is injective, hence is an isomorphism, on $\Im(\calR)$.
But, since $1\simeq \calR\calS\calI$, the functor $\calR:A'\{d\}\to A\{d\}$ is essentially surjective.
We deduce that $1\simeq \calR\calQ'\calI$ on $A\{d\}$.

Therefore, the functor $\calR:A'\{d\}\to A\{d\}$ is exact and
yields an isomorphism
$[A'\{d\}]\to[A\{d\}]$, while $\calQ'\calI :A\{d\}\to A'\{d\}$ is a fully
faithful left adjoint. Hence, Lemma \ref{le:extclosed} shows that 
$\calQ'\calI ,$ $\calR$ are inverse equivalences of categories.
\end{proof}

Recall the set $I=\{\tau_{K,1},\dots,\tau_{K,\ell}\}+\bbZ$.

\vspace{2mm}

\begin{lemma}
The functors $\calQ'\calI$, $\calR$ between
$A\{d\}$, $A'\{d\}$ commute with
$E_i, F_i,X,T$ (whenever $E_i$, $F_i,$ $i\in I$, make sense).
\end{lemma}

\vspace{.5mm}

\begin{proof}
Since $\calQ'\calI$, $\calR$ are inverse equivalences, it is enough to consider the case of $\calR$.
Next, since $(E,F)$ is an adjoint pair, by unicity of the left adjoint, it is enough to consider the
case of the functor $F$.
Let $V_N=\bigoplus_{i=1}^NK v_i$.
Let $M\in\frakg'\MMod$.

If $M\in\calC_{\geqslant\varkappa},$ then
$V_{N+1}\otimes_K M\in\calC_{\geqslant\varkappa}$ and the decomposition
$V_{N+1}=V_N\oplus K e_{N+1}$ yields an isomorphism
$\calR(V_{N+1}\otimes_K M)=V_N\otimes_K \calR(M)$, because $\Ker_M(e_{N+1}-\varkappa+1)=0$.
So, we have an isomorphism of functors 
$\calR\circ f\simeq f\circ \calR:\calC_{\geqslant\varkappa}\to\frakg\MMod$.
Since $\calR$ takes $A'\{d\}$ to $A\{d\}$, and since $f$ preserves the categories $A,A'$,
this yields an isomorphism of functors 
$\calR\circ f\simeq f\circ \calR:A'\{d\}\to A\{d+1\}$.
We deduce that the functors
$F_i\{d\}:A'\{d\}\to A'\{d+1\}$ and $F_i\{d\}:A\{d\}\to A\{d+1\}$
are intertwined by $\calR$ whenever they are defined (i.e., if $i\in I\setminus\{\varkappa-N+1\}$).

Let $i:V_N\otimes M\to V_{N+1}\otimes M$ be the canonical inclusion and 
$p:V_{N+1}\otimes M\to V_N\otimes M$ be the canonical projection.
We have $p\circ \omega_{N+1}\circ i=\omega_N$. It follows that
the action of $X$ commutes with the isomorphism
$\calR\circ f\iso f\circ \calR$.
It is clear that the induced isomorphism $\calR\circ f^2\iso f^2 \circ R$
commutes with the action of $T$. 
\end{proof}

This finishes the proof of Proposition \ref{prop:shift}.
\end{proof}

Now, we can prove Proposition \ref{prop:4.25}.

\begin{proof} [Proof of  Proposition \ref{prop:4.25}]
We may assume that  $\tau_{K,1}-\tau_{K,2}\in\bbZ_{<0}$.
Set $\nu'_1=\nu_1$, $\tau'_{K,1}=\tau_{K,1}$, $\nu'_2=\nu_2+\tau_{K,2}-\tau_{K,1}$ 
and $\tau'_{K,2}=\tau'_{K,1}$. Recall that $s=\nu+\tau$ and $s'=\nu'+\tau'$.
By Proposition \ref{prop:shift}, there is an equivalence of highest weight 
categories
$\Upsilon:A^\nu_{K,\tau}\{d\}\to A^{\nu'}_{K,\tau'}\{d\}$ which intertwines the morphisms
$\varphi^s_{K,d}$, $\varphi^{s'}_{K,d}$ and the functors $\Phi_{K,d}^s$, $\Phi_{K,d}^{s'}$.
In particular, we have $\Upsilon(T_{K,d})=T_{K,d}$, see the proof of Proposition \ref{prop:shift}.
Now, we can apply Propositions \ref{prop:isomBK}
to $A^{\nu'}_{K,\tau'}\{d\}$, because $\tau'_{K,1}=\tau'_{K,2}$. 
This proves the proposition.
\end{proof}

\vspace{2mm}

\begin{rk}
Under the hypothesis in Proposition \ref{prop:4.25}, the tuple $(E,F,X,T)$ is a pre-categorical action on
$A^\nu_{K,\tau}$.
\end{rk}

\vspace{3mm}

\subsection{The categories $A$ and $\scrO$ of a pseudo-Levi subalgebra}\label{sec:notation4}
Fix a pair of distinct elements $u,v\in [1,\ell]$. 
We will represent an $(\ell-1)$-tuple $a$ as a collection of elements
$a_\bullet$, $a_p$ with $p\in[1,\ell]\setminus\{u,v\}$.
If $a$ is an $\ell$-tuple of elements of a ring we write 
$a_\circ=(a_u,a_v)$ and $a_\bullet=a_u+a_v$.
Finally, we consider the positive root system
$\Pi_{\nu,u,v}^+=\Pi^+\cap\Pi_{\nu,u,v}$ with
$\Pi_{\nu,u,v}=\big\{\alpha_{k,l}\,;\, p_k=p_l\,\text{or}\,(p_k,p_l)=(u,v),(v,u)\big\}.$

We will be interested by two types of pseudo-Levi subalgebras of $\frakg_R$ :

\begin{itemize}
\item first, we have  the  Lie subalgebra
$\frakm_{R,\nu}$ associated with
$\Pi_{\nu}$,
\item next, we have the Lie subalgebra
$\frakm_{R,\nu,u,v}$ associated with $\Pi_{\nu,u,v}$. 
\end{itemize}

To each of these Lie algebras we associate a module category.
To do so, fix a composition $\gamma_p$ of $\nu_p$ for each $p.$ 

First, for each tuple $a=(a_p)\in\bbN^\ell$ 
we write $P\{a\}=\{\lambda\in P\,;\,\langle\lambda,\det_p\rangle=a_p,\,\forall p\}$ and
$P^\nu\{a\}=P^\nu\cap P\{a\}$.
Consider the categories of 
$\frakm_{R,\nu}$-modules given by (the tensor product is over $R$)
\begin{equation}\label{4.5}\scrO^\gamma_{R,\tau}(\nu)=
\bigotimes_{p=1}^\ell\scrO^{\gamma_p}_{R,\tau_p}(\nu_p),
\qquad\scrO^\gamma_{R,\tau}(\nu)\{a\}=
\bigotimes_{p=1}^\ell\scrO^{\gamma_p}_{R,\tau_p}(\nu_p)\{a_p\}.
\end{equation}

Next, for each tuple $a=(a_\bullet,a_p)\in\bbN^{\ell-1}$, we set
$P\{a\}=\{\lambda\in P\,;\,\langle\lambda,\det_\bullet\rangle=a_\bullet,\,
\langle\lambda,\det_p\rangle=a_p\}$ and
$P^\nu\{a\}=P^\nu\cap P\{a\}$.
Consider the categories of 
$\frakm_{R,\nu,u,v}$-modules  given by
\begin{equation}\label{2blocs}
\scrO^\gamma_{R,\tau}(\nu,u,v)=
\scrO^{\gamma_\circ}_{R,\tau_\circ}(\nu_\bullet)
\otimes_R\bigotimes_{p\neq u,v}
\scrO^{\gamma_p}_{R,\tau_p}(\nu_p),
\end{equation}
\begin{equation}\label{2blocsd}
\scrO^\gamma_{R,\tau}(\nu,u,v)\{a\}=
\scrO^{\gamma_\circ}_{R,\tau_\circ}(\nu_\bullet)\{a_\bullet\}
\otimes_R\bigotimes_{p\neq u,v}
\scrO^{\gamma_p}_{R,\tau_p}(\nu_p)\{a_p\}.
\end{equation}

We will be mainly interested by the two extreme cases where
$\gamma_p=(\nu_p)$ for each $p$, or where $\gamma_p=(1^{\nu_p})$ for each $p$.
In the first case, we get the categories
$\scrO^\nu_{R,\tau}(\nu)$, $\scrO^\nu_{R,\tau}(\nu,u,v)$, in the second one we get the categories
$\scrO_{R,\tau}(\nu)$, $\scrO_{R,\tau}(\nu,u,v).$

We will also use highest weight subcategories
$A^\nu_{R,\tau}(\nu)\subset\scrO^\nu_{R,\tau}(\nu)$ and
$A^\nu_{R,\tau}(\nu, u,v)\subset\scrO^\nu_{R,\tau}(\nu,u,v)$
which are defined as in Definition \ref{prop:Z+}. 
They decompose in a similar way as in \eqref{4.5}-\eqref{2blocsd}.
We will write
$\Delta(A^\nu_{R,\tau}(\nu))=\{\Delta(\lambda)_{R,\tau}\,;\,\lambda\in\scrP^\nu\}$
and
$\Delta(A^\nu_{R,\tau}(\nu,u,v))=\{\Delta(\lambda)_{R,\tau}\,;\,\lambda\in\scrP^\nu\},$
hoping it will not create any confusion.

Using \eqref{4.5}, \eqref{2blocs} and the pre-categorification $(e,f,X,T)$ on $\scrO^{\nu}_{R,\t}$
introduced in \S \ref{sec:4.5},
we define 
a pre-categorification $(e,f,X,T)$ on $\scrO^\nu_{R,\tau}(\nu)$, $\scrO^\nu_{R,\tau}(\nu,u,v)$
such that, in both cases, the functors $e,$ $f$ are the direct sums of the functors $e,$ $f$
of each of the factors.

Next, using the canonical embeddings we define
tuples $(E,F,X,T)$ on
$A^\nu_{R,\tau}(\nu)$ and
$A^\nu_{R,\tau}(\nu, u,v)$ as in \S \ref{sec:catA}.

\vspace{3mm}

\section{The category $\bfO$}
\label{sec:parabolic}

Fix integers $\ell,N\geqslant 1$ and fix a composition $\nu\in\scrC^\ell_{N,+}$.
Recall that $\frakg_R=\frakg\frakl_{R,N}$.

Let $R$ be a deformation ring.
Thus, we have elements $\kappa_R\in R^\times$ and $\tau_{R,p}\in R$ for $p\in[1,\ell]$.
For each $p,$ we define $s_{R,p}\in R$ by
$s_{R,p}=\nu_p+\t_{R,p}.$

We may abbreviate $\kappa=\kappa_R$, $s_p=s_{R,p}$ and $\tau_p=\tau_{R,p}$.

\vspace{.5mm}

\subsection{Analytic algebras}\label{sec:analytic}
Fix an integer $d\geqslant 1$.

Fix a compact polydisc $D\subset\bbC^d.$ 
Here, we view $\bbC^d$ as a Stein analytic space.
By an \emph{analytic algebra} we'll mean
the localization $R$ of the ring of germs of holomorphic functions on $D$
with respect to some multiplicative subset. 
See \cite{A}, \cite{GR} for more details on analytic algebras.
The following properties hold
\begin{itemize}
\item $R$ is a noetherian regular ring of dimension $d$,
\item $R$ is a UFD, hence every height 1 prime ideal is principal,
\item for any maximal ideal $\frakm\in\frakM$, the localization $R_\frakm$ of $R$
is a henselian local $\bbC$-algebra. 
\end{itemize}

Since $R$ is an analytic algebra, 
for any entire function $f=\sum_{n\in\bbN}a_nz^n$ on $\bbC$ and for any $x\in R,$
the series $\sum_{n\in\bbN}a_nx^n$ is convergent and defines an element $f(x)$ in $R$.
In particular, we have a well-defined element $\exp(x)\in R$. Analogously, for any analytic function
$f: [0,1]\to M_n(R)$ and for any $v\in R^n$, there is a unique analytic function
$v(t)$ on $[0,1]$ with values in $R^n$ such that $v(0)=v$ and $dv(t)/dt=f(t)v(t)$.

An \emph{analytic deformation ring} is an analytic algebra $R$ which is also a deformation ring.
Then, we may view $\kappa_R,$ $\tau_{R,p}$ as germs of holomorphic functions on $D$. 
We will \emph{always} assume that
$\kappa_R(D)\subset\bbC\setminus\bbR_{\geqslant 0}$. 
Thus, for any closed point $R\to\bbC$ 
the element $\kappa_\bbC$ belongs to $\bbC\setminus\bbR_{\geqslant 0}$.

Note that if $R$ is an analytic algebra of dimension $\geqslant 2,$ 
then we can always choose some deformation parameters $\kappa_R,$ $\tau_{R,p}$ 
such that $R$ is in general position.

For an analytic deformation ring $R$ we write
$q_R=\exp(-2\pi\sqrt{-1}/\kappa_R)$ and $Q_{R,p}=q_R^{s_p}=\exp(-2\pi\sqrt{-1}s_{R,p}/\kappa_R).$
We may abbreviate
$q=q_R,$ $Q_p=Q_{R,p}$ and $\kappa=\kappa_R$.

\vspace{3mm}

\subsection{Affine Lie algebras}

\subsubsection{Notations}\label{sec:affine}
Let $L\frakg_R=\frakg\otimes R[t,t^{-1}]$ and
let $\bfg'_R$ be the Kac-Moody central extension of $L\frakg_R$ by $R$.
Let $\bfone$ be the canonical central element 
and let $\partial$ be the derivation of $\bfg'_R$ acting as $t\partial_t$ on $L\frakg_R$
and acting trivially on $\bfone$.

Put $\bfg_R=R\partial\oplus\bfg'_R$ and
$\bft_R=R\partial\oplus R\bfone\oplus\frakt_R.$
Let $\bfb_R,\bfp_{R,\nu}\subset\bfg_R$
be the preimages of $\frakb_R$ and $\frakp_{R,\nu}$ under
the projection 
$R\partial\oplus R\bfone\oplus(\frakg\otimes R[t])\to\frakg_R$.
The element $c=\kappa_R-N$ of $R$ is called the \emph{level}.
Consider the $R$-algebras
$\bfg_{R,\kappa}=U(\bfg_R)/(\bfone-c)$
and
$\bfg'_{R,\kappa}=U(\bfg'_R)/(\bfone-c).$
For $d\in\bbN$ we set
$\bfg_{R,\geqslant d}=\frakg\otimes t^dR[t],$
$\bfg'_{R,+}=R\bfone\oplus\bfg_{R,\geqslant 0}$ and
$\bfg_{R,+}=R\partial\oplus \bfg'_{R,+}.$

For a $\bfg'_{R,+}$-module $M$ of level $c$
we consider the induced module
$
\Indc_R(M)=\bfg'_{R,\kappa}\otimes_{U(\bfg'_{R,+})}M.
$
We can view a $\frakg_R$-module as a
$\bfg'_{R,+}$-module of level $c$ where $\bfg_{R,\geqslant 1}$ acts trivially.
Write again $\Indc_R(M)$ for the corresponding induced module.

For $d\geqslant 1$ let $Q_{R,d}\subset\bfg_{R,\kappa}$ be the $R$-submodule spanned by the
products of $d$ elements of $\bfg_{R,\geqslant 1}$. Set
$Q_{R,0}=R$. Given a $\bfg_{R,\kappa}$-module $M,$ let
$M(d),\,M(-d)\subset M$
be the annihilator of $Q_{R,d}$ and of $Q_{R,-d}={}^\sharp Q_{R,d}$ respectively.
Set
$M(\infty)=\bigcup_{d\in\bbN} M(d)$
and
$M(-\infty)=\bigcup_{d\in\bbN} M(-d)$.
Note that $M(d)$ is a
$\bfg_{R,+}$-submodule of $M$ and that $M(\infty)$,
$M(-\infty)$ are $\bfg_R$-submodules of $M$. 

A $\bfg_{R,\kappa}$-module $M$ is \emph{smooth} if $M=M(\infty)$ and if $M$ 
is flat over $R$. 
Let $\scrS_{R,\kappa}$ be the category of the smooth $\bfg_{R,\kappa}$-modules.

For each $\xi\in\frakg$ and $r\in\bbZ$, let
$\xi^{(r)}$ be the element $\xi\otimes t^r$.
For  each $s\in\bbZ,$ the \emph{Sugawara operator} $\frakL_s$ is the formal sum
\begin{equation*}
\frakL_s={1\over 2\kappa}\sum_{r\geqslant -s/2}\sum_{i,j=1}^Ne_{i,j}^{(-r)}e_{j,i}^{(r+s)}
+{1\over 2\kappa}\sum_{r<-s/2}\sum_{i,j=1}^Ne_{i,j}^{(r+s)}e_{j,i}^{(-r)}
\end{equation*}
It lies in a completion of $\bfg_{R,\kappa}$ and it satisfies the relation
$[\frakL_s,\xi^{(r)}]=-r\xi^{(r+s)}$.
The affine Casimir element is $\afcas=\partial+\frakL_0$.

If $R=\bbC$ we'll drop the subscript $R$ everywhere from the notation.

\vspace{3mm}

\subsubsection{Affine root systems}
The elements of $\bft_R$ and
$\widehat P_R=\bft^*_R$ are called  \emph{affine coweights} and  \emph{affine weights}
respectively. Let $\widehat\Pi$
be the set of roots of $\bfg_R$ and let $\widehat\Pi^+$ be
the set of roots of $\bfb_R$.
We will call an element of $\widehat\Pi$ an
\emph{affine root}.  Let $\widehat\Pi_{re}$ be the system
of \emph{real roots}. The set of simple roots in $\widehat\Pi^+$ is
$\{\alpha_0,\alpha_1,\dots,\alpha_{N-1}\}.$
Let $\check\alpha\in\bft_R$ be the affine coroot associated
with the real affine root $\alpha$.

Let 
$(\bullet:\bullet):\widehat P_R\times\bft_R\to R$
be the canonical pairing.
Let $\delta$, $\Lambda_0$, $\tilde\rho$ be the affine weights given by
$(\delta:\partial)=(\Lambda_0:\bfone)=1,$
$(\Lambda_0:R\partial\oplus\frakt_R)=(\delta:\frakt_R\oplus R\bfone)=0$
and $\tilde\rho=\rho+N\Lambda_0$.
We will use the identification 
$\widehat P_R=R\delta\oplus P_R\oplus R\Lambda_0=R\times P_R\times R$
given by
$\alpha_i\mapsto(0,\alpha_i,0)$ if $i\neq 0,$
$\Lambda_0\mapsto (0,0,1)$ and $\delta\mapsto(1,0,0)$.

Let 
$\langle\bullet :\bullet \rangle:\widehat P_R\times\widehat P_R\to R$
be the non-degenerate symmetric bilinear form given by
$(\lambda:\check\alpha_i)=
2\langle\lambda:\alpha_i\rangle/\langle\alpha_i:\alpha_i\rangle$
and
$(\lambda:\bfone)=\langle\lambda:\delta\rangle$.
It yields an isomorphism $\nu:\frakt_R\to\frakt^*_R$.
Using it we identify $\check\alpha$
with an element of $\widehat P_R$
for any $\alpha\in\widehat\Pi_{re}$.

Let
$\widehat W=W\ltimes\bbZ\Pi$ be the affine Weyl group and
let $s_i=s_{\alpha_i}$ 
be the simple affine reflections relatively to $\alpha_i$.
The group
$\widehat W$ acts on $\widehat P_R$. For
$x\in\frakt_R$ 
let $T_x\in\End(\widehat P_R)$ be the operator given by
$$T_x(\lambda) = \lambda + \langle \lambda :\bfone\rangle\,\nu(x)-\big(\langle\lambda,x\rangle+
(\nu(x):\nu(x))\,\langle\lambda:\bfone\rangle/2\big)\,\delta.$$
The action of the reflection with respect to
the affine real root $\alpha+r\delta$, with $\alpha\in\Pi$ and $r\in\bbZ$, is given by
$s_{\alpha+r\delta}=s_\alpha\circ T_{r\check\alpha}$.
The {\it $\bullet$-action} of $\widehat W$  is given by
$w\bullet\mu=w(\mu+\tilde\rho)-\tilde\rho$ for each
$\lambda\in P_R$ and $\mu\in\widehat P_R$.
Two weights in $\widehat P^\nu_R$ are \emph{linked} if they belong to the same orbit of the $\bullet$-action.

The set of \emph{integral affine weights} is
$\widehat P=\bbZ\delta+P+\bbZ\Lambda_0.$
Replacing $P$ by $P^\nu$ in the definitions above we get the corresponding
sets of integral $\nu$-dominant affine weights $\widehat P^\nu$. 
We define the set $\widehat P^\nu_R\subset \widehat P_R$ 
of $\nu$-dominant affine weights in the obvious way.
To  $\lambda\in P^\nu_R$ we set
$z_\lambda=-\langle\lambda:2\rho+\lambda\rangle/2\kappa$
and we associate the affine weight
$\widehat\lambda=(z_\lambda,\lambda,c)\in\widehat P^\nu_R.$
For $w\in W$, $x\in\bbZ\Pi$ and $\lambda\in P_R$ we have
$w\bullet\widehat\lambda=\widehat{w\bullet\lambda}$ and
$T_x\bullet\widehat\lambda=\widehat{\lambda+\kappa x}.$

\vspace{3mm}

\subsection{The category $\bfO$}

\subsubsection{Definition}

A $\bft_R$-module $M$ is called a \emph{weight $\bft_R$-module} if it is a direct sum
of the \emph{weight submodules}
$M_\lambda=\{m\in M\,;\,xm=\lambda(x)m,\,x\in\bft_R\}$ with $\lambda\in\widehat P_R$.

Let $\bfO^{\nu,\kappa}_R$ be the $R$-linear abelian category of finitely generated $\bfg_{R,\kappa}$-modules $M$ 
such that $M$ is a weight $\bft_R$-module, the $\bfp_{R,\nu}$-action on $M$ is locally finite over $R$, and the 
highest weight of any subquotient of $M$ is of the form $\widehat\lambda$ with $\lambda\in P_R^\nu$.


For  each $\mu\in\widehat P^\nu_R,$ let
$M(\mu)_{R,\nu}$ be the parabolic Verma module with the highest weight $\mu$.
For $\lambda\in P_R^\nu$ we have
$M(\widehat\lambda)_{R,\nu}=\scrI\!nd(M(\lambda)_{R,\nu}).$
Here $\partial,\bfone$ 
act on $M(\lambda)_{R,\nu}$ by multiplication by $z,c$ respectively.
If $R=K$ is a field, let $L(\mu)_K$ denote the top of $M(\mu)_{K,\nu}$.
For $\lambda\in P^\nu_R$ we abbreviate $\bfM(\lambda)_{R,\nu}=M(\widehat\lambda)_{R,\nu}$
and $\bfL(\lambda)_K=L(\widehat\lambda)_K$.

 If $\nu$ is regular,
we write $\bfO_R=\bfO^{\nu,\kappa}_R$ and $\bfM(\lambda)_R=\bfM(\lambda)_{R,\nu}$. 
If $\frakp_\nu=\frakg$ we write $\bfO^{+,\kappa}_R=\bfO^{\nu,\kappa}_R$ and $\bfM(\lambda)_{R,+}=\bfM(\lambda)_{R,\nu}$. 
If $R=\bbC$ we omit the subscript $\bbC$ from the notation.

Let $\bfO^{\nu,\kappa,f}_R\subset \bfO^{\nu,\kappa}_R$ be the full subcategory
consisting of the modules whose weight spaces are free of finite rank over $R$.
Let $\bfO^{\nu,\kappa,\Delta}_R\subset\bfO^{\nu,\kappa,f}_R$ be the full extension closed additive subcategory
generated by the parabolic Verma modules. The category $\bfO^{\nu,\kappa,\Delta}_R$
consists of the modules $M\in\bfO^{\nu,\kappa,f}_R$ such that $\Bbbk M\in\bfO^{\nu,\kappa,\Delta}_\Bbbk$ 
for each $\Bbbk\in\frakM$.

Given $\t\in P_R$ as in \S \ref{sec:liealgebras},
let $\bfO^{\nu,\kappa}_{R,\t}\subset\bfO^{\nu,\kappa}_R$ be the full subcategory consisting of the modules  $M$
such that the highest weight of any subquotient of $M$ is of the form 
$\widehat{\lambda+\tau}$ with $\lambda\in P^\nu$.
We set $\bfO^{\nu,\kappa,\Delta}_{R,\t}=\bfO^{\nu,\kappa}_{R,\t}\cap\bfO^{\nu,\kappa,\Delta}_{R}$.
If $R=\bbC$ or $\tau=0$  we drop the subscripts
$R$ or $\tau$ from the notation.

\vspace{2mm}

\begin{rk} 
\label{rk:gradings}
The operator
$\afcas$ acts locally nilpotently on any module of $\bfO^{\nu,\kappa}$.
Replacing this condition by \emph{$\afcas$ is locally finite} yields a bigger category
which decomposes as the direct sum
$\bigoplus_{a\in\bbZ}\bfO^{\nu,\kappa}[a]$,
where
$\bfO^{\nu,\kappa}[a]$ consists of the modules such that
$\afcas-a$ is locally nilpotent.

More generally, for each $d\in\bbZ$, we may consider the category
$\bfO^{\nu,\kappa}_{R,\tau}[a]\{d\}$ which consists of the modules 
whose subquotients have highest weights of the form
$(z_{\lambda+\tau}+a,\lambda+\tau,c)$ with
$\lambda\in P^\nu\{d\}.$
Here, we set $P\{d\}=\{\lambda\in P\,;\,\langle\lambda,\det\rangle=d\}$
and $P^\nu\{d\}=P^\nu\cap P\{d\}$.
To insist on the rank of $\frakg\frakl_{N}$
we may write $\bfO_{R,\tau}^\nu(N)=\bfO^{\nu,\kappa}_{R,\tau}$. 
We will use similar notation for all related categories, e.g.,
we may write
$\bfO^{\nu,\kappa}_{R,\tau}(N)[a]\{d\}=\bfO^{\nu,\kappa}_{R,\tau}[a]\{d\}$.
\end{rk}

\vspace{2mm}

\begin{rk} 
\label{rk:soergel}
In \cite{KL} the authors set $R=\bbC$ and consider a category $\bfO'$
of $\bfg'$-modules, rather than $\bfg_{\kappa}$-modules as above. 
Forgetting the $\partial$-action gives an equivalence $\bfO^{+,\kappa}\to\bfO'$. 
A quasi-inverse takes a $\bfg'$-module $M$ to itself, with the action of $\partial$ 
equal to the semi-simplification of $-\frakL_0$. See \cite[prop.~8.1]{S1} for details.

More generally, forgetting the $\partial$-action gives again an equivalence from
$\bfO^{\nu,\kappa}_R$ to a category of $\bfg'_R$-modules, and we may identify both categories. In particular,
for $M\in\scrO_R^\nu$ we can view the $\bfg'_{R,\kappa}$-module
$\Indc_R(M)$ as an object of $\bfO^{\nu,\kappa}_R$.

We will use this identification without further comments whenever it is necessary.
\end{rk}

\vspace{3mm}

\subsubsection{Basic properties} 
\label{sec:basic}
Let $R$ be either a field or a local ring. 

Let $e=-\kappa_\Bbbk,$ where $\kappa_\Bbbk$ is the residue class 
of $\kappa_R$. \emph{We will always assume that $e$ is a positive integer}.

For a $\bfg_{R,\kappa}$-module $M$ we set
\begin{itemize}
\item ${}^\sharp M=M$ with the $\bfg_R$-action twisted by the automorphism
$\sharp$ such that $\xi^{(r)}\mapsto(-1)^r\xi^{(-r)}$ and $\bfone\mapsto-\bfone,$

\item ${}^\dag\! M=M$ with the $\bfg_R$-action twisted by the automorphism 
$\dag$ such that $\xi^{(r)}\mapsto-{}^t\xi^{(r)}$ and $\bfone\mapsto\bfone,$

\item $M^*$ is the $R$-dual of $M$
with the $\bfg_R$-action given by 
$(\xi^{(r)}\varphi,m)=-(\varphi,\xi^{(r)}m)$ and
$(\bfone\varphi,m)=-(\varphi,\bfone m).$
\end{itemize}

We define the $\bfg_{R,\kappa}$-modules $DM$, $\scrD M$ by
$DM=({}^\sharp M^*)(\infty)$ and $\scrD M={}^\dag\! DM.$

\vspace{2mm}

\begin{lemma}\label{lem:duality}
The functor $D$ is a duality on
$\bfO^{+,f}_R$ 
and $\scrD$ is a duality on $\bfO^{\nu,\kappa,f}_R$. 
Both commute with base change.
\end{lemma}

\vspace{.5mm}

\begin{proof}
For any $M\in\bfO^{\nu,\kappa}_R,$ 
the $R$-module $D M$ consists of those linear forms in $M^*$
which vanish on $Q_{R,-d}M$ for some $d\geqslant 1$.
Hence, we have $\scrD M={}^{\dag\sharp} M^\circledast$, where
$M^\circledast$ is the set of $\frakg_{\nu}$-finite elements of $M^*$.
Since the automorphism $\dag\sharp$ takes the Borel subalgebra $\bfb_R\subset\bfg_R$
to its opposite, the functor $\scrD$ preserves $\bfO^{\nu,\kappa,f}_R$.  It is the usual BGG duality,
which fixes the simple objects when $R$ is a field.

For any $M\in\bfO^{+}_R,$ the $R$-module $D M$ consists of those linear forms in $M^*$
which vanish on $Q_{R,-d}M$ for some $d\geqslant 1$, 
we have $D M={}^\sharp M^\circledast$, where
$M^\circledast$ is the set of $\frakg$-finite elements of $M^*$.
The functor $D$ preserves $\bfO^{+,f}_R$.  It is the duality introduced in \cite{KL},
which does not fix the simple objects when $R$ is a field.

For the second claim we must prove that for any $S$-point $R\to S$
we have
$D(SM)=SD(M)$ and $S\scrD(N)=\scrD(SN)$ for each
$M\in\bfO^{+,f}_R$, $N\in\bfO^{\nu,\kappa,f}_R$. The proof is the same as in Lemma \cite[lem.~8.16]{KL}.
\end{proof}

A \emph{generalized Weyl module}  is a module in $\bfO_R^{\nu,\kappa,f}$ of the form
$\Indc_R(M),$ where $M$ is a $\bfg_{R,+}$-module with a finite
filtration by $\bfg_{R,+}$-submodules such that the
subquotients are annihilated by $Q_{R,1}$ and lie in
$\scrO_R^\nu$ as $\frakg_R$-modules. 

\vspace{2mm}

\begin{lemma}\label{lem:O0}
A $\bfg_{R,\kappa}$-module which is free over $R$ belongs to
$\bfO^{\nu,\kappa,f}_R$ if and only if it is a quotient of a
generalized Weyl module of $\bfO^{\nu,\kappa,f}_R$.
\qed
\end{lemma}

\vspace{.5mm}

\begin{rk}
The functors $M\mapsto {}^\dag\! M, {}^\sharp M, M^*$ commute with each other and we have
a canonical isomorphism of $\bfg_R$-modules $({}^\sharp M)(\infty)={}^\sharp(M(-\infty)).$
\end{rk}

\vspace{2mm}

\begin{rk}
We define the involution $\dag$ on $\frakg_R$-modules  
and the dualities $\scrD$ on $\scrO^\nu$ and $D$ on
$\scrO^+$ in a similar way as above. We have 
a canonical $\bfg_{R,\kappa}$-module isomorphism
${}^\dag\!\Indc_R(M)=\Indc_R({}^\dag\! M).$
\end{rk}

\vspace{2mm}

For  each $\beta\in \widehat P^\nu_R,$ 
the \emph{truncated category} $\lub \bfO^{\nu,\kappa}_R$ is the 
Serre subcategory of $\bfO^{\nu,\kappa}_R$ consisting of the modules whose 
simple subquotients have a
highest weight in $\beta-\bbN\widehat\Pi^+$. 
The following hold, see e.g.,  \cite{F1}, \cite{F2}, \cite[sec.~3, 7]{S1} for more details.

\vspace{2mm}

\begin{prop} 
(a)  $\bfO^{\nu,\kappa}_R$ is 
the direct limit of the subcategories $\lub \bfO^{\nu,\kappa}_R$,

(b)  $\lub \bfO^{\nu,\kappa}_R$ is 
a highest weight $R$-category with
$\Delta(\lub \bfO^{\nu,\kappa,\Delta}_R)=\lub \bfO^{\nu,\kappa}_R\cap \Delta(\bfO^{\nu,\kappa}_R)$,

(c) for $\beta\leqslant\gamma$ the obvious inclusion
$\lub \bfO^{\nu,\kappa}_R\subset {}^\gamma \bfO^{\nu,\kappa}_R$ 
preserves the tilting modules and commutes with taking extensions.
\qed
\end{prop}

\vspace{.5mm}

In particular, we'll regard the tilting modules as objects of 
$\bfO^{\nu,\kappa,\Delta}_R$, although $\bfO^{\nu,\kappa}_R$ is not a highest weight $R$-category.

Next, from Proposition \ref{prop:introhw} we deduce that
the $R$-category $\bfO^{\nu,\kappa}_R$ is Hom-finite and
that for any local $S$-point $R\to S$ the base change 
preserves the tilting modules.
Further, if $M,$ $N$ are tilting,
then $\Hom_{\bfg_R}(M,N)$ is free over $R$ and the canonical map
$S\Hom_{\bfg_R}(M,N)\to\Hom_{\bfg_S}(SM,SN)$
is invertible.

We call  $\bfO^{+,\kappa}_R$ the \emph{Kazhdan-Lusztig category} of $\frakg_R$, i.e.,
the affine parabolic category O 
associated with the standard maximal parabolic in $\bfg_R$, see \cite{KL}.

\vspace{3mm}

\subsubsection{The linkage principle and the highest weight order on $\bfO$}
\label{sec:linkage}
Assume that $R$ is a local ring.
Let us recall the partial order on $\widehat P^\nu_R$ given in \cite{VV}.

First, to each
$\widehat\lambda=(z,\lambda,c)$ in $\widehat P_R,$
we associate its \emph{integral affine root system} which is given by
$\widehat\Pi(\widehat\lambda)
=\{\alpha\in\widehat \Pi\,;\,\langle\widehat\lambda:\alpha\rangle_\Bbbk\in\bbZ\}.$
Since $\widehat\Pi(\widehat\lambda)=\widehat\Pi(0,\lambda,c)$, we may write
$\widehat\Pi(\lambda,c)$ for $\widehat\Pi(\widehat\lambda)$.

Now, given 
$\widehat\lambda,\widehat\lambda'\in\widehat P^\nu_R$, we write
$\widehat\lambda\Uparrow\widehat\lambda'$
if and only if there are
$\beta\in\widehat\Pi(\widehat\lambda')$, $w\in W_\nu$ such that 
$\beta\notin\Pi_\nu$
and
$\widehat\lambda=ws_\beta\bullet\widehat\lambda'\in\widehat\lambda'-\widehat\Pi^+$ 
modulo $\frakm\,\widehat P_R$.

\vspace{2mm}

\begin{df}
$(a)$ The \emph{linkage ordering} is the partial order $\leqslant_\ell$
on $\widehat P^\nu_R$ is the
transitive and reflexive closure of the relation $\Uparrow$.
For $\lambda,\lambda'\in P^\nu_R$ we abbreviate
$\lambda\leqslant_\ell\lambda'$ if and only if
$\widehat\lambda\leqslant_\ell \widehat\lambda'$.
So, we may view $\leqslant_\ell$ as a partial order on $P^\nu_R$.

$(b)$ The \emph{BGG ordering} $\BGG$ on $P^\nu_R$ is the smallest partial order such that
$\lambda\BGG\lambda'$ if 
$[\bfM(\lambda')_{\Bbbk,\nu}:\bfL(\lambda)_\Bbbk]\neq 0.$
\end{df}

\vspace{.5mm}

\begin{rk}
The definition of $\leqslant_\ell$ is motivated by the following remark :
the parabolic version of the Jantzen formula in \cite{KK} for the determinant of the Shapovalov form
of a parabolic Verma module in $\bfO^{\nu,\kappa}_\Bbbk$ implies that $\leqslant_\ell$ refines $\BGG$.
The BGG order induces an highest weight order on 
${}^\beta\bfO^{\nu,\kappa}_R$ for each $\beta$.
Hence $\leqslant_\ell$ induces also an highest weight order on 
${}^\beta\bfO^{\nu,\kappa}_R$ for each $\beta$.
\end{rk}

\vspace{2mm}

\begin{rk} The partial orders $\leqslant_\ell$, $\BGG$ on $P^\nu_R$ can be viewed as partial orders on
$\scrP^\nu$ under the inclusion $\varpi$. They depend on $\Bbbk$. 
To avoid any confusion we may say that these partial orders are
\emph{relative to the field $\Bbbk$}.
\end{rk}

\vspace{2mm}

\begin{rk}
If $\frakp_\nu=\frakb$,
then $\leqslant_\ell$ coincides with $\BGG$ by \cite{KK}.
\end{rk}

\vspace{3mm}

\subsection{The categorical action on $\bfO$}
\label{sec:categorification}
\emph{From now on,
unless specified otherwise, we'll assume 
that $R$ is a regular local analytic deformation ring of dimension $\leqslant 2$.
}

First, let us briefly recall the main properties of the Kazhdan-Lusztig tensor product $\dot\otimes_R$, 
see \S \ref{sec:KLotimes,II}.
Details will be given in Propositions \ref{prop:TPR1}, \ref{prop:adjoints},
\ref{prop:monoidalR}  and \ref{prop:psi-s}.

Recall that $V_R$ is the natural representation of $\frakg_R$, and that
the modules $\bfV_R,\bfV_R^*\in\bfO^{+,\kappa,\Delta}_R$ are given by
$\bfV_R=\Indc_R(V_R)$, $\bfV_R^*=\Indc_R(V_R^*)$.
We have exact endofunctors $e$, $f$ on $\bfO^{\nu,\kappa,\Delta}_R$ given by
$e(M)=M\dot\otimes_R\bfV_R^*$ and $f(M)=M\dot\otimes_R\bfV_R$.
The functors $e$, $f$  preserve the tilting modules.
If $R=K$ is a field then $e$, $f$ extend to biadjoint endofunctors of $\bfO^{\nu,\kappa}_K$.

Since $R$ is an analytic algebra, the element $q_R=\exp(-2\pi\sqrt{-1}/\kappa_R)$ of $R$ is well-defined and
the operator $\exp(2\pi\sqrt{-1}\frakL_0)$ acts on any module $M\in\bfO^{\nu,\kappa}_R$.
Let $X$ be  the endomorphism of the functor $f$ 
which acts  on $f(M)$  by the operator
$\exp(-2\pi\sqrt{-1}\frakL_0)\,\big(\!\exp(2\pi\sqrt{-1}\frakL_0)\dot
\otimes_R\exp(2\pi\sqrt{-1}\frakL_0)\big),$
see \eqref{balancing}, \eqref{XT-KL}.
Let $T$ be the endomorphism
of $f^2$ defined in \eqref{XT-KL}.
By Remark \ref{rem:3.3} the endomorphisms $X$, $T$ can be viewed as endomorphisms of $e$, $e^2$.

Now, let $R=K$ be a field. 
Let $\t\in P_K$ be as in \S \ref{sec:liealgebras}.
Set $I=\{\t_{K,1},\t_{K,2},\dots,\t_{K,\ell}\}+\bbZ+\kappa_K\bbZ$.
Write $i\sim j$ if $i-j\in\kappa_K\bbZ$.
Put $\scrI=I/\!\sim$. 
We will identify $q_K^i$ with the element $i/\sim$ in $\scrI$.

For each $i\in K$ let $f_i$, $e_i$ be the generalized
$q_K^i$-eigenspace and  $q_K^{-(N+i)}$-eigenspace of $X$ acting on $f$ and $e$. 
The functors $e_i$, $f_i$ are biadjoint, see \cite[rem.~7.22]{CR}.
The action of $e_i$, $f_i$ on parabolic Verma modules can be computed explicitly.
Recall that for $\lambda, \mu\in P^\nu_K$  we write 
$\lambda \overset{i}\to\mu$ if $\mu+\rho$ is obtained from 
$\lambda+\rho$ by replacing an entry equal to $i$ by $i+1$.

\vspace{2mm}

\begin{lemma}\label{lem:pasbete}

(a) For each $\lambda\in P^\nu_K,$ the module $f_i(\bfM(\lambda)_{K,\nu})$ has a filtration with sections of the form 
$\bfM(\mu)_{K,\nu}$, one for each $\mu$ such that $\lambda \overset{j}\to\mu$ for some 
$j\in K$ with $i\sim j$,

(b)  for each $\lambda\in P^\nu_K,$ the module $e_i(\bfM(\lambda)_{K,\nu})$ has a filtration with sections of the form 
$\bfM(\mu)_{K,\nu}$, one for each $\mu$ such that $\mu \overset{j}\to\lambda$ for some
$j\in K$ with $i\sim j$,

(c)  $e$, $f$ are exact endofunctors of  $\bfO^{\nu,\kappa}_{K,\tau}$,

(d) 
$e=\bigoplus_{i\in \scrI}e_i$ and $f=\bigoplus_{i\in \scrI}f_i$
on $\bfO^{\nu,\kappa}_{K,\tau}$.
\end{lemma}

\begin{proof}
Propositions \ref{prop:TPR1}, \ref{prop:adjoints} imply that 
$f(\bfM(\lambda)_{K,\nu})$ has a filtration (not necessarily unique)
whose associated graded consists of the sum of the 
modules $\bfM(\mu)_{K,\nu}$ such that $\lambda\overset{i}\to\mu$ for some $i\in K$. 

Next, the same proof as in \cite[prop.~2.7]{KL}, using the formula
$\frakL_0=\cas/2\kappa+\sum_{r>0}\sum_{i,j=1}^Ne_{ij}^{(-r)}e_{ji}^{(r)}/\kappa,$
shows that the operator
$\exp(2\pi\sqrt{-1}\frakL_0)$ acts on $\bfM(\mu)_{K,\nu}$ by the 
scalar $\exp(-2\pi\sqrt{-1}z'_\mu)$ for any $\mu\in P^\nu_K$, where
$-z'_\mu=\langle\mu:2\rho+(N-1)\det+\mu\rangle/2\kappa$.

Using this, a direct computation shows that any subquotient of $f(\bfM(\lambda)_{K,\nu})$ 
which is isomorphic to 
$\bfM(\mu)_{K,\nu}$, for some affine weight $\mu$ such that
$\lambda\overset{i}\to\mu,$ belongs to the generalized eigenspace of 
$X(\bfM(\lambda)_{K,\nu})$ with eigenvalue $q_K^i$. 
This proves $(a)$.

The discussion above implies that $f=\bigoplus_{i\in K}f_i$, as endofunctors of $\bfO^{\nu,\kappa,\Delta}_K$. 
We deduce that $f=\bigoplus_{i\in K}f_i$ on $\bfO^{\nu,\kappa}_K$, because
$f$ is exact and any object in $\bfO^{\nu,\kappa}_K$ is a quotient of
an object in $\bfO^{\nu,\kappa,\Delta}_K$.
We prove that $e=\bigoplus_{i\in K}e_i$ in a similar way.

For $\lambda, \mu\in P^\nu_K$ such that $\lambda\overset{i}\to\mu$ for some $i\in K$, 
we have $\lambda\in P^\nu+\tau$ if and only if $\mu\in P^\nu+\tau.$
By Lemma \ref{lem:pasbete}, we deduce that $e$, $f$ restrict to exact endofunctors on 
$\bfO^{\nu,\kappa}_{K,\tau}$. 
Note that $e_i$, $f_i$ act by zero on $\bfO^{\nu,\kappa}_{K,\tau}$ whenever $i\not\in I$. 
This proves $(d)$.
\end{proof}

\vspace{2mm}

Now, we define an $\fraksl_{\!\scrI}$-categorical action on $\bfO^{\nu,\kappa}_{K,\tau}$.
For each
$\lambda\in P+\tau$ we write
$m_i(\lambda)=\#\{k\in[1,N]\,;\, q^{\langle\lambda+\rho,\epsilon_k\rangle}_K=i\}$
and
$\wt(\lambda)=\sum_{i\in\scrI}\big(m_i(\lambda)-m_{iq}(\lambda)\big)\Lambda_i.$
For $\beta\in X_{\!\scrI}$ 
let $\bfO^{\nu,\kappa}_{K,\tau,\beta}\subset\bfO^{\nu,\kappa}_{K,\tau}$ be the Serre subcategory generated by  
the modules $\bfL(\lambda)_{K}$ with $\sum_{i\in\scrI}m_i(\lambda)\,\epsilon_i=\beta$. 
\vspace{2mm}

\begin{claim} \label{claim:blocks} For
$\lambda,\mu\in P\{d\}+\tau$ we have

$\widehat\lambda$, $\widehat\mu$ are linked $\iff$
$m_i(\lambda)=m_i(\mu)$ for all $i\in\scrI$ $\iff$ $\wt(\lambda)=\wt(\mu)$.
\end{claim}

\vspace{.5mm}

Hence, we have a decomposition
$\bfO^{\nu,\kappa}_{K,\tau}=\bigoplus_{\beta\in X_{\!\scrI}}\bfO^{\nu,\kappa}_{K,\tau,\beta}$
by the linkage principle. 

\vspace{2mm}

\begin{prop}\label{prop:5.14}
The tuple $(e,f,X,T)$, together with the decomposition of
$\bfO^{\nu,\kappa}_{K,\tau}$ above, is an $\fraksl_{\!\scrI}$-categorification on $\bfO_{K,\tau}^{\nu,\kappa}$.
\end{prop}

\vspace{.5mm}

\begin{proof}
By Lemma \ref{lem:pasbete} we have $e_i(\bfO_{K,\tau,\beta}^{\nu,\kappa})\subset\bfO_{K,\tau,\beta+\epsilon_i-\epsilon_{qi}}^{\nu,\kappa}$
and
$f_i(\bfO_{K,\tau,\beta}^{\nu,\kappa})\subset\bfO_{K,\tau,\beta-\epsilon_i+\epsilon_{qi}}^{\nu}.$
Further, a direct computation using Lemma \ref{lem:pasbete} shows that the operators $e_{i}$, $f_{i}$ with 
$i\in\scrI$ yield a representation of $\fraksl_{\!\scrI}$ on 
$[\bfO^{\nu,\kappa}_{K,\t}]$ such that $[\bfM(\lambda)_{K,\nu}]$ is a weight vector of weight $\wt(\lambda)$. 
The rest follows from Lemma \ref{lem:pasbete} and 
Proposition \ref{prop:psi-s}.
\end{proof}

\vspace{3mm}

\subsection{The category $\bfA$ and the functor $\Psi$}
\label{sec:BKKB}
Let $R$ be either a field or a local deformation ring. We have the following basic fact.

\vspace{2mm}

\begin{lemma}[\cite{VV}]\label{lem:ideal}
The map $\varpi$ identifies $\scrP^\nu$ with an ideal in $P^\nu$ 
for the partial orders $\leqslant_\ell$ or $\BGG$ relative to $\Bbbk$.
\end{lemma}

\vspace{.5mm}

\begin{proof}
It is enough to consider the case of the ordering $\leqslant_\ell$, because it refines $\BGG$.
Since $R$ is a local deformation ring with residue field $\Bbbk$, 
we have $\tau_\Bbbk=0$ and $\kappa_\Bbbk=-e$.
Then, the claim follows from \cite[prop.~A.6.1]{VV}. 
\end{proof}

\vspace{2mm}

For each $\lambda\in\scrP^\nu$, we abbreviate
$\pmb\Delta(\lambda)_{R,\t}=\bfM(\varpi(\lambda))_{R,\nu}$ where
$\varpi$ is the application defined in \eqref{varpi}. 
Following \cite{BK1}, \cite{VV} we introduce the abelian $R$-category
$\bfA^{\nu,\kappa}_{R,\tau}\subset \bfO^{\nu,\kappa}_{R,\tau}$ which is the Serre $R$-linear subcategory
generated by $\{\pmb\Delta(\lambda)_{R,\tau}\,;\,\lambda\in\scrP^\nu\}$.

Since $\tau_\Bbbk=0$, by  Lemma \ref{lem:ideal}, 
$\bfA^{\nu,\kappa}_{\Bbbk}=\bfA^{\nu,\kappa}_{\Bbbk,\tau}$ is a highest weight $\Bbbk$-category.
Using \cite[thm.~4.15]{R1},  this implies that 
$\bfA^{\nu,\kappa}_{R,\tau}$ is a highest weight $R$-category
such that 
$\Delta(\bfA^{\nu,\kappa}_{R,\tau})=\{\pmb\Delta(\lambda)_{R,\tau}\,;\,\lambda\in\scrP^\nu\}.$
The highest weight order on
$\bfA^{\nu,\kappa}_{R,\tau}$ is given by the partial order $\leqslant_\ell$ or $\BGG$ on $\scrP^\nu$ 
relative to $\Bbbk$.

We will write $\bfL(\lambda)$, $\bfP(\lambda)_{R,\tau}$, $\bfT(\lambda)_{R,\tau}$ respectively for the simple, projective, tilting objects associated with 
$\pmb\Delta(\lambda)_{R,\tau}$.
Let $\bfA^{\nu,\kappa,\Delta}_{R,\tau}=(\bfA^{\nu,\kappa}_{R,\tau})^\Delta$ be the full exact subcategory of 
$\Delta$-filtered objects.
For each $d\in\bbN$, let 
$\bfA^{\nu,\kappa}_{R,\tau}\{d\}\subset\bfA^{\nu,\kappa}_{R,\tau}$ be the highest weight
subcategory generated by
$\Delta(\bfA^{\nu,\kappa}_{R,\tau}\{d\})=\{\pmb\Delta(\lambda)_{R,\tau}\,;\,\lambda\in\scrP^\nu_d\}.$

Now, assume that $R$ is analytic of dimension $\leqslant 2$.
By Lemma \ref{lem:pasbete} the endofunctor $f$ of $\bfO^{\nu,\kappa,\Delta}_R$ maps 
$(\bfA^{\nu,\kappa}_{R,\tau}\{d\})^\Delta$ to $(\bfA^{\nu,\kappa}_{R,\tau}\{d+1\})^\Delta$.
We define inductively an object $\bfT_{R,\tau}^{\nu,\kappa}\{d\}$ in $\bfA^{\nu,\kappa}_{R,\tau}\{d\}$ 
by setting
$\bfT_{R,\tau}^{\nu,\kappa}\{0\}=\pmb\Delta(\emptyset)_{R,\t}$ and
$\bfT_{R,\tau}^{\nu,\kappa}\{d\}=f(\bfT_{R,\tau}^{\nu,\kappa}\{d-1\}).$
We will abbreviate
$\bfT_{R,d}=\bfT_{R,\tau}^{\nu,\kappa}\{d\}$ to unburden the notation.
To avoid any confusion we may write
$\bfT_{R,\tau}^{\nu,\kappa}(N)\{d\}=\bfT_{R,\tau}^{\nu,\kappa}\{d\}$ and
$\bfT_{R,d}(N)=\bfT_{R,d}$.
\vspace{2mm}

\begin{lemma}
\label{lem:5.18}
(a) 
We have  $\Bbbk\bfT_{R,d}=\bfT_{\Bbbk,d}$.

(b)
The module $\bfT_{R,d}$ is tilting in $\bfA^{\nu,\kappa}_{R,\tau}.$
\end{lemma}

\begin{proof}
Part $(a)$ follows from Lemma \ref{lem:A20}.
To prove $(b)$, note first that
$\bfT_{R,0}$ is tilting by Proposition \ref{prop:introhw}, because
$\Bbbk\bfT_{R,0}=\bfT_{\Bbbk,0}$ is $\Delta$-filtered and simple.
Since the functor $f$ preserves the tilting modules of
$\bfO^{\nu,\kappa}_{R,\tau}$ by Lemma \ref{lem:A19}, we deduce that $\bfT_{R,d}$ is  tilting. 
\end{proof}

\vspace{2mm}

By Proposition \ref{prop:psi-s}, we have an $R$-algebra homomorphism 
\begin{equation}\label{eq:psi-s}
\psi_{R,d}^s:\bfH^{s}_{R,d}\to\End_{\bfA^{\nu,\kappa}_{R,\t}}\big(\bfT_{R,d}\big)^\op
\end{equation}
and a functor
$$\Psi_{R,d}^s=\Hom_{\bfA_{R,\tau}^{\nu,\kappa}}(\bfT_{R,d},\bullet):
\bfA_{R,\tau}^{\nu,\kappa}\{d\}\to \bfH^{s}_{R,d}\mmod.$$
The main result of the section is Theorem \ref{thm:isom}. To prove it, we will study in the subsequent subsections some properties of $\psi_{R,d}^s$ and $\Psi_{R,d}^s$ when localized to codimension one.

\vspace{2mm}

\begin{rk}\label{rk:dfT}
Since $\bfT_{R,d}$ is tilting, it is uniquely determined by its specialization $\Bbbk\bfT_{R,d}=\bfT_{\Bbbk,d}.$
If  $R$ is a regular local ring of dimension $>2$, then we may define
$\bfT_{R,d}$ as the unique module in $\bfA^{\nu,\kappa}_{R,\tau}\{d\}$ (up to isomorphism)
which specializes to $\bfT_{\Bbbk,d}$.
We do not know how to define either $\psi_{R,d}^s$ or $\Psi_{R,d}^s$ if $\dim R>2$.
\end{rk}

\vspace{2mm}

\begin{rk}\label{rem:5.15}
For each $p\in[1,\ell],$ let $\lambda_p\in\scrP^\ell_1$ be the $\ell$-partition with $(1)$ on the $p$-th component and $\emptyset$ elsewhere. 
The proof of Lemma \ref{lem:pasbete} implies that the module $\bfT_{K,1}$ has a
$\Delta$-filtration with sections of the form $\pmb\Delta(\lambda)_{K,\tau}$
with $\lambda\in\scrP^\ell_1$, and that the operator
$X\in \End(\bfT_{K,1})$ has the eigenvalue $q_K^{s_p}$ on $\pmb\Delta(\lambda_p)_{K,\tau}$.
\end{rk}

\vspace{3mm}

\subsection{The affine Lie algebra of a pseudo-Levi subalgebra}
\label{sec:affine-levi}
Consider the root system 
$
\widehat\Pi_\nu=
\{\alpha+r\delta\,;\,\alpha\in\Pi_\nu,\,r\in\bbZ\}
\cup\{r\delta\,;\,r\in\bbZ^\times\}.
$
Let $\bfm_{R,\nu}$ be the Lie subalgebra of  $\bfg_R$
spanned by $\bft_R$ and the root subspaces associated with
$\widehat\Pi_\nu$.
We may view $\bfm_{R,\nu}$ as the affine Kac-Moody algebra associated 
with the pseudo-Levi subalgebra
$\frakm_{R,\nu}$ of $\frakg_R$. 
We define the associative
$R$-algebra $\bfm_{R,\nu,\kappa}$ in the same way 
as we defined $\bfg_{R,\kappa}$ in \S \ref{sec:affine}.

The Weyl group of $\widehat\Pi_\nu$ is the subgroup
$\widehat W_\nu$ of $\widehat W$ generated by the affine reflections
$s_\beta$ with $\beta\in\widehat\Pi_\nu$. Thus, we have 
$\widehat W_\nu=\{w T_x\,;\,w\in W_\nu,\,x\in\bbZ\Pi_\nu\}.$

Set 
$\bfb_{R,\nu}=\bfm_{R,\nu}\cap\bfb_{R}$.
The category $\bfO_R^\kappa(\nu)$ consists of the finitely generated
$\bfm_{R,\nu,\kappa}$-modules which are weight $\bft_R$-modules with a locally finite
action of $\bfb_{R,\nu}$ (over $R$), and such that the highest weight of any constituent 
is of the form
$\widehat\lambda$ with $\lambda\in P_R$. 
The decomposition $\frakm_{R,\nu}=\bigoplus_{p=1}^\ell\frakg\frakl_{R,\nu_p}$ yields an equivalence
$\bfO_R^\kappa(\nu)=\bigotimes_{p=1}^\ell \bfO_R^\kappa(\nu_p)$, here the tensor product is over $R$.

Given a tuple $\gamma=(\gamma_p)$ of
compositions of the $\nu_p$'s, let 
$\bfO^{\gamma,\kappa}_R(\nu)\subset \bfO_R^\kappa(\nu)$ be the subcategory which is identified under  the equivalence
$\bfO_R^\kappa(\nu)=\bigotimes_{p=1}^\ell \bfO_R^\kappa(\nu_p)$
with the category
$\bigotimes_{p=1}^\ell \bfO^{\gamma_p,\kappa}_R(\nu_p).$
Given a deformation parameter $\tau$ and a tuple $a\in\bbN^\ell,$ 
we also consider the categories
$\bfO^{\gamma,\kappa}_{R,\t}(\nu)=\bigotimes_{p=1}^\ell \bfO_{R,\tau_p}^{\gamma_p,\kappa}(\nu_p)$ 
and
$\bfO_{R,\tau}^{\gamma,\kappa}(\nu)\{a\}=\bigotimes_{p=1}^\ell \bfO_{R,\tau_p}^{\gamma_p,\kappa}(\nu_p)\{a_p\}.$
Setting $\gamma_p=(\nu_p)$ for each $p$, 
we get the \emph{Kazhdan-Lusztig category} 
$\bfO^{+,\kappa}_R(\nu)=\bfO^{\nu,\kappa}_R(\nu)$ of the Lie algebra $\bfm_{R,\nu}$.
Let $\bfO^{+,\kappa}_R(\nu)\{a\}\subset \bfO^{+,\kappa}_R(\nu)$ be the full subcategory defined in the similar way.

To avoid confusions, we may set
$\bfA_{R,\tau}^{\nu,\kappa}(N)=\bfA_{R,\tau}^{\nu,\kappa}$ if $\frakg=\frakgl_N$.
Then, we define
$\bfA^{+,\kappa}_{R,\tau}(\nu)\subset \bfO^{+,\kappa}_{R,\tau}(\nu)$ 
to be the subcategory isomorphic to
$\bigotimes_{p=1}^\ell\bfA^{\nu_p,\kappa}_{R,\tau_p}(\nu_p)$.

As above, we
drop the subscripts $R,$ $\tau$ if $R=\bbC$ or $\tau=0$.

\vspace{3mm}

\subsection{Reductions to codimension one}

\subsubsection{Preliminaries}
\label{sec:preliminaries}
For each $z\in\bbZ$ and $u,v\in[1,\ell]$ we write
$f_{u,v,z}(\tau_R,\kappa_R)=\tau_{R,u}-\tau_{R,v}-z\,\kappa_R$
and $f_{u,v}(\tau_R)=\tau_{R,u}-\tau_{R,v}.$

\vspace{2mm}

\begin{df}
We will say that the deformation ring $R$ is \emph{generic}
if $f_{u,v,z}(\tau_R,\kappa_R)\neq b$ for any tuple $(u,v,z,b)$ with $u<v$ and $z,b\in\bbZ$,
and that it is \emph{subgeneric} if $\kappa_R\notin\bbQ$ and
$f_{u,v,z}(\tau_R,\kappa_R)= b$ for a unique tuple $(u,v,z,b)$ as above (with $u<v$).
\end{df}

\vspace{.5mm}

\begin{rk}
If $R$ is a local deformation ring, i.e., if
$\tau_{\Bbbk,p}=0$ and $\kappa_{\Bbbk}=-e$ with $e\in\bbN^\times$,
then for each $\frakp\in\frakP$ such that
$f_{u,v,z}(\tau_{\Bbbk_\frakp},\kappa_{\Bbbk_\frakp})=b$ we have also
$b=z\, e$.
\end{rk}

\vspace{2mm}

Now, assume that $R$ is a local deformation ring.
Then, the category $\bfA^{\nu,\kappa}_{R,\tau}$ is a highest weight $R$-category by \S \ref{sec:BKKB},
either for the partial order $\leqslant_\ell$ or $\BGG$ relative to $\Bbbk$ by 
Lemma \ref{lem:ideal}.
In other words, the highest weight order on $\bfA^{\nu,\kappa}_{R,\tau}$ is induced from the
highest weight order on $\bfA^{\nu,\kappa}_{\Bbbk}$ via  base change, which yields a canonical
bijection $\Delta(\bfA^{\nu,\kappa}_{R,\tau})\to\Delta(\bfA^{\nu,\kappa}_{\Bbbk}).$

By base change again, these highest weight orders on $\bfA^{\nu,\kappa}_{R,\tau}$
induce highest weight orders on $\bfA^{\nu,\kappa}_{R_\frakp,\tau}$ 
and $\bfA^{\nu,\kappa}_{\Bbbk_\frakp,\tau}$ for each $\frakp\in\frakP$.
Note that $R_\frakp$ is a local ring, but may not be a local deformation ring because
$\tau_{\Bbbk_\frakp,p}$ may be $\neq 0$.
So, we have the posets isomorphisms
$$\xymatrix{\Delta(\bfA^{\nu,\kappa}_{\Bbbk})&\Delta(\bfA^{\nu,\kappa}_{R,\tau})\ar[l]_-{\otimes\Bbbk}
\ar[r]^-{\otimes R_\frakp}&\Delta(\bfA^{\nu,\kappa}_{R_\frakp,\tau})\ar[r]^-{\otimes\Bbbk_\frakp}
&\Delta(\bfA^{\nu,\kappa}_{\Bbbk_\frakp,\tau})}.$$

We will reduce the study of $\bfA^{\nu,\kappa}_{R,\tau}$ to the study of  
$\bfA^{\nu,\kappa}_{\Bbbk_\frakp,\tau}$
for $\frakp\in\frakP_1$.
We will say that \emph{$\frakp$ is generic} if $\Bbbk_\frakp$ is generic and that
\emph{$\frakp$ is subgeneric} if $\Bbbk_\frakp$ is subgeneric.

\vspace{2mm}

\begin{rk}\label{rem:5.29}
Let $R$, $\scrI$ be as in \S \ref{sec:categorification}.
If $R$ is subgeneric then 
each component $\scrI_p$ is a quiver of type $A_\infty$,
while if $R$  is generic then $\Omega=[1,\ell]$ (i.e., the quiver $\scrI$ has exactly $\ell$ components).
\end{rk}

\vspace{2mm}

In order to use the Kazhdan-Lusztig tensor product, we'll be mainly interested by the case where
$R$ is either a field or a regular local deformation ring of dimension $\leqslant 2$.
Note that if $R$ has dimension 2, then we can always choose it in such a way that it is in general position.

The following basic fact is important for the rest of the paper. 

\vspace{2mm}

\begin{prop}\label{prop:preliminaries} 
Assume that $R$ is a local deformation ring in general position.

(a)
If $\frakp\in\frakP_1$ then $\frakp$ is either generic or subgeneric.

(b)
The $K$-category $\bfA^{\nu,\kappa}_{K,\tau}$ is split semi-simple.

(c)
If $R$ is analytic,
then the condition \eqref{(A)} holds in the fraction field $K$.

(d) If $\nu_p\geqslant d$ for all $p$, then the map $\psi^s_{K,d}: \bfH^s_{K,d}\to \End_{\bfA^{\nu,\kappa}_{K,\tau}}(\bfT_{K,d})^\op$ in \eqref{eq:psi-s}
is an isomorphism of $K$-algebras. The functor $\Psi^s_{K,d}$ is an equivalence of categories and it maps $\pmb\Delta(\lambda)_{K,\tau}$ to $S(\lambda)^{s,q}_K$.
\end{prop}

\vspace{.5mm}

\begin{proof}
Since $R$ is a UFD and $\frakp$ has height 1,  we have $\frakp=R g$ for some irreducible element $g\in R$.
Now, if $f_{u,v,z}(\tau_{\Bbbk_\frakp},\kappa_{\Bbbk_\frakp})=b$ and $\kappa_{\Bbbk_\frakp}=c$ 
for some $u\neq v$,  $z,b\in\bbZ$ and $c\in\bbQ$ then
$g$ must be a unit of $R$ because $R$ is in general position. 
This is a contradiction. 
For the same reason, we may have $f_{u,v,z}(\tau_{\Bbbk_\frakp},\kappa_{\Bbbk_\frakp})=b$
for at most one tuple $(u,v,z,b)$.
Therefore, if $\frakp$ is not generic, then we have $\kappa_{\Bbbk_\frakp}\notin\bbQ$.
Part $(a)$ is proved. 

Part $(b)$ follows from the linkage principle. 
More precisely, recall that for $k\in J^\nu_p$ we set $p_k=p$.
Then, since $R$ is in general position, we have
\begin{eqnarray*}
\widehat\Pi(\tau_K,\kappa_K)&=&\{\beta\in\widehat\Pi\,;\, \langle (0,\t_K,\kappa_K):\beta\rangle\in\bbZ\},\\
&=&\{\alpha_{k,l}+z\delta\,;\, f_{p_k,p_l,a}(\tau_{K},\kappa_{K})\in\bbZ\},\\
&=&\Pi_\nu.
\end{eqnarray*}
Thus, the linkage classes are reduced to points, because two $\nu$-dominant weights which are 
$W_\nu$-conjugate under the $\bullet$-action are equal. Hence
$\bfA^{\nu,\kappa}_{K,\tau}$ is split semi-simple.

Part $(c)$ is obvious, 
because $q_K=\exp(-2\pi\sqrt{-1}/\kappa_K)$, $Q_{K,p}=\exp(-2\pi\sqrt{-1}s_p/\kappa_K),$
$\kappa_K\notin\bbQ$ and $(s_{K,u}-s_{K,v}+\kappa_K\bbZ)\cap\bbZ=\emptyset$ 
for each $u\neq v$.

Let us prove part $(d)$. As a finite dimensional split semi-simple $K$-algebra, the center of 
$\bfH^s_{K,d}$ is spanned by the primitive central idempotents. These idempotents are of the form 
$1_\alpha=\sum_{\bfi\in\scrI^\alpha}1_\bfi$ where $\alpha\in Q^+$ has height $d$, see \S 
\ref{sec:3.4}. For each nonzero $1_\alpha,$ there is a unique $\ell$-partition $\lambda$ of $d$ such 
that $\sum_{i\in K}n_i(\lambda)\alpha_i=\alpha$. From Lemma 
\ref{lem:pasbete} we deduce that, if $\nu_p\geqs d$ for all $p,$ then
\begin{equation}\label{eq:haha}
f_\bfi(\bfT_{K,0})=\pmb\Delta(\lambda)_{K,\tau},\quad\forall\ \bfi\in\scrI^\alpha.
\end{equation} 
Since $\psi^s_{K,d}(1_\alpha)$ is the projection from $\bfT_{K,d}$ onto its direct summand
$\bigoplus_{\bfi\in\scrI^\alpha}f_\bfi(\bfT_{K,0})$, the latter is nonzero whenever $1_\alpha$ is 
nonzero. So, the map $\psi^s_{K,d}$ is injective. To prove that it is an isomorphism, we are reduced 
to check the following.

\vspace{2mm}

\begin{claim} $\bfH^s_{K,d}$ and 
$\End_{\bfA^{\nu,\kappa}_{K,\tau}}(\bfT_{K,d})^\op$ have the same dimension over $K$.
\end{claim}

\vspace{.5mm}

To prove the claim, by Proposition \ref{prop:adjoints}, it is enough to check that the $K$-algebras
$H^s_{K,d}$ and 
$\End_{A^{\nu}_{K,\tau}}(T_{K,d})^\op$ have the same dimension. 
This follows from Proposition \ref{prop:isomBK}.

Next, since the $K$-algebra
$\bfA^{\nu,\kappa}_{K,\tau}$ is split semi-simple by part $(b)$, the standard modules 
$\pmb\Delta(\lambda)_{K,\tau},$ with $\lambda\in\scrP^\nu_d,$ form a complete set of 
indecomposable projective 
modules in $\bfA^{\nu,\kappa}_{K,\tau}$. So, formula \eqref{eq:haha} implies that 
$\bfT_{K,d}=\bigoplus_{\bfi\in\scrI^d}f_\bfi(\bfT_{K,0})$ is a projective generator in 
$\bfA^{\nu,\kappa}_{K,\tau}$. So $\Psi^s_{K,d}$ is an equivalence.
Since the unique simple and projective module in the block $\bfH^s_{K,\alpha}$
is the Specht module $S(\lambda)^{s,q}_{K}$, where $\lambda$ is as in
\eqref{eq:haha}, 
we deduce that
$\Psi^s_{K,d}(\Delta(\lambda)_{K,\tau})=S(\lambda)^{s,q}_{K}$.
\end{proof}

\vspace{3mm}

\subsubsection{The reduction to the finite type with $\ell=2$} \label{sec:red1}

For each tuple $a\in\bbN^{\ell-1},$ let $\bfO^{\nu,\kappa}_{R,\t}\{a\}\subset\bfO^{\nu,\kappa}_{R,\t}$ be the full 
subcategory 
consisting of the modules whose simple subquotients have a highest weight of the form
$\widehat{\lambda+\t_\Bbbk}$ with $\lambda\in P^\nu\{a\}.$ 
Set $\bfA^{\nu,\kappa}_{R,\t}\{a\}=\bfO^{\nu,\kappa}_{R,\t}\{a\}\cap\bfA^{\nu,\kappa}_{R,\t}$.
We define
$\bfO^{\nu,\kappa}_{\Bbbk,\t}\{a\}$ and $\bfA^{\nu,\kappa}_{\Bbbk,\t}\{a\}$ in the obvious way.

Let $p\mapsto p^o$ be the permutation of $[1,\ell]$ such that $p^{o}=\ell+1-p$.
Let $k\mapsto k^o$ be the unique permutation of $[1,N]$ which is 
blockwise increasing and which takes
the block $J^\nu_p$ to $J^{\nu^o}_{p^o}$. Applying this permutation
to the entries of a weight $\lambda\in P^\nu_R$ yields a weight $\lambda^o\in P^{\nu^o}_R$. 

Assume that $R$ is a local ring with a subgeneric residue field.
Let $h=(u,v,z)$ be the unique triple such that $u< v$ and 
$f_{u,v,z}(\tau_\Bbbk,\kappa_\Bbbk)=z\,e.$
Given a tuple $\varkappa=\varkappa_R\in R^\ell$, let  $\varkappa_\Bbbk\in \Bbbk^\ell$ be its residue class.
Assume that $f_{u,v}(\varkappa_\Bbbk)=ze$.
We will identify $\varkappa_R$ with the weight 
$\sum_p\varkappa_{R,p}\,\det_p\in P_R$.

If $z\leqslant 0,$ we abbreviate
$\scrO^{\nu}_{R,h}\{a\}=
\scrO^{\nu}_{R,\varkappa}(\nu,u,v)\{a\}$.
If $z>0$, we write $\scrO^{\nu}_{R,h}\{a\}=\scrO^{\nu^o}_{R,\varkappa^o}(\nu^o,v^o,u^o)\{a^o\}.$
See \S \ref{sec:notation4} for the notation.
We define $A^{\nu}_{R,h}\{a\}$ in the same manner. 
For each $d\in\bbN,$ we write $A^{\nu}_{R,h}\{d\}=\bigoplus_a A^{\nu}_{R,h}\{a\},$
where $a$ runs over the set of all $(\ell-1)$-compositions of $d$.
Depending on the sign of $z$, 
we write $M(\lambda)_{R,h}$ for $M(\lambda+\varkappa)_{R,\nu}$ or 
$M(\lambda^o+\varkappa^o)_{R,\nu}$,
and $\Delta(\lambda)_{R,h}$ for $\Delta(\lambda)_{R,\varkappa}$ or 
$\Delta(\lambda^o)_{R,\varkappa^o}$.

\vspace{2mm}

\begin{prop}\label{prop:equivredI}
(a) We have
$\bfO_{R,\t}^{\nu,\kappa}\simeq\bigoplus_{a\in\bbN^{\ell-1}}\bfO^{\nu,\kappa}_{R,\t}\{a\},$
$\bfO_{\Bbbk,\t}^{\nu,\kappa}\simeq\bigoplus_{a\in\bbN^{\ell-1}}\bfO^{\nu,\kappa}_{\Bbbk,\t}\{a\}.$

(b) There are equivalences of highest weight $R$-categories 
$\scrQ_R:\bfO_{R,\t}^{\nu,\kappa}\{a\}\to\scrO^{\nu}_{R,h}\{a\}$ and of highest weight $\Bbbk$-categories
$\scrQ_\Bbbk:\bfO_{\Bbbk,\t}^{\nu,\kappa}\{a\}\to\scrO^{\nu}_{\Bbbk,h}\{a\},$
such that 
$\Bbbk\scrQ_R(M)=\scrQ_\Bbbk(\Bbbk M)$ for each $M\in\bfO_{R,\t}^{\nu,\kappa}\{a\}$
and $\scrQ_R(\bfM(\lambda+\tau_R)_{R,\nu})=M(\lambda)_{R,h}$ for 
each $\lambda\in P^\nu$.

(c)  The equivalences in (b) restrict to equivalences of highest weight categories
$\scrQ_R: \bfA_{R,\tau}^{\nu,\kappa}\{a\}\to A^{\nu}_{R,h}\{a\}$
and $\scrQ_\Bbbk: \bfA_{\Bbbk,\tau}^{\nu,\kappa}\{a\}\to A^{\nu}_{\Bbbk,h}\{a\}.$
In particular $\scrQ_R(\pmb\Delta(\lambda)_{R,\tau})=\Delta(\lambda)_{R,h}$ for all $\lambda$.
\end{prop}

\vspace{.5mm}

\begin{proof}
For $k\in J^\nu_p$ we set $p_k=p$.
Since $\Bbbk$ is subgeneric, the integral root system 
$\widehat\Pi(\tau_\Bbbk,\kappa_\Bbbk)$ is given by
\begin{eqnarray*}
\widehat\Pi(\tau_\Bbbk,\kappa_\Bbbk)&=&
\{\beta\in\widehat\Pi\,;\, \langle (0,\t_\Bbbk,\kappa_\Bbbk):\beta\rangle\in\bbZ\},\\
&=&\{\alpha_{k,l}-r\delta\,;\, f_{p_k,p_l,r}(\tau_\Bbbk,\kappa_\Bbbk)\in\bbZ\},\\
&=&\Pi_\nu\cup\{\pm(\alpha_{k,l}-z\delta)\,;\, p_k=u,\,p_l=v\}.
\end{eqnarray*}
Therefore,  the linkage principle yields a decomposition
$\bfO_{\Bbbk,\t}^{\nu,\kappa}=\bigoplus_{a\in\bbN^{\ell-1}}\bfO^{\nu,\kappa}_{\Bbbk,\t}\{a\}.$
This decomposition holds over $R$ by Proposition \ref{prop:introhw}.
This proves part $(a)$.

The set $\widehat\Pi(\tau_\Bbbk, \kappa_\Bbbk)$ is a Coxeter system whose set of positive roots is
$\widehat\Pi(\tau_\Bbbk, \kappa_\Bbbk)^+=
\widehat\Pi(\tau_\Bbbk, \kappa_\Bbbk)\cap\widehat\Pi^+$,
see \cite[sec.~2.2]{KT}.
The set
$\Pi_{\nu,u,v}=\Pi_\nu\cup\{\pm\alpha_{k,l}\,;\, p_k=u,\,p_l=v\}$
is also a Coxeter system with positive roots 
$\Pi_{\nu,u,v}^+=\Pi_{\nu,u,v}\cap \Pi^+$.

If $z\leqslant  0$ then
$\widehat\Pi(\tau_\Bbbk, \kappa_\Bbbk)^+=
\Pi_\nu^+\cup\{\alpha_{k,l}-z\delta\,;\, p_k=u,\,p_l=v\}.$
Fix an integral coweight $\check\chi$ such that $\alpha_{k,l}(\check\chi)=-z$ if $p_k=u$, $p_l=v$ and
$\alpha_{k,l}(\check\chi)=-z$ if $p_k=p_l.$
The conjugation by $\check\chi$ yields a bijection
$\varphi:\widehat\Pi(\tau_\Bbbk,\kappa_\Bbbk)\iso\,\Pi_{\nu,u,v}$ such that $\alpha+r\delta\mapsto\alpha$
for all $\alpha,r$.
It maps positive roots to positive ones.

If $z> 0$ then
$\widehat\Pi(\tau_\Bbbk, \kappa_\Bbbk)^+=
\Pi_\nu^+\cup\{-\alpha_{k,l}+z\delta\,;\, p_k=u,\,p_l=v\}.$
The permutation $k\mapsto k^o$ of $[1,N]$ induces a bijection
$\Pi_{\nu,u,v}\iso\,\Pi_{\nu^{o},v^{o},u^{o}}$, $\alpha\mapsto\alpha^o$.
The bijection
$\varphi:\widehat\Pi(\tau_\Bbbk,\kappa_\Bbbk)\iso\,\Pi_{\nu^o,v^o,u^o}$ such that $\alpha+r\delta\mapsto\alpha^o$
identifies the subsets of positive roots in both sides.

In both cases the map $\varphi$ is an isomorphism of Coxeter systems.
Now, for each weight $\lambda\in P^\nu+\rho$ we consider the sets of roots 
$\widehat\Pi[\lambda+\tau_\Bbbk, \kappa_\Bbbk]
=\{\beta\in\widehat\Pi\,;\, 
\langle (0,\lambda+\tau_\Bbbk,\kappa_\Bbbk):\beta\rangle=0\}$ and
$\Pi_{\nu,u,v}[\lambda+\varkappa_\Bbbk]=
\{\alpha\in\Pi_{\nu,u,v}\,;\, \langle\lambda+\varkappa_\Bbbk\,:\,\alpha\rangle=0\}$.
Since $\Bbbk$ is subgeneric, we have
\begin{eqnarray*}
\widehat\Pi[\lambda+\tau_\Bbbk, \kappa_\Bbbk]
&=&\{\alpha_{k,l}-r\delta\,;\, f_{p_k,p_l,r}(\tau_\Bbbk,\kappa_\Bbbk)=-\langle\lambda\,:\,\alpha_{k,l}\rangle\},\\
&=&\{\alpha\in\Pi_\nu\,;\, \langle\lambda\,:\,\alpha\rangle=0\}\cup\\
&&\cup\{\pm(\alpha_{k,l}-z\delta)\,;\, p_k=u,\,p_l=v,\,\langle\lambda\,:\,\alpha_{k,l}\rangle=-z\,e\},\\
&=&\Pi_\nu[\lambda]\cup
\{\pm(\alpha_{k,l}-z\delta)\,;\, p_k=u,\,p_l=v,\,\langle\lambda+\varkappa_\Bbbk\,:\,\alpha_{k,l}\rangle=0\},\\
&=&\Pi_\nu[\lambda+\varkappa_\Bbbk]\cup
\{\pm(\alpha_{k,l}-z\delta)\,;\, p_k=u,\,p_l=v,\,\alpha_{k,l}\rangle\in\Pi_{\nu,u,v}[\lambda+\varkappa_\Bbbk]\}.
\end{eqnarray*}

If $z\leqslant  0$ then 
$\varphi\big(\widehat\Pi[\lambda+\tau_\Bbbk, \kappa_\Bbbk])=
\Pi_{\nu,u,v}[\lambda+\varkappa_\Bbbk].$
Therefore, by \cite[thm.~11]{F2}, there is an equivalence
of $\Bbbk$-categories
$\scrQ_\Bbbk:\bfO_{\Bbbk,\t}^\kappa\{a\}\to \scrO_{\Bbbk,\varkappa}(\nu,u,v)\{a\}$
such that
$\bfL(\mu+\tau_\Bbbk)_{\Bbbk}\mapsto L(\mu+\varkappa_\Bbbk)_\Bbbk$
for each $\mu\in P\{a\}$.
The proof of loc. ~cit.~is given by constructing an analogue of Soergel's functor which identifies, block by block, 
the endomorphism rings of projective generators of $\bfO_{R,\t}^\kappa\{a\}$ and $\scrO_{R,\varkappa}(\nu,u,v)\{a\}$
with the endomorphism ring of the same sheaf over a moment graph
(modulo a base change of deformation rings, from a localization of the functions ring of the Cartan subalgebras of $\bfg$ and $\frakm(\nu,u,v)$ to $R$).
This construction yields indeed an equivalence of abelian $R$-categories
$\scrQ_R:\bfO_{R,\t}^\kappa\{a\}\to \scrO_{R,\varkappa}(\nu,u,v)\{a\}$
such that $\Bbbk\scrQ_R(M)=\scrQ_\Bbbk(\Bbbk M)$ for any $M\in \bfO_{R,\t}^\kappa\{a\}$.

If $z\geqslant 0$ then 
$\varphi\big(\widehat\Pi[\lambda+\tau_\Bbbk, \kappa_\Bbbk])=
\big(\Pi_{\nu,u,v}[\lambda+\varkappa_\Bbbk]\big)^o
=\Pi_{\nu^o,v^o,u^o}[\lambda^o+\varkappa^o_\Bbbk].$
For $-\alpha_{k,l}+z\delta\in\widehat \Pi(\tau_\Bbbk,\kappa_\Bbbk)^+,$ we also have
$$
\aligned
\langle \lambda^o+\varkappa^o_\Bbbk:h(-\alpha_{k,l}+z\delta)\rangle
&=\langle \lambda^o+\varkappa^o_\Bbbk:-\alpha_{k^o,l^o}\rangle\\
&=-\langle \lambda,\alpha_{k,l}\rangle-z\,e,\\
&=-\langle \lambda:\alpha_{k,l}\rangle-\tau_{K,u}+\tau_{K,v}+z\,\kappa_\Bbbk\\
&=\langle (0,\lambda+\tau_\Bbbk,\kappa_\Bbbk):-\alpha_{k,l}+z\delta\rangle.
\endaligned$$
Thus, by \cite[thm.~11]{F2} and the discussion above, we have
equivalences of categories 
$$\scrQ_R:\bfO_{R,\t}^\kappa\{a\}\to \scrO_{R,\varkappa}(\nu^o,v^o,u^o)\{a^o\},\quad 
\scrQ_\Bbbk:\bfO_{\Bbbk,\t}^\kappa\{a\}\to \scrO_{\Bbbk,\varkappa}(\nu^o,v^o,u^o)\{a^o\}$$
such that $\Bbbk\scrQ_R(M)=\scrQ_\Bbbk(\Bbbk M)$ and
$\bfL(\mu+\tau_\Bbbk)_{\Bbbk}\mapsto L(\mu^o+\varkappa^o_\Bbbk)_\Bbbk$ for each $\mu\in P\{a^o\}$.

Now, we can prove part $(b)$. To simplify, we assume $z\leqs 0$. 
The case $z>0$ is proved in a similar way. 

First, note that $\scrQ_\Bbbk$ restricts to an equivalence of abelian categories 
$\bfO_{\Bbbk,\t}^{\nu,\kappa}\{a\}\to \scrO_{\Bbbk,\varkappa}^\nu(\nu,u,v)\{a\}$. We denote it again by $\scrQ_\Bbbk$.
Since $\bfO_{R,\t}^{\nu,\kappa}\{a\}$ and $\scrO_{R,\varkappa}^\nu(\nu,u,v)\{a\}$
are the full subcategories of $\bfO_{R,\t}^\kappa\{a\}$ and $\scrO_{R,\varkappa}(\nu,u,v)\{a\},$
respectively, consisting of the modules whose simple subquotients have a highest weight 
of the form $\widehat{\lambda+\tau_\Bbbk}$ and $\lambda+\varkappa_\Bbbk$ respectively,
with $\lambda\in P^\nu$, we deduce that 
$\scrQ_{R}$ restricts to an equivalence of abelian $R$-categories
$\bfO^{\nu,\kappa}_{R,\t}\{a\}\to \scrO_{R,\varkappa}^\nu(\nu,u,v)\{a\}.$

Next, since $\scrQ_\Bbbk(\bfL(\mu+\tau_\Bbbk)_{\Bbbk,\nu})=L(\mu+\varkappa_\Bbbk)_{\Bbbk,\nu}$ 
for each $\mu\in P^\nu$,
the functor $\scrQ_R$ is an equivalence of highest weight 
$R$-categories such that 
$\scrQ_R(\bfM(\mu+\tau_R)_{R,\nu})=M(\mu+\varkappa_R)_{R,\nu}$ 
for each $\mu\in P^\nu$
by Proposition \ref{claim:eqHWC}.

Parts $(b)$ and $(c)$ are proved.
\end{proof}

\vspace{2mm}

\begin{rk} We do not know how to choose the equivalence of categories $\scrQ_R$ in such a way that
it intertwines the endofunctors $e,f$ of $\bfO$ and $\scrO$. We will not need this.
\end{rk}

\vspace{2mm}

In the rest of this section, to unburden the notation, assume that $z\leqslant 0$. 
The case $z>0$ is completely similar.

Fix a $(\ell-1)$-composition $a=(a_\bullet,a_p)$ of the positive integer $d$. Then, we have the tilting module
$T_{R,a_\bullet}(\nu_\bullet)\in A^{\nu_\circ}_{R,\varkappa_\circ}(\nu_\bullet)$ and, for each $p\neq u,v$,
the tilting module
$T_{R,a_p}(\nu_p)\in A^{\nu_p}_{R}(\nu_p)$.
Recall that $\nu_\circ=(\nu_u,\nu_v)$, $\varkappa_\circ=(\varkappa_u,\varkappa_v)$ and $\nu_\bullet=\nu_u+\nu_v$.
Note that, since $f_{u,v}(\varkappa)=ze\leqslant 0$, the (two-blocks) category $A^{\nu_\circ}_{R,\varkappa_\circ}(\nu_\bullet)$
satisfies the assumptions in Proposition \ref{prop:4.25}.
Let $T_{R,h,d}\in A^{\nu}_{R,h}\{d\}$ be the tilting module 
which is identified, under the equivalence \eqref{2blocs},
with the direct sum of the modules 
$T_{R,a_\bullet}(\nu_\bullet)\otimes
\bigotimes_{p\neq u,v}T_{R,a_p}(\nu_p),$
where the sum runs over the set of all $(\ell-1)$-compositions 
$a$ of $d$. We also write
$T_{\Bbbk,h,d}=\Bbbk T_{R,h,d}\in A^{\nu}_{\Bbbk,h}\{d\}.$

Now, let $R$ be either a field or a regular local deformation ring of dimension $2$.
Assume further that $R$ is analytic and in general position.

The category $\bfA_{K,\tau}^{\nu,\kappa}$ is split semi-simple.
We have defined the module $\bfT_{R,d}\in\bfA^{\nu,\kappa}_{R,\tau},$ the
$R$-algebra homomorphism 
$\psi^s_{R,d}:\bfH^s_{R,d}\to\End_{\bfA^{\nu,\kappa}_{R,\tau}}(\bfT_{R,d})^\op,$
and the functor
$\Psi_{R,d}^s : \bfA^{\nu,\kappa}_{R,\tau}\{d\}\to \bfH^s_{R,d}\mmod$.
By base-change, we get $\bfT_{R_\frakp,d},$ $\psi^s_{R_\frakp,d}$ and $\Psi^s_{R_\frakp,d}$
for each $\frakp\in\frakP$, see Remark \ref{rk:4.7}.

\vspace{2mm}

\begin{lemma}\label{lem:isom333}
Assume that $\frakp\in\frakP_1$ is subgeneric.
Then, we have an isomorphism
$\scrQ_{R_\frakp}(\bfT_{R_\frakp,d})=T_{R_\frakp,h,d}$.
\end{lemma}

\vspace{.5mm}

\begin{proof}
The module $\scrQ_{R_\frakp}(\bfT_{R_\frakp,d})$ is tilting,
because $\scrQ_{R_\frakp}$ is an equivalence of highest weight categories.
Since $\bfT_{R_\frakp,0}$ and $T_{R_\frakp,h,0}$ are
parabolic Verma modules, we have $\scrQ_{R_\frakp}(\bfT_{K,0})=T_{R_\frakp,h,0}.$

Next, the functor $\scrQ_{\Bbbk_\frakp}$ induces an isomorphism of the 
(complexified) Grothendieck groups
$[\bfO_{\Bbbk_\frakp,\t}^{\nu,\kappa}\{a\}]\to[\scrO^{\nu}_{\Bbbk_\frakp,h}\{a\}]$ such that
$\scrQ_{\Bbbk_\frakp} ([\bfL(\lambda+\tau_{\Bbbk_\frakp})_{\Bbbk_\frakp}])=
[L(\lambda+\varkappa_{\Bbbk_\frakp})_{\Bbbk_\frakp}].$
Since it also preserves the classes of the standard modules,
the explicit formulae in 
Lemma \ref{lem:pasbete} imply that $\scrQ_{\Bbbk_\frakp}:[\scrO^{\nu}_{\Bbbk_\frakp,h}\{a\}]\to[\bfO_{\Bbbk_\frakp,\t}^{\nu,\kappa}\{a\}]$ 
commutes with the action of the operators $e,$ $f$
on both sides.

Since $\bfT_{\Bbbk_\frakp,d}=f^d(\bfT_{\Bbbk_\frakp,0})$ and 
$T_{\Bbbk_\frakp,h,d}=f^d(T_{\Bbbk_\frakp,h,0}),$
we deduce that 
$[\scrQ_{\Bbbk_\frakp}(\bfT_{\Bbbk_\frakp,d})]=[T_{\Bbbk_\frakp,h,d}]$ in 
$[\scrO^\nu_{\Bbbk_\frakp,h}]$.
Therefore, we have
$\scrQ_{\Bbbk_\frakp}(\bfT_{\Bbbk_\frakp,d})=T_{\Bbbk_\frakp,h,d}$ 
because two tilting modules are isomorphic if they have the same class in the Grothendieck group.
Since $\scrQ_{R_\frakp}(\bfT_{R_\frakp,d})$ is tilting and 
$\Bbbk_\frakp\scrQ_{R_\frakp}(\bfT_{R_\frakp,d})=\scrQ_{\Bbbk_\frakp}(\bfT_{\Bbbk_\frakp,d})$, 
by Proposition \ref{prop:introhw}$(b)$ the isomorphism over $\Bbbk_\frakp$ can be lift to an isomorphism 
$\scrQ_{R_\frakp}(\bfT_{R_\frakp,d})=T_{R_\frakp,h,d}$.
\end{proof}

\vspace{2mm}

\begin{prop}\label{prop:redI}
Let $\frakp\in\frakP_1$ be subgeneric.
Assume that $\nu_p\geqslant d$ for all $p$.
Then,

(a) $\bfT_{R_\frakp,d}$ is projective in $\bfA_{R_\frakp,\tau}^{\nu,\kappa}$, 

(b) $\psi^s_{R_\frakp,d}$ is an isomorphism 
$\bfH^{s}_{R_\frakp,d}\to\End_{\bfA_{R_\frakp,\tau}^{\nu,\kappa}}(\bfT_{R_\frakp,d})^\op$,

(c)  
$\Psi_{R_\frakp,d}^s$ is fully faithful on $(\bfA_{R_\frakp,\tau}^{\nu,\kappa}\{d\})^\Delta$ and 
$(\bfA_{R_\frakp,\tau}^{\nu,\kappa}\{d\})^\nabla$.
\end{prop}

\vspace{.5mm}

\begin{proof}
Since $\Bbbk_\frakp$ is subgeneric, we may fix $u,v,z$ as above. So, we have
$f_{u,v,z}(\tau_{\Bbbk_\frakp},\kappa_{\Bbbk_\frakp})=z\,e.$
Hence, by Proposition \ref{prop:equivredI} and
Lemma \ref{lem:isom333}, there is an equivalence of highest weight $R_\frakp$-categories
$\scrQ_{R_\frakp}:\bfA_{R_\frakp,\tau}^{\nu,\kappa}\{d\}\to 
A^\nu_{R_\frakp,h}\{d\}$ taking
$\pmb\Delta(\lambda)_{R_\frakp,\tau}$ to $\Delta(\lambda)_{R_\frakp,\varkappa}$
and
$\bfT_{R_\frakp,d}$ to $T_{R_\frakp,h,d}$.
By base change, it specializes to 
an equivalence of highest weight $\Bbbk_\frakp$-categories
$\scrQ_{\Bbbk_\frakp}:\bfA_{\Bbbk_\frakp,\tau}^{\nu,\kappa}\{d\}\to 
A^\nu_{\Bbbk_\frakp,h}\{d\}.$

Recall that $\nu_\circ=(\nu_u,\nu_v)$ and $\nu_\bullet=\nu_u+\nu_v$.
To unburden the notation, we may identify the highest weight $R_\frakp$-categories 
$A^\nu_{R_\frakp,h}\{d\}$  and
$A^{\nu_\circ}_{R_\frakp,\varkappa_\circ}(\nu_\bullet)\{d\}$ via the equivalence
\eqref{2blocs}. The later is a particular case of the categories which have been studied in \S \ref{sec:2blocs}.
Note that we have $\varkappa_{\Bbbk_\frakp,u}-\varkappa_{\Bbbk_\frakp,v}=ze\notin\bbN^\times$.
Thus, Proposition \ref{prop:4.25}$(c)$  implies that $T_{\Bbbk_\frakp,h,d}$ is projective.
Hence, part $(a)$ follows from Proposition \ref{prop:introhw} 
and Lemma \ref{lem:isom333}.

To prove $(b)$
we use Proposition \ref{prop:inj}. Let us check the assumptions. First, 
the fraction field of $R_\frakp$ is $K$.
Since $R$ is in general position, 
the $K$-algebra $\bfH^s_{K,d}$ is split semi-simple.
Next, by \cite[thm.~3.30]{DJM}, the decomposition map
$K_0(\bfH^s_{K,d})\to K_0(\bfH^s_{\Bbbk_\frakp,d})$ is surjective.

Now, let us construct an endomorphism $\theta_{R_\frakp}$
of $\bfH^s_{R_\frakp,d}$.
By Remark \ref{rk:4.7}, we have a pre-categorification
$(E,F,X,T)$ on $A_{R_\frakp,h}^{\nu}$. Let
$\varphi^s_{R_\frakp,d}:
H^s_{R_\frakp,d}
\to\End_{A^{\nu}_{R_\frakp,h}}(T_{R_\frakp,h,d})^\op$
be the corresponding $R_\frakp$-algebra homomorphism.
It is an isomorphism by Proposition \ref{prop:4.25} and the Nakayama's lemma.
Next, by Proposition \ref{prop:isomHecke}, we have an $R_\frakp$-algebra isomorphism
$\alpha_{R_\frakp}:\bfH^s_{R_\frakp,d}\to H^s_{R_\frakp,d}.$
Since 
$\scrQ_{R_\frakp}(\bfT_{R_\frakp,d})=T_{R_\frakp,h,d},$
by functoriality, we have an isomorphism
$\beta_{R_\frakp}:
\End_{\bfA_{R_\frakp,\tau}^{\nu,\kappa}}(\bfT_{R_\frakp,d})^\op
\to
\End_{A^{\nu}_{R_\frakp,h}}(T_{R_\frakp,h,d})^\op.$
We set $\theta_{R_\frakp}=
\alpha^{-1}_{R_\frakp}\circ(\varphi_{R_\frakp,d}^s)^{-1}\circ\beta_{R_\frakp}\circ\psi_{R_\frakp,d}^s$ 
and we write $\theta_K=K\theta_{R_\frakp}$.

To prove $(b)$, we must check that $\theta_{R_\frakp}$ is invertible.
By Proposition \ref{prop:inj}, this follows from the following.

\vspace{2mm}

\begin{claim}\label{claim:5.41}
The endomorphism $\theta_K$ of $\bfH^s_{K,d}$ is an automorphism and it yields the identity on the Grothendieck group.
\end{claim}

\vspace{.5mm}

Now, we prove the claim.
Since $R$ is in general position, by Proposition \ref{prop:preliminaries},
the $K$-algebra morphisms $\psi^s_{K,d}:\bfH^s_{K,d}\to
\End_{\bfA_{K,\tau}^{\nu,\kappa}}(\bfT_{K,d})^\op$ is an
 isomorphism. Hence $\theta_K$ is an automorphism. 
 
Consider the 
the equivalences of categories
$\Psi_{K,d}^s : \bfA^{\nu,\kappa}_{K,\tau}\{d\}\to \bfH^s_{K,d}\mmod$
and
$\Phi_{K,d}^s : A^{\nu}_{K,h}\{d\}\to H^s_{K,d}\mmod$
induced by $\psi_{K,d}^s$ and $\varphi_{K,d}^s$.
The corresponding maps
between isomorphism classes of simple modules fit into the commutative square
$$\xymatrix{
\Irr(\bfA^{\nu,\kappa}_{K,\tau}\{d\})\ar[r]^-{\Psi_{K,d}^s}\ar[d]_-{\scrQ_K}&\Irr(\bfH^s_{K,d})\\
\Irr(A^{\nu}_{K,h}\{d\})\ar[r]^-{\Phi_{K,d}^s}
&\Irr(H^s_{K,d})\ar[u]_-{\alpha_K},
}$$
because we have
$\Psi_{K,d}^s(\pmb\Delta(\lambda)_{K,\tau})=S(\lambda)_K^{s,q}$,
$\Phi_{K,d}^s(\Delta(\lambda)_{K,\varkappa})=S(\lambda)_K^s$ 
by Propositions \ref{prop:preliminaries}$(d)$, \ref{prop:isomBK}$(d)$, 
and we have
$\alpha_K(S(\lambda)_K^s)=S(\lambda)_K^{s,q}$,
$\scrQ_K(\pmb\Delta(\lambda)_{K,\tau})=\Delta(\lambda)_{K,\varkappa}$.
This implies that  
$\theta_K$ is identity on
the Grothendieck group. The claim is proved.

Finally, let us prove part $(c)$. 
Let
$\varphi^s_{R_\frakp,d}$,
$\beta_{R_\frakp}$
and $\alpha_{R_\frakp}$
be as above.
Then, we can view $\varphi^s_{R_\frakp,d}$ as an isomorphism
$\bfH^s_{R_\frakp,d}
\to
\End_{\bfA_{R_\frakp,\tau}^{\nu,\kappa}}(\bfT_{R_\frakp,d})^\op.$
We don't know whether $\psi^s_{R_\frakp,d}=\varphi^s_{R_\frakp,d}$.
However, since they are both invertible, they differ obviously by an automorphism of
$\bfH^s_{R_\frakp,d}$.
Thus, the equivalence 
$\scrQ_{R_\frakp}$ 
intertwines the functors $\Psi_{R_\frakp,d}^s$ and 
$\Phi_{R_\frakp,d}^s,$
up to a twist by an automorphism of
$\bfH^s_{R_\frakp,d}$.
Therefore, it is enough to prove that
$\Phi_{R_\frakp,d}^s$ is fully faithful on $(A_{R_\frakp,h}^{\nu}\{d\})^\Delta$ and 
$(A_{R_\frakp,h}^{\nu}\{d\})^\nabla$.

By Proposition \ref{prop:4.25},
a simple module of $A^\nu_{\Bbbk_\frakp,h,d}$ is a submodule
of a parabolic Verma module if and only if it lies in the top of $T_{\Bbbk_\frakp,h,d}$.
Thus, the functor
$\Phi_{\Bbbk_\frakp,d}^s$ is  faithful on $(A_{\Bbbk_\frakp,h}^{\nu}\{d\})^\Delta.$
By \cite[cor.~4.18]{BK6}, the category $A_{\Bbbk_\frakp,h}^{\nu}\{d\}$
is Ringel self-dual, i.e., we have an equivalence $A_{\Bbbk_\frakp,h}^{\nu}\{d\}\simeq (A_{\Bbbk_\frakp,h}^{\nu}\{d\})^\diamond.$
Therefore, by Lemma \ref{le:hwcoverRingel}, the functor
$\Phi_{\Bbbk_\frakp,d}^s$ is also faithful on $(A_{\Bbbk_\frakp,h}^{\nu}\{d\})^\nabla.$
Note that \cite{BK6} considers the category $A^\nu$ without any shift $\varkappa$, but our situation reduces to this one by Proposition \ref{prop:shift}.
Now, part $(c)$ follows from Proposition \ref{prop:0to1}.
\end{proof}

\vspace{2mm}

\begin{rk}
If $\nu_u-\nu_v\not\in\bbZ\,e$ for all $u\neq v$, then
$\Psi_{R_\frakp,d}^s$ is a 1-faithful highest weight cover.
\end{rk}

\vspace{3mm}

\subsubsection{The reduction to $\ell=1$}\label{sec:red2}
Assume that the deformation ring $R$ is a local ring with a generic residue field $\Bbbk$.
We have the following lemma.

\vspace{2mm}

\begin{lemma}
\label{lem:reduction2}
For  $\lambda,\lambda'\in P^\nu$, if
${\lambda+\tau_\Bbbk}\leqslant_\ell {\lambda'+\tau_\Bbbk}$ then
$\widehat{\lambda+\tau_\Bbbk}\in \widehat W_\nu\bullet\widehat{\lambda'+\tau_\Bbbk}.$
\end{lemma}

\vspace{.5mm}

\begin{proof}
By an easy induction we may assume that there are elements
$\beta\in\widehat\Pi(\tau_\Bbbk,\kappa_\Bbbk)\setminus\Pi_\nu$ and $w\in W_\nu$ with 
$\widehat{\lambda+\tau_\Bbbk}=ws_\beta\bullet\widehat{\lambda'+\tau_\Bbbk}.$
We have
$$\aligned
\widehat\Pi(\tau_\Bbbk,\kappa_\Bbbk)\subset\widehat\Pi_\nu
&\iff
\langle (0,\tau_\Bbbk,\kappa_\Bbbk):\beta\rangle\notin\bbZ,\quad
\forall\beta\in\widehat\Pi\setminus\widehat\Pi_\nu,\\
&\iff
\langle \tau_\Bbbk:\alpha\rangle+r\kappa\not\in\bbZ,
\quad\forall\alpha\in\Pi\setminus\Pi_\nu,
\quad\forall r\in\bbZ,\\
&\iff
\Bbbk\ \text{is\ generic}.
\endaligned
$$
Thus $\beta\in\widehat\Pi_\nu$, hence $ws_\beta\in\widehat W_\nu$.
\end{proof}

\vspace{2mm}

For $a\in\bbN^\ell$ let
$\bfO^{\nu,\kappa}_{R,\t}\{a\}\subset\bfO^{\nu,\kappa}_{R,\t}$ be the full subcategory 
of the modules whose simple subquotients have a highest weight of the form
$\widehat{\lambda+\t_\Bbbk}$ with
$\lambda\in P^\nu\{a\}.$

\vspace{2mm}

\begin{prop}
\label{prop:reduction2}
(a) We have
$\bfO_{R,\t}^{\nu,\kappa}=\bigoplus_{a\in\bbN^\ell}\bfO^{\nu,\kappa}_{R,\t}\{a\}$ and
$\bfO_{\Bbbk,\t}^{\nu,\kappa}=\bigoplus_{a\in\bbN^\ell}\bfO^{\nu,\kappa}_{\Bbbk,\t}\{a\}$. 

(b) There are equivalences of highest weight $R$-categories 
$\scrQ_R:\,\bfO_{R,\t}^{\nu,\kappa}\{a\}\to \bfO^{+,\kappa}_R(\nu)\{a\}$ and of 
highest weight $\Bbbk$-categories 
$\scrQ_\Bbbk:\,\bfO_{\Bbbk,\t}^{\nu,\kappa}\{a\}\to \bfO^{+,\kappa}_\Bbbk(\nu)\{a\}$
such that $\Bbbk\scrQ_R(M)=\scrQ_\Bbbk(\Bbbk M)$ and
$\scrQ_R(\bfM(\lambda+\tau)_{R,\nu})=\bfM(\lambda)_{R,+}$.

(c) The equivalences in (b) restricts to equivalences of highest weight categories
$\scrQ_R:\bfA_{R,\t}^{\nu,\kappa}\to \bfA^{+,\kappa}_R(\nu)$
and $\scrQ_\Bbbk:\bfA_{\Bbbk,\t}^{\nu,\kappa}\to \bfA^{+,\kappa}_\Bbbk(\nu)$. In particular,
we have $\scrQ_R(\pmb\Delta(\lambda)_{R,\tau})=\pmb\Delta(\lambda)_R$ for all $\lambda$.
\end{prop}

\vspace{.5mm}

\begin{proof} 
Since $\Bbbk$ is generic, the linkage principle and Lemma \ref{lem:reduction2} 
imply that if a parabolic Verma module in $\bfO_{\Bbbk,\t}^{\nu,\kappa}$
has a highest weight of the form $\widehat{\lambda+\t_\Bbbk}$ with
$\lambda\in P^\nu\{a\}$, then
any constituent has also a highest weight of the same form. So we have a decomposition 
$\bfO_{\Bbbk,\t}^{\nu,\kappa}=\bigoplus_{a\in\bbN^\ell}\bfO^{\nu,\kappa}_{\Bbbk,\t}\{a\}$. 
The decomposition over $R$ follows from Proposition \ref{prop:introhw}.
Part $(a)$ is proved.

For the same reason as above, we have $\bfO_{R,\t}^\kappa=\bigoplus_{a\in\bbN^\ell}\bfO_{R,\t}^\kappa\{a\}$,
where $\bfO_{R,\t}^\kappa\{a\}$ is the full subcategory 
of the modules whose simple subquotients have a highest weight of the form
$\widehat{\lambda+\t_\Bbbk}$ with $\lambda\in P\{a\}$.

Further, by \cite[thm.~11]{F2} there is an equivalence
of highest weight $\Bbbk$-categories
$\scrQ_\Bbbk:\bfO_{\Bbbk,\t}^\kappa\{a\}\to\bfO_{\Bbbk}^\kappa(\nu)\{a\}$
such that $\bfL(\lambda+\tau_\Bbbk)_{\Bbbk}\mapsto\bfL(\lambda)_{\Bbbk}$.
For the same reason as explained in the proof of Proposition \ref{prop:equivredI}, the proof of \cite[thm.~11]{F2} yields an equivalence 
$\scrQ_R:\bfO_{R,\t}^\kappa\{a\}\to\bfO_{R}^\kappa(\nu)\{a\}$ such that $\Bbbk\scrQ_R(M)=\scrQ_\Bbbk(\Bbbk M)$ for any $M\in\bfO_{R,\t}^\kappa\{a\}$.

Since $\lambda+\tau_\Bbbk$ is $\nu$-dominant if and only if $\lambda$ is $\nu$-dominant,
this equivalence restricts to an equivalence of abelian categories
$\bfO_{\Bbbk,\tau}^{\nu,\kappa}\{a\}\to\bfO^{+,\kappa}_{\Bbbk}(\nu)\{a\}$. We denote it again by $\scrQ_\Bbbk$.
Since $\bfO_{R,\tau}^{\nu,\kappa}\{a\}$ and $\bfO_{R}^{+,\kappa}(\nu)\{a\}$ are full subcategories of $\bfO_{R,\tau}^\kappa\{a\}$ and 
$\bfO^{\kappa}_{R}(\nu)\{a\}$
consisting of the modules whose simple subquotients have a highest weight of the form $\widehat{\lambda+\tau_\Bbbk}$ and $\widehat{\lambda}$, respectively, with 
$\lambda\in P^\nu\{a\}$, we deduce that $\scrQ_R$ restricts to an equivalence of abelian $R$-categories 
$\scrQ_R:\,\bfO_{R,\tau}^{\nu,\kappa}\{a\}\to\bfO^{+,\kappa}_{R}(\nu)\{a\}$. 
Since $\scrQ_\Bbbk(\bfL(\lambda+\tau_\Bbbk)_{\Bbbk})=\bfL(\lambda)_{\Bbbk}$ for all $\lambda\in P^\nu\{a\}$, 
by Proposition \ref{claim:eqHWC} we deduce that $\scrQ_\Bbbk$ and $\scrQ_R$ are indeed equivalences of highest weight 
categories and that
$\scrQ_R(\bfM(\lambda+\tau)_{R,\nu})=\bfM(\lambda)_{R,+}$. This proves parts $(b)$, $(c)$.
\end{proof}

\vspace{2mm}

Now, let $R$ be either a field or a regular local deformation ring of dimension $2$.
Assume further that $R$ is analytic and in general position.

Consider the Kazhdan-Lusztig category $\bfO^{+,\kappa}_R(\nu_p)$ of $\frakgl_{R,\nu_p}$. 
The equivalence of categories $\bfO_R^\kappa(\nu)=\bigotimes_{p=1}^\ell \bfO_R^\kappa(\nu_p)$ yields an equivalence of categories 
$\bfO^{+,\kappa}_R(\nu)=\bigotimes_{p=1}^\ell\bfO^{+,\kappa}_R(\nu_p)$

Let $\bfV(\nu_p)\in\bfO^{+,\kappa}_R(\nu_p)$ be the module induced from the natural representation of $\frakgl_{R,\nu_p}$. 
The endofunctor $f_p=\bullet\,\dot\otimes_R \bfV(\nu_p)$ of $\bfO^{+,\kappa}_R(\nu_p)$ extends to an 
endofunctor of $\bfO^{+,\kappa}_R(\nu)$ in the obvious way. We denote it again by $f_p$. Let $f=\bigoplus_{p=1}^\ell f_p$.

We set $\bfT_{R,0}(\nu)=\bigotimes_{p=1}^\ell\bfT_{R,0}(\nu_p)$. 
For each $d\in\bbN,$ we consider the tilting module 
$\bfT_{R,d}(\nu)=f^d(\bfT_{R,0}(\nu))$ in $\bfO^{+,\kappa}_R(\nu)$.
We have introduced a module $\bfT_{R,d}$  in 
$\bfO^{\nu,\kappa}_{R,\tau}.$

By base change, for each $\frakp\in\frakP$, we get the modules
$\bfT_{R_\frakp,d}\in \bfO^{\nu,\kappa}_{R_\frakp,\tau}$ and $\bfT_{R_\frakp,d}(\nu)\in \bfO^{+,\kappa}_{R_\frakp}(\nu).$
The same proof as in Lemma \ref{lem:isom333} yields the following.

\vspace{2mm}

\begin{lemma}\label{lem:Q} 
Assume that $\frakp\in\frakP_1$ is generic. Then, we have an isomorphism 
$\scrQ_{R_\frakp}(\bfT_{R_\frakp,d})=\bfT_{R_\frakp,d}(\nu).$
\end{lemma}

\vspace{.5mm}

On the other hand, for each $a\in\bbN^{\ell},$ we set $\bfH^\ell_{R,a}=\bigotimes_{p=1}^\ell \bfH^+_{R,a_p}$.
By base change, it yields the $R_\frakp$-algebra $\bfH^\ell_{R_\frakp,a}$.

\vspace{2mm}

\begin{lemma}
Let $\frakp\in\frakP_1$ be generic. Then, we have an $R_\frakp$-algebra isomorphism
\begin{equation}\label{eq:plane2}
\bfH^s_{R_\frakp,d}=\bigoplus_{a\in\scrC^\ell_d}\Mat_{\frakS_d/\frakS_a}\bigl(\bfH^\ell_{R_\frakp,a}\bigr).
\end{equation}
\end{lemma}

\vspace{.5mm}

\begin{proof} 
Let $I=\{\t_{R_\frakp,1},\t_{R_\frakp,2},\dots,\t_{R_\frakp,\ell}\}+\bbZ+\kappa_{R_\frakp}\bbZ$ and
$\scrI=\scrI_{R_\frakp}=I/\!\!\sim$. 
Let $\scrI_{\Bbbk_\frakp}$ be the image of $\scrI$ in the residue field $\Bbbk_\frakp$.
Since $\frakp$ is generic, the quiver $\scrI_{\Bbbk_\frakp}$ has exactly $\ell$ components 
given by $\scrI_{\Bbbk_\frakp,p}=(\t_{\Bbbk_\frakp,p}+\bbZ+\kappa_{\Bbbk_\frakp}\bbZ)/\!\!\sim$ with $p\in[1,\ell]$.
Hence, the quiver $\scrI_{R_\frakp}$ has also $\ell$ components $\scrI_1=\scrI_{R_\frakp,1},\dots,\scrI_\ell=\scrI_{R_\frakp,\ell}$
which specialize to $\scrI_{\Bbbk_\frakp,1},\dots,\scrI_{\Bbbk_\frakp,\ell}$ respectively.

For each tuple $\bfp=(p_1,p_2,...,p_d)$ in $[1,\ell]^d,$
we consider the idempotent in $\bfH^s_{\Bbbk_\frakp,d}$ given by
$1_\bfp=\sum_\bfi 1_\bfi$, where $\bfi=(i_1,i_2,\dots,i_d)$ runs over the set 
$\scrI_{\Bbbk_\frakp,\bfp}=\prod_{r=1}^d\scrI_{\Bbbk_\frakp,p_r}$
and $1_\bfi$ is as in \S \ref{sec:3.4}.
Note that, although there may be an infinite number of such tuples $\bfi$,
this sum contains only a finite number of non zero terms. 
Next, for each $a\in\scrC^\ell_d$, we define a central idempotent  $1_{(a)}$ in $\bfH^s_{\Bbbk_\frakp,d}$ by
$1_{(a)}=\sum_{\bfp\in(a)}1_\bfp,$ where $a$ is identified with the tuple $(1^{a_1}2^{a_2}\dots \ell^{a_\ell})$ and $(a)$ is the set of all permutations of $a$
in $[1,\ell]^d$.
Then, we
write $\bfH^s_{\Bbbk_\frakp,(a)}=1_{(a)}\bfH^s_{\Bbbk_\frakp,d}.$

It is well-known that there are $\Bbbk_\frakp$-algebra isomorphisms
$\bfH^s_{\Bbbk_\frakp,d}=\bigoplus_{a\in\scrC^\ell_d}\bfH^s_{\Bbbk_\frakp,(a)},$
$\bfH^\ell_{\Bbbk_\frakp,a}=1_a\,\bfH^s_{\Bbbk_\frakp,d}1_a$
and
$\bfH^s_{\Bbbk_\frakp,(a)}=\Mat_{\frakS_d/\frakS_a}\bigl(\bfH^\ell_{\Bbbk_\frakp,a}\bigr),$
where $\frakS_a=\frakS_{a_1}\times\cdots\times\frakS_{a_\ell}$,
see \cite{DM}.
We must prove that the isomorphism 
$\bigoplus_{a\in\scrC^\ell_d}\Mat_{\frakS_d/\frakS_a}\bigl(\bfH^\ell_{\Bbbk_\frakp,a}\bigr)\to\bfH^s_{\Bbbk_\frakp,d}$
lifts to an isomorphism of $R_\frakp$-algebras.
To do that, by the Nakayama's lemma, it is enough to prove that this isomorphism
 lifts to an $R_\frakp$-algebra homomorphism 
 $\bigoplus_{a\in\scrC^\ell_d}\Mat_{\frakS_d/\frakS_a}\bigl(\bfH^\ell_{R_\frakp,a}\bigr)\to\bfH^s_{R_\frakp,d}.$

First, by Proposition \ref{prop:isomHecke}, for each tuple $\bfi\in(\scrI_{\Bbbk_\frakp})^d$, the
sum $1_\bfi=\sum_{\bfi'} 1_{\bfi'}$ over all elements $\bfi'\in\scrI^d$ whose residue class is equal to $\bfi,$
is an idempotent in the $R_\frakp$-subalgebra $\bfH^s_{R_\frakp,d}$ of $\bfH^s_{K,d}.$
Therefore, for each tuple $\bfp\in [1,\ell]^d,$
the idempotent $1_\bfp\in\bfH^s_{K,d}$ given by
$1_\bfp=\sum_{\bfi'} 1_{\bfi'}$, where $\bfi'$ runs over the set 
$\scrI_{\bfp}=\prod_{r=1}^d\scrI_{p_r},$ belongs also to the $R_\frakp$-subalgebra $\bfH^s_{R_\frakp,d}$ and it specializes to the
idempotent $1_\bfp\in\bfH^s_{\Bbbk_\frakp,d}$ given above.
In particular, for each $a\in\scrC^\ell_d$, the idempotent  in $\bfH^s_{K,d}$ given by $1_{(a)}=\sum_{\bfp\in(a)}1_\bfp$
belongs indeed to $\bfH^s_{R_\frakp,d}$ and it specializes to the
idempotent $1_{\{a\}}\in\bfH^s_{\Bbbk_\frakp,d}$ given above.
Further, setting $\bfH^s_{R_\frakp,(a)}=1_{(a)}\bfH^s_{R_\frakp,d},$ we get $R_\frakp$-algebra isomorphisms
$\bfH^s_{R_\frakp,d}=\bigoplus_{a\in\scrC^\ell_d}\bfH^s_{R_\frakp,(a)}$ and
$\bfH^\ell_{R_\frakp,a}=1_a\,\bfH^s_{R_\frakp,d}1_a.$
 
 Now, we construct an $R_\frakp$-algebra homomorphism 
 $\bigoplus_{a\in\scrC^\ell_d}\Mat_{\frakS_d/\frakS_a}\bigl(1_a\bfH^s_{R_\frakp,d}1_a\bigr)\to\bfH^s_{R_\frakp,d}$
 which lifts the isomorphism over the residue field $\Bbbk_\frakp$ mentioned above.
 
To do that, it is convenient to use the formalism of \emph{quiver Hecke algebras}.
Let $\bfR^s_{K,d}$ be the 
\emph{cyclotomic quiver Hecke algebra} of rank $d$ associated with $s$.
It is the $K$-algebra generated by elements
$1_{\bfi},$ $x_{\bfi,k},$ $\tau_{\bfi,l}$  with
$\bfi\in \scrI^d,$ $k\in[1,d]$ and $l\in[1,d)$,
subject to the relations in \cite[sec.~3.2.1]{R2} associated with the quiver $\scrI$ and to the cyclotomic relations given by
$(x_{\bfi,1})^{\sharp\{p\,;\,q^{s_p}=i_1\}}=0$ for all $\bfi$'s.
Note that the $K$-algebra $\bfR^s_{K,d}$ is finite dimensional, and that we have $1_\bfi=0$ except for a finite number of $\bfi$'s.

By \cite{BK3}, \cite{R2}
we have a $K$-algebra isomorphism
$\bfR^s_{K,d}=\bfH^s_{K,d}$ which identifies the idempotents $1_\bfi,$ $\bfi\in\scrI^d$, from both sides.
In particular, for each integer $l\in[1,d)$ and each $d$-tuple $\bfp$ such that $p_l\neq p_{l+1}$,
the element $\tau_{\bfp,l}=\sum_{\bfi\in\scrI_\bfp}\tau_{\bfi,l}$ in
$\bfR^s_{K,d}$ can be viewed as an element of $\bfH^s_{K,d}$ which belongs to
$\bfH^s_{R_\frakp,d}$ and which satisfies the relations $\tau_{s_l(\bfp),l}\,\tau_{\bfp,l}=1_\bfp$
and $\tau_{\bfp,l}\,\tau_{s_l(\bfp),l}=1_{s_l(\bfp)}.$

Next, let $w\in \frakS_d.$ Assume that $w$ is of minimal length in its right 
$\frakS_a$-coset. Fix a reduced decomposition $w=s_{r_m}\cdots s_{r_2} s_{r_1}$.
Consider the elements $\tau_{w,a}$ and $\bar\tau_{a,w}$ of $\bfH^s_{R_\frakp,d}$ given by
$\tau_{w,a}=\tau_{s_{r_m},s_{r_{m-1}}\dots s_{r_1}(a)}\cdots 
\tau_{s_{r_2},s_{r_1}(a)}\tau_{s_{r_1},a}$ and
$\bar\tau_{a,w}=\tau_{s_{r_1},s_{r_{2}}\dots s_{r_m}w(a)}\cdots 
\tau_{s_{r_{m-1}},s_{r_{m}}w(a)}\tau_{s_{r_m},w(a)}$.
We have  $\bar\tau_{a,w}\,\tau_{w,a}=1_a$ and $\tau_{w,a}\,\bar\tau_{a,w}=1_{w(a)}.$

The expected map
$\bigoplus_{a\in\scrC^\ell_d}\Mat_{\frakS_d/\frakS_a}\bigl(1_a\bfH^s_{R_\frakp,d}1_a\bigr)\to\bfH^s_{R_\frakp,d}$ takes
the square matrix $(x_{v(a),w(a)})_{v,w}$ in $\Mat_{\frakS_d/\frakS_a}\bigl(1_a\bfH^s_{R_\frakp,d}1_a\bigr)$ 
with $x_{v(a),w(a)}\in 1_a\bfH^s_{R_\frakp,d}1_a$ and $v,w\in\frakS_d$ as above to the sum
$\sum_{v,w}\tau_{w,a}\,x_{v(a),w(a)}\,\bar\tau_{v,a}$.
\end{proof}

\vspace{2mm}

Now, given $\bfp=(p_1,p_2,...,p_d)$ in $[1,\ell]^d,$
we write $f_\bfp=f_{p_1}f_{p_2}...f_{p_d}$ and $\bfT_{R,\bfp}(\nu)=f_\bfp(\bfT_{R,0}(\nu))$. 
We have an isomorphism $\bfT_{R,d}(\nu)=\bigoplus_\bfp\bfT_{R,\bfp}(\nu)$.

Recall that we identify a composition $a=(a_1,\ldots, a_\ell)$ in $\scrC_d^{\ell}$ with the $\ell$-tuple $(1^{a_1}2^{a_2}....\ell^{a_\ell})$. 
Then, we have an isomorphism $\bfT_{R,a}(\nu)=\bigotimes_{p=1}^\ell\bfT_{R,a_p}(\nu_p).$

Next, consider the action of the symmetric group $\frakS_d$ on $[1,\ell]^d$ by permutation. 
Each orbit contains a unique element given by a composition $a\in \scrC_d^{\ell}$. 
Let $(a)$ denote this orbit. We have a bijection $\frakS_d/\frakS_a\overset{\sim}\to\{a\}$ given by $w\mapsto w(a)$, 
where $\frakS_a$ is the stabilizer of $a$. 
We write $\bfT_{R,(a)}(\nu)=\bigoplus_{\bfp\in\{a\}}\bfT_{R,\bfp}(\nu)$. 

For each $\bfp\in(a)$ we have a canonical isomorphism  $\bfT_{R,\bfp}(\nu)=\bfT_{R,a}(\nu)$. 
Therefore, we have  $\bfT_{R,(a)}(\nu)=\bigoplus_{w\in\frakS_d/\frakS_a}\bfT_{R,a}(\nu)$. We deduce that
$$\End_{\bfA^{+,\kappa}_{R}(\nu)}(\bfT_{R,(a)}(\nu))=\Mat_{\frakS_d/\frakS_a}\bigl(\End_{\bfA^{+,\kappa}_{R}(\nu)}(\bfT_{R,a}(\nu))\bigr).$$
Next, recall that $\bfO^{+,\kappa}_R(\nu)=\bigoplus_{a\in\scrC^\ell_d}\bfO^{+,\kappa}_R(\nu)(a)$ and that
$\bfT_{R,\bfp}(\nu)\in\bfO^{+,\kappa}_R(\nu)\{a\}$ if and only if $\bfp\in(a)$.
Therefore, we have 
\begin{eqnarray}\label{eq:plane1}
\End_{\bfA^{+,\kappa}_{R}(\nu)}(\bfT_{R,d}(\nu))&=&\bigoplus_{a\in\scrC^\ell_d}\End_{\bfA^{+,\kappa}_{R}(\nu)}(\bfT_{R,(a)}(\nu))\nonumber\\
&=&\bigoplus_{a\in\scrC^\ell_d}\Mat_{\frakS_d/\frakS_a}\bigl(\End_{\bfA^{+,\kappa}_{R}(\nu)}(\bfT_{R,a}(\nu))\bigr).
\end{eqnarray}

For each $p\in[1,\ell],$  the $\bfg_{R,\nu_p}$-module $\bfT_{R,a_p}(\nu_p)\in\bfA^{+,\kappa}_R(\nu_p)$
gives rise to an $R$-algebra homomorphism $\bfH^+_{R,a_p}\to\End_{\bfA^{+,\kappa}_R(\nu_p)}(\bfT_{R,a_p}(\nu_p))^\op$ given by \eqref{eq:psi-s}. 
Taking the tensor product , we get an $R$-algebra homomorphism $\bfH^\ell_{R,a}\to \End_{\bfA^{+,\kappa}_{R}(\nu)}(\bfT_{R,a}(\nu))^\op$. 

Now, assume that $\frakp\in\frakP_1$ is generic.
Combining the $R$-algebra homomorphism above with base change, \eqref{eq:plane2} and  \eqref{eq:plane1}, 
we get an $R_\frakp$-algebra homomorphism
$\psi_{R_\frakp,d}^+(\nu) :
\bfH^s_{R_\frakp,d}\to\End_{\bfA^{+,\kappa}_{R_\frakp}(\nu)}(\bfT_{R_\frakp,d}(\nu))^\op$.
Further, the composition with $\psi_{R_\frakp,d}^+(\nu)$ yields a functor
$\Psi_{R_\frakp,d}^+(\nu)=\Hom_{\bfA^{+,\kappa}_{R_\frakp}(\nu)}\big(\bfT_{R_\frakp,d}(\nu),\bullet\big): 
\bfA^{+,\kappa}_{R_\frakp}(\nu)\{d\}\to \bfH^s_{R_\frakp,d}\mmod.$

\vspace{2mm}

\begin{lemma}\label{lem:levelone}
Let $\frakp\in\frakP_1$ be generic. The following holds

(a)
$\bfT_{R_\frakp,d}(\nu)$ is projective in $\bfA^{+,\kappa}_{R_\frakp}(\nu)\{d\}$,

(b) 
$\psi_{R_\frakp,d}^+(\nu)$ is an isomorphism
$\bfH^s_{R_\frakp,d}\to\End_{\bfm_{R_\frakp,\nu}}(\bfT_{R_\frakp,d}(\nu))^\op$,

 (c) 
 $\Psi^+_{R_\frakp,d}(\nu)$ is fully faithful on
 $(\bfA^{+,\kappa}_{R_\frakp}(\nu)\{d\})^\Delta$ and $(\bfA^{+,\kappa}_{R_\frakp}(\nu)\{d\})^\nabla$.
\end{lemma}

\vspace{.5mm}

\begin{proof}
We have an equivalence of highest weight categories $\bfO^{+,\kappa}_{\Bbbk_\frakp}(\nu)=\bigotimes_{p=1}^\ell\bfO^{+,\kappa}_{\Bbbk_\frakp}(\nu_p)$ 
and each factor $\bfO^{+,\kappa}_{\Bbbk_\frakp}(\nu_p)$ is a copy of the Kazhdan-Lusztig 
category. Therefore $\bfO^{+,\kappa}_{\Bbbk_\frakp}(\nu)$
is equivalent to a category of modules over a quantum group
by \cite{KL}. Hence $\bfA^{+,\kappa}_{\Bbbk_\frakp}(\nu)\{d\}$ is equivalent to the module category of a $q$-Schur algebra (with $\ell=1$)
as a highest weight category,
and this equivalence takes $\Psi^+_{\Bbbk_\frakp,d}(\nu)$ to the q-Schur functor.

Hence, some standard facts on $q$-Schur algebras imply that
$\bfT_{\Bbbk_\frakp,d}(\nu)$ is projective in $\bfA^{+,\kappa}_{\Bbbk_\frakp}(\nu)\{d\}$, proving part $(a)$, and that the 
$\Bbbk_\frakp$-algebra homomorphism $\bfH^s_{\Bbbk_\frakp,d}\to\End_{\bfm_{\Bbbk_\frakp,\nu}}(\bfT_{\Bbbk_\frakp,d}(\nu))^\op$ is an isomorphism, 
proving part $(b)$ by Nakayama's lemma and 
\eqref{eq:plane1}, \eqref{eq:plane2}, see e.g.,  \cite{R1}.

Now, we concentrate on part $(c)$.
A standard argument due to Donkin implies that the q-Schur functor 
$\Xi^s_{\Bbbk_\frakp,d}$ is faithful on $(\bfS^s_{\Bbbk_\frakp,d}\mmod)^\nabla$ for $\ell=1$.
More precisely, recall that $\Xi^s_{\Bbbk_\frakp,d}=\Hom_{\bfS^s_{\Bbbk_\frakp,d}}(\bfS^s_{\Bbbk_\frakp,d}\,e,\bullet)$ for some idempotent 
$e\in\bfS^s_{\Bbbk_\frakp,d}.$
Recall also that the $\bfS^s_{\Bbbk_\frakp,d}$-module $\bfS^s_{\Bbbk_\frakp,d}\,e$ is faithful and that
any Weyl module embeds in $\bfS^s_{\Bbbk_\frakp,d}\,e$, see e.g., \cite[p.~188]{Martin}.
Thus, the claim follows from \cite[thm.~4.5.5]{Martin}.
So, from the equivalence above, we deduce that $\Psi^+_{\Bbbk_\frakp,d}(\nu)$ is faithful on
$(\bfA^{+,\kappa}_{\Bbbk_\frakp}(\nu)\{d\})^\nabla$.
Since the q-Schur algebra is Ringel self-dual, we deduce that
$\Psi^+_{\Bbbk_\frakp,d}(\nu)$ is also faithful on
$(\bfA^{+,\kappa}_{\Bbbk_\frakp}(\nu)\{d\})^\Delta$.
Therefore, the part $(c)$ of the lemma follows from Proposition \ref{prop:0to1}.
\end{proof}

\vspace{2mm}

We can now prove the main result of this section.
Recall that we have introduced a module $\bfT_{R,d}$  in 
$\bfA^{\nu,\kappa}_{R,\tau},$ an
$R$-algebra homomorphism 
$\psi^s_{R,d}:\bfH^s_{R,d}\to\End_{\bfA^{\nu,\kappa}_{R,\tau}}(\bfT_{R,d})^\op,$
and a functor
$\Psi_{R,d}^s : \bfA^{\nu,\kappa}_{R,\tau}\{d\}\to\bfH^s_{R,d}\mmod$.

By base-change, we get $\bfT_{R_\frakp,d},$ $\psi^s_{R_\frakp,d}$ and $\Psi^s_{R_\frakp,d}$ for each $\frakp\in\frakP$, see Remark \ref{rk:4.7}.
Recall also that, since $R$ is in general position,  the $K$-category $\bfA_{K,\tau}^{\nu,\kappa}$ is split semi-simple and 
condition \eqref{(A)} holds in $K$.

\vspace{2mm}

\begin{prop}\label{prop:red2}
Let $\frakp\in\frakP_1$ be generic.
Assume that $\nu_p\geqslant d$ for all  $p$.
Then

(a) $\bfT_{R_\frakp,d}$ is projective in $\bfA_{R_\frakp,\tau}^{\nu,\kappa}$, 

(b) $\psi^s_{R_\frakp,d}$ is an isomorphism 
$\bfH^{s}_{R_\frakp,d}\to\End_{\bfA_{R_\frakp,\tau}^{\nu,\kappa}}(\bfT_{R_\frakp,d})^\op$,

(c) $\Psi_{R_\frakp,d}^s$ is fully faithful on $(\bfA_{R_\frakp,\tau}^{\nu,\kappa}\{d\})^\Delta$ and $(\bfA_{R_\frakp,\tau}^{\nu,\kappa}\{d\})^\nabla$.
\end{prop}

\vspace{.5mm}

\begin{proof}
Assuming part $(b)$, the Proposition \ref{prop:reduction2} and Lemma \ref{lem:levelone} imply parts $(a)$ and $(c)$.
Let us prove $(b)$. 

The proof is similar to the proof of Proposition \ref{prop:redI}. It is based on Proposition \ref{prop:inj}. Recall that $\bfH^s_{K,d}$ is a split semi-simple 
$K$-algebra, and by \cite[thm.~3.30]{DJM}, that the decomposition map $K_0(\bfH^s_{K,d})\to K_0(\bfH^s_{\Bbbk_\frakp,d})$ is surjective.
We construct an endomorphism $\theta_{R_\frakp}$
of $\bfH^s_{R_\frakp,d}$ as follows. 
By Lemma \ref{lem:levelone}, we have an isomorphism $\psi_{R_\frakp,d}^+(\nu):\bfH^s_{R_\frakp,d}\to\End_{\bfm_{R_\frakp,\nu}}(\bfT_{R_\frakp,d}(\nu))^\op$.
Next, by Proposition \ref{prop:reduction2} and Lemma \ref{lem:Q},
we have an equivalence of categories
$\scrQ_{R_\frakp}:\bfO_{R_\frakp,\t}^{\nu,\kappa}\{a\}\to \bfO_{R_\frakp}^{+,\kappa}(\nu)\{a\}$ which maps
$\bfT_{R_\frakp,d}$ to $\bfT_{R_\frakp,d}(\nu).$
By functoriality, it induces an isomorphism
$\beta_{R_\frakp}:
\End_{\bfA_{R_\frakp,\tau}^{\nu,\kappa}}(\bfT_{R_\frakp,d})^\op
\to
\End_{\bfA_{R_\frakp}^{+,\kappa}(\nu)}(\bfT_{R_\frakp,d}(\nu))^\op.$
We set $\theta_{R_\frakp}=
(\varphi_{R_\frakp,d}^s)^{-1}\circ\beta_{R_\frakp}\circ\psi_{R_\frakp,d}^s$. 
The same proof as in Proposition \ref{prop:redI} implies that the map $\theta_K=K\theta_{R_\frakp}$ induces the identity on the Grothendieck group.
So $\theta_{R_\frakp}$ is an automorphism by Proposition \ref{prop:inj}.
This implies that $\psi^s_{R_\frakp,d}$ is an isomorphism.
\end{proof}

\vspace{2mm}

\begin{rk}
If $q_{\Bbbk_\frakp}\neq 1$ then $\Psi_{R_\frakp,d}^s$ is a 1-faithful highest weight cover.
\end{rk}

\vspace{3mm}

\subsection{The category $\bfA$ as a highest weight cover}
\label{sec:refined}
Let $R$ be a local analytic deformation ring of dimension 2 in general position. Let $\kappa_{\Bbbk}=-e$ and $s_{R}=\nu+\tau_{R}$.
Recall the module $\bfT_{R,d}\in\bfA^{\nu,\kappa}_{R,\tau}$ and the functor 
$\Psi^s_{R,d}:\bfA^{\nu,\kappa}_{R,\tau}\to \bfH^{s}_{R,d}\mmod$.

The first main result of this paper is the following theorem.
\vspace{2mm}

\begin{thm}\label{thm:isom}
Assume that $\nu_p\geqslant d$ for all $p$.

(a) The map $\psi^s_{R,d}:\bfH^{s}_{R,d}\to\End_{\bfA^{\nu,\kappa}_{R,\tau}}\big(\bfT_{R,d}\big)^\op$ 
is an $R$-algebra isomorphism. 

(b) The module $\bfT_{R,d}$ is projective in $\bfA^{\nu,\kappa}_{R,\tau}$.

(c) The functor $\Psi^s_{R,d}$ is fully faithful on $\bfA^{\nu,\kappa,\Delta}_{R,\tau}$ and $\bfA^{\nu,\kappa,\nabla}_{R,\tau}$.
\end{thm}

\vspace{1mm}

\begin{proof}
First, by Proposition \ref{prop:preliminaries},  the category $\bfA_{K,\tau}^{\nu,\kappa}$ 
is split semi-simple and 
condition \eqref{(A)} holds in the fraction field $K$.

To prove part $(a)$, observe that since $\bfT_{R,d}$ is tilting,  
the $R$-module $\End_{\bfA^{\nu,\kappa}_{R,\tau}}(\bfT_{R,d})$ is projective.
Since $\bfH^s_{R,d}$ is also projective over $R$, we have
$$\bfH^s_{R,d}=\bigcap_{\frakp\in\frakP_1} R_\frakp \bfH^s_{R,d},\qquad
\End_{\bfA^{\nu,\kappa}_{R,\tau}}(\bfT_{R,d})=
\bigcap_{\frakp\in\frakP_1} R_\frakp\End_{\bfA^\nu_{R,\tau}}(\bfT_{R,d}),$$
see [Bourbaki, \emph{Alg\`ebre commutative}, \text{ch.}~VII, \S 4, $\text{n}^\circ 2$].
Next, we have
$R_\frakp \bfH^s_{R,d}=\bfH^s_{R_\frakp,d}$ and
$R_\frakp\End_{\bfA^{\nu,\kappa}_{R,\tau}}(\bfT_{R,d})=
\End_{\bfA^{\nu,\kappa}_{R_\frakp,\tau}}(\bfT_{R_\frakp,d})$ for each $\frakp\in\frakP$.
Thus, it is enough to prove that the map $\psi_{R_\frakp,d}^s$ is invertible for each $\frakp\in\frakP_1$.
By Proposition \ref{prop:preliminaries}, the prime $\frakp$ is
generic or subgeneric. Thus part $(a)$ follows from 
Proposition \ref{prop:redI} and Proposition \ref{prop:red2}. 

\smallskip

Now, let us prove that $\Psi^s_{R,d}$ is fully faithful on $\bfA^{\nu,\kappa,\nabla}_{R,\tau}$. Since $\bfT_{R,d}$ is tilting, by Corollary \ref{cor:2.16}
it is enough to check that
$\Psi^s_{R_\frakp,d}$ is fully faithful on 
$(\bfA^{\nu,\kappa}_{R_\frakp,\tau}\{d\})^\nabla$ for $\frakp\in\frakP_1$. This has already been proved in Propositions \ref{prop:redI} and \ref{prop:red2}.

\smallskip
As a consequence, the tilting module $\bfT_{R,d}$ is projective by Lemma \ref{le:projfromFrob}, because the algebra $\End_{\bfA^{\nu,\kappa}_{R,\tau}}(\bfT_{R,d})$ being isomorphic to $\bfH^s_{R,d}$ is symmetric. Part $(b)$ is proved.
\smallskip

We deduce that $\Psi^s_{R,d}$ is quotient functor. Therefore by Lemma \ref{prop:RR} it is fully faithful over $\bfA^{\nu,\kappa,\Delta}_{R,\tau}$ if 
$\Psi^s_{R_\frakp,d}$ is fully faithful on $(\bfA^{\nu,\kappa}_{R_\frakp,\tau}\{d\})^\Delta$ for $\frakp\in\frakP_1$. 
Again, this has been proved in Propositions \ref{prop:redI} and \ref{prop:red2}. The theorem is proved.
\end{proof}

\medskip

The following corollary is a straightforward consequence of the theorem by specializing to the residue field, see also \cite{L3}.

\vspace{2mm}

\begin{cor}
Assume that $\nu_p\geqslant d$ for all $p$.

(a) The map $\psi^\nu_{\Bbbk,d}:\bfH^{\nu}_{\Bbbk,d}\to\End_{\bfA^{\nu,-e}_\Bbbk}\big(\bfT_{\Bbbk,d}\big)^\op$ 
is a $\Bbbk$-algebra isomorphism. 

(b) The module $\bfT_{\Bbbk,d}$ is projective in $\bfA^{\nu,-e}_\Bbbk$.
\end{cor}

\vspace{2mm}

\begin{rk}
The module $\bfT_{\Bbbk,d}$ may not be projective in $\bfO^{\nu,\kappa}_\Bbbk$.
\end{rk}

\begin{rk} \label{rk:projective}
Let $R$ be any local deformation ring. 
Assume that $\nu_p\geqslant d$ for each $p$.
From Theorem \ref{thm:isom}(b), Proposition \ref{prop:introhw} and Remark \ref{rk:dfT}
we deduce that $\bfT_{R,d}$ is well-defined and is projective in $\bfA^{\nu,\kappa}_{R,\tau}$.
\end{rk}

\vspace{3mm}

\subsection{The functor $F$ and induction}\label{ss:Finduction}

In \S \ref{sec:catA} we defined a pre-categorical action $(E,F,X,T)$ on $A_{R,\tau}^{\nu,\kappa}$.
Now, we define a tuple $(E,F,X,T)$ on
$\bfA_{R,\tau}^{\nu,\kappa}$ in the following way.
Let $h:\bfA_{R,\tau}^{\nu,\kappa}\to\bfO^{\nu,\kappa}_{K,\tau}$
be the canonical embedding.
Consider the endofunctors $E,F$ of $\bfA_{R,\t}^{\nu,\kappa}$ given by
$E=h^* e h,$ $F=h^* f h$. 
Since $f$ preserves the subcategory $\bfA_{R,\t}^{\nu,\kappa},$ we have
$F=h^!fh=f|_{\bfA_{R,\t}^{\nu,\kappa}}$.
In particular, the functor $F$ is exact, $(E,F)$ is an adjoint pair and 
we have $E(\bfA^{\nu,\kappa}_{R,\tau}\{d+1\}^\Delta)\subset (\bfA^{\nu,\kappa}_{R,\tau}\{d\})^\Delta$.
Then, we define
$X\in\End(F)=\End(f)$ and $T\in\End(F^2)=\End(f^2)$ as in Proposition \ref{prop:5.14}.

Let $d$, $k$ be positive integers such that $d+k\leqslant\nu_p$ for all $p$.
In  this section we compare the functors $F^k:\bfA^{\nu,\kappa}_{R,\tau}\{d\}\to\bfA^{\nu,\kappa}_{R,\tau}\{d+k\}$ and
$\Ind_d^{d+k}=\bfH^s_{R,d+k}\otimes_{\bfH^s_{R,d}}\bullet:\bfH^s_{R,d}\mmod\to\bfH^s_{R,d+k}\mmod$.

By definition $F^k(\bfT_{R,d})=\bfT_{R,d+k}$ and we have a commutative diagram
\[\xymatrix{\bfH^s_{R,d}\ar[rr]^{\psi^s_{R,d}\qquad\qquad}_{\sim\qquad\qquad}\ar@{^{(}->}[d] 
&&\End_{\bfA^{\nu,\kappa}_{R,\tau}}(\bfT_{R,d})^{\op}\ar[d]^{F^k}\\
\bfH^s_{R,d+k}\ar[rr]^{\psi^s_{R,d+k}\qquad\qquad}_{\sim\qquad\qquad}&&\End_{\bfA^{\nu,\kappa}_{R,\tau}}(F^k\bfT_{R,d})^{\op}.\\}\]
Therefore, we have a morphism of functors on $\bfA^{\nu,\kappa}_{R,\tau}\{d\}$
\begin{eqnarray*}
\vartheta_k: \quad\Ind_d^{d+k}\Psi^s_{R,d}&\simeq&\Hom_{\bfA^{\nu,\kappa}_{R,\tau}}(\bfT_{R,d+k},F^k\bfT_{R,d})
\otimes_{\End_{\bfA^{\nu,\kappa}_{R,\tau}}(\bfT_{R,d})^{\op}}\Hom_{\bfA^{\nu,\kappa}_{R,\tau}}(\bfT_{R,d},\bullet)\\
&\simeq&\Hom_{\bfA^{\nu,\kappa}_{R,\tau}}(E^k\bfT_{R,d+k},\bfT_{R,d})\otimes_{\End_{\bfA^{\nu,\kappa}_{R,\tau}}(\bfT_{R,d})^{\op}}
\Hom_{\bfA^{\nu,\kappa}_{R,\tau}}(\bfT_{R,d},\bullet)\\
&\to&\Hom_{\bfA^{\nu,\kappa}_{R,\tau}}(E^k\bfT_{R,d+k},\bullet)\\
&\simeq&\Hom_{\bfA^{\nu,\kappa}_{R,\tau}}(\bfT_{R,d+k},F^k\bullet)\\
&=&\Psi^s_{R,d+k}F^k,
\end{eqnarray*}
where the map in the third line is given by composition.

\medskip
\begin{lemma}
\label{lem:phiFcom}
Assume that $d+k\leqslant\nu_p$ for all $p$. Then $\vartheta_k: \Ind_d^{d+k}\Psi^{s}_{R,d}\to\Psi^{s}_{R,d+k} F^k$ is an isomorphism.
\end{lemma}

\begin{proof}
It is enough to prove that $\vartheta_k$ is an isomorphsim of functors on $( \bfA^{\nu,\kappa}_{R,\tau}\{d\})^\Delta$.
We must prove that the map
$$\Hom_{\bfA^{\nu,\kappa}_{R,\tau}}(E^k\bfT_{R,d+k},\bfT_{R,d})\otimes_{\End_{\bfA^{\nu,\kappa}_{R,\tau}}(\bfT_{R,d})^{\op}}
\Hom_{\bfA^{\nu,\kappa}_{R,\tau}}(\bfT_{R,d},M)\to\Hom_{\bfA^{\nu,\kappa}_{R,\tau}}(E^k\bfT_{R,d+k},M)$$
given by composition is an isomorphism for each $M\in( \bfA^{\nu,\kappa}_{R,\tau}\{d\})^\Delta$.

Since $\Psi^s_{R,d}$ is $0$-faithful, $E(\bfA^{\nu,\kappa}_{R,\tau}\{d+1\}^\Delta)\subset (\bfA^{\nu,\kappa}_{R,\tau}\{d\})^\Delta$
and $\Psi^s_{R,d}(\bfT_{R,d})\simeq\bfH^s_{R,d}$ as $(\bfH^s_{R,d},\bfH^s_{R,d})$-bimodules, the left hand side is isomorphic to
$$\Hom_{\bfH^s_{R,d}}(\Psi^s_{R,d}(E^k\bfT_{R,d+k}),\bfH^s_{R,d})\otimes_{\bfH^s_{R,d}}\Hom_{\bfH^s_{R,d}}(\bfH^s_{R,d},\Psi^s_{R,d}(M))$$
and the right hand side is isomorphic to
$\Hom_{\bfH^s_{R,d}}(\Psi^s_{R,d}(E^k\bfT_{R,d+k}),\Psi^s_{R,d}(M)).$
Hence, we are reduced to prove that the natural map
\begin{multline*}
\Hom_{\bfH^s_{R,d}}(\Psi^s_{R,d}(E^k\bfT_{R,d+k}),\bfH^s_{R,d})\otimes_{\bfH^s_{R,d}}\Hom_{\bfH^s_{R,d}}(\bfH^s_{R,d},\Psi^s_{R,d}(M))
\to\\
\to\Hom_{\bfH^s_{R,d}}(\Psi^s_{R,d}(E^k\bfT_{R,d+k}),\Psi^s_{R,d}(M))
\end{multline*}
given by composition is an isomorphism. We claim that $\Psi^s_{R,d}(E^k\bfT_{R,d+1})\simeq \bfH^s_{R,d+k}$ as $\bfH^s_{R,d}$-modules. 
Thus it is a projective $\bfH^s_{R,d}$-module, and the isomorphism follows.

To prove the claim,
note that since $\Psi^s_{R,d}=\Hom_{\bfA^{\nu,\kappa}_{R,\tau}}(\bfT_{R,d},\bullet)$ is fully faithful on 
$(\bfA^{\nu,\kappa}_{R,\tau}\{d\})^\nabla$, using the duality $\scrD$ on $\bfA^{\nu,\kappa}_{R,\tau}\{d\}$ and the fact that 
$\scrD(\bfT_{R,d})\simeq\bfT_{R,d}$, we deduce that the contravariant functor 
$\Hom_{\bfA^{\nu,\kappa}_{R,\tau}}(\bullet,\bfT_{R,d}): \bfA^{\nu,\kappa}_{R,\tau}\{d\}\to( \bfH^{s}_{R,d})^\op\mmod$ is fully faithful on 
$(\bfA^{\nu,\kappa}_{R,\tau}\{d\})^\Delta$. Therefore, we have isomorphisms
\begin{eqnarray*}
\Psi^s_{R,d}(E^k\bfT_{R,d+k})&\simeq&\Hom_{\bfA^{\nu,\kappa}_{R,\tau}}(\bfT_{R,d},E^k\bfT_{R,d+k})\\
&\simeq&
\Hom_{(\bfH^s_{R,d})^\op}\big(\Hom_{\bfA^{\nu,\kappa}_{R,\tau}}(E^k\bfT_{R,d+k},\bfT_{R,d}),
\bfH^s_{R,d}\big)\\
&\simeq&\Hom_{(\bfH^{s}_{R,d})^\op}\big(\Hom_{\bfA^{\nu,\kappa}_{R,\tau}}(\bfT_{R,d+k},F^k\bfT_{R,d}),\bfH^s_{R,d}\big)\\
&\simeq&\Hom_{(\bfH^{s}_{R,d})^\op}\big(\bfH^s_{R,d+k},\bfH^s_{R,d}\big).\end{eqnarray*}
Finally, since $\bfH^{s}_{R,d+k}$ is a Frobenius algebra, 
there is an isomorphism of $\bfH^s_{R,d}$-modules 
$\bfH^{s}_{R,d+k}\simeq \Hom_{(\bfH^s_{R,d})^\op}\big(\bfH^s_{R,d+k},\bfH^s_{R,d}\big).$
The claim is proved.
\end{proof}

\medskip

\begin{rk}\label{rk:commXXX}
Recall that $X$ acts on $\Ind_d^{d+1}=\bfH^s_{R,d+1}\otimes_{\bfH^s_{R,d}}\bullet$ by right multiplication by $X_{d+1}$ on $\bfH^s_{R,d+1}$, and 
the action of $X$ on $E$ is the transposition of its action on $F$ under the adjunction, see Remark \ref{rem:3.3}.
Hence, it follows from the definition of $\vartheta_1$ that it intertwines the action of $X$ on $\Ind_d^{d+1}$ and on $F$, i.e., we have 
\begin{equation*}
\vartheta_1\circ(\Psi^s_{R,d+1}X)=(X\Psi^s_{R,d})\circ\vartheta_1.
\end{equation*}
Similarly we have
\begin{equation*}
\vartheta_2\circ(\Psi^s_{R,d+2}T)=(T\Psi^s_{R,d})\circ\vartheta_2,
\end{equation*}
for the action of $T$ on $\Ind^{d+2}_d$ and on $F^2$.
\end{rk}

\vspace{3mm}

\vspace{3mm}

\section{The category $\bfA$ and CRDAHA's}


\vspace{3mm}

\subsection{Reminder on rational DAHA's}

\subsubsection{Definition of the category $\calO$}

Let $R$ be a local ring with residue field $\bbC$.
Let  $W$ be a complex reflection group,
let $\frakh$ be the reflection representation of $W$ over $R$ and
let $S$ be the set of pseudo-reflections in $W$.
Let $\calA$ be the set of reflection hyperplanes in $\frakh$. We write $\frakh_{reg}=\frakh\backslash\bigcup_{H\in\calA} H$.

Let $c:S\to R$ be a map that is constant on the
$W$-conjugacy classes.
The \emph{RDAHA (=rational double affine Hecke algebra)} attached to $W$ with parameter $c$ is the
quotient $H_c(W,\frakh)_R$ of the smash product of $R W$ and the tensor algebra of
$\frakh\oplus\frakh^*$ by the relations
$$[x, x' ] = 0, \quad [y, y' ] = 0, \quad
[y, x] = \langle x, y \rangle-\sum_{s\in S}c_s\langle\alpha_s,y\rangle
\langle x,\check\alpha_s\rangle s,$$
for all $x, x'\in\frakh^*$, $y, y'\in\frakh$.
Here $\langle\bullet,\bullet\rangle$
is the canonical pairing between $\frakh^*$ and $\frakh$,
the element $\alpha_s$ is a generator of
$\Im(s|_{\frakh^*}-1)$ and $\check\alpha_s$ is the generator of
$\Im(s|_{\frakh}-1)$
such that $\langle\alpha_s , \check\alpha_s\rangle = 2$.

Let $R[\frakh]$, $R[\frakh^*]$ be the subalgebras of $H_c(W,\frakh)_R$ generated by $\frakh^*$ and $\frakh$ respectively.
The category $\calO$ of $H_c(W,\frakh)_R$ is the full subcategory of the category of
$H_c(W,\frakh)_R$-modules consisting of objects that are finitely generated as
$R[\frakh]$-modules and $\frakh$-locally nilpotent, see \cite[\S~ 3]{GGOR}. We denote it by $\calO_c(W,\frakh)_R$. 
It is a highest weight $R$-category.
The standard modules are labeled
by the set $\Irr(\bbC W)$ of isomorphism classes of irreducible $W$-modules.
The standard module associated with
$E\in\Irr(\bbC W)$ is the induced module
$\Delta(E)_R=\Ind^{H(W)_R}_{W\ltimes R[\frakh^*]}(RE).$
Here $RE$ is regarded as a $W\ltimes R[\frakh^*]$-module such that 
$\frakh\subset R[\frakh^*]$ acts by zero.
Let $L(E)$ be the unique simple quotient of $\Delta(E)_R$, and let $P(E)_R$
be the projective cover of $\Delta(E)_R$.

By \cite[\S~4.2.1]{GGOR} there is a functor
$$(\bullet)^\vee:\calO_c(W,\frakh)_R\to\calO_{c^\vee}(W,\frakh^*)^{\op}_R,$$
which is an equivalence over the subcategories of modules in $\calO_c(W,\frakh)_R$, $\calO_{c^\vee}(W,\frakh^*)^{\op}_R$ that are free over $R$.
Here $c^\vee:S\to R$ is defined by $c^\vee(s)=c(s^{-1})$.
For any $E\in\Irr(\bbC W)$ we write $E^\vee=\Hom_R(E,R)$. We have $\Delta(E)_R^\vee\simeq \nabla(E^\vee)_R$ and
$\nabla(E)_R^\vee\simeq \Delta(E^\vee)_R$. 

\medskip

\subsubsection{The $\KZ$-functor}

Let $R$ be an analytic regular local ring. There is a quotient functor
$$\KZ_R:\calO_c(W,\frakh)_R\to \bfH (W,\frakh)_R\mmod$$
defined in \cite[\S~5.3]{GGOR}, where $\bfH (W,\frakh)_R$ is the Hecke algebra associated with $W$ and a parameter which depends on $c$. 
Note that loc.~cit.~uses regular complete local rings, but the same construction can be done for analytic ones.

\begin{prop}
\label{prop:ffKZ}
The functor $\KZ_R$ is $0$-faithful.
\end{prop}
\begin{proof}
By Proposition \ref{prop:0to1} it is enough to prove that over the residue field $\bbC$ the functor $\KZ$ is $(-1)$-faithful. 
In other words, we must prove that $\KZ$ is faithful on $\calO_c(W,\frakh)^\Delta_\bbC$.

Write $\calO=\calO_c(W,\frakh)_\bbC$, let $\calO_{tor}\subset\calO$ be the full subcategory consisting of the objects $M$ such that 
$M\otimes_{\bbC[\frakh]}\bbC[\frakh_{reg}]=0$. By \cite[thm.~5.14]{GGOR}, the functor $\KZ$ is isomorphic to the quotient functor 
$\calO\to\calO/\calO_{tor}$. A $\Delta$-filtered object $M$ is free over $\bbC[\frakh]$ by \cite[prop.~2.21]{GGOR}, 
so it has no torsion submodules. Therefore, the map
$\Hom_{\calO}(M,N)\to\Hom_{\calO}(\KZ(M),\KZ(N))$
is injective for each $M, N\in\calO^\Delta$. We are done.
\end{proof}

\medskip

\subsubsection{Induction and restriction functors}

A \emph{parabolic subgroup} $W'\subset W$ is the stabilizer of some point $b\in\frakh$. It is a complex reflection group with the set of 
reflections $S'=S\cap W$ and with reflection representation $\frakh/\frakh^{W'}$, where $\frakh^{W'}$ is the subspace of points fixed by $W'$.
Bezrukavnikov and Etingof \cite{BE} defined parabolic induction and restriction functors
$$\OInd_{W'}^W:\calO_{c}(W',\frakh/\frakh^{W'})_R\to\calO_c(W,\frakh)_R,\quad
\ORes_{W'}^W:\calO_c(W,\frakh)_R\to\calO_{c}(W',\frakh/\frakh^{W'})_R.$$
Here we view $c$ as a parameter for $W'$ by identifying it with its restriction to $S'$.
In loc.~cit.~the authors work over a field. The definition is the same over a ring $R$. 
The functor $\OInd_{W'}^W$ is left adjoint to $\ORes_{W'}^W$, and both functors are exact.
In particular $\OInd_{W'}^W$ maps projective objects to projective objects.

Let $R$ be an analytic regular local ring. By \cite[thm.~2.1]{S} we have isomorphisms of functors
\begin{equation}\label{eq:commm}
\KZ_R\circ\ORes_{W'}^{W}\simeq\HRes_{W'}^{W}\circ\KZ_R,\quad \KZ_R\circ\OInd_{W'}^{W}\simeq\HInd_{W'}^{W}\circ\KZ_R,
\end{equation}
where $\ORes_{W'}^{W}$ and $\HInd_{W'}^{W}$ refer to the restriction and induction functors for Hecke algebras $\bfH(W',\frakh/\frakh^{W'})_R\hookrightarrow\bfH(W,\frakh)_R$, see loc.~cit.~for more details. Again, in loc.~cit.~we work over a field, but the same proof works over $R$.

We will be mainly interested in the case where $W'=W_H$ is the pointwise stabilizer of a hyperplane $H$. We will abbreviate 
$\calO(W_H)_R=\calO_c(W_H,\frakh/H)_R$ and $\OInd_H=\OInd_{W_H}^W$.
\medskip

\subsubsection{Support of modules}
Let $R$ be a local ring with residue field $\Bbbk$. We abbreviate $\calO_R=\calO_c(W,\frakh)_R$. If $R=K$ is a field, let $\mathrm{Ch}(M)$ denote 
\emph{the characteristic variety} of $M$ as defined
in \cite[\S~4.3.4]{GGOR}. It is a closed subvariety of $\frakh\oplus\frakh^*$. Recall the notation $\mathrm{lcd}_{\calO}$ and $\mathrm{rcd}_{\calO}$
from \eqref{eq:lcdrcd}. 

\begin{lemma}
\label{lem:lcdrcd}
Assume $R=K$ is a field.
For any $M\in\calO_K$ we have $$\mathrm{lcd}_{\calO_K}(M)=\mathrm{rcd}_{\calO_K}(M)=\dim\frakh-\dim\mathrm{Ch}(M).$$
\end{lemma}

\begin{proof}
The equality $\dim\mathrm{Ch}(M)=\dim\frakh-\mathrm{rcd}_{\calO}(M)$ is proved in \cite[cor.~4.14]{GGOR}. Further, the proof of \cite[lem.~5.2]{GGOR} 
yields $\dim\mathrm{Ch}(M)=\dim\mathrm{Ch}(M^\vee)$. 
This implies that $\mathrm{rcd}_{\calO}(M)=\mathrm{rcd}_{\calO}(M^\vee)$. 
On the other hand,
by \cite[prop.~4.7]{GGOR}, if $T$ is a tilting generator of 
$\calO_K$ then $T^\vee$ is a tilting generator of $\calO^\vee_K$ and
$\Ext^i_{\calO_K}(T,M)\simeq \Ext^i_{\calO^\vee_K}(M^\vee,T^\vee)$. 
We deduce that $\mathrm{lcd}_{\calO_K}(M)=\mathrm{rcd}_{\calO_K}(M^\vee)=\mathrm{rcd}_{\calO_K}(M)$.
\end{proof}

\smallskip

\begin{lemma}
\label{lem:indsupp}
For $E\in\Irr(\bbC W)$ we have $\mathrm{rcd}_{\calO_\Bbbk}(L(E))\leqslant 1$ if and only if there exist $H\in\calA$ and 
$P\in\calO(W_H)^\proj_R$ such that 
$P(E)$ is a direct summand of $\OInd_H(P)$. 
\end{lemma}

\begin{proof}
By \cite[thm.~6.8]{G} we have $\mathrm{Ch}(L(E))=\frakh^{W'}\bigoplus\{0\}\subset\frakh\bigoplus\frakh^*$ for some parabolic subgroup $W'\subset W$.
So $\mathrm{rcd}_{\calO_\Bbbk}(L(E))\leqslant 1$ is equivalent,
by Lemma \ref{lem:lcdrcd}, to the fact that $\frakh^{W'}$ has codimension $\leqslant 1$ in $\frakh$, which is equivalent to $W'\subset W_H$ for some 
hyperplane $H$ in $\calA$. By \cite[prop.~2.2]{SV}, the latter is true if and only if $\ORes_{W_H}^W(L(E))\neq 0$, which is equivalent to
$$\Hom_{\calO_R}\big(\OInd_{W_H}^W(P),L(E)\big)=\Hom_{\calO(W_H)_R}\big(P,\ORes_{W_H}^W(L(E))\big)\neq 0,$$ for some $P\in\calO(W_H)^\proj_R$.
Hence $\mathrm{rcd}_{\calO_\Bbbk}(L(E))\leqslant 1$ is equivalent to 
$P(E)$ being a direct summand of $\OInd_{W_H}^W(P)$ for some $H\in\calA$ and $P\in\calO(W_H)^\proj_R$.
\end{proof}

\vspace{3mm}

\subsection{The category $\calO$ of cyclotomic rational DAHA's}\label{sec:Ocher}

Let $R$ be a local ring. Fix $\kappa_R\in R^\times$ and
$s=(s_{R,1},\ldots,s_{R,\ell})\in R^\ell$.

\subsubsection{Definition}

Recall that $\Gamma$ is the group of $\ell$-th roots of unity in $\bbC^\times$
and that $\Gamma_d$ is the semi-direct product
$\frakS_d\ltimes\Gamma^d$, where $\Gamma^d$ is the Cartesian product of $d$ copies of $\Gamma$.
For $\gamma\in\Gamma$ let $\gamma_i\in\Gamma^d$ be the element with $\gamma$ at the
$i$-th place and with 1 at the other ones.
Let $s_{ij}\in\frakS_d$ be the transposition $(i,j)$.
Write $s_{ij}^{\gamma}=s_{ij}\gamma_i\gamma_j^{-1}$ for $\gamma\in\Gamma$ and $i\neq j$.

Fix a basis $(x,y)$ of $R^2$. Let $x_i$,
$y_i$ denote the elements $x,y$ respectively in the $i$-th summand
of $(R^2)^{\oplus d}$. There is a unique action of the
group $\Gamma_d$ on $(R^2)^{\oplus d}$
such that for distinct $i$, $j$, $k$ we have
$\gamma_i(x_j)=\gamma^{-\delta_{ij}} x_j,$
$\gamma_i(y_j)=\gamma^{\delta_{ij}} y_j$ and
$s_{ij}(x_i)=x_j,$ $s_{ij}(x_k)=x_k$, 
$s_{ij}(y_i)=y_j$ and
$s_{ij}(y_k)=y_k$.

Fix $k\in R$ and $c_\gamma\in R$ for each $\gamma\in\Gamma$.
Note that $\Gamma_d$ is a complex reflection group with reflection
representation $\frakh=R^{\oplus d}$
and $S=\{s_{ij}^\gamma\}_{1\le i\neq j\le d,\gamma\in\Gamma}\coprod
\{\gamma_i\}_{1\le i\le d}$.
Let $c:S\to R$ be the map given by $c(s_{ij}^\gamma)=k$, $c(\gamma_i)=c_\gamma/2$.
We consider the algebra
$H_c(W,\frakh)_R$ for $W=\Gamma_d$.
We will call $H_c(\Gamma_{d},\frakh)_R$ the {\it CRDAHA(=cyclotomic RDAHA)}.
It is the quotient
of the smash product of $R \Gamma_{d}$ and the tensor algebra of
$(R^2)^{\oplus d}$ by the relations
$$[y_i,x_i]=1-k\sum_{j\neq i}\sum_{\gamma\in\Gamma}s_{ij}^{\gamma}
-\sum_{\gamma\in\Gamma\setminus\{1\}}c_\gamma\gamma_i,$$
$$[y_i,x_j]=k\sum_{\gamma\in\Gamma}\gamma s_{ij}^{\gamma}
\quad \text{if}\ i\neq j,$$
$$[x_i,x_j]=[y_i,y_j]=0.$$

We will use a presentation of $H_c(\Gamma_d,\frakh)_R$ where the
parameters are $h, h_0,h_1,\dots h_{\ell-1}$ with (setting $h_{-1}=h_{\ell-1}$)
$$k=-h,\qquad -c_\gamma=\sum_{p=0}^{\ell-1}\gamma^{-p}(h_{p}-h_{p-1})\quad \text{for}\ \gamma\neq1.$$ 
The notation $h=h_R$, $h_p=h_{R,p}$ here is the same as in \cite[sec.~6.1.2]{R1}.
Finally, we choose the elements $h_R$, $h_{R,p}$ in the following way :
\begin{equation}\label{eq:paarameters}
h_R=-1/\kappa_R,\qquad h_{R,p}=-s_{R,p+1}/\kappa_R-p/\ell,\qquad p=0,1,\dots,\ell-1.
\end{equation}

\emph{
In the rest of this section we assume that the residue field is $\Bbbk=\bbC$ and that 
$s_{\Bbbk,p}\in\bbZ$ for all $p$.} 

Write $\kappa=\kappa_\Bbbk$ and $s_p=s_{\Bbbk,p}$.
We abbreviate $\calO_R^{s,\kappa}\{d\}=\calO_c(\Gamma_d,\frakh)_R$. If $\ell=1$, then $c$ only depends on $\kappa$, we abbreviate 
$\calO_R^{\kappa}(\frakS_d)=\calO_c(\frakS_d,\frakh)_R$. The category
$\calO_R^{s,\kappa}\{d\}$ is a highest weight $R$-category such that
$\Delta(\calO^{s,\kappa}_R\{d\})=\{\Delta(\lambda)^{s,\kappa}_R\,;\,\lambda\in\scrP^\ell_d\}$
and
$\Delta(\lambda)^{s,\kappa}_R=\Delta(\scrX(\lambda)_\bbC)_R.$
We write $L(\lambda)^{s,\kappa}$, $P(\lambda)^{s,\kappa}_R$, $T(\lambda)^{s,\kappa}_R$, $I(\lambda)^{s,\kappa}_R$ for the corresponding  simple, projective, tilting, injective object in $\calO^{s,\kappa}_R\{d\}$.

\vspace{3mm}

\subsubsection{Comparison of partial orders}

The partial order on the set
$\Delta(\calO^{s,\kappa}_R\{d\})\simeq\scrP^\ell_d$ is defined as follows.
Let $A$, $B$ be boxes of $\ell$-partitions. We say $A\succ_s B$ if we have $\ct^s(A)<\ct^s(B)$ or if 
$\ct^s(A)=\ct^s(B)$ and $p(A)> p(B)$. 
We define a partial order $\geqslant_{s,\kappa}$ on $\scrP^\ell_d$ by setting 
$\lambda\geqslant_{s,\kappa}\mu$ if and only if there are orderings 
$Y(\lambda)=\{A_n\}$ and $Y(\mu)=\{B_n\}$ such that 
$A_n\succcurlyeq_s B_n$ for all $n$. 

\vspace{2mm}

\begin{lemma}
Assume $\kappa<0$. Then $\geqslant_{s,\kappa}$ is a highest weight order on $\calO^{s,\kappa}_R\{d\}$.
\end{lemma}

\vspace{.5mm}

\begin{proof}
By the proof of \cite[thm.~4.1]{DG}, if $[\Delta(\lambda)^{s,\kappa}_\Bbbk: L(\mu)^{s,\kappa}_\Bbbk]\neq 0$ then there exist orderings
$Y(\lambda)=\{A_n\}$ and $Y(\mu)=\{B_n\}$ 
and  non negative integers $D_n$ such that 
$$D_n=p(A_n)-p(B_n)+\ell\,\big(\ct^s(A_n)-\ct^s(B_n)\big)/\kappa,$$
for all $n$ and $(\ct^s(A_n)-\ct^s(B_n))/\kappa\in\bbZ$. Our notation matches those of 
loc.~cit.~in the following way : $r=\ell$, $c_0=\kappa^{-1}$, $d_p=-\ell h_p.$ Now, since $\kappa<0$ and 
$p(A_n)$, $p(B_n)\in[1,\ell]$, we have $D_n\geqs 0$ if and only if $A_n\succcurlyeq_s B_n$.
\end{proof}

\vspace{2mm}


Set $s^\star=(-s_\ell,-s_{\ell-1},\dots,-s_1).$
For each $\lambda\in\scrP^\ell_d$ we write $\lambda^\star=({}^t\lambda^\ell,\dots,{}^t\lambda^2,{}^t\lambda^1)$.
We have the following lemma which is similar to \cite[lem.~2.2]{L3}.

\vspace{2mm}

\begin{lemma}\label{lem:order3}
Assume that $\kappa<0$ and that $s_p=\nu_p\geqslant d$ for all $p$.
Then the order $\geqslant_{s^\star,\kappa}$ refines the order $\geqslant_\ell$, i.e.,
for any $\lambda,\mu\in\scrP^\ell_d$ such that $\mu\geqslant_\ell\lambda$ we have
$\mu^\star\geqslant_{s^\star,\kappa}\lambda^\star$.
\end{lemma}

\vspace{.5mm}

\begin{proof}
First, for any $\lambda\in\scrP^\nu_d$ and $A\in Y(\lambda)$, we have the transposed  box
$A^\star\in Y(\lambda^\star)$ such that
$\ct^{s^\star}(A^\star)=-\ct^{s}(A)$ and $p(A^\star)=\ell+1-p(A)$. 
Therefore, we have $A\prec_s B$ if and only if $A^\star\succ_{s^\star} B^\star$.

Let $\lambda,\mu\in\scrP^\ell_d$ be such that $\mu\geqslant_\ell\lambda$.
Assume that $w\in W_\nu$, $\beta\in\widehat\Pi^+$ are such that
$\langle\widehat{\varpi(\mu)}+\tilde\rho\,:\,\beta\rangle>0$ and
$ws_\beta\bullet\widehat{\varpi(\mu)}=\widehat{\varpi(\lambda)}.$
We must prove that $\mu^\star\geqslant_{s^\star,\kappa}\lambda^\star$, which is equivalent to 
$\mu\leqslant_{s,\kappa}\lambda$.

Write $\beta=\alpha_{k,l}+r\delta$ and $\lambda'=s_{k,l}(\mu+\rho_\nu)+er\alpha_{k,l}-\rho_\nu.$
Set  $n=\langle\mu+\rho_\nu\,:\,\alpha_{k,l}\rangle-er>0.$
We have $w(\lambda'+\rho_\nu)=\lambda+\rho_\nu$ and
$(\lambda'+\rho_\nu)_k=(\mu+\rho_\nu)_k-n,$ $(\lambda'+\rho_\nu)_l=(\mu+\rho_\nu)_l+n.$

For $k\in[1,N]$ let $k'=k-\nu_1-\nu_2-\ldots-\nu_{p_k-1}$ where $p_k$ is 
such that $k\in J^\nu_{p_k}$. Then the diagram $Y(\lambda')$ is obtained from the diagram 
$Y(\mu)$ by removing $n$ boxes from the right 
end of the $k'$-th row of the $p_k$-th partition of $\mu$ and adding $n$ 
boxes the right end of the $l'$-th row of 
the $p_l$-th partition of $\mu$. 

We number the removed boxes by $B_1,B_2,\ldots B_n$ ordered from left to 
right, and the added boxes by $A_1,A_2,\ldots A_n$ ordered from left to right. 
We claim that $B_j\preccurlyeq A_j$ 
for $1\leqs j\leqs n$. 

To prove this, note first that $B_j\preccurlyeq A_j$ if and only if $B_n\preccurlyeq A_n$, 
because we have $\ct^s(B_j)-\ct^s(A_j)=\ct^s(B_n)-\ct^s(A_n)$, $p(B_j)=p(B_n)$ and 
$p(A_j)=p(A_n)$. 

Now let 
us compare $B_n$ and $A_n$. Observe that 
$$\ct^s(B_n)=(\mu+\rho_\nu)_k-1,\quad \ct^s(A_n)=
(\lambda'+\rho_\nu)_l-1=(\mu+\rho_\nu)_l+n-1.$$
Recall that $\beta=\al_{k,l}+r\delta$ is a positive root. 
Therefore, we have either $r>0$, and  then
$$\ct^s(B_n)-\ct^s(A_n)=\langle\mu+\rho_\nu\,:\,\alpha_{k,l}\rangle-n=er>0,$$
or we have $r=0$ and $k\leqs l$, and
then $\ct^s(B_n)=\ct^s(A_n)$ and $p(B_n)=p_k\leqs p_l=p(A_n)$. 
We deduce that $B_n\preccurlyeq A_n$. Hence we have shown $\mu\leqs_{s,\kappa}\lambda'$.

Next, recall that $w\in W_\nu$ is such that the tuple $\lambda+\rho_\nu$ is $\nu$-dominant. Thus, we can write $w=s_{\beta_m}s_{\beta_{m-1}}\ldots s_{\beta_1}$ such that $\beta_i=\alpha_{k,l}$ for some $k< l$ with 
$p_k=p_l$, that 
$\gamma_i=s_{\beta_{i-1}}s_{\beta_{i-2}}\ldots s_{\beta_1}(\lambda'+\rho_\nu)-\rho_\nu\in\bbN^N$ 
and that $n=-\langle\gamma_i+\rho_\nu,\beta_i\rangle>0$. We set $\gamma_0=\lambda'$. Repeating 
the argument of the last paragraph with $\beta=\beta_i$ yields that $Y(\gamma_{i+1})$ is obtained from 
$Y(\gamma_{i})$ by removing $n$ boxes in the $l'$-th row of $\gamma_i^{p_l}$ and adding them to the 
$k'$-th row. Order the removed boxes by $B_1,B_2,\ldots,B_n$ and the added one by $A_1,A_2,\ldots,A_n$ in 
the same way as above. Then the same computation as above yields that $\ct^s(A_j)=\ct^s(B_j)$ and 
$p(A_j)=p(B_j)$ for all $j=1,2,\ldots n$.
Therefore we have $\gamma_{i+1}=\gamma_i$ for the order $\leqs_{s,\kappa}$. 
We deduce that $\lambda=\gamma_{m+1}=\gamma_{m+1}=\lambda'$.
Therefore $\mu\leqs_{s,\kappa}\lambda$. The lemma is proved.
\end{proof}

\vspace{2mm}

\subsubsection{The $\KZ$-functor}

Now, let $R$ be a local analytic deformation ring and set
$q_R=\exp(-2\pi \sqrt{-1}/\kappa_R)\in R^\times$.
Consider the $\KZ$-functor
$\KZ_{R,d}^s:\calO^{s,\kappa}_R\{d\}\to\bfH^s_{R,d}\mmod.$

\vspace{2mm}

\begin{lemma}\label{lem:6.3}
Assume that \eqref{(A)} holds in $K$. Then
$\Irr(\bfH^{s}_{K,d})=\{S(\lambda)^{s,q}_K\,;\,\lambda\in\scrP^\ell_d\}$,
$\Irr(\calO^{s,\kappa}_K\{d\})=
\{\Delta(\lambda)^{s,\kappa}_K\,;\,\lambda\in\scrP^\ell_d\}$
and the bijection
$\Irr(\calO^{s,\kappa}_K\{d\})\iso\Irr(\bfH^{s}_{K,d})$ 
induced by $\KZ^{s}_{K,d}$ takes
$\Delta(\lambda)^{s,\kappa}_K$ to $S(\lambda)^{s,q}_K$.
\end{lemma}

\begin{proof}
The first statement follows from the semi-simplicity of $\bfH^{s}_{K,d}$ 
and from \cite[thm.~2.19]{GGOR}.
The second one follows from Tits' deformation Theorem, because
the modules $\KZ^{s}_{K,d}(\Delta(\lambda)^{s,\kappa}_K)$ and $S(\lambda)^{s,q}_K$
are both the generic point of a flat family of modules whose fiber at the special point 
is the $\bbC\Gamma_d$-module $\scrX(\lambda)_\bbC$, see \cite[lem.~3.1]{S} for details.
\end{proof}

\vspace{3mm}

\subsubsection{The Ringel duality}
\label{sec:ringel}
By \cite[prop.~4.10]{GGOR}, there is an equivalence of 
categories $\scrR:\calO^{s^\star,\kappa}_R\{d\}^\Delta\iso\bigl(\calO^{s,\kappa}_R\{d\}^\Delta\bigr)^\op$ that restricts to an equivalence
$\calO^{s^\star,\kappa}_R\{d\}^\tilt\iso\bigl(\calO^{s,\kappa}_R\{d\}^\proj\bigr)^\op$. Hence, it induces an equivalence
$\calO^{s^\star,\kappa}_R\{d\}^\blackdiamond\iso \,
\calO^{s,\kappa}_R\{d\}^\op$. We have
$\scrR(\Delta(\lambda^\star)_R^{s^\star,\kappa})\simeq\Delta(\lambda)_R^{s,\kappa}.$
Consider the isomorphism of $R$-algebras
$$\iota:\bfH^{s}_{R,d}\iso(\bfH^{s^\star}_{R,d})^\op,\ T_i\mapsto -q_RT_i^{-1},\
X_j\mapsto X_j^{-1}.$$
It induces an equivalence
$$\scrR_\bfH=\iota^*(\bullet^{\!\vee}):\bfH^{s^\star}_{R,d}\mmod\cap R\mproj\iso
(\bfH^{s}_{R,d}\mmod)^\op\cap R\mproj,$$
where $\bullet^{\!\vee}$ is the dual as an $R$-module.

By \cite[\S 5.4.2]{GGOR}, there is a commutative diagram
\begin{equation}\label{eq:ringelcom}
\begin{split}
\xymatrix{
\calO^{s^\star,\kappa}_R\{d\}^\Delta\ar[rr]^{\scrR}_\sim\ar[d]_{\KZ^{s^\star}_{R,d}}&&\bigl(\calO^{s,\kappa}_R\{d\}^\Delta\bigr)^\op\ar[d]^{\KZ^{s}_{R,d}}\\
\bfH^{s^\star}_{R,d}\mmod\cap R\mproj\ar[rr]^{\scrR_\bfH}_\sim&&
(\bfH^{s}_{R,d}\mmod)^\op\cap R\mproj.}
\end{split}
\end{equation}
In particular, if $R=K$ is a field satisfying the condition \eqref{(A)},
then Lemma \ref{lem:6.3} yields $K\scrR_\bfH(S(\lambda^\star)^{s^\star,q}_K)\simeq S(\lambda)^{s,q}_K$.

We will also consider the $R$-algebra isomorphisms
$$\IM:\bfH_{R,d}^s\iso\bfH_{R,d}^{s^\star},\
T_i\mapsto -q_RT_i^{-1},\ X_j\mapsto X_j^{-1}$$
and
$$\sigma: (\bfH_{R,d}^s)^\op\iso\bfH_{R,d}^{s},\
T_i\mapsto T_i,\ X_j\mapsto X_j.$$
Note that the composition $\IM^*\scrR_\bfH^{-1}$ is given by
$\sigma^*(\bullet^{\!\vee})$.

\vspace{3mm}

\subsection{Proof of Varagnolo-Vasserot's conjecture}

Let $R$ be a local analytic deformation ring of dimension $2$ in general position with residue field $\Bbbk=\bbC$. Fix $e$, $\ell$, $N\in\bbN^\times$. 
Fix $\kappa_R\in R^\times$ such that $\kappa_\Bbbk=-e$ and $\nu\in\scrC^\ell_{N,+}$. We set $s_{R,p}=\nu_p+\t_{R,p}$, 
$q_R=\exp(-2\pi \sqrt{-1}/\kappa_R)$ and $Q_{R,p}=\exp(-2\pi \sqrt{-1}s_{R,p}/\kappa_R)$. 
We may abbreviate $\kappa=\kappa_\Bbbk$, $s_p=s_{\Bbbk,p}$.

\subsubsection{Small rank cases}

As a preparation for the proof, we start by comparing the highest weight covers $\KZ^s_{R,d}:\calO^{s,\kappa}_R\{d\}\to\bfH^s_{R,d}\mmod$ and 
$\Psi^s_{R,d}:\bfA^{\nu,\kappa}_{R,\tau}\{d\}\to\bfH^s_{R,d}\mmod$ for $d=1$, $2$.

First, assume that $d=1$. Then $\Gamma_d=\Gamma$ is a cyclic group. 
The Hecke algebra associated with $\Gamma$ is $\bfH^s_{R,1}=R[X_1]/\big(\prod_{p=1}^\ell (X_1-Q_{R,p})\big)$. 

\vspace{2mm}

\begin{prop}
\label{prop:rankone}
We have $\KZ^s_{R,1}(P(\lambda)^{s,\kappa}_R)\simeq \Psi^s_{R,1}(\bfT(\lambda)_{R,\tau})$ for any $\lambda\in\scrP^\ell_1$.
\end{prop}

\vspace{.5mm}

\begin{proof}
For each $p\in[1,\ell]$ let $\lambda_p\in\scrP^\ell_1$ be the $\ell$-partition with $1$ on the $p$-th component and $\emptyset$ elsewhere. 
By Remark \ref{rem:5.15}, Proposition \ref{prop:preliminaries}(d) and Lemma \ref{lem:6.3}, we have 
$$\KZ^s_{K,1}(\Delta(\lambda_p)^{s,\kappa}_K)\simeq K[X_1]/(X_1-Q_{R,p})\simeq\Psi^s_{K,1}(\pmb\Delta(\lambda_p)_{K,\tau}).$$

Since $\KZ^s_{R,1}\,\scrR$ and $\Psi^s_{R,1}$ are highest weight covers
of $\bfH^s_{R,1}$ with compatible orders,
Proposition \ref{pr:coverrank1} shows that
$\KZ^s_{R,1}\,\scrR(T(\lambda^\star)_R^{s^\star,\kappa})\simeq
\Psi^s_{R,1}(\bfT(\lambda)_{R,\tau})$ for any $\lambda\in\scrP^\ell_1$ and
the result follows.
\end{proof}

\vspace{2mm}

Now, assume that $d=2$. Recall the Hecke algebra $\bfH^+_{R,2}=R[T_1]/(T_1+1)(T_1-q_R)$ associated with the group 
$\frakS_2$. Write $\lambda_+=(2)$ and $\lambda_-=(1^2)$ in $\scrP^1_2$.
The category $\calO^\kappa_R(\frakS_2)$ is a special case of $\calO^{s,\kappa}_R\{1\}$ with $\ell=2$. Proposition \ref{prop:rankone} yields
\begin{equation}\label{eq:reprep}
\KZ_{R,d}^s(P(\lambda_\sharp)^{\kappa}_R)\simeq \Psi^+_{R,2}(\bfT(\lambda_\sharp)_{R,\tau}),\quad \sharp=+,-.
\end{equation}
Consider the induction functor $\Ind_{2,+}^{2,s}:
\bfH^+_{R,2}\mmod\to\bfH^s_{R,2}\mmod$.

\vspace{2mm}

\begin{prop}
\label{prop:ranktwo}
Assume $\nu_p\geqslant 2$ for all $p$.
For $\sharp=+,-,$ there exists a tilting object $\bfT_\sharp\in\bfA^{\nu,\kappa}_{R,\tau}\{2\}$ such that 
$\Psi^s_{R,2}(\bfT_\sharp)\simeq \Ind_{2,+}^{2,s}(\Psi^+_{R,2}(\bfT(\lambda_\sharp)_{R,\tau})$.
\end{prop}

\vspace{.1mm}

\begin{proof}
By Theorem \ref{thm:isom}(a), the module $\Psi^{s}_{R,2}(\bfT_{R,2})$ is the regular representation of $\bfH^{s}_{R,2}$. 
Write $\bfT^+_{R,2}=\bfV^{\dot\otimes 2}_R$. We have $\Psi^{+}_{R,2}(\bfT^+_{R,2})\simeq \bfH^+_{R,2}$.
Thus, there is an isomorphism of 
$\bfH^{s}_{R,2}$-modules
\begin{equation}\label{eq:iso+s2}
\Psi^{s}_{R,2}(\bfT_{R,2})\simeq \Ind_{2,+}^{2,s}(\Psi^{+}_{R,2}(\bfT^+_{R,2})).
\end{equation}

\vspace{2mm}

If $e>2$ then $\kappa_\Bbbk\neq -2$, hence $\bfH^{+}_{\Bbbk,2}$ is semi-simple and
$\bfT^+_{R,2}\simeq\bfT(\lambda_+)_{R,\tau}\bigoplus\bfT(\lambda_-)_{R,\tau}$.
Since $\Psi^s_{R,2}$ is $0$-faithful, it maps indecomposable factors of $\bfT_{R,2}$ to indecomposable $\bfH^s_{R,2}$-modules. 
So, the proposition follows from \eqref{eq:iso+s2} and the Krull-Schmidt theorem.

Now, assume that $e=2$, then $q_\Bbbk=-1$. The indecomposable tilting modules in $\bfA^{+,\kappa}_{R,\tau}\{2\}$ are 
$\bfT(\lambda_+)_R=\bfT^+_{R,2}$ and $\bfT(\lambda_-)_R=\pmb\Delta(\lambda_-)_R$. We need to prove the proposition for $\bfT(\lambda_-)_R$.
We have $\Psi^+_{R,2}(\bfT(\lambda_-)_R)\simeq R[T_1]/(T_1+1)$.
Consider the action of $\bfH^{+}_{R,2}$ on $\bfT^+_{R,2}$. Then $\bfT(\lambda_-)_R$ is
the image of $T_1-q_R$ acting on $\bfT^+_{R,2}$.
Since the functor $\bfT_{R,0}\dot\otimes_R\,\bullet$ is exact, we deduce that 
$\bfT_{R,0}\dot\otimes_R\bfT(\lambda_-)_R$ is the image of $T_1-q_R$ acting on 
$\bfT_{R,2}$. By consequence, $\Psi^s_{R,2}\big(\bfT_{R,0}\dot\otimes_R\bfT(\lambda_-)_R\big)$ is the image of the right multiplication by $T_1-q_R$ on 
$\Psi^s_{R,2}(\bfT_{R,2})=\bfH^s_{R,2}$. Therefore, we have
\begin{equation}\label{eq:eqeqeq}
\Psi^s_{R,2}\big(\bfT_{R,0}\,\dot\otimes_R\bfT(\lambda_-)_R\big)\simeq\Ind_{2,+}^{2,s}\big(\Psi^+_{R,2}(\bfT(\lambda_-)_R)\big).
\end{equation}

We claim that $\bfT_{R,0}\dot\otimes_R\bfT(\lambda_-)_R$ is tilting in $\bfA^{\nu,\kappa}_{R,\tau}\{2\}$. 
Indeed, by Proposition \ref{prop:monoidalR}, the specialization map $\End_{\bfA^{\nu,\kappa}_{R,\tau}}(\bfT_{R,2})\to
\End_{\bfA^{\nu,\kappa}_{\Bbbk,\tau}}(\bfT_{\Bbbk,2})$ takes $T_1-q_R$ to $T_1-q_\Bbbk$. Since
$\bfT_{R,0}\dot\otimes_R\bfT(\lambda_-)_R$ is free over $R$ by Lemma \ref{lem:A16} and since it is the image of the operator
$T_1-q_R:\bfT_{R,2}\to\bfT_{R,2}$, 
the image of $T_1-q_\Bbbk:\bfT_{\Bbbk,2}\to\bfT_{\Bbbk,2}$ is $\Bbbk(\bfT_{R,0}\dot\otimes_R\bfT(\lambda_-)_R)$. 
The same argument as above implies that
$\bfT_{\Bbbk,0}\dot\otimes_\Bbbk\bfT(\lambda_-)_\Bbbk$ is also the image of the operator
$T_1-q_\Bbbk:\bfT_{\Bbbk,2}\to\bfT_{\Bbbk,2}$. 
We deduce that there is an isomorphism
$\Bbbk(\bfT_{R,0}\dot\otimes_R\bfT(\lambda_-)_R)\simeq \bfT_{\Bbbk,0}\dot\otimes_\Bbbk\bfT(\lambda_-)_\Bbbk$. 
Since
$\bfT_{\Bbbk,0}\dot\otimes_\Bbbk\bfT(\lambda_-)_\Bbbk$ is tilting
by Proposition \ref{prop:TP2bis}, the claim follows from Proposition \ref{prop:introhw}(c).
The proposition is proved.
\end{proof}

\vspace{3mm}

\subsubsection{Proof of the main theorem}
\label{sec:comparison2}
We can now prove Conjecture \cite[conj.~8.8]{VV}.

\vspace{2mm}
 
 \begin{thm}\label{thm:main3}
Assume that $\nu_p\geqslant d$ for each $p$.
Then, we have an equivalence of highest weight categories 
$\Upsilon^{\nu,-e}_d:\bfA^{\nu,-e}\{d\}\iso\,\calO^{\nu^\star,-e}\{d\}$ such that 
$\Upsilon^{\nu,-e}_d(\pmb\Delta(\lambda))\simeq \Delta(\lambda^\star)^{\nu^\star,-e}$ and 
$\Psi^{\nu}_d\simeq\IM^*\KZ^{\nu^\star}_d\Upsilon^{\nu,-e}_d$.
 \end{thm}
 
\vspace{.5mm}

\begin{proof}
Let $R$ be a local analytic deformation ring of dimension 2 in general position with residue field $\Bbbk=\bbC$. Assume that $\kappa_\Bbbk=-e$. 
Set $s_{R,p}=\nu_p+\t_{R,p}$.

Let $\scrC=\calO^{s^\star,\kappa}_R\{d\}$ and  $\scrC'=\bfA^{\nu,\kappa}_{R,\tau}\{d\}$. We consider the highest weight covers
$F=\IM^*\KZ_{R,d}^{s^\star}:\scrC\to \bfH^{s}_{R,d}\mmod$ and
$F'=\Psi^s_{R,d}:\scrC'\to \bfH^{s}_{R,d}\mmod.$
We claim that they satisfy the conditions in Proposition \ref{prop:key}, so the theorem holds.
Let us check these conditions.

First, $\bfH^{s}_{R,d}$ is Frobenius. Since $R$ is in general position, the condition \eqref{(A)} 
holds in $K$, hence $\bfH^{s}_{K,d}$ is semi-simple.

We have $KF(\Delta(\lambda^\star)^{s^\star,\kappa}_K)=S(\lambda)^{s,q}_K$ by Lemma \ref{lem:6.3} and \S \ref{sec:ringel}, and we have
and $KF'(\pmb\Delta(\lambda)_{K,\tau})=S(\lambda)^{s,q}_K$ by Proposition \ref{prop:preliminaries}. So the order on 
$\Irr(\bfH^{s}_{K,d})$ induced by $(\scrC,F)$ refines the order induced by $(\scrC',F')$ by Lemma \ref{lem:order3}.

Since $\IM^*$ is an equivalence, by Proposition \ref{prop:ffKZ} the functor $F$ is fully faithful on $\scrC^\Delta$. 
Hence it is also fully faithful on $\scrC^\nabla$, by \eqref{eq:ringelcom} and \cite[\S 4.2.1]{GGOR}. 
Theorem \ref{thm:isom}(c) gives the fully faithfulness of $F'$ on ${\scrC'}^\Delta$ and ${\scrC'}^\nabla$.

It remains to check that $F(T(\lambda)^{s^\star,\kappa}_R)\in F'({\calC'}^\tilt)$ for all $\lambda\in\scrP^\ell_d$ such that 
$\mathrm{lcd}_{\Bbbk\scrC^\diamond}(L^\diamond(\lambda))\leqslant 1$ or $\mathrm{rcd}_{\Bbbk\scrC^\diamond}(L^\diamond(\lambda))\leqslant 1$. 
Recall from \S \ref{sec:ringel} that $\scrC^\diamond\simeq\scrC^\blackdiamond\simeq\calO^{s,\kappa}_R\{d\}^\op$ and 
$L^\diamond(\lambda)$ corresponds to $L(\lambda^\star)^{s,\kappa}$. 
By Lemma \ref{lem:lcdrcd}, we have
$$\mathrm{rcd}_{\scrC^\diamond}(L^\diamond(\lambda))=
\mathrm{lcd}_{\calO^{s,\kappa}_k\{d\}}(L(\lambda^\star)^{s,\kappa})=\mathrm{rcd}_{\calO^{s,\kappa}_k\{d\}}(L(\lambda^\star)^{s,\kappa})=
\mathrm{lcd}_{\scrC^\diamond}(L^\diamond(\lambda)).$$
We have
$$F(T(\lambda)^{s^\star,\kappa}_R)\simeq
\IM^*\KZ_{R,d}^{s^\star}(\scrR^{-1}(P(\lambda)^{s,\kappa}_R))\simeq
\sigma^*(\KZ_{R,d}^s(P(\lambda)^{s,\kappa}_R)^\vee).$$

By Lemma \ref{lem:indsupp} and the Krull-Schmidt theorem, it is enough to prove that for any reflection hyperplane $H$ of $\Gamma_d$
and any $P\in\calO(W_H)^{\proj}_R$, we have
$$\sigma^*(\KZ_{R,d}^s(\OInd_{W_H}^{\Gamma_d}(P))^\vee)
\in F'({\scrC'}^\tilt).$$
Since $\sigma^*(\bullet^{\!\vee})$ commutes with induction functors and
fixes isomorphism classes of $R$-free $\bfH_{R,1}^s$-modules and
$\bfH_{R,2}^+$-modules, we deduce from (\ref{eq:commm}) that
$$\sigma^*(\KZ_{R,d}^s(\OInd_{W_H}^{\Gamma_d}(P))^\vee)\simeq
\KZ_{R,d}^s(\OInd_{W_H}^{\Gamma_d}(P)).$$

There are two possibilities for $H$:
\begin{itemize}
\item either $H$ is conjugate to $\ker(\gamma_i-1)$ for some $i\in[1,d]$. Then $W_H\simeq\Gamma$. 
We identify $\calO(W_H)_R\simeq\calO^{s,\kappa}_R\{1\}$ and $\OInd_{W_H}^{\Gamma_d}\simeq\OInd_{\Gamma_1}^{\Gamma_d}$. 
By Proposition \ref{prop:rankone}, for any projective $P\in\calO^{s,\kappa}_R\{1\}$, there exists $\bfT\in\bfA^{\nu,\kappa}_{R,\tau}\{1\}^\tilt$ such that 
$\KZ_{R,1}^s(P)\simeq\Psi^s_{R,1}(\bfT)$. By \eqref{eq:commm}, we have 
$\KZ_{R,d}^s\,\OInd_{W_H}^{\Gamma_d}\simeq\Ind_1^d\,\KZ_{R,1}^s$. 
Using Lemma \ref{lem:phiFcom}, this yields
$\KZ_{R,d}^s(\OInd_{W_H}^{\Gamma_d}(P))\simeq \Ind_1^d(\Psi_{R,1}^s(\bfT))\simeq\Psi^s_{R,d}(F^{d-1}(\bfT)).$
The module $F^{d-1}(\bfT)$ is tilting by Proposition \ref{prop:adjoints}(a), 
so $\KZ_{R,d}^s(\OInd_{W_H}^{\Gamma_d}(P))\in F'({\calC'}^\tilt)$;

\vspace{1mm}
\item or $H$ is conjugate to $\ker(s_{ij}^\gamma-1)$ for some $\gamma\in\Gamma$ and $i\neq j$. 
Then $W_H\simeq\frakS_2$ and $\calO(W_H)_R\simeq\calO^{\kappa}_R(\frakS_2)$. 
By \eqref{eq:commm}, we have 
$\KZ_{R,d}^s(\OInd_{W_H}^{\Gamma_d}(P))\simeq\Ind_{2,+}^{d,s}(\KZ(P))$.
By \eqref{eq:reprep} and Proposition \ref{prop:ranktwo} there exists $\bfT\in\bfA^{\nu,\kappa}_{R,\tau}\{2\}^\tilt$ such that $\Psi^s_{R,2}(\bfT)\simeq\Ind_{2,+}^{2,s}(\KZ(P))$.
Using Lemma \ref{lem:phiFcom}, this yields $$\Ind_{2,+}^{d,s}(\KZ(P))\simeq \Ind_{2}^d(\Psi^s_{R,2}(\bfT))\simeq \Psi^s_{R,d}(F^{d-2}(\bfT)).$$
Since $F^{d-2}(\bfT)$ is tilting, we have $\KZ_{R,d}^s(\OInd_{W_H}^{\Gamma_d}(P))\in F'({\scrC'}^\tilt)$.
\end{itemize}
We have checked that $(\scrC,F)$, $(\scrC',F')$ satisfy all the conditions in Proposition 
\ref{prop:key}, the theorem is proved.
\end{proof}

\vspace{2mm}

\begin{rk}
In \cite[(8.2)]{VV} the parameters of the CRDAHA are chosen in a different way.
More precisely, the symbol $h_p$ in \cite{VV} corresponds to
our parameter $h_p-h_{p-1}$. Further, the parameters
$(h,h_p)$ are specialized to $(-1/e,s_{p+1}/e-p/\ell)$ in \cite{VV} instead of $(1/e,s_{p+1}/e-p/\ell)$ as above.
\end{rk}

\vspace{3mm}

\subsubsection{Proof of the main theorem for irrational levels}
\label{sec:comparison1}

Let $\kappa\in\bbC\setminus\bbQ$.
We will prove the following result, which was conjectured in 
\cite[rem.~8.10$(b)$]{VV}, as a degenerate analogue of \cite[conj.~8.8]{VV}.
If $\nu$ is dominant, a proof was given in \cite[thm.~6.9.1]{GL}.

\vspace{2mm}

\begin{thm}\label{thm:comparisonI} Assume that $\kappa\in\bbC\setminus\bbQ$ and that $\nu_p\geqslant d$ for each $p$.
Then, we have an equivalence of highest weight categories 
$\Upsilon^{\nu,\kappa}_d: A^\nu\{d\}\iso\calO^{\nu^\star,\kappa}\{d\}$ 
such that $\Upsilon^{\nu,\kappa}_d(\Delta(\lambda))\simeq\Delta(\lambda^\star)^{\nu^\star,\kappa}$ and 
$\Phi_{d}^\nu\simeq\IM^*\,\KZ_{d}^{\,\nu^\star}\Upsilon^{\nu,\kappa}_d$.
\end{thm}

\vspace{.5mm}

Let $R$ be the completion at $(\kappa,0,\dots,0)$ of the ring of 
polynomials on $\bbC^{\ell+1}$.
It is a local deformation ring such that
$\kappa_R,\tau_{R,1},\dots,\tau_{R,\ell}$ are the standard coordinates.
The residue field is $\Bbbk=\bbC$ and 
we have $\kappa_\Bbbk=\kappa$, $\tau_{\Bbbk,p}=0$.
Further, for each $u$, $v$ and each $\frakp\in\frakP$, we have $\tau_{\Bbbk_\frakp,u}-\tau_{\Bbbk_\frakp,v}\not\in\bbZ^\times$.
We set $s_{R,p}=\nu_p+\t_{R,p}.$ 
Now, we consider the functor
$\Phi_{R,d}^s: A^\nu_{R,\tau}\{d\}\to H^s_{R,d}\mmod$ given in \S \ref{sec:catA}. 

\vspace{2mm}

\begin{lemma}\label{lem:6.333}
The functor $\Phi_{R,d}^s$ is a highest weight cover. It is fully faithful on 
$(A^{\nu}_{R,\tau}\{d\})^\Delta$ and $(A^{\nu}_{R,\tau}\{d\})^\nabla$.
\end{lemma}

\vspace{.5mm}

The proof is by reduction to codimension one, and is very similar to the proof of Theorem \ref{thm:isom}. We will be sketchy. 

We say that a prime ideal $\frakp\in\frakP$ is \emph{generic} if 
$\tau_{\Bbbk_\frakp,u}\neq\tau_{\Bbbk_\frakp,v}$ for each $u\neq v$,
and that it is
\emph{subgeneric} if there is a unique pair $u\neq v$ such that 
$\tau_{\Bbbk_\frakp,u}=\tau_{\Bbbk_\frakp,v}$.

\vspace{2mm}

\begin{claim}\label{claim:6.17}  For each $\frakp\in\frakP_1$ the following hold.

(a)  $\frakp$ is either generic or subgeneric,

(b) if $\frakp$ is generic, then there is an equivalence of highest weight categories
$\scrQ_{\Bbbk_\frakp}:\scrO^\nu_{\Bbbk_\frakp,\tau}\{a\}\to \scrO^+_{\Bbbk_\frakp}(\nu)\{a\},$

(c)  if  $\frakp$ is subgeneric with
$\tau_{\Bbbk_\frakp,u}=\tau_{\Bbbk_\frakp,v}$ and $u\neq v$, then
there is an equivalence of highest weight categories
$\scrQ_{\Bbbk_\frakp}:\scrO^\nu_{\Bbbk_\frakp,\tau}\{a\}\to 
\scrO^\nu_{\Bbbk_\frakp}(\nu,u,v)\{a\}$.
\end{claim}

\vspace{.5mm}

\begin{proof}
Part $(a)$ is easy.
Parts $(b)$, $(c)$ are proved as in Propositions \ref{prop:equivredI}, \ref{prop:reduction2}, using
\cite[thm.~11]{F2}. 
The details are left to the reader.
\end{proof}

\vspace{2mm}

\begin{proof}[Proof of Lemma \ref{lem:6.333}]
Now, the module $\scrQ_{\Bbbk_\frakp}(T_{\Bbbk_\frakp,d})$ 
can be identified explicitly, using the same argument as in the proof of
Lemmas \ref{lem:isom333}, \ref{lem:Q}.
Indeed, it is enough to check that
$\scrQ_{\Bbbk_\frakp}$ takes a parabolic Verma module to a parabolic Verma module with the same 
highest weight and that the induced linear map
$[\scrO^\nu_{\Bbbk_\frakp,\tau}]\to 
[\scrO^\nu_{\Bbbk_\frakp}(\nu,u,v)]$
commutes with the linear operators induced by the
categorification functors $e,f$.

Using the same argument as in the proof of Theorem \ref{thm:isom}, we only need to prove the lemma for 
$\Phi^s_{\Bbbk_\frakp,d}$ and $\frakp\in\frakP_1$. Hence, by Claim \ref{claim:6.17}, we are reduced to the case $\ell=1$ or $2$.
If $\ell=1$ everything is obvious, because the category $\scrO^+(\nu)$ is semi-simple.
If $\ell=2$, we may assume $\tau_{\Bbbk_\frakp}=0$, the result follows from Proposition \ref{prop:isomBK} and the last paragraph of the proof of Proposition \ref{prop:redI}.
\end{proof}

\vspace{3mm}

\begin{proof}[Proof of Theorem \ref{thm:comparisonI}]
Consider the highest weight cover $\IM^*\,\KZ_{R,d}^{s^\star}:\calO^{s^\star,\kappa}_R\{d\}\to\bfH^{s}_{R,d}\mmod$. 
Since the $R$-algebras $\bfH^{s}_{R,d}$ and $H^{s}_{R,d}$ 
are isomorphic by Proposition \ref{prop:isomHecke},
we can regard $\IM^*\,\KZ_{R,d}^{s^\star}$ and $\Phi_{R,d}^s$ as highest weight covers
of the category $\bfH^{s}_{R,d}\mmod.$
We claim that they satisfy the conditions in Proposition \ref{prop:key}, so the theorem follows. Let us check the conditions.

First, $\bfH^{s}_{R,d}$ is Frobenius, and $\bfH^{s}_{R,d}$ is semi-simple because \eqref{(A)} holds obviously in $K$.
The compatibilities of orders is again given by Lemma \ref{lem:order3}.

Since $\IM^*$ is an equivalence, $\IM^*\,\KZ_{R,d}^{s^\star}$ is fully faithful on $\Delta$- and $\nabla$-filtered objects
by Proposition \ref{prop:ffKZ} and \eqref{eq:ringelcom}.
The corresponding property for $\Phi^s_{R,d}$ follows from Lemma \ref{lem:6.333}.

It remains to check that $\IM^*\,\KZ_{R,d}^{s^\star}(T(\lambda)^{s^\star,\kappa}_R)\in \Phi_{R,d}^s((A^\nu_{R,\tau}\{d\})^\tilt)$ 
for all $\lambda\in\scrP^\ell_d$ such that 
$\mathrm{lcd}_{\Bbbk\calO^{s^\star,\kappa}_R\{d\}^\diamond}(L^\diamond(\lambda))\leqslant 1$ or $\mathrm{rcd}_{\Bbbk\calO^{s^\star,\kappa}_R\{d\}^\diamond}(L^\diamond(\lambda))\leqslant 1$. 
The proof is the same as in Theorem \ref{thm:main3}. Details are left to the reader.
\end{proof}

\vspace{3mm}

\vspace{3mm}

\section{Consequences of the main theorem}

\subsection{Reminder on the Fock space}\label{ss:Fock}
Let $R$, $q_R$,  $\scrI=\scrI(q)$ and
$Q_{R,1},Q_{R,2},\dots,Q_{R,\ell}$ be as in \S \ref{sec:quivers}.
Consider the dominant weight in $P=P_\scrI$ given by
$\Lambda^Q=\sum_{p=1}^\ell\Lambda_{Q_{p}}.$
Note that
$\Lambda^Q=\sum_{p\in \Omega}\Lambda_p,$
with
$\Lambda_{p}=\sum_{u\,;\,Q_{u}\equiv Q_{p}}\Lambda_{Q_{u}}.$
Let $s=(s_1,\dots,s_\ell)$ be as in \S \ref{sec:combinatorics}. 
Then, we may write $\Lambda^s=\Lambda^Q$.

The \emph{Fock space of multi-charge} $s$ is the vector space
$\bfF(\Lambda^s)=\bigoplus_{\lam\in\scrP^\ell}\bbC\,|\lam,s\rangle.$
We will abbreviate $\Lambda=\Lambda^s$.
We will call $\{|\lambda,s\rangle\,;\,\lambda\in\scrP^\ell\}$ the \emph{standard monomial basis} of 
$\bfF(\Lambda)$.

There is an integrable representation of $\fraksl_{\!\scrI}$ on $\bfF(\Lambda)$ given by
\begin{equation}\label{eq:EFaction}
F_i(|\lam,s\rangle)=\sum_{q\text{-}\!\res^s(\mu-\lam)=i}|\mu,s\rangle,\qquad
E_i(|\lam,s\rangle)=\sum_{q\text{-}\!\res^s(\lam-\mu)=i}|\mu,s\rangle.
\end{equation}
Let $n_i(\lam)$ be the
number of boxes of residue $i$ in $\lam$.
To avoid any confusion we may write $n^s_i(\lam)=n^Q_i(\lam)=n_i(\lam)$.
Each basis vector $|\lam,s\rangle$ is a weight vector of weight
$\bfwt(|\lambda,s\rangle)=\Lambda-\sum_{i\in\scrI}n_i(\lambda)\,\alpha_i.$

The $\Lambda$-weight space of $\bfF(\Lambda)$ has
dimension one and is spanned by the element $|\emptyset,s\rangle$.
The $\fraksl_{\!\scrI}$-submodule $\bfL(\Lambda)\subset\bfF(\Lambda)$ generated by
$|\emptyset,s\rangle$ is the simple module of highest weight $\Lambda$.
It decomposes as the tensor product 
$\bfL(\Lambda)=\bigotimes_{p\in \Omega}\bfL(\Lambda_{p})$,
where $\bfL(\Lam_{p})$ is the simple $\fraksl_{\scrI_p}$-module 
of highest weight $\Lambda_{p}$. 

\vspace{2mm}

\begin{rk}
Assume that the quiver $\scrI(q)$ is the disjoint union of $\ell$ components 
of type $A_\infty$. Then, we have $\bfF(\Lambda)=\bfL(\Lambda)=\bigotimes_{p=1}^\ell\bfL(\Lambda_{p})$.
\end{rk}

\vspace{.5mm}

\begin{rk}\label{rk:weights}
The weight $\wt(\lambda)$ associated with the element $\lambda\in P^\nu+\tau$
should not be confused with the weight $\bfwt(|\lambda,s\rangle)$ above, which is associated with the
$\ell$-partition $\lambda\in\scrP^\ell$. 
The former has the level 0 while the latter has the level $\ell$.
We have
$\wt(\varpi(\lambda))=\bfwt(|\lambda,s\rangle)-\sum_{p=1}^\ell\Lambda_{\tau_p}$
mod $\bbZ\,\delta.$

Indeed, the equation above holds for 
$\lambda=\emptyset$. Thus, it is proved by induction
using the following equivalences for $\lambda,\mu\in\scrP^\nu$, see \S \ref{ss:Fock},
$$\aligned
\varpi(\lambda)\overset{i}\to\varpi(\mu)
&\Leftrightarrow  q_K\text{-}\!\res^s(\mu-\lambda)=q_K^i,\\
&\Rightarrow
\wt(\varpi(\lambda))-\wt(\varpi(\mu))=
\bfwt(|\lambda,s\rangle)-\bfwt(|\mu,s\rangle)=\alpha_i.
\endaligned
$$
\end{rk}

\vspace{3mm}

\subsection{Rouquier's conjecture}
Let $K=\bbC$.
Fix integers $e,\ell\geqslant 1$ and fix $s=(s_1,\dots,s_\ell)\in\bbZ^\ell$.
Set $\Lambda=\Lambda^s$.
Set $I=\bbZ$ and $\scrI=I/e\bbZ$.
So, we have $\fraksl_\scrI=\widehat\fraksl_e$ and
the Fock space $\bfF(\Lambda)$ is an integrable $\widehat\fraksl_e$-module.
Consider the \emph{Uglov's canonical}
bases $\{\calG^\pm(\lambda,s)\,;\,\lambda\in\scrP^\ell\}$ 
of $\bfF(\Lambda)$ introduced in \cite[sec.~4.4]{U}.

Set $\calO^{s^\star,-e}=\bigoplus_{d\in\bbN}\calO^{s^\star,-e}\{d\}$.
We identify the complexified Grothendieck group
$[\calO^{s^\star,-e}]$ with $\bfF(\Lambda)$ via the linear map
$\theta:[\calO^{s^\star,-e}]\iso \bfF(\Lambda)$ such that
$[\Delta(\lambda^\star)^{s^\star,-e}]\mapsto|\lambda,s\rangle.$

Since the category $\calO^{s^\star,-e}$ is preserved under the substitution
$s\mapsto (1+s_1,1+s_2,\dots, 1+s_\ell)$
we may assume that $s_p=\nu_p\geqslant d$ for each $p$. 
Set $\bfA^{\nu,-e}=\bigoplus_{d\in\bbN}\bfA^{\nu,-e}\{d\}$

\smallskip
The following result has been conjectured by Rouquier \cite[sec.~6.5]{R1}. 

\vspace{2mm}

\begin{thm}\label{thm:rouquier}
We have $\theta([T(\lambda^\star)^{{s^\star,-e}}])=\calG^+(\lambda,s)$ and $\theta([L(\lambda^\star)^{{s^\star,-e}}])=\calG^-(\lambda,s)$.
\end{thm}

\vspace{.5mm}

\begin{proof}
Let $c^\pm_{\lambda,\mu}(s)\in\bbZ$ be such that
$\calG^\pm(\lambda,s)=\sum_{\mu} c^\pm_{\lambda,\mu}(s)\,|\mu,s\rangle.$

Let $\bfF(\Lambda)\{d\}\subset\bfF(\Lambda)$ be the subspace spanned by the set $\{|\lambda,s\rangle\,;\,\lambda\in\scrP^\ell_d\}$.
Assume that $\nu_p\geqslant d$ for each $p$.
We identify the complexified Grothendieck group
$[\bfA^{\nu,-e}\{d\}]$ with $\bfF(\Lambda)\{d\}$ via the linear map such that
$[\pmb\Delta(\lambda)]\mapsto|\lambda,s\rangle$.

Let $\bfL(\lambda)$ be the top of $\pmb\Delta(\lambda)$
in $\bfA^{\nu,-e}\{d\}$.
By \cite[prop.~8.2]{VV}, we have
$[\bfL(\lambda)]=\sum_\mu c^-_{\lambda,\mu}(s)\,[\pmb\Delta(\mu)]$
in $[\bfA^{\nu,-e}\{d\}]$. Therefore, the isomorphism
$[\bfA^{\nu,-e}\{d\}]\iso\bfF(\Lambda)\{d\}$
maps $[\bfL(\lambda)]$ to $\calG^-(\lambda,s)$. 

Since the 
equivalence of categories $\bfA^{\nu,-e}\{d\}\iso\calO^{s^\star,-e}\{d\}$ in Theorem \ref{thm:main3}
maps $\bfL(\lambda)$ to $L(\lambda^\star)^{s^\star,-e}$ and since 
the isomorphism $[\calO^{s^\star,-e}\{d\}]\iso\bfF(\Lambda)\{d\}$ is the composition of the map
$[\calO^{s^\star,-e}\{d\}]\iso[\bfA^{\nu,-e}\{d\}]$ induced by the inverse of the equivalence with 
the isomorphism $[\bfA^{\nu,-e}\{d\}]\iso\bfF(\Lambda)\{d\}$ above, we 
deduce that the map $[\calO^{s^\star,-e}\{d\}]\iso\bfF(\Lambda)\{d\}$
takes $[L(\lambda^\star)^{s^\star,-e}]$ to $\calG^-(\lambda,s)$.

Next, let $\bfP(\lambda)$ be the projective cover of $\pmb\Delta(\lambda)$
in $\bfA^{\nu,-e}\{d\}$. By the Brauer reciprocity we have 
$(\bfP(\lambda):\pmb\Delta(\mu))=[\pmb\Delta(\mu):\bfL(\lambda)]$. Therefore we have 
$[\bfP(\lambda)]=\sum_{\mu}d^-_{\lambda,\mu}(s)[\pmb\Delta(\mu)]$
in $[\bfA^{\nu,-e}\{d\}],$ where the matrix $\big(d^-_{\lambda,\mu}(s)\big)$ 
is the transpose of the inverse matrix of 
$\big(c^-_{\lambda,\mu}(s)\big)$. 

By \cite[thm.~5.15]{U}, we have
$d^-_{\lambda,\mu}(s)=c^+_{\lambda^\star,\mu^\star}(s^\star)$. 
Using the equivalence of categories $\bfA^{\nu,-e}\{d\}\iso\calO^{s^\star,-e}\{d\},$ we get
$[P(\lambda^\star)^{s^\star,-e}]=
\sum_{\mu}c^+_{\lambda^\star,\mu^\star}(s^\star)[\Delta(\mu^\star)^{s^\star,-e}]$
in $[\calO^{s^\star,-e}\{d\}]$. By removing $\star$ everywhere, we get
the following equality in $[\calO^{s,-e}\{d\}]$
\begin{equation}\label{PP}
[P(\lambda)^{s,-e}]=\sum_{\mu}c^+_{\lambda,\mu}(s)[\Delta(\mu)^{s,-e}].
\end{equation}

Next, by \S \ref{sec:ringel} we have the equivalence $\scrR: \calO^{s^\star,-e,\Delta}\{d\}\iso\,
\calO^{s,-e,\Delta}\{d\}^\op$ such that $\Delta(\lambda^\star)^{s^\star,-e}\mapsto\Delta(\lambda^{s,-e})$
and $T(\lambda^\star)^{s^\star,-e}\mapsto P(\lambda)^{s,-e}$. 
The inverse of $\scrR$ yields an isomorphism of 
Grothendieck groups
$[\calO^{s,-e}\{d\}]\iso[\calO^{s^\star,-e}\{d\}]$ such that
$[\Delta(\lambda)^{s,-e}]\mapsto [\Delta(\lambda^\star)^{s^\star,-e}]$
and
$[P(\lambda)^{s,-e}]\mapsto [T(\lambda^\star)^{s^\star,-e}].$
The image of the equality \eqref{PP} under this isomorphism gives the identity
$[T(\lambda^\star)^{s^\star,-e}]=\sum_{\mu}c^+_{\lambda,\mu}(s)[\Delta(\mu^\star)^{s^\star,-e}]$
in $[\calO^{s^\star,-e}\{d\}]$.
We deduce that the isomorphism
$[\calO^{s^\star,-e}\{d\}]\iso\bfF(\Lambda)\{d\}$ maps the element
$[T(\lambda^\star)^{s^\star,-e}]$ to 
$\sum_{\mu}c^+_{\lambda,\mu}(s)|\mu,s\rangle=\calG^+(\lambda,s)$. We are done.
\end{proof}

\vspace{3mm}

\subsection{The category $\calO$ of CRDAHA's is Koszul}
Recall that $\scrI\simeq[0,e)$ and that $\Lambda_i,$ $\alpha_i$ are the fundamental weights and the simple roots of $\widehat{\fraks\frakl}_e$.
For $t=(t_1,\dots,t_e)\in\bbZ^e$
let $\calO^{s}_t\subset\calO^{s,-e}$ be the Serre subcategory generated by the modules
$\Delta(\lambda)^{s,-e}$ such that the following condition holds
\begin{equation}\label{7.1}
\Lambda^s-\sum_{i=1}^{e-1}(n^s_i(\lambda)-n^s_0(\lambda))\,\alpha_i
=\sum_{i=1}^{e-1}(t_{i}-t_{i+1})\Lambda_i+(\ell+t_e-t_1)\Lambda_0.
\end{equation}

Set $|s|=s_1+\cdots+s_\ell$ and $|t|=t_1+\cdots+t_e$. From \eqref{7.1} we get that
$|t|=|s|$ modulo $\bbZ\,e$. Hence, up to translating the $t_i$'s 
simultaneously by the same integer,
we may assume that $t\in\bbZ^e(|s|)$. 
Note that the left hand side of \eqref{7.1} is equal 
to  $\bfwt(|\lambda,s\rangle)$ modulo $\bbZ\,\delta$.

Since the category $\calO^s_t$ is preserved under the substitutions
$s\mapsto (1+s_1,1+s_2,\dots, 1+s_\ell)$ and $t\mapsto (1+t_1,1+t_2,\dots, 1+t_\ell),$ we may assume that 
$s=\nu^\star$, $t=\mu^\star$ for some
compositions $\nu\in\scrC^\ell_d$, $\mu\in\scrC^e_d$ such that $\mu_i,\nu_p\geqslant d$ for each $i,p$. 

The following result has been conjectured by
Chuang and Miyachi \cite[conj.~6]{CM}.

\vspace{2mm}

\begin{thm}\label{thm:chuang-miyachi}
The category  $\calO^s_t$ is (standard) Koszul and its Koszul dual coincides with the Ringel dual
of $\calO^{t^\star}_{s^\star}$.
\end{thm}

\vspace{.5mm}

\begin{proof}
Theorem \ref{thm:main3} yields an equivalence $\bfA^{\nu,-e}\simeq\calO^{s,-e}$. 
Let $\bfA^{\nu}_\mu\subset\bfA^{\nu,-e}$ be the image of $\calO^{s}_t$ by this equivalence. 
By \cite[thm.~B.4]{SVV}, the highest weight category $\bfA^{\nu,-e}$
is (standard) Koszul, and the Koszul dual of $\bfA^{\nu}_\mu$ is equivalent to the Ringel dual of $\bfA^\mu_\nu$.
The theorem follows.
\end{proof}

\vspace{3mm}

\subsection{Categorical actions on $\bfA$}\label{sec:7.4}
Recall that $\scrI=\bbZ/e\bbZ.$
By \cite[thm.~5.1]{S}, there is an $\fraksl_{\!\scrI}$-categorical action $(E,F,X,T)$ on $\calO^{s,-e}$ with 
$E=\bigoplus_{d\in\bbN}\ORes_d^{d+1}$, $F=\bigoplus_{d\in\bbN}\OInd_d^{d+1}$, and
such that the functor $\KZ^s=\bigoplus_{d\in\bbN}\KZ_{d}^s$ is a morphism of
$\fraksl_{\!\scrI}$-categorifications $\calO^{s,-e}\to\scrL(\Lambda^s)_\scrI.$
In this section we construct a similar $\fraksl_{\!\scrI}$-categorification for the category $\bfA$.

\smallskip

Let $R=\bbC$, $\tau=0$, and let $\nu\in\scrC^\ell_{N,+}$. Assume that $d\leqs\nu_p$ for all $p$. Recall the tuple $(E,F,X,T)$ on $\bfA^{\nu,-e}$ from Section \ref{ss:Finduction}. Let
$\Upsilon_d=\Upsilon^{\nu,-e}_d:\bfA^{\nu,-e}\{d\}\iso\calO^{\nu^\star,-e}\{d\}$ 
be the equivalence in Theorem \ref{thm:main3}.
We have the following.

\vspace{2mm}

\begin{lemma}\label{lem:cor}
Assume that $d+1\leqs\nu_p$ for all $p$.
Then, the functors $F:\bfA^{\nu,-e}\{d\}\to\bfA^{\nu,-e}\{d+1\}$ and $E:\bfA^{\nu,-e}\{d+1\}\to\bfA^{\nu,-e}\{d\}$ are biadjoint.
Further, there are isomorphisms of functors 
$\OInd_{d}^{d+1}\,\Upsilon_d\simeq \Upsilon_{d+1}\, F$ and $\ORes_{d}^{d+1}\,\Upsilon_{d+1}\simeq \Upsilon_d\, E,$
which intertwines $X\Upsilon_d$ with $\Upsilon_{d+1}\IM(X)$, and $T\Upsilon_d$ with $\Upsilon_{d+2}\IM(T)$.
\end{lemma}

\vspace{.1mm}

\begin{proof}
We abbreviate $\KZ_d=\KZ^{\nu}_d$, $\KZ_d^\star=\KZ^{\nu^\star}_d$, $\Psi_d=\Psi^{\nu}_d$, $\bfA=\bfA^{\nu,-e}$, $\bfO=\bfO^{\nu,-e}$  
and $\calO=\calO^{\nu^\star,-e}$.
By Theorem \ref{thm:main3}  we have
$\Psi_d\simeq\IM^*\KZ^\star_d\,\Upsilon_d$ on $\bfA\{d\}$.

Recall from Proposition \ref{prop:adjoints} that $e$, $f$ are biadjoint functors on $\bfO$. 
Let $F'=\Upsilon_{d+1}\, F\,\Upsilon_d^{-1}:\calO\{d\}\to\calO\{d+1\}$.
We claim that there is an isomorphism of functors $F'\simeq\OInd_d^{d+1}$. Let us prove it.

Since 
$\Psi_{d+1}\, F\simeq\Ind_d^{d+1}\, \Psi_d$ by Lemma \ref{lem:phiFcom}, 
and since $\IM^*$ commutes with the induction functor,
we have $\KZ_{d+1}F'\simeq\Ind_d^{d+1}\KZ_d$.
By \eqref{eq:commm}, we also have
$\KZ_{d+1} \OInd_d^{d+1}\simeq\Ind_d^{d+1}\, \KZ_d$. 
Hence, we get an isomorphism of functors
$$\theta:\KZ_{d+1} F'\iso\,\KZ_{d+1}\OInd_d^{d+1}.$$

The functor $\OInd_d^{d+1}$ maps projectives to projectives.
Let $G_d$ be the right adjoint to $\KZ_d$. Since $\KZ_d$ is a highest weight cover,
the unit $\eta:1\to G_d\,\KZ_d$ is invertible on projective modules.
Hence, the isomorphism $\theta$ yields an isomorphism of functors on projective modules
$G_{d+1}\,\KZ_{d+1}F'
\simeq G_{d+1}\,\KZ_{d+1}\OInd_d^{d+1}\simeq \OInd_d^{d+1}$.
Composing it with $\eta$, we get a morphism
$\theta':F'
\to \OInd_d^{d+1}$ on the projectives, such that $\KZ_{d+1}\theta'=\theta$.
Since $\KZ_{d+1}$ is $(-1)$-faithful, it follows from Lemma  \ref{lem:B2} and Remark \ref{rem:2.9} that $\theta'$ is injective, hence invertible because both terms coincide in
the Grothendieck group by Lemma \ref{lem:pasbete} and \cite[prop.~4.4(3)]{S}.
Thus, $\theta'$ is an isomorphism on the projective modules. 

Now, since $\Upsilon_{d+1}$, $\Upsilon_{d}$ are equivalences, both $F'$ and 
$\OInd_d^{d+1}$ are exact on $\calO\{d\}$.
Thus, $\theta'$ extends to an isomorphism of functors
$\theta':F'\iso\, \OInd_d^{d+1}$ on $\calO\{d\}$ such that $\KZ_{d+1}\theta'=\theta$
by \cite[lem.~1.2]{S}. 
The claim is proved.

Let $E':\bfA\{d+1\}\to\bfA\{d\}$ be the right adjoint of $F$. 
The uniqueness of right adjoints implies that $\Upsilon_d\, E'\, \Upsilon_{d+1}^{-1}\simeq\ORes^{d+1}_d$.
Now, since $\ORes_d^{d+1}$ is also left adjoint to $\OInd_d^{d+1}$ by \cite[prop.~2.9]{S} and since
$\Upsilon_d$ is an equivalence, we deduce that 
$E'$ is left adjoint to $F$, hence $E\simeq E'$ on 
$\bfA\{d+1\}$. 


Now, let $X_\bfH\in\End(\Ind_d^{d+1})$ and $T_\bfH\in\End(\Ind_d^{d+2})$ be as in Example \ref{ex:3.6}.
The isomorphism 
$\Ind_d^{d+1}\KZ_{d}\simeq\KZ_{d+1}\OInd_d^{d+1}$ in \eqref{eq:commm}
intertwines $X_\bfH^{-1}\KZ_{d}$ with $\KZ_{d+1} X^{-1}$.
The isomorphism 
$\Ind_d^{d+1}\Psi_{d}\simeq\Psi_{d+1} F$
in Lemma \ref{lem:phiFcom}
intertwines
$X_\bfH\Psi_d$ with $\Psi_{d+1}X$ by Remark \ref{rk:commXXX}.
Hence, $\theta$ intertwines  
$\KZ_{d+1}\Upsilon_{d+1}X\Upsilon_{d}^{-1}$ with $\KZ_{d+1} X^{-1}.$
We deduce that $\theta'$ intertwines $\Upsilon_{d+1}X\Upsilon_{d}^{-1}$ with $X^{-1}.$
The proof for $T$  is similar.
The lemma is proved.
\end{proof}

\medskip

For each $a\in\bbN,$ set $\nu+a=(\nu_1+a,\nu_2+a,...,\nu_\ell+a)$.

\vspace{2mm}

\begin{lemma}\label{lem:limit}
For any $d\in\bbN$ and any $a\leqslant a'\in\bbN$ such that $d\leqslant\nu_p+a$ for all $p$, there is an equivalence of highest weight categories 
$\Sigma=\Sigma^{a,a'}:\bfA^{\nu+a,-e}\{d\}\iso\bfA^{\nu+a',-e}\{d\}$ 
which maps $\pmb\Delta(\lambda)$ to $\pmb\Delta(\lambda)$, 
intertwines $(E,F,X,T)$ on both sides and such that $\Psi^{\nu+a}_d\simeq \Psi^{\nu+a'}_d\Sigma^{a,a'}$.
\end{lemma}

\vspace{.5mm}

\begin{proof}
The CRDAHA's associated with $(\nu+a)^\star$ and $(\nu+a')^\star$ are the same. 
Hence, we have $\calO^{(\nu+a)^\star,-e}\{d\}=\calO^{(\nu+a')^\star,-e}\{d\}$. 
We define $\Sigma^{a,a'}=(\Upsilon^{\nu+a',-e}_d)^{-1}\,\Upsilon^{\nu+a,-e}_d$. 
By Lemma \ref{lem:cor}, the functor $\Sigma$ intertwines $(E,F)$ on both sides.
\end{proof}

\vspace{2mm}

For each $d$, we define the category $\widetilde\bfA^{\nu,-e}\{d\}$ as the limit of the inductive system of categories
$(\bfA^{\nu+a,-e}\{d\},\Sigma^{a,a'})_{a,a'\in\bbN}$. 
We have an equivalence of highest weight categories 
$\tilde{\Upsilon}^{\nu,-e}_d:\widetilde\bfA^{\nu,-e}\{d\}\iso\calO^{\nu^\star,-e}\{d\}$
and a  highest weight cover $\tilde{\Psi}^{\nu}_d:\widetilde\bfA^{\nu,-e}\{d\}\to \bfH^{\nu}_d\mmod$. 
In particular, the blocks of $\widetilde\bfA^{\nu,-e}\{d\}$ are in bijection with the blocks of $\bfH^{\nu}_d\mmod$ via $\tilde{\Psi}^{\nu}_d$. 
For $\mu=\Lambda^\nu-\alpha$, let $\widetilde\bfA^{\nu,-e}_\mu$ be the block corresponding to $\bfH^\nu_{\alpha}\mmod$.

Now, let $\widetilde\bfA^{\nu,-e}=\bigoplus_{d\in \bbN}\widetilde\bfA^{\nu,-e}\{d\}$.  
The category $\widetilde\bfA^{\nu,-e}$ carries a pre-categorical action $(E,F,X,T)$ given by Lemma \ref{lem:limit}. 
The following is now obvious.

\vspace{2mm}

\begin{prop} The tuple $(E,F,X,T)$ and the decomposition $\widetilde\bfA^{\nu,-e}=\bigoplus_{\mu\in X_{\!\scrI}}\widetilde\bfA^{\nu,-e}_{\mu}$
define an $\fraksl_{\!\scrI}$-categorical action on $\tilde\bfA^{\nu,-e}$.
\end{prop}

\vspace{.5mm}

\begin{proof}
We have $E=\bigoplus_{i\in \scrI}E_i$ and $F=\bigoplus_{i\in \scrI}F_i$, where $E_i$, $F_i$ are defined as in \S \ref{sec:categorification}.
By Theorem \ref{thm:main3}, the equivalence 
$\tilde{\Upsilon}^{\nu,-e}=\bigoplus_{d\in\bbN}\tilde{\Upsilon}^{\nu,-e}_d:\tilde\bfA^{\nu,-e}\iso\calO^{\nu^\star,-e}$
yields a linear isomorphism $[\tilde\bfA^{\nu,-e}]\iso[\calO^{\nu^\star,-e}]$ which 
maps $[\pmb\Delta(\lambda)]$ to $[\Delta(\lambda^\star)^{\nu^\star,-e}]$. Hence by Lemma \ref{lem:pasbete} and \cite[prop.~4.4]{S}, it
intertwines the operators $E_i,F_i$ on the left hand side 
with the operators the operators $E_{-i},F_{-i}$ on the right hand side.
Thus, the operators $E_{i}$, $F_{i}$ with 
$i\in\scrI$ yield a representation of $\fraksl_{\!\scrI}$ on 
$[\tilde\bfA^{\nu,-e}]$.
\end{proof}

\vspace{3mm}

\subsection{The category $\bfA$ and the cyclotomic q-Schur algebra}\label{sec:7.5}

Let $R=K$ be a field. The following propositions generalize some of the results in \cite{BK1} and follow from the arguments used in this paper.

\vspace{2mm}

\begin{prop}\label{prop:BKschur}
Let $\nu_p\geqslant d$ and $\tau_{K,u}-\tau_{K,v}\notin\bbN^\times$ for all $p$  and all $u<v$.
Set $s=\nu+\tau$.
Assume that $\nu$ is either dominant or anti-dominant.
Then, there is an equivalence of highest weight $K$-categories
$\scrG^s_{K,d}:A^\nu_{K,\tau}\{d\}\iso S^{s^\star}_{K,d}\mmod$ which intertwines the functors
$\Phi_{K,d}^s:A^\nu_{K,\tau}\{d\}\to H^s_{K,d}\mmod$ and
$\IM^*\,\Xi_{K,d}^{s^\star}:S^{s^\star}_{K,d}\mmod\to H^s_{K,d}\mmod$.
We have $\scrG^{s}_{K,d}(\Delta(\lambda)_{K,\tau})\simeq W(\lambda^\star)^{s^\star}_K$ for all $\lambda$'s.
\qed
\end{prop}


\vspace{2mm}

\begin{prop}
Let $\nu_p\geqslant d$ for all $p$.
Assume that $\nu$ is either dominant or anti-dominant. Set $s=\nu$.
Then, there is an equivalence of highest weight $K$-categories
$\scrG^s_{K,d}:\bfA^{\nu,\kappa}_{K}\{d\}\iso \bfS^{s^\star}_{K,d}\mmod$ which intertwines the functors
$\Phi_{K,d}^s:\bfA^{\nu,\kappa}_{K}\{d\}\to \bfH^s_{K,d}\mmod$ and
$\IM^*\,\Xi_{K,d}^{s^\star}:\bfS^{s^\star}_{K,d}\mmod\to \bfH^s_{K,d}\mmod$.
We have $\scrG^{s}_{K,d}(\pmb\Delta(\lambda)_{K})\simeq W(\lambda^\star)^{s^\star,q}_K$ for all $\lambda$'s.
\qed
\end{prop}

\vspace{3mm}

\section{The Kazhdan-Lusztig category}\label{sec:KL}
\label{sec:8}

Fix integers $\ell,N\geqslant 1$ and fix a composition $\nu\in\scrC^\ell_{N,+}$.
Let $R$ be a deformation ring. We may abbreviate $\kappa=\kappa_R$.

\vspace{3mm}

\subsection{Coinvariants}
\label{sec:coinvariants}
Fix a finite totally ordered set $A$.
Set $R^{A}=\bigoplus_{a\in A}R((t_a))$, where $t_a$ is a formal variable. 
Let $\bfg^{A}_{R}$ be the central extension of
$\frakg\otimes R^{A}$ by $R$ associated with the cocycle
$(\xi\otimes f,\zeta\otimes g)\mapsto \langle\xi:\zeta\rangle\sum_{a\in
A}\Res_{t_a=0}(gdf).$

Write $\bfone$ for the canonical central element
of $\bfg^{A}_{R}$, and let
$U(\bfg^{A}_R)\to\bfg^{A}_{R,\kappa}$
be the quotient of the enveloping algebra (over $R$)
by the two-sided ideal generated by $\bfone-c$.
By the symbol $\bigotimes_{R,a}$ 
we'll mean the (ordered) tensor product of 
$R$-modules with respect to the ordering of $A$.
Given $M_a\in\scrS_{R,\kappa}$ for each $a\in A$, the Lie algebra
$\bfg^{A}_{R,\kappa}$ acts naturally on $\bigotimes_{R,a}M_a$.

Let $C$ be a connected projective curve isomorphic to $\bbP^1$.
By a \emph{chart} on $C$ centered at $x$
we mean an automorphism $\gamma$ of $\bbP^1$ such that $\gamma(x)=0$.
We will say that $\gamma=\{\gamma_a\,;\,a\in A\}$ is an \emph{admissible system of charts} if the conditions
$(a)$, $(b)$ in \cite[sec.~13.1]{KL} hold.
Let $\eta_a=\gamma_a^{-1}$ denote the automorphism which is the inverse of $\gamma_a$ and
let $x_a$ be the center of $\gamma_a$. 
The $x_a$'s are distinct points of $C$. 
We write 
$C_\gamma=C\setminus\{x_a; a\in A\},$
$D_R=D_{R,\gamma}=R[C_\gamma]$ and
$\Gamma_{R}=\Gamma_{R,\gamma}=\frakg\otimes D_R.$

For $f\in D_R$ let ${}^a\!f\in R((t_a))$ be the power series expansion at 0 of the rational
function $f\circ\eta_a$ on $\bbP^1$.
Taking $f$ to the $A$-tuple ${}^A\!f=({}^a\!f)$ gives a $R$-algebra homomorphism 
$D_R\to R^{A}$ and a $R$-Lie algebra homomorphism 
$\Gamma_{R}\to\bfg^{A}_{R}$
by the residue theorem.

We can now define the sets of \emph{coinvariants}.

\vspace{2mm}

\begin{df} Given $N_a\in U(\frakg_R)\mmod$
and  $M_a\in\scrS_{R,\kappa}$ for $a\in A=[1,n]$, we set
$\langle N_1,\dots,N_n\rangle_R= H_0(\frakg_R,\Otimes_{R,a}N_a)$
and
$\llangle M_1,\dots,M_n\rrangle_R=
H_0(\Gamma_{R},\Otimes_{R,a}M_a).$
\end{df}

\vspace{2mm}

By \cite[sec.~13.3]{KL} the $R$-module $\llangle M_1,\dots,M_n\rrangle_R$
does not depend on the choice of the admissible system of charts, up to a canonical isomorphism.
Further, it only depends on the cyclic ordering of $A$, so that we have a canonical isomorphism
$\llangle M_1,\dots,M_n\rrangle_R=\llangle M_2,\dots,M_n,M_1\rrangle_R.$

\vspace{2mm}

\begin{lemma}
\label{lem:TP0}
(a) Taking coinvariants is a right exact functor. It commutes with base change, i.e., for any morphism of deformation rings
$R\to S$ the obvious map $\Otimes_{R,a}M_a\to \Otimes_{S,a}S M_a$
induces a $S$-module isomorphism 
$S\llangle M_1,\dots,M_n\rrangle_R=\llangle SM_1,\dots,SM_n\rrangle_S$.

(b)
Assume that for each $a$ there exists integers $d_a\geqslant 1$ such that $M_a(d_a)$ generates $M_a$ as a 
$\bfg_{R,\kappa}$-module. 
Then the inclusion
$\Otimes_{R,a}M_a(d_a)\to \Otimes_{R,a}M_a$
induces a surjective $R$-module homomorphism
$\langle M_{1}(d_1),\dots, M_n(d_n)\rangle_R\to\llangle M_{1},\dots,M_n\rrangle_R.$
In particular, given modules $M_a$ as above with $a\in A=[-m,n]$ such that
$M_0\in\bfO_R$ and $M_a\in\bfO^{+,\kappa}_R$ for $a\neq 0$, the $R$-module 
$\llangle M_{-m},\dots, M_n\rrangle_R$ is finitely generated.

(c) 
Assume that $M_a=\Indc_R(N_a)$ is a generalized Weyl module for each $a\in A$. Then the inclusion
$\Otimes_{R,a}N_a\to \Otimes_{R,a}M_a$
induces an isomorphism of $R$-modules
$\langle N_{1},\dots, N_n\rangle_R\to\llangle M_{1},\dots,M_n\rrangle_R.$

(d) If  $M_{1}=\bfM(0)_{R,+}$ the canonical inclusion
$\Otimes_{R,a\neq 1}M_a\to \Otimes_{R,a}M_a$
induces  an $R$-module isomorphism 
$\llangle M_2,\dots, M_n\rrangle_R\to\llangle M_1,M_{2},\dots, M_n\rrangle_R.$
\end{lemma}

\vspace{.5mm}

\begin{proof}
Part $(a)$ is obvious. See, e.g.,  \cite[sec.~9.13]{KL}.
For part $(b)$ note that the $\frakg_R$-action on $M_a$ preserves the $R$-submodule
$M_a(d_a)$ for each $a\in A$. Then, the first claim follows from \cite[prop.~9.12]{KL}. 
The proof in loc.~cit.~is done under the hypothesis that $\kappa\in\bbC$.
It extends easily to the case of any $\kappa=\kappa_R\in R$.
The second claim follows from Lemma \ref{lem:O0}$(b)$, the first claim, part $(a)$ and from the fact that 
$\scrO_R$ is Hom finite (over $R$)
and that the tensor product maps $\scrO_R\times\scrO^+_R$ into $\scrO_R$.
Part $(c)$ is proved in \cite[prop.~9.15]{KL}. The proof in loc.~cit.~is done under
some restrictive conditions on the $\frakg_R$-modules $N_a$ and under the hypothesis that $\kappa\in\bbC$, but
it extends  to our setting. 
Part $(d)$ is proved in \cite[prop.~9.18]{KL}. The loc.~cit.~the proof is given
for $R=\bbC$, but it generalizes to our case.
\end{proof}

\vspace{2mm}

From now on, for $R=\bbC$ we drop the subscript $R$ from the notation.

\vspace{2mm}

\begin{ex}\label{ex:8.3}
Let $\bfone=\bfone_R$ denote the parabolic Verma module $\bfM(0)_{R,+}$. 
Assume that $R=K$ is a field and that $\kappa_K\notin\bbQ_{\geqslant 0}$.
Then, we have $\bfone=\bfL(0)_K,$ see e.g., \cite[prop.~2.12]{KL}. 
\end{ex}

\vspace{3mm}

\subsection{The monoidal structure on $\bfO$} 
\label{sec:monoidal}
Let $R=K$ be a field which is an analytic algebra.
Fix $\kappa_K$ in $K\setminus\bbQ_{\geqslant 0}$.

There is a  braided monoidal structure $(\bfO^{+,\kappa}_K,\dot\otimes_K,\bfa_K,\bfc_K)$
on $\bfO^{+,\kappa}_K$ with unit $\bfone$, see \cite{KL} and \S \ref{sec:TP2} below.

We have
a \emph{bimodule category} $(\bfO^{\nu,\kappa}_K,\dot\otimes_K,\bfa,\bfc)$ over
$(\bfO^{+,\kappa}_K,\dot\otimes_K,\bfa_K,\bfc_K)$. This means first that
$\bfO^{\nu,\kappa}_K$ is a left and right \emph{module category} over
$\bfO^{+,\kappa}_K$, see \cite{Os} and \S \ref{sec:TP2} for details,
and that the functors $\bfa_K$, $\bfc_K$
satisfy an analogue of the hexagon axiom that expresses the commutativity
of the left and right actions. 
More precisely,
\begin{itemize}
\item there are bilinear functors 
$\dot\otimes_K:\bfO^{+,\kappa}_K\times\bfO^{\nu,\kappa}_K\to\bfO^{\nu,\kappa}_K$ and
$\dot\otimes_K:\bfO^{\nu,\kappa}_K\times\bfO^{+,\kappa}_K\to\bfO^{\nu,\kappa}_K$
such that $V\dot\otimes_K\,\bullet$ and $\bullet\,\dot\otimes_K V$ are exact for each $V\in\bfO^{+,\kappa}_K$,
\vspace{1mm}
\item there are functorial (left and right) unit isomorphisms for $M\in\bfO^{\nu,\kappa}_K$,
\begin{equation*}
\bfone\dot\otimes_K M\to M,\qquad M\dot\otimes_K \bfone\to M,
\end{equation*}
\item there are functorial associativity isomorphisms
for $V_1,V_2\in\bfO^{+,\kappa}_K$, $M\in\bfO^{\nu,\kappa}_K$,
\begin{equation*}
\gathered
\bfa_{V_1,V_2,M}:(V_1\dot\otimes_K V_2)\dot\otimes_K M\to 
V_1\dot\otimes_K(V_2\dot\otimes_K M),\\
\bfa_{V_1,M,V_2}:(V_1\dot\otimes_K M)\dot\otimes_K V_2\to 
V_1\dot\otimes_K(M\dot\otimes_K V_2),\\
\bfa_{M,V_1,V_2}:(M\dot\otimes_K V_1)\dot\otimes_K V_2\to 
M\dot\otimes_K(V_1\dot\otimes_K V_2),
\endgathered
\end{equation*}
\item there are functorial commutativity isomorphisms
for $V\in\bfO^{+,\kappa}_K$, $M\in\bfO^{\nu,\kappa}_K$,
\begin{equation*}
\gathered
\bfc_{V,M}:V\dot\otimes_K M\to M\dot\otimes_K V,
\endgathered
\end{equation*}
\item $\bfone$, $\bfa$ satisfy the triangle axioms (left and right)
for $V\in\bfO^{+,\kappa}_K$, $M\in\bfO^{\nu,\kappa}_K$,
\begin{equation*}
\begin{split}
\xymatrix{(V\dot\otimes_K\bfone)\dot\otimes_K M\ar[rr]\ar[dr]&&&V\dot\otimes_K(\bfone\dot\otimes_K M)\ar[dl]\\
&V\dot\otimes_K M,&}
\end{split}
\end{equation*}
\item  $\bfa$ satisfies the pentagon axiom (left and right) for $V_1,V_2, V_3\in\bfO^{+,\kappa}$, $M\in\bfO^{\nu,\kappa}$,
\begin{equation*}
\begin{split}
\xymatrix{
V_1\dot\otimes_K((V_2\dot\otimes_K V_3)\dot\otimes_K M)\ar[rr]&&&V_1\dot\otimes_K(V_2\dot\otimes_K
(V_3\dot\otimes_K M))\\
(V_1\dot\otimes_K(V_2\dot\otimes_K V_3))\dot\otimes_K M\ar[u]&&&(V_1\dot\otimes_K V_2)
\dot\otimes_K(V_3\dot\otimes_K M)\ar[u]\\
&((V_1\dot\otimes_K V_2)\dot\otimes_K V_3)\dot\otimes_K M,\ar[ur]\ar[ul]&}
\end{split}
\end{equation*}
plus the diagrams obtained by cyclic permutation of $M,V_1,V_2,V_3,$
\item  $\bfa_K$, $\bfc_K$ satisfy the hexagon axiom for $V_1,V_2\in\bfO^{+,\kappa}_K$, $M\in\bfO^{\nu,\kappa}_K$,
\begin{equation*}
\begin{split}
\xymatrix{
&V_1\dot\otimes_K(M\dot\otimes_K V_2)\ar[r]&(M\dot\otimes_K V_2)\dot\otimes_K V_1\ar[dr]&\\
(V_1\dot\otimes_K M)\dot\otimes_K V_2\ar[ur]\ar[dr]&&&M\dot\otimes_K(V_2\dot\otimes_K V_1)\\
&(M\dot\otimes_K V_1)\dot\otimes_K V_2\ar[r]&M\dot\otimes_K(V_1\dot\otimes_K V_2),\ar[ur]}
\end{split}
\end{equation*}
plus the diagrams obtained by cyclic permutation of $M,V_1,V_2$.
\end{itemize}

\vspace{2mm}

\begin{rk}
The notion of \emph{bimodule functors}, and, in particular, of equivalence of bimodule
categories is defined in the obvious way. Generally one impose the functor
$\dot\otimes$ to be biexact. Our choice simplifies the exposition in the rest of the paper.
\end{rk}

\vspace{2mm}

Now, let us recall the construction of $\dot\otimes_K$.
Let $A=[-m,n]$ and fix an admissible system of charts $\gamma$.
Given a smooth $\bfg_\kappa$-module $M_a$ for each $a\in A$, we consider the functor
$M\mapsto\llangle M_{-m},\dots,M_n,DM\rrangle_K^*.$

\vspace{2mm}

\begin{prop} 
\label{prop:TP1}
Assume that $M_0\in\bfO^{\nu,\kappa}_K$  and that $M_a\in\bfO^{+,\kappa}_K$ for $a\neq 0$.

(a) There is a module 
$\dot\bigotimes_{K,a}M_a\in\bfO^{\nu,\kappa}_K$ such that,
for each $M\in\scrS_\kappa,$ we have
$\Hom_{\bfg_K}(\dot\Otimes_{K,a}M_a, M)=
\llangle M_{-m},\dots,M_n,DM\rrangle_K^*.$

(b) We have a functorial  isomorphism
$\llangle \dot\Otimes_{K,a}M_a, DM\rrangle_K
=\llangle M_{-m},\dots,M_{n},D M\rrangle_K.$
\end{prop}

\vspace{.5mm}

\begin{proof} 
It is easy to prove that $\bfO^{\nu,\kappa}_K$ is the category of the finitely generated smooth
$\bfg_{K,\kappa}$-modules $M$ such that $M(d)$ belongs to $\scrO^\nu_K$ for all $d\geqslant 1$.
Thus, part $(a)$ follows from \cite[def.~1.2, thm.~1.6]{Y1}.
Part $(b)$ is proved as in \cite[sec.~7.10, 13.4]{KL}.
\end{proof}

\vspace{2mm}

\begin{rk}
The $\bfg_{K,\kappa}$-module $\dot\bigotimes_{K,a}M_a$
does not depend on the choice of the admissible system of charts, up to a canonical isomorphism.
\end{rk}

\vspace{2mm}

Consider the charts $\gamma_{-1}$, $\gamma_0$, $\gamma_1$ centered at 
$1$, $\infty$, $0$ respectively, given in \cite[sec.~13.5]{KL}.  Let $A=[-1,0]$.
Then, Proposition \ref{prop:TP1} yields modules $V\dot\otimes_K M$ and
$M\dot\otimes_K V$ in $\bfO^{\nu,\kappa}_K$
for each $V\in\bfO^{+,\kappa}$, $M\in\bfO^{\nu,\kappa}_K$. 
The following is proved in \S \ref{sec:TP2}.

\vspace{2mm}

\begin{prop} \label{prop:TP2}
The functors $\dot\otimes_K:\bfO^{+,\kappa}_K\times\bfO^{\nu,\kappa}_K\to\bfO^{\nu,\kappa}_K$ and 
$\dot\otimes_K:\bfO^{\nu,\kappa}_K\times\bfO^{+,\kappa}_K\to\bfO^{\nu,\kappa}_K$ give
a bimodule category $(\bfO^{\nu,\kappa}_K,\dot\otimes_K,\bfa_K,\bfc_K)$ over 
the braided monoidal category $(\bfO^{+,\kappa}_K,\dot\otimes_K,\bfa_K,\bfc_K)$.
\qed
\end{prop}

\vspace{2mm}

Note that $\scrO^{+,\Delta}_K=\scrO^{+}_K$, because the category is semi-simple.
The tensor product equips $[\scrO^{+,\Delta}_K]$ with a commutative 
$\bbC$-algebra structure and $[\scrO^{\nu,\Delta}_K]$ with a bimodule structure over 
$[\scrO^{+,\Delta}_K]$.

\vspace{2mm}

\begin{prop} \label{prop:TP2bis}
(a) The functor $\dot\otimes_K$ preserves the $\Delta$-filtered modules.

(b) The functor $\dot\otimes_K$ preserves the tilting modules.

(c)  The functor $\dot\otimes_K$ is biexact on $\bfO^{+,\kappa,\Delta}_K\times\bfO^{\nu,\kappa,\Delta}_K$ and on
$\bfO^{\nu,\kappa,\Delta}_K\times\bfO^{+,\kappa,\Delta}_K$. It equips $[\bfO^{+,\kappa,\Delta}_K]$ with a commutative 
$\bbC$-algebra structure and $[\bfO^{\nu,\kappa,\Delta}_K]$ with a bimodule structure over 
$[\bfO^{+,\kappa,\Delta}_K]$.

(d) The functor $\scrI\!nd$ gives a $\bbC$-algebra isomorphism
$[\scrO^{+,\Delta}_K]\to[\bfO^{+,\kappa,\Delta}_K]$ and a module isomorphism
$[\scrO^{\nu,\Delta}_K]\to[\bfO^{\nu,\kappa,\Delta}_K]$.
\qed
\end{prop}

\vspace{2mm}

\begin{rk} \label{rk:balancing} By \cite[prop.~31.2]{KL} the braided monoidal category 
$(\bfO^{+}_K,\dot\otimes_K,\bfa,\bfc)$ admits a \emph{balancing}. 
More precisely, we have
\begin{equation}\label{balancing}
\bfc^2=\exp(-2\pi\sqrt{-1}\frakL_0)\big(\exp(2\pi\sqrt{-1}\frakL_0)\dot\otimes_K\exp(2\pi\sqrt{-1}\frakL_0)\big).
\end{equation}
The proof in loc.~cit.~
implies that \eqref{balancing} holds also for tensor products of modules from
$\bfO^{\nu,\kappa}_K$ and $\bfO^{+,\kappa}_K$.
\end{rk}

\vspace{2mm}

\begin{rk} 
In \cite{KL}, \cite{Y1} the tensor product $\dot\otimes_K$ is defined over the field $K=\bbC$ for any
$\kappa\notin\bbQ_{\geqslant 0}$.
The same definition works equally well over any field $K$ containing $\bbC$ and for any
$\kappa\in K\setminus\bbQ_{\geqslant 0}$, see Remark \ref{rem:TPK} below.
If $K=\bbC$ we remove the subscript $K$ everywhere from the notation.
\end{rk}

\vspace{3mm}

\subsection{The monoidal structure on $\bfO$ over a ring}
\label{sec:KLotimes,II}
Let $R$ be either a field or a regular local ring of dimension $\leqslant 2$ 
with residue field $\Bbbk$.
Assume that $\kappa_\Bbbk=-e$ where $e$ is a positive integer.
The following is proved in \S \ref{sec:TPR1}.

\vspace{2mm}

\begin{prop}\label{prop:TPR1} 
(a) Assume that $A=[-m,n]$, $M_0\in \bfO^{\nu,\kappa,\Delta}_R$ and $M_a\in\bfO^{+,\kappa,\Delta}_R$ for $a\neq 0$.
Then, there is a module $\dot\Otimes_{R,a}M_a$
in $\bfO^{\nu,\kappa,\Delta}_R$ which is functorial in the $M_a$'s and such that
$\Hom_{\bfg_R}(\dot\Otimes_{R,a}M_a,M)=\llangle M_{-m},\dots,M_{n},D M\rrangle_R^*$
for each $M\in\bfO^{\nu,\kappa,f}_R$.

(b) The functor $\dot\otimes_R$ commutes with flat base change (of the ring $R$).
\qed
\end{prop}

\vspace{2mm}

We abbreviate
$\bfV_R=\bfM(\epsilon_1)_{R,+}$ and $\bfV^*_R=\bfM(-\epsilon_N)_{R,+}$ 
in $\bfO^{+,\kappa}_R$. 

\vspace{2mm}

\begin{lemma}\label{lem:V}
(a) If $R=K$ is a field then $\bfV_K$, $\bfV_K^*$ are simple.

(b) The modules $\bfV_R$, $\bfV_R^*$ in $\bfO_R^+$ are tilting.

(c) We have $D\bfV_R={}^\dag\bfV_R=\bfV_R^*,$ $\bfV_R=\scrI\!nd(V_R)$ and $\bfV^*_R=\scrI\!nd(V^*_R)$.
\end{lemma}

\vspace{.5mm}

\begin{proof} 
If $R=K$ is a field, then $\bfV_K$, $\bfV_K^*$ are simple, proving part $(a)$.
To prove $(b)$, note that under base change we get $\bfV_\Bbbk=\Bbbk\bfV_R$ and 
$\bfV_\Bbbk^*=\Bbbk\bfV_R^*$. 
The modules $\bfV_\Bbbk$ and 
$\bfV_\Bbbk^*$ are simple and standard.
Thus, they are both tilting.
Therefore $\bfV_R$, $\bfV_R^*$ are also tilting modules by Proposition \ref{prop:introhw}.
Part $(c)$ is clear, because $(b)$ implies that
$\scrD\bfV_R=\bfV_R$ and $\scrD\bfV_R^*=\bfV_R^*$.
\end{proof}

\vspace{2mm}

We define the endofunctors $e,$ $f$ of $\bfO^{\nu,\kappa,\Delta}_R$ by
$e=\bullet\,\dot\otimes_R\bfV^*_R$ and $f=\bullet\,\dot\otimes_R\bfV_R,$
and  the endofunctors $e,$ $f$ of $\scrO^{\nu,\Delta}_R$ by
$e=\bullet\,\otimes_RV^*_R$ and $f=\bullet\,\otimes_RV_R.$
The following is proved in \S \ref{sec:TPR1}.

\vspace{2mm}

\begin{prop}\label{prop:adjoints} 
(a) The functors $e$, $f$ are exact and preserve $\bfO^{\nu,\kappa,\tilt}_R$.

(b)  We have functorial isomorphisms
$\Bbbk e(M)\simeq e(\Bbbk M)$ and $\Bbbk f(M)\simeq f(\Bbbk M)$
for each $M\in\bfO^{\nu,\kappa,\tilt}_R$.

(c) If $R=K$ is a field then $e$, $f$ extend to exact biadjoint endofunctors on $\bfO^{\nu,\kappa}_K$.

(d) The functor $\Indc_R$ gives a $\bbC$-vector space isomorphism
$[\scrO^{\nu,\Delta}_R]\to[\bfO^{\nu,\kappa,\Delta}_R]$ which commutes with the $\bbC$-linear maps
$e,$ $f$.
\qed
\end{prop}

\vspace{2mm}

\begin{prop}\label{prop:monoidalR}
Assume that $R$ is a local analytic algebra.

(a) There is
a braided monoidal category $(\bfO^{+,\kappa,\Delta}_R,\dot\otimes_R,\bfa_R,\bfc_R)$ and
a bimodule category $(\bfO^{\nu,\kappa,\Delta}_R,\dot\otimes_R,\bfa_R,\bfc_R)$ over it.

(b) For each $M\in\bfO^{\nu,\kappa,\tilt}_R$ and each $d\geqslant 1$, we have a $\Bbbk$-algebra isomorphism
$\Bbbk\End_{\bfg_R}(f^d(M))\to\End_{\bfg_\Bbbk}(f^d(\Bbbk M))$
which commutes with the associativity and the commutativity constraints.
\qed
\end{prop}

\vspace{3mm}

\subsection{From $\bfO$ to the cyclotomic Hecke algebra}
\label{sec:endo-KL}
Let $R$ be a local analytic deformation ring of dimension $\leqslant 2$.
Set $v=v_R=\exp(-\sqrt{-1}\pi/\kappa_R)$, $q=q_R=v_R^{2}.$

The endomorphisms $X$, $T$ of the functors $f$, $f^2$ are given by
$X=\bfc_R\circ\bfc_R$ and $T=v_R\cdot\bfa_R^{-1}\circ(1\dot\otimes_R\bfc_R)\circ \bfa_R$, i.e.,
for each $M\in\bfO^{\nu,\kappa}_{R}$ we have
\begin{equation}\label{XT-KL}
\gathered
X_M=\bfc_{\bfV_R,M}\circ\bfc_{M,\bfV_R},\\
T_M=v_R\cdot\bfa^{-1}_{M,\bfV_R,\bfV_R}\circ(1_M\dot\otimes_R \bfc_{\bfV_R,\bfV_R})
\circ\bfa_{M,\bfV_R,\bfV_R}.\endgathered
\end{equation} 

Next, fix an integer $d\geqslant 1$ and
consider the endomorphisms of $f^d$ given by
$X_j=1^{d-j}X1^{j-1}$ and
$T_i=1^{d-i-1}T1^{i-1}$ with $j\in [1,d]$,  $i\in [1,d)$.

We can now prove the following.

\vspace{2mm}

\begin{prop}\label{prop:psi-s} 
(a) The endomorphisms $X_j$, $T_i$ 
yield an $R$-algebra homomorphism $\psi_{R,d}:\bfH_{R,d}\to\End(f^d)^\op$.

(b) The map $\psi_{R,d}$ gives an $R$-algebra homomorphism
$\psi_{R,d}^s:\bfH^{s}_{R,d}\to\End_{\bfg_R}(\bfT_{R,d})^\op.$
\end{prop}

\vspace{.5mm}

\begin{proof}
The braid relations $T_iT_{i+1}T_i=T_{i+1}T_iT_{i+1}$ and $T_iT_j=T_jT_i$ if $|i-j|>1$
are well-known formal consequences of the axioms of a braided monoidal category. 

Next, consider the relation
$T_iX_{i}T_i=v_R^2X_{i+1}$.
The hexagon axiom yields the following relations
$$\gathered
\bfa_{\bfV_R,M,\bfV_R}\circ(\bfc_{M,\bfV_R}\dot\otimes_R 1_{\bfV_R})\circ T_M=
v_R\cdot\bfc_{f(M),\bfV_R},
\cr T_M\circ(\bfc_{\bfV_R,M}\dot\otimes_R 1_{\bfV_R})\circ\bfa_{\bfV_R,M,\bfV_R}^{-1}=
v_R\cdot\bfc_{\bfV_R,f(M)}.\endgathered$$ 
Therefore we have
$T_M\circ(X_M\dot\otimes_R 1_{\bfV_R})\circ T_M=v_R^2\cdot X_{f(M)}.$ 
We deduce that 
$(T_i)_M\circ(X_i)_M\circ(T_i)_M=v_R^2\cdot (X_{i+1})_M.$



Now, let us prove the relations
$X_iX_j=X_jX_i$ and $X_iT_j=T_jX_i$ for $i\neq j,j+1.$
We are reduced to check the 
relations
$(X_1)_M\circ(X_i)_M=(X_i)_M\circ(X_1)_M$
and $(T_i)_M\circ(X_1)_M=(X_1)_M\circ(T_i)_M$ for $i\neq 1.$
They follow from the functoriality of $\bfc$ and 
$\bfa$. Let us check the first one in details for $i=1$, $j=2$. The diagram
$$\xymatrix{
f(M)\dot\otimes_R\bfV_R\ar[r]^{X_M\dot\otimes 1}\ar[d]_{\bfc_{f(M),\bfV_R}}
&f(M)\dot\otimes_R\bfV_R\ar[d]^{\bfc_{f(M),\bfV_R}}\cr
\bfV_R\dot\otimes_Rf(M)\ar[r]^{1\dot\otimes
X_M}\ar[d]_{\bfc_{\bfV_R,f(M)}}
&\bfV_R\dot\otimes_Rf(M)\ar[d]^{\bfc_{\bfV_R,f(M)}}\cr
f(M)\dot\otimes_R\bfV_R\ar[r]^{X_M\dot\otimes 1}
&f(M)\dot\otimes_R\bfV_R}$$
is commutative because $X_M$ is an
endomorphism of $f(M)$ and $\bfc_R$ is a morphism of
functors. Now the composition of both vertical maps is equal to
$X_{f(M)}=(X_2)_M$, while the upper and the lower horizontal
maps are both equal to $(X_1)_M$. We are done.

To prove the Hecke relation $(T_i+1)(T_i-q_R)=0$, observe that the action of
$\omega$ on $V_R\otimes_R V_R$ is a diagonalizable operator
with eigenvalues $1$ and $-1$.
Thus, from \eqref{XT-KL} we get that
$(T_i-v_R^2)(T_i+1)=0$.

Finally, let us check the cyclotomic relation.
By \eqref{balancing}, \eqref{XT-KL}, the endomorphism
$(X_1)_{\bfT_{R,0}}$ of $\bfT_{R,d}$ is identified with the
endomorphism
$f^{d-1}(X_{\bfT_{R,0},\bfV_R})$ of $f^{d-1}(\bfT_{R,1})$,
where $X_{\bfT_{R,0},\bfV_R}$ 
is an operator on
$\bfT_{R,1}=\bfT_{R,0}\dot\otimes_R \bfV_R$.
We must prove that this operator satisfies the equation
$\prod_{p=1}^\ell(X_{\bfT_{R,0},\bfV_R}-q_R^{s_p})=0.$
We may assume that $R=K$ is a field. Then, the claim follows from Remark \ref{rem:5.15}.
\end{proof}

\vspace{3mm}

\appendix

\section{Proofs for \S \ref{sec:KL}}

\subsection{Proof of Propositions \ref{prop:TP2}, \ref{prop:TP2bis}}
\label{sec:TP2}
First, let us explain briefly the definition of the
bimodule category $(\bfO^{\nu,\kappa}_K,\dot\otimes_K,\bfa_K,\bfc_K)$ over $(\bfO^{+,\kappa}_K,\dot\otimes_K,\bfa_K,\bfc_K)$.

The bifunctor $\dot\otimes_K$ on $\bfO^{+,\kappa}_K$ is defined in \cite{KL}.
By \cite{Y1}, the same definition yields functors
$\bfO^{+,\kappa}_K\times\bfO^{\nu,\kappa}_K\to\bfO^{\nu,\kappa}_K$ and $\bfO^{\nu,\kappa}_K\times\bfO^{+,\kappa}_K\to\bfO^{\nu,\kappa}_K$.
The endomorphisms of functors $\bfa_K$, $\bfc_K$ are defined in \cite[sec.~14, 18]{KL}. 

Note that \cite{KL}, \cite{Y1} deal only with the field $K=\bbC$. For $K\neq\bbC$ a field containing $\bbC$, one define
$\dot\otimes_{K}$ as over $\bbC$, see Remark \ref{rem:TPK}.
Further, the morphisms $\bfa_K$, $\bfc_K$ 
are defined for $\bfO^{+,\kappa}$, 
but for  $\bfO^{\nu,\kappa}_K$ one can 
proceed in the same way. 

More precisely, since the spaces of coinvariants are finite dimensional 
by Lemma \ref{lem:TP0} and since $K$ is an analytic algebra,
the proof of \cite[thm.~17.29]{KL} works equally well in our case and
standard facts about linear ordinary differential equations yield a canonical isomorphism
$
\llangle V_1\dot\otimes_K M_1,V_2\dot\otimes_K DM_2\rrangle_K=
\llangle V_1, M_1,V_2, DM_2\rrangle_K$
for all
$M_1,M_2\in\bfO^{\nu,\kappa}_K$ and all $V_1,V_2\in\bfO^{+,\kappa}_K$.
Then, we define $\bfa_{V_1,M_1,V_2}$ using this isomorphism and
Proposition \ref{prop:TP1} as in \cite[sec.~18.2]{KL}. 
The other associativity constraints are constructed in the same way using the cyclic invariance of coinvariants.
The braiding is also defined as in \cite{KL}, since any module from
$\bfO^{\nu,\kappa}_K$ admits an action of the Sugawara operators. See the proof of Lemma \ref{lem:c} for details.

Now, we must check the axioms of a bimodule category.
The proofs are essentially the same as in \cite{KL}.
We will give a few details for the comfort of the reader.

The unit of $(\bfO^{+,\kappa}_K,\dot\otimes_K,\bfa_K,\bfc_K)$ is the module $\bfone$ from Example \ref{ex:8.3}. 
By \cite[sec.~31,32]{KL}, this braided monoidal category  is rigid
with the duality functor $D$.
This means that $D$ is exact and that for any $M\in\bfO^{+,\kappa}_K$ there are functorial morphisms 
$i_M:\bfone\to M\dot\otimes_K DM$ and
$e_M: DM\dot\otimes_K M\to\bfone$
such that  $DM\dot\otimes_K\,\bullet$ is left adjoint to $M\dot\otimes_K\,\bullet.$
Equivalently, the functor $\bullet\,\dot\otimes_K M$ is left adjoint to the functor $\bullet\,\dot\otimes_K DM.$
Since $D$ is an involution, the functors above are indeed biadjoint.

\vspace{2mm}

\begin{lemma}\label{lem:unit}
For $M\in\bfO^{\nu,\kappa}_K$ there are functorial isomorphisms
$\bfone\dot\otimes_K M\to M$ and $M\dot\otimes_K\bfone\to M$
which satisfy the triangle axioms.
\end{lemma}

\vspace{.5mm}

\begin{proof} 
The proof is the same as in \cite{KL}. We will be brief.
We define the unit isomorphism 
$\pmb\ell_M:\bfone\dot\otimes_K M\to M$
to represent the isomorphism of functors given as follows,
for each $M_1, M_2\in\bfO^{\nu,\kappa}_K$,
\begin{equation*}
\aligned
\Hom_{\bfg_K}(M_1,M_2)
&=\llangle M_1,D M_2\rrangle_K^*\cr
&=\llangle\bfone,M_1,D M_2\rrangle_K^*\cr
&=\Hom_{\bfg_K}(\bfone\dot\otimes_K M_1,M_2).
\endaligned
\end{equation*}
The second isomorphism is given Lemma \ref{lem:TP0}$(b)$, the other ones are as in
Proposition \ref{prop:TP1}. A similar construction yields the isomorphism 
$\pmb r_M:M\dot\otimes_K\bfone\to M.$ 

Now, it is enough to check the triangle axiom for
$V=\bfone\in\bfO^{+,\kappa}_K$ and $M\in\bfO^{\nu,\kappa}_K$ (then, the general version follows using the pentagon axiom
for the quadruple $V,\bfone,\bfone, M$). So we must check that the composition
\begin{equation*}
\xymatrix{(\bfone\dot\otimes_K\bfone)\dot\otimes M\ar[r]^{\bfa_{\bfone,\bfone,M}}&
\bfone\dot\otimes_K(\bfone\dot\otimes_K M)\ar[r]^-{1\dot\otimes_K \pmb\ell_M}&\bfone\dot\otimes_K M}
\end{equation*}
is given by the unit $\pmb r_\bfone:\bfone\dot\otimes_K\bfone\to\bfone.$ 
This follows from Proposition \ref{prop:TP1}$(b)$ and the invariance of coinvariants under cyclic permutation as in \cite[sec. ~18.2]{KL}.
This allows us to identify the morphism
\begin{equation*}
\llangle(\bfone\dot\otimes_K\bfone)\dot\otimes_K M,N\rrangle_K\to
\llangle\bfone\dot\otimes_K(\bfone\dot\otimes_K M),N\rrangle_K\to
\llangle\bfone\dot\otimes_KM,N\rrangle_K
\end{equation*}
induced by $\pmb\ell_M$, $\bfa_{\bfone,\bfone,M}$ with the morphism
$$\llangle(\bfone\dot\otimes_K\bfone)\dot\otimes_K M,N\rrangle_K\to
\llangle\bfone\dot\otimes_KM,N\rrangle_K$$
induced by $\pmb r_\bfone$.
\end{proof}

\vspace{2mm}

\begin{lemma}\label{lem:TP2a}
For $M\in\bfO^{+,\kappa}_K$ the functors $\bullet\,\dot\otimes_K M$ and $\bullet\,\dot\otimes_K DM$ on $\bfO^{\nu,\kappa}_K$
are exact and biadjoint  to each other.
The same holds for the functors $M\dot\otimes_K\,\bullet$ and $DM\dot\otimes_K\,\bullet$.
\end{lemma}

\vspace{.5mm}

\begin{proof}
If $\bfO^{\nu,\kappa}_K=\bfO^{+,\kappa}_K$ the lemma follows from the rigidity of $(\bfO^{+,\kappa}_K,\dot\otimes_K,\bfa_K,\bfc_K)$.
The general case is proved in the same way, using the rigidity of $M$, $DM$ in $\bfO^{+,\kappa}_K$ and Lemma
\ref{lem:unit} instead of the unit axiom of $(\bfO^{+,\kappa}_K,\dot\otimes_K,\bfa_K,\bfc_K)$.
\end{proof}

\vspace{2mm}

We can now prove Propositions \ref{prop:TP2} and \ref{prop:TP2bis}.

\begin{proof}[Proof of Proposition \ref{prop:TP2}]
We must check that $\bfa_K$, $\bfc_K$ satisfy the pentagon and the hexagon axioms.
These are proved as in Proposition \cite[prop.~31.2]{KL},
using an auxiliary module category called
the \emph{Drinfeld category}.

Set $R_\infty=K[[\varpi]]$ and $K_\infty=K((\varpi))$.
Put $\kappa_{R_\infty}=\kappa_{K_\infty}=-1/\varpi$.
Consider the elements $v=\exp(\sqrt{-1}\pi\varpi)$ and $q=v^{-2}$
in $R_\infty$. 
Whenever this makes sense we write
$v^z=\sum_{r\in\bbN}\varpi^r(\sqrt{-1}\pi z)^r/r!$.

The category $\scrO_\infty$
of \emph{deformation representations} of $\frakg$ consists of the
representations of $\frakg_{R_\infty}$ on topologically free $R_\infty$-modules $M$ such that
$M$ is a weight $\frakt_{R_\infty}$-module and
the weights of $M$ belong to a union of finitely many cones $\lambda-Q^+$
and the weight subspaces are free of finite type over $R_\infty$.

Following Drinfeld and \cite{EK}, \cite{KL}
we put on $\scrO_\infty$ a structure of a braided monoidal category
$(\scrO_\infty,\otimes_{R_\infty}, \bfa_{\infty}, \bfc_{\infty})$ where
$\otimes_{R_\infty}$ is the tensor product of $R_\infty$-modules and
$\bfa_\infty$ is the Knizhnik-Zamolodchikov associator, i.e.,
\begin{equation*}
 \bfa_\infty=\{\bfa_{M_1,M_2,M_3}:(M_1\otimes_{R_\infty} M_2)\otimes_{R_\infty} M_3\to
M_1\otimes_{R_\infty}(M_2\otimes_{R_\infty} M_3)\}
\end{equation*}
is defined in \cite[sec.~19.10]{KL}.
Note that we do not impose $M$ to be of finite rank over $R_\infty$.
However, since the weight subspaces of $M$ are free of finite type over $R_\infty$,
by standard facts about linear ordinary differential equations, the series
obtained by restricting $\bfa_\infty$ to a weight
subspace in the tensor product of three objects of $\scrO_\infty$ 
is well-defined. 
The braiding is given by the following formula, see \cite[sec.~19.12]{KL},
\begin{equation*}
 \bfc_\infty=\{\bfc_{M_1,M_2}=v^\omega \sigma:M_1\otimes_{R_\infty} M_2\to M_2\otimes_{R_\infty} M_1\}
\end{equation*}
where $\sigma$ flips the factors and $\omega$ is the Casimir element. 
Recall that
\begin{itemize}
\item the functor $\otimes_{R_\infty}$ is  $R_\infty$-bilinear and biexact, 
\item there is a unit object $\bfone$, which is simple  
(equal to $R_\infty$ with the trivial action),
with functorial unit isomorphisms
$\bfone\otimes_{R_\infty}M\to M$, $M\otimes_{R_\infty}\bfone\to M$,
\item the unit $\bfone$ and the functor $\bfa_\infty$ satisfy the triangle axiom,
\item the functor $\bfa_\infty$ satisfies the pentagon axiom, 
\item the functors $\bfa_\infty$ and $\bfc_\infty$ satisfy the hexagon axiom.
\end{itemize}

Restricting
the braided monoidal structure on $\scrO_\infty$ to  some parabolic subcategories, we
define in the obvious way 
\begin{itemize}
\item 
a braided monoidal category $(\scrO^+_\infty,\otimes_{R_\infty}, \bfa_\infty, \bfc_\infty)$
called the \emph{Drinfeld category,} which consists of
the modules which are free of finite rank over $R_\infty$,

\item
a bimodule category $(\scrO^\nu_\infty,\otimes_{R_\infty}, \bfa_\infty, \bfc_\infty)$ over 
$(\scrO^+_\infty,\otimes_\infty, \bfa_\infty, \bfc_\infty)$.
\end{itemize}

Now, we may assume that there is a local analytic deformation ring $R\subset R_\infty$ of dimension 1 with residue field $K$ such that
$\kappa_R=\kappa_{R_\infty}=-1/\varpi$ is the germ of an holomorphic function over some polydisc and
the inclusion  $R\subset R_\infty$ is given by the expansion at $\varpi=\infty$.
Further, the specialization map
$R\to K$ takes $\kappa_R$ to $\kappa_K$.
Since $R_\infty$ is flat over $R$, the base change yields an exact functor 
$\bfO_{R}^{\kappa,\Delta}\to\bfO_{R_\infty}^{\kappa,\Delta}$.

\vspace{2mm}

\begin{lemma}
\label{prop:Comp0}
(a) There is a faithful braided functor  and a faithful bimodule functor
$$\gathered
(\bfO^{+,\kappa,\Delta}_{R},\dot\otimes_R,\bfa_R,\bfc_R)
\to(\bfO^{+,\kappa,\Delta}_{R_\infty},\dot\otimes_{R_\infty},\bfa_{R_\infty},\bfc_{R_\infty}),\\
(\bfO^{\nu,\kappa,\Delta}_{R},\dot\otimes_R,\bfa_R,\bfc_R)
\to(\bfO^{\nu,\kappa,\Delta}_{R_\infty},\dot\otimes_{R_\infty},\bfa_{R_\infty},\bfc_{R_\infty}).
\endgathered$$

(b) There is a braided equivalence and a bimodule equivalence
$$\gathered
(\bfO^{+,\kappa,\Delta}_{R_\infty},\dot\otimes_{R_\infty},\bfa_{R_\infty},\bfc_{R_\infty})
\to(\scrO^{+,\Delta}_\infty,\otimes_{R_\infty},\bfa_\infty,\bfc_\infty),\\
(\bfO^{\nu,\kappa,\Delta}_{R_\infty},\dot\otimes_{R_\infty},\bfa_{R_\infty},\bfc_{R_\infty})
\to(\scrO^{\nu,\Delta}_\infty,\otimes_{R_\infty},\bfa_\infty,\bfc_\infty).
\endgathered$$

(c) The specialization gives a braided functor  and a bimodule functor
$$\gathered
(\bfO^{+,\kappa,\Delta}_{R},\dot\otimes_R,\bfa_R,\bfc_R)
\to(\bfO^{+,\kappa,\Delta}_K,\dot\otimes_K,\bfa_K,\bfc_K),\\
(\bfO^{\nu,\kappa,\Delta}_{R},\dot\otimes_R,\bfa_R,\bfc_R)
\to(\bfO^{\nu,\kappa,\Delta}_K,\dot\otimes_K,\bfa_K,\bfc_K).
\endgathered$$
\end{lemma}

\vspace{.5mm}

\begin{proof}
Part $(b)$ is proved as in \cite[sec.~31]{KL}. Part $(a)$ follows from Lemma \ref{lem:A11} below.
Part $(c)$ is proved as in \cite[thm.~29.1]{KL}.
\end{proof}

\vspace{2mm}

We can now finish the proof of Proposition \ref{prop:TP2}.
Composing the functors in $(a)$, $(b),$ we get faithful functors
$(\bfO^{+,\kappa,\Delta}_{R},\dot\otimes_R,\bfa_R,\bfc_R)
\to(\scrO^{+,\Delta}_\infty,\otimes_{R_\infty},\bfa_\infty,\bfc_\infty)$
and 
$(\bfO^{\nu,\kappa,\Delta}_{R},\dot\otimes_R,\bfa_R,\bfc_R)
\to(\scrO^{\nu,\Delta}_\infty,\otimes_{R_\infty},\bfa_\infty,\bfc_\infty).$
This implies that $\bfa_R$, $\bfc_R$ satisfy the pentagon and the hexagon axioms.
Hence, from $(c)$, we deduce that 
$\bfa_K$, $\bfc_K$ also satisfy the pentagon and the hexagon axioms.
The details are left to the reader.
\end{proof}

\vspace{2mm}

\begin{rk} Since $R$ has dimension 1, the specialization functor in Lemma \ref{prop:Comp0}$(c)$ is constructed in \cite{KL}.
In Lemma \ref{lem:A20}, we'll consider specializations from regular local rings of dimension 2.
\end{rk}

\vspace{2mm}

\begin{proof}[Proof of Proposition \ref{prop:TP2bis}]
First, we prove part $(a)$.
Fix $M_1\in\bfO^{+,\kappa,\Delta}_K$ and $M_2\in\bfO^{\nu,\kappa,\Delta}_K$. 
The module $M_1\dot\otimes_K M_2$
belongs to the category $\bfO^{\nu,\kappa}_K$ by
Proposition \ref{prop:TP1}. We must prove that it lies in $\bfO^{\nu,\kappa,\Delta}_K$.

Fix $\beta$ such that $M_2$ and
$M_1\dot\otimes_K M_2$ belong to the Serre subcategory
$\lub\bfO^{\nu,\kappa}_K$ of $\bfO^{\nu,\kappa}_K$. Since $\lub\bfO^{\nu,\kappa}_K$ is
a highest weight category with a duality functor $\scrD$, it is enough to check that for
$M_3\in\Delta(\lub\bfO^{\nu,\kappa}_K)$ we have the equality
$\Ext^1_{\lub\bfO^{\nu,\kappa}_K}(M_1\dot\otimes_K M_2,\scrD M_3)=0.$

Fix an exact sequence
$0\to Q\to P\to M_3\to 0$
with $P$ a projective module in $\lub\bfO^{\nu,\kappa}_K$. Since $P$, $M_3$ have $\Delta$-filtrations, the module
$Q$ is again a $\Delta$-filtered object of $\lub\bfO^{\nu,\kappa}_K$. Since $P$ is projective, we have 
$\Ext^1_{\lub\bfO^{\nu,\kappa}_K}(M_1\dot\otimes_K
M_2,\scrD P)=0.$
Therefore, since $\scrD$ is exact and contravariant, the long exact
sequence of the Ext-group and Proposition \ref{prop:TP1}  yield a vector space exact sequence
\begin{equation*}
\gathered
0\to \llangle M_1,M_2,{}^\dag M_3\rrangle_K^*\to
\llangle M_1,M_2,{}^\dag P\rrangle_K^*\to \llangle M_1,M_2,{}^\dag Q\rrangle_K^*\to\\
\to\Ext^1_{\lub\bfO^{\nu,\kappa}_K}(M_1\dot\otimes_K M_2,\scrD M_3)\to 0.
\endgathered
\end{equation*}
Thus, we get the equality of dimensions
\begin{equation*}
\gathered
\dim\Ext^1_{\lub\bfO^{\nu,\kappa}_K}(M_1\dot\otimes_K M_2,\scrD M_3)=
\dim\llangle M_1,M_2,{}^\dag P\rrangle_K -\dim\llangle M_1,M_2,{}^\dag Q\rrangle_K-\\
-\dim\llangle
M_1,M_2,{}^\dag M_3\rrangle_K.
\endgathered
\end{equation*}
The right hand side is zero by the following lemma.

\vspace{2mm}

\begin{lemma}\label{lem:TP3}
For $M_2,M_3\in\bfO^{\nu,\kappa,\Delta}_K$ and 
$M_1\in\bfO^{+,\kappa,\Delta}_K$ we have
$$\gathered
\dim\llangle M_1,M_2,{}^\dag M_3\rrangle_K=\\
\sum
(M_1:\bfM(\lambda_1)_+)\,
(M_2:\bfM(\lambda_2)_\nu)\,
(M_3:\bfM(\lambda_3)_\nu)\,
(L(\lambda_1)\otimes M(\lambda_2)_\nu:M(\lambda_3)_\nu),
\endgathered$$
where the sum is over all  $\lambda_1\in P_K^+$ and $\lambda_2,\lambda_3\in P^\nu_K$.
\end{lemma}

\vspace{.5mm}

\begin{proof} 
Let $d(M_1,M_2,{}^\dag M_3)$ denote the right hand side in the equality of the lemma.

First, assume that
$M_3=\bfM(\lambda_3)_\nu,$
$M_2=\bfM(\lambda_2)_\nu$ and $M_1=\bfM(\lambda_1)_+$.
We have
${}^\dag M_3=\Indc({}^\dag M(\lambda_3)_\nu)$. Thus ${}^\dag M_3$ is again a
generalized Weyl module and Lemma \ref{lem:TP0}$(a)$ yields
\begin{equation*}
\aligned
\llangle M_1,M_2,{}^\dag M_3\rrangle_K &=\langle L(\lambda_1),
M(\lambda_2)_\nu,{}^\dag M(\lambda_3)_\nu\rangle_K,\cr
&=\Hom_{\frakg_K}(L(\lambda_1)\otimes_K M(\lambda_2)_\nu,\scrD
M(\lambda_3)_\nu)^*.
\endaligned
\end{equation*}
Since 
$L(\lambda_1)\otimes_K M(\lambda_2)_\nu\in\scrO^{\nu,\Delta}_K$, by \cite[prop.~A.2.2$(ii)$]{Do} we get
$\dim\,\llangle M_1,M_2,{}^\dag M_3\rrangle_K=
(L(\lambda_1)\otimes_K M(\lambda_2)_\nu: M(\lambda_3)_\nu).$

The same argument implies that
$\dim\llangle M_1,M_2,{}^\dag M_3\rrangle_K= d(M_1,M_2,{}^\dag M_3)$
if $M_1$, $M_2$, $M_3$ are
generalized Weyl modules. 

Now, we concentrate on the general case.
First, observe that using the \emph{third} construction of $\dot\otimes_K$ in 
\cite[sec.~6]{KL} it is easy to check that
$\dot\otimes_K$ is right biexact. 
Further, by Proposition \ref{prop:TP1}, we have
$\llangle M_1,M_2,{}^\dag M_3\rrangle_K^*=\Hom_{\bfg_K}(M_1\dot\otimes_K M_2,\scrD M_3).$
Thus the left hand side is left exact in each of its variables.
So, given exact sequences 
$M_a^{(2)}\to M_a^{(1)}\to M_a^{(3)}\to 0$ of $\Delta$-filtered modules with $a=1,2,3$, we have
\begin{equation}\label{ineq}
\dim\llangle M_1^{(1)},M_2^{(1)},{}^\dag M_3^{(1)}\rrangle_K\leqslant
\sum_{\alpha,\beta,\gamma=2,3}\dim\llangle M_1^{(\alpha)},M_2^{(\beta)},{}^\dag M_3^{(\gamma)}\rrangle_K.
\end{equation}
Using the first part of the proof (i.e., the case of generalized Weyl modules) and \eqref{ineq}, we get that
for any $M_2,M_3\in\bfO^{\nu,\kappa,\Delta}_K$ and $M_1\in\bfO^{+,\kappa,\Delta}_K$ we have
$$\dim\llangle M_1,M_2,{}^\dag M_3\rrangle_K\leqslant d(M_1,M_2,{}^\dag M_3).$$

To prove the equality, for each $a$ we fix an exact sequence 
$0\to M_a^{(2)}\to M_a^{(1)}\to M_a^{(3)}\to 0$ of $\Delta$-filtered modules such that 
$M_a^{(1)}$ is a generalized Weyl module and $M_a^{(3)}=M_a$. 
Clearly, such exact sequences always exist. Then, we have
$$\gathered
\dim\llangle M_1^{(1)},M_2^{(1)},{}^\dag M_3^{(1)}\rrangle_K=
d(M_1^{(1)},M_2^{(1)},{}^\dag M_3^{(1)}),\\
\dim\llangle M_1^{(\alpha)},M_2^{(\beta)},{}^\dag M_3^{(\gamma)}\rrangle_K\leqslant 
d(M_1^{(\alpha)},M_2^{(\beta)},{}^\dag M_3^{(\gamma)}),\quad\forall\alpha,\beta,\gamma.
\endgathered$$
Thus the equality follows from \eqref{ineq}.
\end{proof}

\vspace{2mm}

Next, we prove part $(b)$. Assume that $M_1\in\bfO^{+,\kappa}_K$ and 
$M_2\in\bfO^{\nu,\kappa}_K$ are tilting. We must check that $M_1\dot\otimes_K M_2$ is still tilting.
For $N\in\bfO^{\nu,\kappa,\Delta}_K$ we must prove that 
$\Ext^1_{\bfO^{\nu,\kappa}_K}(N,M_1\dot\otimes_K M_2)=0$.
Since the functors $\bullet\,\dot\otimes_K M_2$ and $\bullet\,\dot\otimes_K DM_2$ are exact 
and biadjoint by Lemma \ref{lem:TP2a}, we have
$\Ext^1_{\bfO^{\nu,\kappa}_K}(N,M_1\dot\otimes_K M_2)=\Ext^1_{\bfO^{\nu,\kappa}_K}(N\dot\otimes_K DM_2,M_1)$.
Since  $DM_2$ is $\Delta$-filtered, 
part $(a)$ yields
$\Ext^1_{\bfO^{\nu,\kappa}_K}(N\dot\otimes_K DM_2,M_1)=0$.

Finally, we prove parts $(c)$, $(d)$.
The functor $\dot\otimes_K$ is right biexact.
The same argument as above using Proposition \ref{prop:TP1} and Lemma \ref{lem:TP3}
implies that it is biexact on $\Delta$-filtered modules. More precisely, for each 
$M\in\bfO^{\nu,\kappa,\Delta}_K$ the functor $\Hom_{\bfg_K}(\bullet\,\dot\otimes\,\bullet,\scrD M)$ on
$\bfO^{+,\kappa,\Delta}_K\times\bfO^{\nu,\kappa,\Delta}_K$ is exact and $\bullet\,\dot\otimes_K\bullet$ takes values
in $\bfO^{\nu,\kappa,\Delta}_K$. Thus, if $0\to N_1\to N_2\to N_3\to 0$ is exact in 
$\bfO^{+,\kappa,\Delta}_K$ and if $N\in\bfO^{\nu,\kappa,\Delta}_K,$ then we have an exact sequence
$N_1\dot\otimes_K N\to N_2\dot\otimes_K N\to N_3\dot\otimes_K N\to 0$
and $(N_2\dot\otimes N:M_3)=(N_1\dot\otimes_K N:M_3)+(N_3\dot\otimes_K N:M_3)$ 
if $M_3$ is a parabolic Verma module. Thus, also
we have an exact sequence
$0\to N_1\dot\otimes_K N\to N_2\dot\otimes_K N\to N_3\dot\otimes_K N\to 0.$

Since it is exact, the tensor product $\dot\otimes_K$
factors to the Grothendieck groups $[\bfO^{+,\kappa,\Delta}_K]$ and $[\bfO^{\nu,\kappa,\Delta}_K]$.
The exact functor $\Indc(\bullet)$ gives $\bbC$-linear isomorphisms
$[\scrO^{+,\Delta}_K]\to[\bfO^{+,\kappa,\Delta}_K]$ and $[\scrO^{\nu,\Delta}_K]\to[\bfO^{\nu,\kappa,\Delta}_K],$
because the parabolic Verma modules form bases of the Grothendieck groups
of $\Delta$-filtered objects. The compatibility with the monoidal structures follows from
Proposition \ref{prop:TP1} and Lemma \ref{lem:TP3}.
\end{proof}

\vspace{3mm}

\subsection{Proof of Propositions \ref{prop:TPR1}, \ref{prop:adjoints}, \ref{prop:monoidalR}}
\label{sec:TPR1}
Let $A$ and the $\bfg_{R,\kappa}$-modules $M_a$, $a\in A$, be as in Proposition \ref{prop:TPR1}.
We must define the $\bfg_{R,\kappa}$-module $\dot\bigotimes_{R,a}M_a$.
The construction
is essentially the same as in \cite[sec.~29]{KL}.
However, our setting differs from that of \cite{KL} from several aspects
\begin{itemize}
\item the category $\bfO_R^{\nu,\kappa}$ is defined over a regular local ring $R$ of dimension $\leqslant 2$,
\item the modules $M_a$ do not all belong to the category $\bfO^{+,\kappa}_R$, 
\item the modules $M_a$ may not have integral weights,
\item we work with $\bfg_R$-modules, rather than $\bfg'_R$-modules.
\end{itemize}

The last point is easy to deal with :
we'll switch from $\bfg_R$-modules to $\bfg'_R$-modules as in Remark \ref{rk:soergel} without 
mentioning it explicitly.

First, assume that $R=\bbC$ and $M_a\in\scrS_\kappa$ for each $a$.
Then, the smooth $\bfg'_\kappa$-module $\dot\Otimes_{a}M_a$ is defined in \cite{KL}. 
If $M_a\in\bfO^{+,\kappa}$ for each $a$ then
$\dot\Otimes_{a}M_a\in\bfO^{+,\kappa}$ by \cite{KL}. 
If $M_a\in\bfO^{+,\kappa}$ for $a\neq 0$ and $M_0\in\bfO^{\nu,\kappa}$ then 
$\dot\Otimes_{a}M_a\in\bfO^{\nu,\kappa}$ by
\cite{Y1}, see \S \ref{sec:TP2}. 

Next, if $R=K$ is a field  and $M_a\in\scrS_\kappa$ for each $a$, then one can define
$\dot\Otimes_{K,a}M_a$ as over $\bbC$, see Remark \ref{rem:TPK} below.

Now, let $R$ be any commutative noetherian $\bbC$-algebra with 1. 
Set $A=[-m,n+1]$ and consider the subsets
$\spadesuit=[-m,n]$ and $\heartsuit=\{n+1\}$.
To simplify the notation we'll also write $\heartsuit$ for $n+1$. 

Recall that $\gamma_a$ is a chart on $C$ centered at $x_a$ for each $a\in A,$
and that $D_R=R[C\setminus\{x_a; a\in A\}].$
Let $\widehat\Gamma_{R}$ be the central extension of 
$\Gamma_R=\frakg\otimes D_{R}$ by $R$ associated with the cocycle
$(\xi_1\otimes f_1,\xi_2\otimes f_2)\mapsto\Res_{\gamma_\heartsuit=0}(f_2df_1).$

Set $\kappa=c+N$ and $\kappa'=-c+N$. The quotient by the ideal
$(\bfone-c)$ yields an algebra homomorphism
$U(\widehat\Gamma_{R})\to\Gamma_{R,\kappa}.$

\vspace{2mm}

\begin{lemma}\label{lem:map2}
(a) There is an $R$-algebra homomorphism
$\Gamma_{R,\kappa'}\to\bfg^{\spadesuit}_{R,\kappa}$ 
such that $\xi\otimes f\mapsto\xi\otimes {}^\spadesuit\! f.$

(b) There is an $R$-algebra homomorphism
$\bfg'_{R,\kappa'}\to\Gamma_{R,\kappa'}$ such that
$\xi\otimes f(t)\mapsto\xi\otimes f(\gamma_\heartsuit)$, and
an $R$-algebra homomorphism
$\Gamma_{R,\kappa'}\to\bfg^{\heartsuit}_{R,\kappa'}$ such that
$\xi\otimes f\mapsto\xi\otimes{}^\heartsuit\!f.$

(c) Composing the maps in $(b)$ we get an $R$-algebra embedding
$\bfg'_{R,\kappa'}\to\bfg^{\heartsuit}_{R,\kappa'}$ such that
$\xi\otimes f(t)\mapsto\xi\otimes f(t_\heartsuit).$
\end{lemma}

\vspace{0.5mm}

\begin{proof}
Part $(a)$ is standard.
To prove $(b),$ observe that
the chart $\gamma_\heartsuit$ can be regarded as an element in the subalgebra
$\{f\in D_R\,;\,f(x_\heartsuit)=0\}$.
Thus, we have an $R$-algebra homomorphism $R[t, t^{-1}]\to D_R$ such that 
$f(t)\mapsto f(\gamma_\heartsuit)$ and an $R$-algebra homomorphism
$D_R\to R((t_\heartsuit))$ such that $f\mapsto {}^\heartsuit\!f$.
\end{proof}

\vspace{2mm}

Now, for each $a\in\spadesuit$ we fix a smooth module $M_a\in\scrS_{R,\kappa}$ which
is a weight $\frakt_R$-module.
Set
$W_R=\Otimes_{R,a\in\spadesuit}M_a$.
Since the $M_a$'s are smooth, the $R$-module $W_R$
has a natural structure of $\bfg^{\spadesuit}_{R,\kappa}$-module. 
We view $W_R$ as a $\Gamma_{R,\kappa'}$-module via the map in Lemma \ref{lem:map2}$(a)$. 
Note that $W_R$ is a weight $\frakt_R$-module.

For $d\geqslant 1$ let $G_{R,d}$
be the $R$-submodule of
$\Gamma_{R,\kappa'}$
spanned by the products of $d$ elements in
$\frakg\otimes D_{R}^1$ with $D_{R}^1=\{f\in D_R\,;\,f(x_\heartsuit)=0\}.$
Note that $G_{R,d}$ is a weight $\frakt_R$-module for the adjoint action.
We have the following natural decreasing filtration of weight $\frakt_R$-modules
$W_R\supset G_{R,1} W_R\supset G_{R,2} W_R\supset \cdots.$
Consider the weight $\frakt_R$-module $W_{R,d}$ given by
$W_{R,d}=W_R/G_{R,d} W_R$. Let
$W_{R,d}=\bigoplus_{\lambda\in P_R} W_{R,d,\lambda}$
be its decomposition into the sum of its weight $R$-submodules.

The modules $W_{R,d}$ with $d\geqslant 1$ form a projective system.
The limit $\widehat W_R=\pro\, W_{R,d}$
in the category of weight $\frakt_R$-modules decomposes as the direct sum of weight $R$-submodules
$\widehat W_R=\bigoplus_{\lambda\in P_R}\widehat W_{R,\lambda},$
where $\widehat W_{R,\lambda}$ is the projective limit of $R$-modules $\pro\, W_{R,d,\lambda}.$

For each $\lambda\in P_R$ and each $d\geqslant 1$, we set $Z_{R,d,\lambda}=(W_{R,d,\lambda})^*$.
The $R$-modules $Z_{R,d,\lambda}$ with $d\geqslant 1$ form an inductive system of submodules
$Z_{R,1,\lambda}\subset Z_{R,2,\lambda}\subset\cdots.$
Consider the weight $\frakt_R$-module $Z_{R,\infty}$ given by
$Z_{R,\infty}=\bigoplus_{\lambda\in P_R} Z_{R,\infty,\lambda},$
where
$Z_{R,\infty,\lambda}=\ind\, Z_{R,d,\lambda}$.

From now on, we assume  that $R$ is a regular ring of dimension $\leqslant 2$.
Assume in addition that $M_0\in\bfO^{\nu,\kappa,f}_R$ and $M_a\in\bfO^{+,\kappa,f}_R$ for $a\neq 0$. 

\vspace{2mm}

\begin{lemma}\label{lem:fg}
(a) The $R$-module $W_{R,d,\lambda}$ is finitely generated.

(b) The $R$-module $Z_{R,d,\lambda}$ is finitely generated and projective.
\end{lemma}

\vspace{.5mm}

\begin{proof} 
Since $M_a$ belongs to $\bfO_R^\kappa,$ there is an integer $d_a\geqslant 0$ such that 
$M_a$ is generated by the $R$-submodule $M_a(d_a)$ as a
$\bfg_{R,\kappa}$-module.
Then, the same proof as in \cite[prop.~7.4]{KL} implies that  
\begin{equation}\label{fg}
W_R=X_{R,d}W_R^1+G_{R,d}W_R,\qquad
W_R^1=\Otimes_{R,a}M_a(d_a),
\end{equation}
where $X_{R,d}$ is the $R$-submodule of $\Gamma_{R,\kappa'}$ 
spanned by the product of $<\!d$ elements in
$\frakg\otimes \gamma_\heartsuit$. The right hand side of the first equality in 
\eqref{fg} is defined using the $\Gamma_{R,\kappa'}$-module
structure on $W_R$.

Now, since $M_a\in\bfO_R^\kappa$ and $R$ is noetherian, the weight $\frakt_R$-submodules of the $\frakt_R$-submodule $M_a(d_a)\subset M_a$ are
finitely generated over $R$. Indeed, the weight $\bft_R$-submodules of $M_a$ are finitely generated because $M_a\in\bfO_R^\kappa$, and each 
weight $\frakt_R$-submodule of $M_a(d_a)$ is contained in the sum of a finite number of weight $\bft_R$-submodules
of $M_a$ (because $M_a$ is flat over $R$ and the result is well-known over the fraction field $K$ of $R$).
Therefore, part $(a)$ of the lemma is an easy consequence of \eqref{fg}.

Since $R$ is a regular ring of dimension $\leqslant 2$,
any finitely generated reflexive $R$-module is projective.
Since it is the dual of a finitely generated $R$-module, the $R$-module $Z_{R,d,\lambda}$ is
finitely generated and reflexive.
Hence it is projective as an $R$-module for each $d$, $\lambda$.
\end{proof}

\vspace{2mm}

Under the previous hypothesis, we can now prove the following.

\vspace{2mm}

\begin{lemma} \label{lem:A.7}
(a) There is a natural representation of $\bfg'_{R,\kappa'}$ on $\widehat W_{R}$.

(b) There is a natural smooth representation of $\bfg'_{R,\kappa}$ on $Z_{R,\infty}$.
\end{lemma}

\vspace{.5mm}

\begin{proof}
The proof is adapted from \cite{KL}. We will be sketchy.
Recall that $W_R$ is a $\Gamma_{R,\kappa'}$-module.
The $\Gamma_{R,\kappa'}$-action does not  induce a $\Gamma_{R,\kappa'}$-action on $\widehat W_R$ in a natural way.
However,
under the second map in Lemma \ref{lem:map2}$(b),$ it
descends  to a representation of $\bfg^{\heartsuit}_{R,\kappa'}$ on $\widehat W_R$ as in \cite[sec.~4.9]{KL}.
More precisely, given $f(t_\heartsuit)$ in $t_\heartsuit^{-n}R[[t_\heartsuit]]$ for some $n\in\bbN$,
we fix a sequence of elements
$g_1, g_2,\dots$ in $D_R$ such that ${}^\heartsuit \!g_d-f(t_\heartsuit)\in t^d_\heartsuit R[[t_\heartsuit]]$ for each $d$,
and we define the action of $\xi\otimes f(t_\heartsuit)$ on the element $(w_d)\in \widehat W_R$, with
$w_d\in W_{R,d}$ and $d\geqslant 1$, by setting
$\xi\otimes f(t_\heartsuit)\cdot (w_d)=(\xi\otimes g_d\cdot w_{n+d}).$

Twisting this representation by the map $\bfg'_{R,\kappa'}\to\bfg^{\heartsuit}_{R,\kappa'}$
in Lemma \ref{lem:map2}$(c)$, we get a representation of
$\bfg'_{R,\kappa'}$ on $\widehat W_R$.
Taking its dual, we get a representation of $\bfg'_{R,\kappa}$ on $Z_{R,\infty}$. 
See \cite[sec.~6.3]{KL} for details.

The $R$-module $Z_{R,\infty}$ is flat, because the direct summand $Z_{R,\infty,\lambda}$ is the limit of 
the inductive system of flat submodules $(Z_{R,d,\lambda})$. 
To prove that it
is smooth, it is enough to check that $Z_{R,\infty}=Z_{R,\infty}(\infty)$. This is obvious, because we have
$Z_{R,d}\subset Z_{R,\infty}(d)$, where $Z_{R,d}=\bigoplus_\lambda Z_{R,d,\lambda}$.
\end{proof}

\vspace{2mm}

Now, we consider the behavior of $Z_{R,\infty}$ under flat base changes.

\vspace{2mm}

\begin{lemma}\label{lem:A7}
Let $S$ be a commutative noetherian $R$-algebra with 1 which is flat as an $R$-module.
Then, we have a canonical $\bfg'_{S,\kappa}$-module isomorphism $SZ_{R,\infty}=Z_{S,\infty}$.
\end{lemma}

\vspace{.5mm}

\begin{proof} 
Since taking tensor products is right exact,
we have a canonical $S$-module isomorphism $SW_{R,d,\lambda}=W_{S,d,\lambda}.$
Since $S$ is flat over $R$, for any $R$-modules $X,Y$
such that $X$ is finitely presented over $R$, the canonical homomorphism
$S\Hom_R(X,Y)\to\Hom_S(SX,SY)$ is an isomorphism. 
By Lemma \ref{lem:fg}, the $R$-module $W_{R,d,\lambda}$ is finitely generated.
Therefore, since direct limits commute with tensor products, we have
$$
SZ_{R,\infty}
=\bigoplus_\lambda\ind\, S\Hom_R(W_{R,d,\lambda},R)
=\bigoplus_\lambda\ind\, \Hom_S(W_{S,d,\lambda},S)
=Z_{S,\infty}.
$$
\end{proof}

\vspace{2mm}

\begin{rk} \label{rem:TPK}
The tensor product $\dot\bigotimes_{R,a}M_a$ is constructed for $R=\bbC$ and for any smooth modules
$M_a\in\scrS_\kappa$ in \cite[sec.~4]{KL}. It is called there the \emph{first construction}. 
If $R=K$ is a field (containing $\bbC$), one can define it as in loc. cit., 
because everything works equally well. 
We can also allow the modules $M_a$ to have non integral weights as in \cite{Y1}.

The smooth $\bfg'_\kappa$-module $DZ_{K,\infty}$ is precisely the one given by the 
\emph{third construction} in \cite[sec.~6]{KL}. If $M_a\in\bfO^{+,\kappa}_K$ for all $a,$ then there is an isomorphism 
of $\bfg'_\kappa$-modules
$DZ_{K,\infty}=\dot\Otimes_{K,a}M_a$
by \cite[thm.~7.9]{KL}. If $M_0\in\bfO^{\nu,\kappa}_K$ 
and $M_a\in\bfO^{+,\kappa}_K$ for $a\neq 0,$ then these modules are again isomorphic by \cite[prop.~5.8]{Y1}.

\end{rk}

\vspace{2mm}

We can now prove the following

\vspace{2mm}

\begin{lemma}\label{lem:A8} Assume that $R=K$ is a field.

(a) The $\bfg'_{K,\kappa}$-module $Z_{K,\infty}$ belongs to $\bfO^{\nu,\kappa}_K$.

(b) The Sugawara operator $\frakL_0$ preserves the finite dimensional
$K$-subspace $Z_{K,d,\lambda}$ of $Z_{K,\infty}$ for each $d,\lambda$.

(c) The characteristic polynomial of $\frakL_0$ on  $Z_{K,d,\lambda}$ is a product of linear 
factors with coefficients in $R$.
\end{lemma}

\vspace{.5mm}

\begin{proof} 
By Remark \ref{rem:TPK}, the $\bfg'_{K,\kappa}$-module $DZ_{K,\infty}$ is isomorphic to the
tensor product $\dot\Otimes_{K,a}M_a$ considered there.  Thus, part $(a)$ follows from \cite{Y1}.

Part $(b)$ is a standard computation using the relation $[\frakL_0,\xi^{(r)}]=-r\xi^{(r)}$ for each
$\xi\in\frakg$ and $r\in\bbZ$. 

Since $Z_{K,\infty}\in\bfO^{\nu,\kappa}_K$, part $(c)$ follows from 
elementary properties of the action of the Sugawara operator on objects of $\bfO^{\nu,\kappa}_K$.
\end{proof}

\vspace{2mm}

Now, we come back to the case where $R$ is a regular local ring of dimension $\leqslant 2$.

\vspace{2mm}

\begin{lemma} \label{lem:A.10}
There is a natural smooth representation of $\bfg_{R,\kappa}$ on $Z_{R,\infty}$.
\end{lemma}

\vspace{.5mm}

\begin{proof}
Since the $\bfg'_{R,\kappa}$-module $Z_{R,\infty}$ is smooth, it is equipped with a canonical action of the 
Sugawara operator $\frakL_0$. For each $r\in R$ we set
$${}^r\!Z_{R,\infty}=\{v\in Z_{R,\infty}\,;\,(\frakL_0-r)^nv=0,\, n\gg 0\}.$$

Replacing $R$ by $K$ everywhere in the construction above, we get
the $\bfg'_{K,\kappa}$-module $Z_{K,\infty}$. 
Since the $\bfg'_{R,\kappa}$-module $Z_{R,\infty}$ is smooth, it is flat over $R$.
Thus, we have an obvious inclusion 
$Z_{R,\infty}\subset KZ_{R,\infty}=Z_{K,\infty}.$
Hence, by Lemma \ref{lem:A8}, we have a direct sum decomposition $Z_{R,\infty}=\bigoplus_{r}{}^r\!Z_{R,\infty}$.

Therefore, we can consider the $R$-linear operator $\partial$ on $Z_{R,\infty}$ which acts by multiplication with
$(-r)$ on the $R$-submodule ${}^r\!Z_{R,\infty}$.
It equips $Z_{R,\infty}$ with the structure of a smooth
$\bfg_{R,\kappa}$-module.
\end{proof}

\vspace{2mm}

\begin{rk} Under the identification of $Z_{R,\infty,\lambda}$ with the topological dual of the $R$-module $\widehat W_{R,\lambda}$, 
the $\frakL_0$-action on $Z_{R,\infty}$ is identified with 
the transpose of the operator $\widehat\frakL_0$ on $\widehat W_R$ 
defined in the following way.
Via the first map in Lemma \ref{lem:map2}$(b)$, we can view the Sugawara operator
$\frakL_0$, which lies in a completion of $\bfg'_{R,\kappa'}$,
as an element in a completion of $\widehat\Gamma_{R,\kappa'}$ such that
\begin{equation}\label{der}
[\frakL_0,\xi\otimes f]=\xi\otimes\delta(f),\qquad
\forall\xi\otimes f\in\Gamma_R.
\end{equation}
Here $\delta$ is the unique $R$-linear derivation of $D_R$ such that 
$\delta(\gamma_\heartsuit)=\gamma_\heartsuit$.
This implies that $\frakL_0$ defines a $R$-linear endomorphism
$(\frakL_{0,d})$ of the projective system $(W_{R,d})$ as in \cite[lem.~26.3]{KL}.
Thus, the projective limit $\widehat\frakL_0=\pro\,\frakL_{0,d}$ is
an operator on $\widehat W_R$.
\end{rk}

\vspace{2mm}

\begin{df} Assume that $R$ is a regular local ring of dimension $\leqslant 2$.
Then, we define the $\bfg_{R,\kappa}$-module
$\dot\Otimes_{R,a}M_a$ to be $DZ_{R,\infty}$. 
\end{df}

\vspace{.5mm}

Note that $\dot\Otimes_{R,a}M_a$ is smooth by Lemma \ref{lem:A.10} and by definition of the functor $D$.
We must prove that it satisfies the properties in Proposition \ref{prop:TPR1}.

First, note the following.

\vspace{2mm}

\begin{lemma}\label{lem:A11}
Let $S$ be a commutative noetherian $R$-algebra with 1 which is regular of dimension $\leqslant 2$ and which is flat as an $R$-module.
We have canonical $\bfg_{S,\kappa}$-module isomorphism
$S(\dot\Otimes_{R,a}M_a)=\dot\Otimes_{S,a}SM_a$.
\end{lemma}

\vspace{.5mm}

\begin{proof}
By Lemma \ref{lem:A7} we have $SZ_{R,\infty}=Z_{S,\infty}$.
Thus, the lemma follows from the proof of Lemma \ref{lem:duality}, which insures that $D$ commutes with base change.
\end{proof}

\vspace{2mm}

Next, we prove the following.

\vspace{2mm}

\begin{lemma}
We have $\dot\Otimes_{R,a}M_a\in\bfO^{\nu,\kappa,f}_R$ and
$\dot\Otimes_{R,a}$ is a right exact functor on $\bfO^{\nu,\kappa,f}_R$ and $\bfO^{+,\kappa,f}_R$.
\end{lemma}

\vspace{.5mm}

\begin{proof}
The $\bfg'_{R,\kappa'}$-action on $W_R$ yields a representation of
$\bfg'_{R,\kappa}$ on ${}^\sharp W_R$.
Consider the $R$-submodule $W_R^1\subset W_R$ introduced in \eqref{fg}. 
We claim that ${}^\sharp W^1_R$ is a
$\bfg'_{R,+}$-submodule of ${}^\sharp W_R$.
Indeed, the element $\xi\otimes f(t)$ in $\bfg'_{R,\kappa}$ acts on ${}^\sharp W_R$ by
the operator
$\sum_{a\in\spadesuit} \xi\otimes{}^a\!f(-1/\gamma_\heartsuit).$
Further, for each $a\in\spadesuit$ the function $1/\gamma_\heartsuit$ is regular at $x_a$ and, thus,
since the system of charts is defined over $\bbC$, the expansion
${}^a(1/\gamma_\heartsuit)$ is a well-defined Laurent formal series in $\bbC[[t_a]]$.
Therefore we have ${}^a\!f(-1/\gamma_\heartsuit)\in R[[t_a]]$ for each $f\in R[t]$.

We deduce that there is a $\bfg'_{R,\kappa}$-homomorphism
\begin{equation}\label{map3}\scrI\! nd({}^\sharp W^1_R)\to{}^\sharp W_R.\end{equation}

Next, recall that the first map in Lemma \ref{lem:map2}$(b)$ yields a
$\bfg'_{R,\kappa'}$-action on $W_R$ and that $\bfg'_{R,\kappa'}$ acts on $\widehat W_R$
by Lemma \ref{lem:A.7}. By definition of the actions, the canonical $R$-module homomorphism
$W_R\to\widehat W_R$ is a $\bfg'_{R,\kappa'}$-module homomorphism.
Taking the dual of $W_R$ in the category of weight $\frakt_R$-modules, 
we get the $\bfg'_{R,\kappa}$-module $Z_R$ given by
\begin{equation}\label{Z}
Z_{R}=\bigoplus_{\lambda\in P_R} Z_{R,\lambda},\qquad
Z_{R,\lambda}=(W_{R,\lambda})^*.
\end{equation}

Twisting \eqref{map3} by $\sharp$ and taking its transpose, we get a $\bfg'_{R,\kappa}$-homomorphism
$Z_{R}\to {}^\sharp\!\scrI\!nd({}^\sharp W^1_R)^*.$
Since $Z_{R,\infty}\subset Z_R(\infty)$, 
this map restricts to a $\bfg'_{R,\kappa}$-homomorphism
\begin{equation}\label{inclusion}
{}^\dag\! Z_{R,\infty}\to \scrD\scrI\!nd({}^\sharp W^1_R).
\end{equation}
Using \eqref{fg} it is easy to see that the map \eqref{inclusion} is an inclusion.

\vspace{2mm}

\begin{claim}
Let $N\in\bfO^{\nu,\kappa,f}_R$ and let $M\subset N$ be a submodule which is flat as an $R$-module.
Then, we have $M\in\bfO^{\nu,\kappa,f}_R$.
\end{claim}

\vspace{.5mm}

To prove the claim, observe first that, since $N$ is a weight $\bft_R$-module 
with finitely generated weight subspaces over $R$, 
so is also $M$. Thus, since $M$ is flat and since any flat finitely generated $R$-module is free
(because $R$ is a noetherian local ring),  the $R$-module $M$ is indeed free. It is easy to check that $M$ 
satisfies the other axioms of the category $\bfO^{\nu,\kappa,f}_R$, except the fact that it is finitely generated.
For this last property, recall that for each $\beta$ the category $\lub\bfO^{\nu,\kappa}_R$ 
is a highest weight category over $R$. 
Since it is equivalent to the category 
of finitely generated modules over a finitely generated projective $R$-algebra, it is noetherian.
Therefore $M$ is finitely generated. 
The claim is proved.

Now, recall that ${}^\dag\! Z_{R,\infty}$ is flat over $R$ and that
$\scrI\!nd({}^\sharp W^1_R)$ is a generalized Weyl module of  
$\bfO^{\nu,\kappa,f}_R$.
Thus, the claim implies that ${}^\dag\! Z_{R,\infty}\in\bfO^{\nu,\kappa,f}_R$.
Hence $DZ_{R,\infty}\in\bfO^{\nu,\kappa,f}_R$. This proves the first part of the lemma.

To prove the second part, it is enough to observe that the functor
$(M_a)\mapsto {}^\dag\!Z_{R,\infty}$ is left exact, 
because it is the composition of a tensor product over $R$ of free
$R$-modules, of a dual over $R$ of free $R$-modules, and of the functor of taking smooth vectors (which is
left exact), and that $\scrD$ is an exact endofunctor of $\bfO^{\nu,\kappa,f}_R$.
\end{proof}

\vspace{2mm}

Now, we consider the functor represented by the module $\dot\Otimes_{R,a}M_a$. 
Recall that $A=[-m,n]$. The lemma below
gives a functorial isomorphism for $M\in\bfO^{\nu,\kappa,f}_R$
\begin{equation}\label{represent}
\Hom_{\bfg_R}(\dot\Otimes_{R,a}M_a,M)=\llangle M_{-m},\dots,M_{n},D M\rrangle_R^*.
\end{equation}

\vspace{2mm}

\begin{lemma} \label{lem:A14}
For $M,N\in\bfO^{\nu,\kappa,f}_R$ we have functorial $R$-module isomorphisms
\begin{equation*}
\gathered
\Hom_{\bfg_R}(N, M)=\llangle N, DM\rrangle_R^*,\quad
\llangle \dot\Otimes_{R,a}M_a, DM\rrangle_R^*
=\llangle M_{-m},\dots,M_{n},D M\rrangle_R^*.
\endgathered
\end{equation*}
\end{lemma}

\vspace{.5mm}

\begin{proof} There is a natural $R$-module inclusion
$\Hom_{\bfg_R}(N, M)\to\Hom_{\bfg_R}(N\otimes_R DM,R)$,
because $M$, $N$ are weight $\bft_R$-modules whose weight subspaces are free $R$-modules
of finite type. We must prove that this inclusion is indeed an isomorphism. 
The proof is the same as in  \cite[prop.~A.2.6]{VV}, see also
\cite[prop.~2.31]{KL}. 

Next, by definition of coinvariants, we also
have a canonical $R$-module isomorphism
$\Hom_{\bfg_R}(N\otimes_R DM,R)\to\llangle N,DM\rrangle_R^*.$
This proves the first isomorphism in the lemma. 

Now, we concentrate on the second one.
Consider the $\Gamma_{R,\kappa}$-module
$Z_{R}$.
By construction, we have
$\Hom_R(DM,Z_R)=\Hom_R(W_R\otimes_RDM,R).$
Thus, we have also
$\Hom_{\Gamma_{R,\kappa}}(DM,Z_R)=\Hom_{\Gamma_R}(W_R\otimes_RDM,R).$
Thus, since $DM$ is smooth and $Z_{R,\infty}=Z_R(\infty)$,
the canonical inclusion
$\Hom_{\Gamma_{R,\kappa}}(DM,Z_{R,\infty})\subset\Hom_{\Gamma_{R,\kappa}}(DM,Z_R)$
is indeed an isomorphism. So we get an isomorphism
$\Hom_{\bfg_{R,\kappa}}(DM,Z_{R,\infty})=\llangle M_{-m},\dots,M_{n},D M\rrangle_R^*.$
Finally, since ${}^\dag\!Z_{R,\infty}$ belongs to $\bfO^{\nu,\kappa,f}_R,$ we have
$$\Hom_{\bfg_{R,\kappa}}(\dot\Otimes_{R,a}M_a,M)=
\Hom_{\bfg_{R,\kappa}}(DZ_{R,\infty}, M)=\Hom_{\bfg_{R,\kappa}}(DM,Z_{R,\infty}).$$
\end{proof}

\vspace{2mm}

Next, we consider the behavior of the tensor product $\dot\otimes_R$ on $\Delta$-filtered modules.
Assume that $M_0\in\bfO^{\nu,\kappa,\Delta}_R$ and $M_a\in\bfO^{+,\kappa,\Delta}_R$ for 
$a\neq 0$. 
First, note the following.

\vspace{2mm}

\begin{lemma}\label{lem:A10}
For each $M\in\bfO^{\nu,\kappa,\Delta}_R$ the $R$-module 
$\llangle M_{-m},\dots,M_{n},{}^\dag\! M\rrangle_R$ is free of finite type.
\end{lemma}

\vspace{.5mm}

\begin{proof}
Since this $R$-module is finitely generated by Lemma \ref{lem:TP0}, it is enough to check that its rank
is the same at the special point and at the generic point of $\Spec(R)$. 
By Lemma \ref{lem:TP0} we must check that
$$
\dim_\Bbbk\,\llangle \Bbbk M_{-m},\dots,\Bbbk M_{n},{}^\dag \Bbbk M\rrangle_\Bbbk=
\dim_K\,\llangle K M_{-m},\dots,K M_{n},{}^\dag\! K M\rrangle_K.
$$
For each $M\in\bfO^{\nu,\kappa,\Delta}_R$, $N\in\Delta(\bfO^{\nu,\kappa}_R)$, we have $(KM:KN)=(\Bbbk M:\Bbbk N).$
Therefore, the claim follows from Lemma \ref{lem:TP3}.
\end{proof}

\vspace{2mm}

Now, we can prove the following.

\vspace{2mm}

\begin{lemma}\label{lem:A16}
We have $\dot\Otimes_{R,a}M_a\in\bfO^{\nu,\kappa,\Delta}_R$.
\end{lemma}

\vspace{.5mm}

\begin{proof} Taking $\beta$ large enough we can assume that all modules belong to
the category $\lub\bfO^{\nu,\kappa,f}_R$. 
Since $\lub\bfO^{\nu,\kappa}_R$ is a highest weight category
over $R$, to prove that $\dot\Otimes_{R,a}M_a$ has a $\Delta$-filtration, 
it suffices to check that 
$\Ext^1_{\lub\bfO^{\nu,\kappa}_R}(\dot\Otimes_{R,a}M_a,M)=0$
for each $M\in\nabla(\lub\bfO^{\nu,\kappa}_R)$,
see \cite[lem.~4.21]{R1}.
Since the category $\lub\bfO^{\nu,\kappa}_R$ is preserved by taking extensions in 
$\bfO^{\nu,\kappa}_R$, we must check that
$\Ext^1_{\bfO^{\nu,\kappa}_R}(\dot\Otimes_{R,a}M_a,\scrD M)=0$
for each $M\in\Delta(\bfO^{\nu,\kappa}_R).$
To simplify the notation, we assume that $[-m,n]=[1,2]$.
By \eqref{represent}, it is enough to check that, given an exact sequence
$$0\to Q\to P\to M\to 0$$
with $P$ projective, the following left exact sequence of $R$-modules is indeed exact
\begin{equation*}\label{exactsequence2}
\gathered
0\to \llangle M_1,M_2,{}^\dag\! M\rrangle_R^*\to
\llangle M_1,M_2,{}^\dag\! P\rrangle_R^*\to \llangle M_1,M_2,{}^\dag\! Q\rrangle_R^*\to 0.
\endgathered
\end{equation*}
Note that $M_1,M_2,M_3,Q,P,M$ are $\Delta$-filtered. To prove the claim we may consider the right exact 
sequence of free $R$-modules of finite type
\begin{equation*}\label{exactsequence3}
\gathered
0\to \llangle M_1,M_2,{}^\dag\! Q\rrangle_R\to
\llangle M_1,M_2,{}^\dag P\rrangle_R\to \llangle M_1,M_2,{}^\dag\! M\rrangle_R\to 0.
\endgathered
\end{equation*}
We must prove that it is exact. To do so, it is enough to prove that it is exact after specialization
at the special point and at the generic point of $\Spec(R)$. Now, the sequences
\begin{equation*}
\gathered
0\to \llangle \Bbbk M_1,\Bbbk M_2,{}^\dag\Bbbk Q\rrangle_\Bbbk\to
\llangle \Bbbk M_1,\Bbbk M_2,{}^\dag\Bbbk P\rrangle_\Bbbk\to 
\llangle \Bbbk M_1,\Bbbk M_2,{}^\dag\Bbbk M\rrangle_\Bbbk\to 0,\\
0\to \llangle K M_1,K M_2,{}^\dag\!K Q\rrangle_K\to
\llangle K M_1,K M_2,{}^\dag\!K P\rrangle_K\to 
\llangle K M_1,K M_2,{}^\dag\!K M\rrangle_K\to 0
\endgathered
\end{equation*}
are both exact by Lemma \ref{lem:TP3}. Thus, the lemma follows from Lemma \ref{lem:TP0}.
\end{proof}

\vspace{2mm}

Next, we define the associativity and the commutativity constraints for $\dot\otimes_R$.
From now on we assume that $R$ is an analytic algebra.

\vspace{2mm}

\begin{lemma}\label{lem:a}
Assume that $V_1,V_2\in\bfO^{+,\kappa,f}_R$ and $M\in\bfO^{\nu,\kappa,f}_R$.
Then, there are functorial isomorphisms
$\bfa_{V_1,M,V_2}:(V_1\dot\otimes_R M)\dot\otimes_R V_2\to V_1\dot\otimes_R(M\dot\otimes_RV_2)$
\end{lemma}

\vspace{.5mm}

\begin{proof}
We apply the same construction as in the case $R=\bbC$ in \cite[sec.~17,18]{KL}. 
We will be very brief.
We allow the system of charts $\gamma$ to vary in the set of $\bbC$-points
of an affine scheme $\scrV$.
Taking the coinvariants, we construct a bundle of $R$-modules of finite rank over
$\scrV$. This bundle is equipped with an integrable $R$-linear connection.
Since $R$ is an analytic algebra, it admits a flat section.
This section gives $R$-linear isomorphisms, see \cite[thm.~17.29]{KL},
\begin{equation*}
\gathered
\llangle V_1\dot\otimes_R M,V_2\dot\otimes_R DN\rrangle_R=
\llangle V_1, M,V_2, DN\rrangle_R,\\
\llangle DN\dot\otimes_RV_1,M\dot\otimes_R V_2\rrangle_R=
\llangle DN,V_1,M,V_2\rrangle_R
\endgathered
\end{equation*}
for each $M,N\in\bfO^{\nu,\kappa,f}_R$ and $V_1,V_2\in\bfO^{+,\kappa}_R$.

Using these isomorphisms and the invariance of coinvariants under cyclic permutation,
we get functorial isomorphisms
\cite[sec.~18.2]{KL} 
$$\xymatrix{
\llangle V_1\dot\otimes_R M,V_2\dot\otimes_R DN\rrangle_R\ar@{=}[r]\ar@{=}[d]&
\llangle (V_1\dot\otimes_R M)\otimes_RV_2,DN\rrangle_R,\\
\llangle DN\dot\otimes_R V_1,M\dot\otimes_R V_2\rrangle_R\ar@{=}[r]&
\llangle V_1\dot\otimes_R (M\otimes_RV_2),DN\rrangle_R.
}$$

Hence, from  \eqref{represent} we deduce a functorial isomorphism
\begin{equation*}
\gathered
\Hom_{\bfg_R}((V_1\dot\otimes_R M)\dot\otimes_RV_2,DN)=
\Hom_{\bfg_R}(V_1\dot\otimes_R (M\dot\otimes_RV_2),DN),
\endgathered
\end{equation*}
which yields a module isomorphism
$\bfa_{V_1,M,V_2}:(V_1\dot\otimes_R M)\dot\otimes_R V_2\to V_1\dot\otimes_R(M\dot\otimes_RV_2).$
\end{proof}

\vspace{2mm}

The isomorphisms $\bfa_{M,V_1,V_2}$ and $\bfa_{M,V_1,V_2}$ are constructed in a similar way. The details
are left to the reader. 

\vspace{2mm}

\begin{lemma} \label{lem:c}
Assume that $V\in\bfO^{+,\kappa,f}_R$ and $M\in\bfO^{\nu,\kappa,f}_R$.
Then, there is a functorial isomorphism
$\bfc_{V,M}:V\dot\otimes_R M\to M\dot\otimes_RV$.
 It represents the morphism of functors
$\llangle V,M,DN\rrangle_R\to\llangle M,V,DN\rrangle_R$ induced by the $R$-linear map
$V\otimes_RM\otimes_RDN\to M\otimes_RV\otimes_RDN$ such that
$x\otimes y\otimes z\mapsto\tau y\otimes\tau x\otimes \bar\tau z.$
Here, we set
$\tau=\exp(\sqrt{-1}\pi\frakL_0)\exp(\frak L_1)$ and $\bar\tau=\exp(-\sqrt{-1}\pi\frakL_0)\exp(\frak L_1).$
\end{lemma}

\vspace{.5mm}

\begin{proof}
The isomorphism $\bfc_{V,M}$ is defined as in \cite[sec.~14]{KL}.
More precisely, setting $A=[0,1]$, $M_0=V$ and $M_1=M$, we consider
the $\Gamma_{R,\kappa}$-module $Z_R$ in \eqref{Z}.
Switching $V$ and $M$ we define $Z'_R$ in a similar way.
Since the Sugawara operators $\frakL_0$, $\frakL_1$ act on $V$, $M$ 
and since $R$ is an analytic algebra, we can define the 
$R$-module isomorphism $P:Z'_R\to Z_R$ which is the transpose of the $R$-linear map
$V\otimes_RM\to M\otimes_RV$ such that $x\otimes y\mapsto\tau y\otimes\tau x.$
Now, recall that $Z'_{R,\infty}=Z'_R(\infty)$ and $Z_{R,\infty}=Z_R(\infty)$. One check as in loc.~cit.~that 
$P$ induces a $\bfg'_{R,\kappa}$-isomorphism $Z'_{R,\infty}\to Z_{R,\infty}.$ We define the isomorphism
$\bfc_{V,M}$ to be the map
$DZ_{R,\infty}\to DZ'_{R,\infty}$ which is the transpose of $P$. 
The second part of the lemma is proved as in \cite[sec.~14.6]{KL}.
\end{proof}

\vspace{2mm}

Next, we consider the behavior of the tensor product $\dot\otimes_R$ on tilting modules.
To do that, we'll restrict ourselves to
the endofunctors $e,$ $f$ of $\bfO^{\nu,\kappa,f}_R$ given by
$e=\bullet\,\dot\otimes_R\bfV^*_R$ and $f=\bullet\,\dot\otimes_R\bfV_R$.

\vspace{2mm}

\begin{lemma}\label{lem:A19}
The functors $e$, $f$ are exact on $\bfO^{\nu,\kappa,\Delta}_R$ and 
preserve $\bfO_R^{\nu,\kappa,\tilt}$.
\end{lemma}

\vspace{.5mm}

\begin{proof}
Let $S\subset R$ be the $\bbC$-subalgebra of $R$ generated by $\kappa.$
The modules $\bfV_R$, $\bfV^*_R$ are defined over $S$, i.e., we have 
$\bfV_R=R\bfV_S$ and $\bfV_R^*=R\bfV_S^*$ with $\bfV_S=\Indc_S(V_S)$,
$\bfV_S^*=\Indc_S(V_S^*)$. 

Now, the second claim is proved as Proposition \ref{prop:TP2bis}$(b)$. Since
$\bfV_R$, $\bfV_R^*$ are tilting by Lemma \ref{lem:V}, it is enough to
check that $e$, $f$ are biadjoint on $\bfO^{\nu,\kappa,\Delta}_R$ (hence exact) proving the first claim on the way.
To do so, since $R$ is a regular ring, we may assume that $R$ is flat over $S$.
Then, since $e$, $f$ commute with flat base change by Lemma \ref{lem:A11}, we may assume that $R=S$ is
a regular local ring of dimension 1. So, we are in the same setting as in \cite[sec.~31]{KL}.

Next, proving the lemma is equivalent to proving that  $\bfV_{R}$ and $\bfV_{R}^*$ are \emph{rigid}, see the  appendix to part IV
of \cite{KL} for details.
This is proved in the proof of \cite[prop.~31.3]{KL}, modulo a technical assumption
which is checked in \cite[lem.~31.6]{KL}.
\end{proof}

\vspace{2mm}

\begin{rk}\label{rem:A.24} The same argument implies that the functor $\bullet\,\dot\otimes_R M$ preserves $\bfO_R^{\nu,\kappa,\tilt}$
whenever $M=RM_S$ with $S$ as above and $M_S\in\bfO^{+,\kappa,f}_S$ is a tilting module such that $M_S$ and $DM_S$
satisfy the technical condition in \cite[prop.~31.3]{KL}. This technical assumption is probably not necessary, but this is enough for us.

\end{rk}

\vspace{2mm}

Finally, we consider the behavior of the tensor product $\dot\otimes_R$ under specialization of $R$ to the residue field $\Bbbk$.
The particular case where $R$ is a regular local ring of dimension $1$ has already been
considered in \cite[thm.~29.1]{KL}. 

\vspace{2mm}

\begin{lemma}\label{lem:A20}
For $M\in\bfO^{\nu,\kappa,\tilt}_R,$ we have functorial isomorphisms
$\Bbbk e(M)= e(\Bbbk M)$ and $\Bbbk f(M)= f(\Bbbk M).$
\end{lemma}

\vspace{.5mm}

\begin{proof} 
By \eqref{represent}, for $N\in\bfO^{\nu,\kappa,f}_R$ we have functorial isomorphisms 
$$\gathered
\Hom_{\bfg_R}(\dot\Otimes_{R,a}M_a,N)=\llangle M_{-m},\dots,M_n,DN\rrangle^*_R,\\
\Hom_{\bfg_\Bbbk}(\dot\Otimes_{\Bbbk,a}\Bbbk M_a,\Bbbk N)=
\llangle \Bbbk M_{-m},\dots,\Bbbk M_n,D\Bbbk N\rrangle^*_\Bbbk.
\endgathered$$
If $M_a$, $N$ are tilting, then 
$\llangle M_{-m},\dots,M_n,DN\rrangle_R$ is free of finite type over $R$ by Lemma \ref{lem:A10}.
Therefore, by Lemma \ref{lem:TP0} we have a functorial isomorphism
$\Bbbk\Hom_{\bfg_R}(\dot\Otimes_{R,a}M_a,N)=
\Hom_{\bfg_\Bbbk}(\dot\Otimes_{\Bbbk,a}\Bbbk M_a,\Bbbk N).$
So, for $M,N\in\bfO^{\nu,\kappa,\tilt}_R$ we have functorial isomorphisms
$\Bbbk\Hom_{\bfg_R}(e(M),N)=\Hom_{\bfg_\Bbbk}(e(\Bbbk M),\Bbbk N)$
and similar isomorphisms for $f$.

On the other hand, by Lemma \ref{lem:A19} the modules $e(M)$, $f(M)$ are tilting.
Thus, we have functorial isomorphisms
\begin{equation}\label{spe-funct}
\gathered
\Hom_{\bfg_\Bbbk}(\Bbbk e(M),\Bbbk N)=\Hom_{\bfg_\Bbbk}(e(\Bbbk M),\Bbbk N),\\
\Hom_{\bfg_\Bbbk}(\Bbbk f(M),\Bbbk N)=\Hom_{\bfg_\Bbbk}(f(\Bbbk M),\Bbbk N).
\endgathered
\end{equation}
This proves the lemma.
\end{proof}

\vspace{2mm}

We can now finish the proof of Propositions \ref{prop:TPR1}, \ref{prop:adjoints}, \ref{prop:monoidalR}.

\vspace{2mm}

\begin{proof}[Proof of Proposition \ref{prop:TPR1}]
Clear, by Lemmas \ref{lem:A11}, \ref{lem:A14}, \ref{lem:A16}. 
\end{proof}


\begin{proof}[Proof of Proposition \ref{prop:adjoints}]
Parts $(a)$, $(b)$, $(c)$ follow from  Lemmas \ref{lem:A11}, \ref{lem:A19}, \ref{lem:A20} and \ref{lem:TP2a}. 
Part $(d)$ is proved as Proposition \ref{prop:TP2bis}$(b)$.
\end{proof}


\begin{proof}[Proof of Proposition \ref{prop:monoidalR}]
The isomorphisms of functors $\bfa_R$, $\bfc_R$ are constructed in Lemmas \ref{lem:a}, \ref{lem:c}.
For part $(a)$ we must prove that $\bfa_R$, $\bfc_R$ satisfy the hexagon and the pentagon axioms. 
The tensor product $\dot\otimes_R$ commutes with a flat base change of the ring $R$ by Lemma \ref{lem:A11}.
The isomorphisms of functors $\bfa_R$, $\bfc_R$
commute also with a flat base change. Therefore, embedding $R$ in its fraction field $K$, 
we are reduced to prove that $\bfa_K$, $\bfc_K$ satisfy the hexagon and the pentagon axioms. 
This is proved in Proposition \ref{prop:TP2}.

Now, let $M\in\bfO^{\nu,\kappa,\tilt}_R$ and $N\in\bfO^{\nu,\kappa,\tilt}_R.$
By \eqref{spe-funct} and Propositions  \ref{prop:introhw}, \ref{prop:adjoints}, the specialization at $\Bbbk$ gives functorial isomorphisms
\begin{equation*}
\gathered
\Bbbk\Hom_{\bfg_R}(\bfV^*_R\dot\otimes_RM,N)=
\Hom_{\bfg_\Bbbk}(\bfV^*_\Bbbk\dot\otimes_\Bbbk (\Bbbk M), \Bbbk N),\\
\Bbbk\Hom_{\bfg_R}(M\dot\otimes_R\bfV^*_R,N)=
\Hom_{\bfg_\Bbbk}((\Bbbk M)\dot\otimes_\Bbbk\bfV^*_\Bbbk, \Bbbk N).
\endgathered
\end{equation*}
They are induced by the base-change homomorphisms
\begin{equation}\label{bc}
\gathered
\Bbbk\llangle \bfV^*_R,M,DN\rrangle_R\to\llangle \bfV^*_\Bbbk,\Bbbk M,D(\Bbbk N)\rrangle_\Bbbk,\\
\Bbbk\llangle M,\bfV^*_R,DN\rrangle_R\to\llangle \bfV^*_\Bbbk,\Bbbk M,D(\Bbbk N)\rrangle_\Bbbk.
\endgathered
\end{equation}
We must check that they intertwine the  isomorphisms
$$\bfc_{\bfV^*_R,M}:\bfV^*_R\dot\otimes_RM\to M\dot\otimes_R\bfV^*_R,\qquad
\bfc_{\bfV^*_\Bbbk,\Bbbk M}:\bfV^*_\Bbbk\dot\otimes_\Bbbk (\Bbbk M)\to
(\Bbbk M)\dot\otimes_\Bbbk\bfV^*_\Bbbk.
$$ 
To do so, recall that $\bfc_{\bfV^*_R,M}$ represents the transpose of the morphism of functors
$$P_R:
\llangle \bfV^*_R,M,DN\rrangle_R\to\llangle M,\bfV^*_R,DN\rrangle_R,\quad
x\otimes y\otimes z\mapsto\tau y\otimes\tau x\otimes \bar\tau z.$$
So the claim follows from the commutativity of the following square
$$\xymatrix{
\Bbbk\llangle \bfV^*_R,M,DN\rrangle_R\ar[d]^{\eqref{bc}}\ar[r]^{\Bbbk P_R}&
\Bbbk\llangle M,\bfV^*_R,DN\rrangle_R\ar[d]^{\eqref{bc}}
\\
\llangle \bfV^*_\Bbbk,\Bbbk M,D(\Bbbk N)\rrangle_\Bbbk\ar[r]^{P_\Bbbk}&
\llangle \Bbbk M,\bfV^*_\Bbbk,D(\Bbbk N)\rrangle_\Bbbk.
}$$
The commutation of the specialization with the associativity constraint is proved in a similar way.
\end{proof}

\clearpage

\section*{Index of notation}

\begin{itemize}\setlength{\itemsep}{1mm}
\item[\textbf{2\ \ \ \ }:] $R$, $K$, $\Bbbk$, $\frakm$.

\item[\textbf{2.1\ \ }:] $M^*$, $SM$, $S\phi$, $\frakP$, $\frakM$, $\frakP_1$, $R_\frakp$, $\frakm_\frakp$, $\Bbbk_\frakp$.

\item[\textbf{2.2\ \ }:] $A^\op$, $\scrC^\op$, $1_\scrC$, $K_0(\scrC)$, $[\scrC]$, $[M]$, $A\mmod$, $\scrC^*$, $\Irr(\scrC)$, $\scrC^\proj$, $\scrC^\inj$, $\Irr(A)$, 
$A\mproj$, $A\minj$, $S\scrC$, $SF$, $h$, $h^\ast$, $h^!$.

\item[\textbf{2.3\ \ }:] $\Delta(\scrC)$, $\leqslant$, $\Lambda$, $P(\lambda)$, $I(\lambda)$, $T(\lambda)$, $\nabla(\lambda)$, $\Delta^*(\lambda)$, $P^*(\lambda)$, $I^*(\lambda)$, $T^*(\lambda)$, $\nabla^*(\lambda)$, $L(\lambda)$, 
$\scrC^\Delta$, $\scrC^\nabla$, 
$\scrC^\tilt$, $\scrC^\diamond$, $\scrR$, 
$\Delta^\diamond(\lambda)$, $P^\diamond(\lambda)$, $T^\diamond(\lambda)$, $\scrC^\blackdiamond$, 
$\mathrm{lcd}_\scrC(M)$, $\mathrm{rcd}_\scrC(M)$.

\item[\textbf{2.4.1}:] $F:\scrC\to B\mmod$, $G$, $(B\mmod)^{F\Delta}$, $F^\Delta$, $F^*$, $F^\diamond$. 

\item[\textbf{2.4.3}:] $(KB')_{\le E}$, $S(\lambda)$, $S'(\lambda)$.

\item[\textbf{3\ \ \ \ }:] $q$, $q_R$.

\item[\textbf{3.1\ \ }:] $\scrI$, $\scrI(q)$, $Q_{R,p}$, $\scrI_p$, $I_1$.

\item[\textbf{3.2\ \ }:] $\mathfrak{sl}_{\!\scrI}$, $\Omega$,
$\al_i$, $\check\al_i$, $\Lam_i$, $Q=Q_{\!\scrI}$, $Q^+=Q^+_{\!\scrI}$, $P=P_{\!\scrI}$, $P^+=P^+_{\!\scrI}$, $X=X_{\!\scrI}$, $\varepsilon_i$, $\scrI^\al$, $\mathfrak{sl}_I$.

\item[\textbf{3.3\ \ }:] $\bbZ^\ell(n)$, $\scrC_n^\ell$, $\scrC_{n,+}^\ell$, $\scrP_n$, $|\lam|$, $l(\lam)$,  ${}^t\lambda$, $Y(\lambda)$, $\scrP^\ell_n$, $\scrP$, $\scrP^\ell$, $\scrP^\nu$, $\scrP^\nu_d$, $p(A)$, $q\text{-}\!\res^Q$, $q\text{-}\!\res^s$, $\ct^s$, $q^{s_p}_R=Q_{R,p}$, $Q_p=Q_{R,p}$, $\Gamma$, $\frakS_d$, 
$\Gamma_d$, $\scrX(\lambda)_\bbC$.

\item[\textbf{3.4.1}:] $\bfH_{R,d}$, $T_i$, $X_i$, $\bfH^Q_{R,d}$, $\bfH^+_{R,d}$, $\bfH^s_{R,d}$, 
$\Ind^{d'}_d$, $\Res^{d'}_d$, $\Ind^{d',s}_{d,+}$, $\Res^{d',s}_{d,+}$,
$M_\bfi$, $1_\bfi$, $1_\al$, $\bfH^s_{K,\al}$.

\item[\textbf{3.4.2}:] $H_{R,d}$, $H^s_{R,d}$, $t_i$, $x_i$, $H^s_{K,\alpha}$, $H^s_{I}$, $H^s_{I,d}$.

\item[\textbf{3.4.3}:] $\zeta$, $S(\lambda)^{s,q}_R$, $\unlhd$, $S(\lambda)^{s}_R$. 

\item[\textbf{3.5\ \ }:] $w_\lambda$, $x_\lambda$, $\frakS_\lambda$, $\bfS_{R,d}^s$, $W(\lambda)^{s,q}_R$, $\Xi^s_{R,d}$, $S^s_{R, d}$, $W(\lambda)_R^{s}$.

\item[\textbf{3.6\ \ }:] $(E,F,X,T)$, $\phi_{E^d}$, $\Lambda=\Lambda^s$, $\bfH^s_{\scrI,d}$, $\scrL(\Lambda)_\scrI$, $\scrL(\Lambda)_{\scrI,\Lambda-\al}$, 
$\bfL(\Lambda)$, $\scrL(\Lambda)_I$.

\item[\textbf{4\ \ \ \ }:] $\ell$, $N$, $\nu$.

\item[\textbf{4.1\ \ }:] $\kappa_R=\kappa$, $\tau_{R,p}=\tau_p$, $\tau_R$, $s_{R,p}=s_p$, $\kappa_S$, $\tau_{S,p}$, $e$.

\item[\textbf{4.2\ \ }:] $\frakg_R$, $U(\frakg_R)$, $\frakt_R$, $\frakb_R$, $\frakp_{R,\nu}$, $\frakm_{R,\nu}$, $e_{i,j}$, $e_i$, $\epsilon_i$, $\frakt^*_R$, $\Pi$, $\Pi^+$, $\Pi_\nu$, $\Pi^+_\nu$, $W$, $w\bullet\lambda$, $\rho$, $i_p$, $j_p$, $J^\nu_p$, $p_k$, $\det_p$, $\det$, $P$, $P_R$, $S^\nu$, $P^\nu_R$, $\rho_\nu$, $\tau=\tau_R$, $\varpi$, $\omega=\omega_N$, $\cas=\cas_N$.

\item[\textbf{4.3\ \ }:] $M_\lambda$, $\scrO^\nu_R$, $V(\lambda)_{R,\nu}$, $M(\lambda)_{R,\nu}$, $L(\lambda)_K$, 
$\scrO^\nu_{R,\t}$, $\Delta(\lambda)_{R,\tau}$.

\item[\textbf{4.4\ \ }:] $\Uparrow$, $A^\nu_{R,\tau}$, $A^\nu_{R,\tau}\{d\}$.

\item[\textbf{4.5\ \ }:] $V_R$, $V^*_R$, $e$, $f$, $I$, $\wt(\mu)$, $m_i(\mu)$, $\scrO^\nu_{K,\tau,\lambda}$, $V_I$, $\lambda \overset{i}\to\mu$.

\item[\textbf{4.6\ \ }:] $h$, $E$, $F$, $T_{R,d}=T^\nu_{R,\tau}\{d\}$, $\varphi^s_{R,d}$, 
$\Phi_{R,d}^s$, $A_{R,\tau}^\nu(N)=A^{\nu}_{R,\t}$,
$T_{R,d}(N)=T_{R,d}$.

\item[\textbf{4.7\ \ }:] $a_\bullet$, $a_p$, $a_\circ$, $\Pi_{\nu,u,v}$, $\Pi_{\nu,u,v}^+$, $\frakm_{R,\nu}$, $\frakm_{R,\nu,u,v}$,
$P\{a\}$, $P^\nu\{a\}$, $\det_\bullet$, $\nu_\circ$, $\nu_\bullet$, $\Pi_{\nu,u,v}$, $\Pi_{\nu,u,v}^+$, $\scrO^\gamma_{R,\tau}(\nu)$, $\scrO^\gamma_{R,\tau}(\nu)\{a\}$, $\scrO^\gamma_{R,\tau}(\nu,u,v)$, $\scrO^\gamma_{R,\tau}(\nu,u,v)\{a\}$, $A^\nu_{R,\tau}(\nu)$, $A^\nu_{R,\tau}(\nu, u,v)$.

\item[\textbf{5.1\ \ }:] $q_R=\exp(-2\pi\sqrt{-1}/\kappa_R)$, $Q_{R,p}=q_R^{s_p}=\exp(-2\pi\sqrt{-1}s_{R,p}/\kappa_R).$

\item[\textbf{5.2.1}:] $L\frakg_R$, $\bfg'_R$, $\bfone$, $\partial$, $\bfg_R$, $\bft_R$, $\bfb_R$, $\bfp_{R,\nu}$, $c$, $\bfg_{R,\kappa}$, $\bfg'_{R,\kappa}$, $\bfg_{R,\geqslant d}$, $\bfg'_{R,+}$, $\bfg_{R,+}$, $\Indc_R(M)$, $Q_{R,d}$, $M(d)$, $M(-d)$, $M(\infty)$, $M(-\infty)$, $\scrS_{R,\kappa}$, $\xi^{(r)}$, $\frakL_s$, $\afcas$.

\item[\textbf{5.2.2}:] $\widehat P_R$, $\widehat\Pi$, $\widehat\Pi^+$, $\widehat\Pi_{re}$, $\check\alpha$, $(\bullet:\bullet)$, $\delta$, $\Lambda_0$, $\tilde\rho$, $\langle\bullet :\bullet \rangle$, $\widehat W$, $s_i$, $T_x$, $w\bullet\mu$, $\widehat P$, $\widehat P^\nu$, $\widehat P^\nu_R$, $\widehat\lambda$, $z_\lambda$.

\item[\textbf{5.3.1}:] $\bfO^{\nu,\kappa}_R$, $M(\mu)_{R,\nu}$, $L(\mu)_K$, $\bfM(\lambda)_{R,\nu}$, $\bfL(\lambda)_K$, $\bfO_R$, $\bfM(\lambda)_R$, $\bfO^{+,\kappa}_R$, $\bfM(\lambda)_{R,+}$, $\bfO^{\nu,\kappa,f}_R$, $\bfO^{\nu,\kappa,\Delta}_R$, $\bfO^{\nu,\kappa}_{R,\t}$, $\bfO^{\nu,\kappa,\Delta}_{R,\t}$, $\bfO^{\nu,\kappa}[a]$, $P\{d\}$, $P^\nu\{d\}$, $\bfO_{R,\tau}^\nu(N)$, $\bfO^{\nu,\kappa}_{R,\tau}(N)[a]\{d\}$, $\bfO'$.

\item[\textbf{5.3.2}:] ${}^\sharp M$, ${}^\dag\! M$, $M^*$, $DM$, $\scrD M$, $\lub \bfO^{\nu,\kappa}_R$.

\item[\textbf{5.3.3}:] $\widehat\Pi(\widehat\lambda)$, $\widehat\Pi(\lambda,c)$, $\widehat\lambda\Uparrow\widehat\lambda'$, $\leqslant_\ell$, $\BGG$.

\item[\textbf{5.4\ \ }:] $\dot\otimes_R$, $\bfV_R$, $\bfV_R^*$, $e$, $f$, $X$, $T$, $I$,  $i\sim j$, $\scrI$, $e_i$, $f_i$, $m_i(\lambda)$, $\wt(\lambda)$, $\bfO^{\nu,\kappa}_{K,\tau,\beta}$.

\item[\textbf{5.5\ \ }:] $\pmb\Delta(\lambda)_{R,\t}$, $\bfA^{\nu,\kappa}_{R,\tau}$, $\bfL(\lambda)$, $\bfP(\lambda)_{R,\tau}$, $\bfT(\lambda)_{R,\tau}$, $\bfA^{\nu,\kappa}_{R,\tau}\{d\}$, $\bfT_{R,\tau}^{\nu,\kappa}\{d\}$, $\bfT_{R,d}$, $\bfT_{R,d}(N)$, $\psi_{R,d}^s$, $\Psi_{R,d}^s$.

\item[\textbf{5.6\ \ }:] $\bfm_{R,\nu}$, $\bfm_{R,\nu,\kappa}$, $\widehat W_\nu$, $\bfb_{R,\nu}$, $\bfO_R^\kappa(\nu)$, $\bfO^{\gamma,\kappa}_R(\nu)$, $\bfO^{\gamma,\kappa}_{R,\t}(\nu)$, $\bfO_{R,\tau}^{\gamma,\kappa}(\nu)\{a\}$, $\bfO^{+,\kappa}_R(\nu)$, $\bfO^{+,\kappa}_R(\nu)\{a\}$, $\bfA^{+,\kappa}_{R,\tau}(\nu)$.

\item[\textbf{5.7.1}:] $f_{u,v,z}(\tau_R,\kappa_R)$.

\item[\textbf{5.7.2}:] $\bfO^{\nu,\kappa}_{R,\t}\{a\}$, $\bfA^{\nu,\kappa}_{R,\t}\{a\}$, $p^o$, $\lambda^o$, $h$, $\varkappa=\varkappa_R$, $\scrO^{\nu}_{R,h}\{a\}$, $A^{\nu}_{R,h}\{a\}$, $M(\lambda)_{R,h}$, $\Delta(\lambda)_{R,h}$,
$\scrQ_R$, $T_{R,a_\bullet}(\nu_\bullet)$, $T_{R,h,d}$. 

\item[\textbf{5.7.3}:] $\bfV(\nu_p)$, $f_p$, $\bfT_{R,d}(\nu)$, $\bfH^\ell_{R,a}$, 
$f_\bfp$, $\bfT_{R,\bfp}(\nu)$, $\bfT_{R,(a)}(\nu)$, $\psi_{R_\frakp,d}^+(\nu)$, $\Psi_{R_\frakp,d}^+(\nu)$.

\item[\textbf{5.9\ \ }:] $E$, $F$.

\item[\textbf{6.1.1}:] $W$, $\frakh$, $S$, $\calA$, $\frakh_{reg}$, $c$, $H_c(W,\frakh)_R$, $\alpha_s$, $\check\alpha_s$, $R[\frakh]$, $R[\frakh^*]$, $\calO_c(W,\frakh)_R$, $\Delta(E)_R$, $L(E)$, $P(E)_R$,  $(\bullet)^\vee$, $c^\vee$. 

\item[\textbf{6.1.2}:] $\KZ_R$.

\item[\textbf{6.1.3}:] $W'$, $S'$, $\frakh^{W'}$, $\OInd_{W'}^W$, $\ORes_{W'}^W$, 
$W_H$, $\calO(W_H)_R$, $\OInd_H$.

\item[\textbf{6.1.4}:] $\mathrm{Ch}(M)$.

\item[\textbf{6.2.1}:] $\gamma_i$, $s_{ij}^{\gamma}$, $x_i$, $y_i$, $k$, $c_\gamma$, $h_R$, $h_{R, p}$, $\calO_R^{s,\kappa}\{d\}$, $\calO_R^{\kappa}(\frakS_d)$, $\Delta(\lambda)^{s,\kappa}_R$, $L(\lambda)^{s,\kappa}$, $P(\lambda)^{s,\kappa}_R$, $T(\lambda)^{s,\kappa}_R$, $I(\lambda)^{s,\kappa}_R$.

\item[\textbf{6.2.2}:] $\succ_s$, 
$\geqslant_{s,\kappa}$, $s^\star$, $\lambda^\star$.

\item[\textbf{6.2.3}:] $\KZ_{R,d}^s$.

\item[\textbf{6.2.4}:] $\scrR_\bfH$.
 
 \item[\textbf{6.3.1}:] $\lambda_+$, $\lambda_-$.

 \item[\textbf{7.1\ \ }:] $\Lambda^Q$, $\bfF(\Lambda^s)$, $|\lambda,s\rangle$,
$n_i(\lam)=n^s_i(\lam)=n^Q_i(\lam)$, $\bfwt(|\lambda,s\rangle)$.

 \item[\textbf{7.2\ \ }:] $\calG^\pm(\lambda,s)$, $\calO^{s^\star,-e}$.
 
\item[\textbf{7.3\ \ }:] $\calO^{s}_t$, $|s|$. 

\item[\textbf{7.4\ \ }:] $\Upsilon_d$, $\Sigma=\Sigma^{a,a'}$, $\widetilde\bfA^{\nu,-e}\{d\}$, $\tilde{\Upsilon}^{\nu,-e}_d$, $\tilde{\Psi}^{\nu}_d$, $\widetilde\bfA^{\nu,-e}_\mu$, $\widetilde\bfA^{\nu,-e}$.

\item[\textbf{8.1\ \ }:] $R^{A}$, $\bfg^{A}_{R}$, $\bfg^{A}_{R,\kappa}$, $\bigotimes_{R,a}$, $C$, $\gamma=\{\gamma_a\,;\,a\in A\}$, $\eta_a$, $x_a$, $C_\gamma$, $D_R=D_{R,\gamma}$, $\Gamma_{R}=\Gamma_{R,\gamma}$, ${}^a\!f$, ${}^A\!f$, $\langle N_1,\dots,N_n\rangle_R$, $\llangle M_1,\dots,M_n\rrangle_R$.

\item[\textbf{8.2\ \ }:] $(\bfO^{+,\kappa}_K,\dot\otimes_K,\bfa_K,\bfc_K)$, $(\bfO^{\nu,\kappa}_K,\dot\otimes_K,\bfa,\bfc)$, $\gamma_{-1}$, $\gamma_0$, $\gamma_1$, 

\item[\textbf{8.4\ \ }:] $v=v_R$.



\end{itemize}

\end{document}